\documentclass[11pt]{article}

\usepackage{amsmath,amscd,amsthm,amssymb}

\usepackage{amsfonts}

\usepackage{ytableau}
\usepackage{mathrsfs}
\usepackage{graphics}
\usepackage[new]{old-arrows}
\usepackage{hyperref}
\usepackage{epsfig}
\usepackage[utf8]{inputenc}
\usepackage[T1]{fontenc}
\usepackage{xfrac} 
\usepackage{stmaryrd}
\usepackage{physics}
\usepackage{tikz}
\usepackage{mathdots}
\usepackage{yhmath}
\usepackage{cancel}
\usepackage{color}
\usepackage{siunitx}
\usepackage{array}
\usepackage{authblk}
\usepackage{multirow}
\usepackage{gensymb}
\usepackage{tabularx}
\usepackage{extarrows}
\usepackage{booktabs}
\usetikzlibrary{ decorations.markings}
\usetikzlibrary{fadings}
\usetikzlibrary{patterns}
\usetikzlibrary{shadows.blur}
\usetikzlibrary{shapes}

\usepackage{graphicx}
\usepackage{xcolor}
\usepackage{float}
\usepackage[top=1.5in,bottom=1in,left=1.25in,right=1.25in]{geometry}
\usepackage{tikz-cd}
\usepackage{braket}
\DeclareMathOperator{\Map}{Map}
\DeclareMathOperator{\Hom}{Hom}
\DeclareMathOperator{\RHom}{RHom}
\DeclareMathOperator{\Isom}{Isom}
\DeclareMathOperator{\Ext}{Ext}
\DeclareMathOperator{\Gal}{Gal}
\DeclareMathOperator{\Bun}{Bun}
\DeclareMathOperator{\Pic}{Pic}
\DeclareMathOperator{\Gr}{Gr}

\DeclareMathOperator{\gl}{\mathfrak{gl}}
\DeclareMathOperator{\pgl}{\mathfrak{pgl}}
\DeclareMathOperator{\psl}{\mathfrak{psl}}
\DeclareMathOperator{\SL}{SL}
\DeclareMathOperator{\GL}{GL}
\DeclareMathOperator{\PGL}{PGL}
\DeclareMathOperator{\PSL}{PSL}
\DeclareMathOperator{\Sp}{Sp}
\DeclareMathOperator{\Aut}{Aut}
\DeclareMathOperator{\an}{an}
\DeclareMathOperator{\pt}{pt}
\DeclareMathOperator{\Char}{char}
\DeclareMathOperator{\End}{End}

\DeclareMathOperator{\Spec}{Spec}
\DeclareMathOperator{\Bl}{Bl}
\DeclareMathOperator{\Sch}{Sch}
\DeclareMathOperator{\QCoh}{QCoh}
\DeclareMathOperator{\Coh}{Coh}
\DeclareMathOperator{\MC}{MC}
\DeclareMathOperator{\Fitt}{Fitt}
\DeclareMathOperator{\leng}{length}
\DeclareMathOperator{\cone}{Cone}

\DeclareMathOperator{\Hilb}{Hilb}
\DeclareMathOperator{\Quot}{Quot}
\DeclareMathOperator{\Spf}{Spf}
\DeclareMathOperator{\Proj}{Proj}
\DeclareMathOperator{\Sing}{Sing}
\DeclareMathOperator{\Perv}{Perv}
\DeclareMathOperator{\IC}{IC}
\DeclareMathOperator{\Loc}{Loc}
\DeclareMathOperator{\Shv}{Shv}
\DeclareMathOperator{\Supp}{Supp}
\DeclareMathOperator{\pr}{pr}

\newcommand*{\invlim}{\varprojlim}                               
\newcommand*{\tensorhat}{\widehat{\otimes}}                      

\newcommand*{\m}{\mathfrak{m}}                                   
\newcommand*{\n}{\mathfrak{n}}                                   
\newcommand*{\p}{\mathfrak{p}}                                   
\newcommand*{\q}{\mathfrak{q}}                                   
\newcommand*{\X}{\mathfrak{X}}                                   
\newcommand*{\Y}{\mathfrak{Y}}                                   
\renewcommand*{\r}{\mathfrak{r}}                                 

\newcommand*{\cA}{\mathcal{A}}
\newcommand*{\cB}{\mathcal{B}}
\newcommand*{\cC}{\mathcal{C}}
\newcommand*{\cM}{\mathcal{M}}

\newcommand*{\bA}{\mathbb{A}}
\newcommand*{\bC}{\mathbb{C}}
\newcommand*{\bN}{\mathbb{N}}
\newcommand*{\bZ}{\mathbb{Z}}
\newcommand*{\bR}{\mathbb{R}}
\newcommand*{\bP}{\mathbb{P}}

\newcommand*{\vv}{\ensuremath{\mathbf{v}}}
\newcommand*{\uu}{\ensuremath{\mathbf{u}}}
\newcommand*{\ww}{\ensuremath{\mathbf{w}}}
\newcommand*{\ee}{\ensuremath{\mathbf{e}}}
\newcommand*{\pp}{\ensuremath{\mathbf{p}}}

\newcommand*{\g}{\ensuremath{\mathfrak{g}}}

\newtheorem{innercustomthm}{Theorem}
\newenvironment{customthm}[1]
  {\renewcommand\theinnercustomthm{#1}\innercustomthm}
  {\endinnercustomthm}

\theoremstyle{plain}
  \newtheorem{theorem}{Theorem}[section]
  \newtheorem{proposition}{Proposition}[section]
  \newtheorem{lemma}{Lemma}[section]
  \newtheorem{corollary}{Corollary}[section]
  \newtheorem{conjecture}{Conjecture}[section]

\theoremstyle{definition}
  \newtheorem{definition}{Definition}[section]

\theoremstyle{remark}
  \newtheorem{example}{Example}[section]
  \newtheorem{remark}{Remark}[section]

\numberwithin{equation}{section}

\title{Quantum Algebra of Chern-Simons Matrix Model and Large $N$ Limit}

\author[1, 2]{Sen Hu\thanks{shu@ustc.edu.cn}}
\author[3]{Si Li\thanks{sili@tsinghua.edu.cn}}
\author[1]{Dongheng Ye\thanks{ydrj163@mail.ustc.edu.cn}}
\author[4]{Yehao Zhou\thanks{yehao.zhou@ipmu.jp}}

\affil[1]{School of Mathematical Sciences, University of Science and Technology of China, Hefei, China}
\affil[2]{Shanghai Institute for Mathematics and Interdisciplinary Sciences, Shanghai, China}
\affil[3]{Department of Mathematical Sciences, Tsinghua University, Beijing, China}
\affil[4]{Kavli Institute for the Physics and Mathematics of the Universe (WPI), The University of Tokyo, Kashiwa, Japan}

\date{}

\begin{document}

\maketitle

\begin{abstract}

In this paper we study the algebra of quantum observables of the Chern-Simons matrix model which was originally proposed by Susskind and Polychronakos to describe electrons in fractional quantum Hall effects. We establish the commutation relations for its generators and study the large $N$ limit of its representation. We show that the large $N$ limit algebra is isomorphic to the uniform in $N$ algebra studied by Costello, which is isomorphic to the deformed double current algebra studied by Guay. Under appropriate scaling limit, we show that the large $N$ limit algebra degenerates to a Lie algebra which admits a surjective map to the affine Lie algebra of $\mathfrak{u}(p)$. We conjecture that the large $N$ limit algebra acts on the large $N$ limit Hilbert space via the aforementioned degeneration limit, and we prove this conjecture in the $p=1$ case by a detailed study of large $N$ limit Hilbert space. This leads to a complete proof of the large $N$ emergence of the $\widehat{\mathfrak{u}}(p)$ model as proposed by Dorey, Tong and Turner in the case $p=1$. This also suggests a rigorous derivation of edge excitation of a fractional quantum Hall droplet. 


\end{abstract}

{\tableofcontents}

\section{Introduction}
Susskind \cite{Susskind} proposed non-commutative Chern-Simons theory as an effective theory to describe fractional quantum Hall effect (FQHE). The theory describes FQHE as an in-compressible quantum fluid under a large magnetic field $B$. The effective fields to describe FQHE are fluctuations of density of electrons. Relabeling of the electrons is the gauge symmetry of the theory. It becomes area preserving diffeomorphism symmetry of the fluid in the macroscopic limit \cite{Tong_2015}.
    
Like any ordinary field theories one needs regularization to make sense of the field theory. Polychronakos \cite{Poly} proposed a matrix model as a regularization of non-commutative Chern-Simons theory which leads to a quantum mechanical system. This matrix model is called Chern-Simons matrix model. The ground state of the theory is related to the Laughlin wave function \cite{Poly,Hellerman-Raamsdonk,Karabali-Sakita}. In \cite{Dorey-Tong-Turner}, a generalized Chern-Simons matrix model with internal degrees of freedom transforming under $U(p)$ was considered. This model becomes the original Chern-Simons matrix model when $p=1$. For the non-Abelian case $p\geq 2$, they found the ground states for the matrix model to describe a class of non-Abelian quantum Hall states with filling fraction $\nu=\frac{p}{k+p}$, where $k$ is the level of the Chern-Simons matrix model. It is also known that this matrix model is closely related to the Calogero-Moser-Sutherland model \cite{Poly,Polychronakos_2006,Karabali_2002}. 

Recently, the large $N$ limit of this quantum matrix model was considered by Dorey, Tong and Turner \cite{ Dorey-Tong-Turner, Dorey_Tong_Turner-Matrix}. They showed that this matrix model is closely related to the $d=1+1$ WZW conformal field theory, which is supposed to be the boundary CFT of the fractional quantum Hall droplet. In particular, they defined the following operators in this matrix model (see Sections \ref{subsec: large N limit algebra} and \ref{subsec: large N limit representation} for details)\footnote{Our notation is different from \cite{Dorey_Tong_Turner-Matrix}. In this paper, we put $\sim$ on top of operators to denote rescaled operators.}:
$$
\begin{aligned}
&\tilde J^a_{b;n}:=\left(\frac{(k+p)N}{p}\right)^{-\frac{n}{2}}({{\lambda}^{\dagger}}^aZ^{n}\lambda_{b}-\frac{1}{p}\delta^a_b\sum_{c=1}^{p}{{\lambda}^{\dagger}}^cZ^{n}\lambda_{c}), \quad n\geq 0\\
&\tilde J^a_{b;n}:=\left(J^a_{b,-n}\right)^{\dagger}, \quad n<0\\
&\tilde{\alpha}_{n}:=\left(\frac{(k+p)N}{p}\right)^{-\frac{n}{2}}\mathrm{Tr}(Z^{n}), \quad n> 0\\
&\tilde{\alpha}_{-n}:=\left(\frac{(k+p)N}{p}\right)^{-\frac{n}{2}}\mathrm{Tr}({Z^{\dagger}}^n), \quad n>0
\end{aligned}
$$  
Here $Z,Z^{\dagger}$ are $N \cross N$ matrix valued operators and $\lambda_a,{\lambda^{\dagger}}^a$ are $N$-vector valued operators of the matrix model, and $k$ is the level. The $J^a_{b;n},\tilde{\alpha}_{n}$ are gauge invariant operators. They proposed the emergence of the Kac-Moody algebra commutation relations in the large $N$ limit. Explicitly, they propose the following convergence:
\begin{align}
 \label{1.1}    &[\tilde J^a_{b;n},\tilde J^c_{d;m}]\rightarrow (\delta^a_d \tilde J^c_{b;n+m}-\delta^c_b\tilde J^a_{d;n+m})+km\delta_{m+n,0}(\delta^c_b\delta^a_d-\frac{1}{p}\delta^a_b\delta^c_d)   \quad \text{when $N\to \infty$} \quad
\end{align}
In \cite{Dorey_Tong_Turner-Matrix} they provided strong evidence of the above large $N$ convergence based on several conjectured formulae, though there is still lack of a complete proof.

\bigskip

Inspired by Dorey, Tong and Turner's work, we study the algebra of quantum observables of the Chern-Simons matrix model and its representations in the large $N$ limit. In order to obtain a better understanding of the commutation relations between the operators $\tilde J,\tilde \alpha$ in the $N\to \infty$ limit, we study the full algebra of gauge invariant observables and its action on the Hilbert spaces of the Chern-Simons matrix model, both at finite $N$ and in the $N\to \infty$ limit. 

The algebra of gauge invariant observables at finite $N$, denoted by $\mathscr O^{(p)}_N$, is isomorphic to the quantization of the Nakajima quiver variety \cite{gan2006almost,losev2012isomorphisms} associated to the ADHM quiver (see the quiver diagram below), with quantization parameter $1$ and deformation parameter $k+p$ (see Section \ref{sec: Quantum Operators} for details).
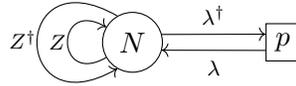
\begin{center}
\begin{tikzpicture}[x={(2cm,0cm)}, y={(0cm,2cm)}, baseline=0cm]
  \node[draw,circle,fill=white] (Gauge) at (0,0) {$N$};
  \node[draw,rectangle,fill=white] (Framing) at (1,0) {$p$};
  \node (Z) at (-.5,0) {\scriptsize $Z$};
  \node (Zdag) at (-.73,0.02) {\scriptsize $Z^{\dagger}$};
 \draw[->] (Gauge.15) -- (Framing.155) node[midway,above] {\scriptsize $\lambda^{\dagger}$};
 \draw[<-] (Gauge.345) -- (Framing.205) node[midway,below] {\scriptsize $\lambda$};

  \draw[->,looseness=5] (Gauge.210) to[out=210,in=150] (Gauge.150);
  \draw[<-,looseness=6] (Gauge.240) to[out=210,in=150] (Gauge.120);

\end{tikzpicture}
\end{center}
By invariant theory \cite{ginzburg2009lectures}, $\mathscr O^{(p)}_N$ is generated by the following set of operators
\begin{align*}
    e^a_{b;n,m}&:={{\lambda}^{\dagger}}^a\mathrm{Sym}(Z^{n}Z^{\dag m})\lambda_{b},\\
    t_{n,m}&:=\mathrm{Tr}(\mathrm{Sym}(Z^{n}Z^{\dag m})),
\end{align*}
for all $(n,m)\in \mathbb N^2$, see Section \ref{sec: Quantum Operators}. The first main result of this paper is as follows.
\begin{customthm}{A}[{Theorem \ref{Main theorem}}]
The operators $e^a_{b;n,m}$ and $t_{n,m}$ satisfy the explicit basic relations \eqref{trace of e}-\eqref{[t03,tmn]}.
\end{customthm}

We will define a large $N$ limit algebra $\mathscr O^{(p)}_{\infty}$, which captures all relations that are found in Theorem \ref{Main theorem}. We show that the basic commutation relations in Theorem \ref{Main theorem} implies all other commutation relations between $e$ and $t$ operators by induction. Using this set of relations, it is shown in \cite{Gaiotto-Rapcek-Zhou} that $\mathscr O^{(p)}_{\infty}$ is isomorphic to the deformed double current algebra, studied by Guay \cite{guay2007affine} \footnote{The dictionary between notations in this paper and \cite{Gaiotto-Rapcek-Zhou} is as follows: $e^a_{b;m,n}$ is this paper is $\mathsf T_{m,n}(E^a_b)$ there, and $t_{m,n}$ here is $\mathsf t_{m,n}$ there, and $p$ here is $K$ there, and $\mathscr O^{(p)}_{\infty}$ here is $\mathsf A^{(K)}_{\infty}$ there.}. The algebra $\mathscr O^{(p)}_{\infty}$ also show up as the algebra of operators on M2 branes in the $\Omega$-deformed M-theory \cite{Costello,gaiotto2019aspects}.

Define the following rescaled operators
\begin{align*}
&\tilde{J}^a_{b;m,n}:=\left(\frac{(k+p)N}{p}\right)^{-\frac{m+n}{2}}\left(e^a_{b;m,n}-\frac{1}{p}\delta^a_b t_{m,n}\right) \quad,\quad \tilde{t}_{m,n}:=\left(\frac{(k+p)N}{p}\right)^{-\frac{m+n+2\delta_{m,n}}{2}}t_{m,n}.
\end{align*}
Using the generators $\tilde{J}^a_{b;m,n}$ $\tilde{t}_{m,n}$ and relations of the algebra $\mathscr O^{(p)}_{\infty}$, a rescaled algebra $\widetilde{\mathscr{O}}^{(p)}_{\infty}$ can be defined, see Section \ref{subsec scaling limit}. It turns out that the degeneration limit $\widetilde{\mathscr{O}}^{(p)}_{\infty}/(t_{0,0}^{-\frac{1}{2}})$ captures the large $N$ commutation relations of the $\tilde J$ and $\tilde t$ operators. More precisely we have the second main result of this paper as follows.

\begin{customthm}{B}[{Theorem \ref{cor: conj1 implies conj2} in the case $\epsilon_1=1,\epsilon_2=k+p$}]
    The degeneration limit $\widetilde{\mathscr{O}}^{(p)}_{\infty}/(t_{0,0}^{-\frac{1}{2}})$ of the rescaled algebra is isomorphic to the universal enveloping algebra of a certain Lie algebra underlying ${\cal O}({\mathbb C}^{2})\otimes \mathfrak{gl}_{p}$with Lie bracket
\begin{align*}
    [f\otimes A, g\otimes B] =fg\otimes [A, B] +\{f, g\}\otimes  \kappa(A, B)\cdot{\bf{1}}_p,
\end{align*}
where $\{-, -\}$ is the bilinear pairing on $\mathcal O(\mathbb C^2)=\mathbb{C}[z,w]$ such that 
$$\{z^nw^m,z^rw^s\}=\delta_{n-m,s-r}(ns-mr)z^{n+r-1}w^{m+s-1}$$
and $\kappa(-,-)$ is the invariant bilinear form on $\mathfrak{gl}_p$ given by
\begin{align*}
    \kappa(A,B)=\begin{cases}
        \frac{k(k+p)}{p}\mathrm{Tr}(AB), & A\in \mathfrak{sl}_p,\\
        \frac{1}{p}\mathrm{Tr}(AB), & A\in \mathbb{C} {\bf 1}_{p}.
    \end{cases}
\end{align*}
The isomorphism is given by ${\tilde{J}}^{a}_{b;n, m}\mapsto z^{n} w^{m}\otimes J^{a}_{b}$ and $\tilde{t}_{n,m}\mapsto z^{n} w^{m}\otimes {\bf 1}_{p} $. Here $J^a_b\in \mathfrak{sl}_p$ is the trace-free part of the elementary matrix $E^a_b$, i.e. $J^a_b=E^a_b-\frac{1}{p}\delta^a_b {\bf 1}_{p}$, and ${\bf 1}_{p}$ is the identity $p\times p$ matrix. 

Moreover there is a surjective Lie algebra map from ${\cal O}({\mathbb C}^{2})\otimes \mathfrak{gl}_{p}$
to the affine Lie algebra $\widehat{\mathfrak{sl}}(p)_{ k}\oplus \widehat{\mathfrak{gl}}(1)_{\frac{p}{k+p}}$,
where the latter has generators $\tilde J^{a}_{b;n}, \alpha_{m}, (m \ne 0)$, and Lie brackets 
\begin{align*}
    [\tilde J^{a}_{b;n}, \tilde J^{c}_{d;m}] &= \delta^c_b\tilde J^{a}_{d;n+m}-\delta^a_d\tilde J^{c}_{b;n+m} +k n \delta_{n, -m}  \left(\delta^a_d\delta^c_b-\frac{1}{p}\delta^c_d\delta^a_b\right),\\
[\tilde{\alpha}_{n},\tilde{\alpha}_{m}] &= \frac{p}{k+p} n \delta_{n, -m}.
\end{align*}
The map is given by
\begin{align}\label{Sujective Lie map}
    \tilde{J}^a_{b;n,m} \mapsto \tilde J^{a}_{b;n-m},\quad {\tilde{t}}_{n,m} \mapsto \tilde{\alpha}_{n-m},\quad \tilde{t}_{n,n} \mapsto  \frac{1}{n+1} \frac{p}{k+p}.
\end{align}
\end{customthm}

In this sense, we have observed the emergence of a Kac-Moody Lie algebra in the large $N$ scaling limit. 

\bigskip

Our next part of the paper studies the large $N$ asymptotic behavior of representation of the algebra $\mathscr O^{(p)}_N$ on the large $N$ limit Hilbert space of the matrix model. We have the following.


\begin{customthm}{C}[{Theorem \ref{thm: large N limit}}]
   In the Abelian case $p=1$, the generators $\tilde{t}_{m,n}$ converge in the large $N$ limit (in an appropriate sense). The limit $\lim\limits_{N\to\infty} {\tilde{t}_{m,n}}$ are independent of $m-n$, if $m\neq n$, and let
   $${\tilde{\alpha}}_{m-n}:=\lim_{N\to\infty} {\tilde{t}_{m,n}} \quad m\neq n.$$
   Then these operators $\{{\tilde{\alpha}}_{n}, n\neq 0\}$ form a $\widehat{\mathfrak{u}(1)}_{\frac{1}{k+1}}$ current algebra, representing on a bosonic Fock space.
\end{customthm}

For the non-Abelian case $p\geq 2$, we have the following conjecture (Section \ref{sec 4.2.2})

\begin{conjecture}\label{conj3}
 There exists an appropriate energy eigenbasis $\mathcal{B}$ for the Hilbert space.  The operators $\tilde{J}^a_{b;m,n}$ and $\tilde{t}_{m,n}$ converge as operators acting on the large $N$ limit of the Hilbert space. Moreover in the basis $\mathcal{B}$, $\tilde{t}_{n,n}$ has the following asymptotic behaviour in the large $N$ limit
    \begin{align*}
    \tilde{t}_{n,n}=\frac{1}{n+1}\frac{p}{k+p}+O({\sqrt{N}}^{-2})
    \end{align*}
    and the operators $\left \{\lim\limits_{N\to \infty}\tilde{J}^a_{m,n}\right \}$ and $\left \{\lim\limits_{N\to \infty}\tilde{t}_{m,n}, m\neq n\right \}$ only depend on the difference $m-n$.
\end{conjecture}

If Conjecture \ref{conj3} holds, then together with \textbf{Theorem B}, we will have an emergent $\widehat{\mathfrak{su}}(p)_{k}\oplus \widehat{\mathfrak{u}}(1)_{\frac{p}{k+p}}$ Kac-Moody algebra as well as its representations in the large $N$ quantum matrix model. Moreover for all $p\in \mathbb N_{>0}$, we observe that the central charge of the $\mathfrak{u}(1)$ factor 
\begin{align*}
    c=\lim_{N \to \infty}(n+1)\tilde{t}_{n,n}=\frac{p}{k+p}
\end{align*}
is equal to the filling factor of the corresponding non-Abelian fractional quantum Hall systems \cite{Dorey-Tong-Turner}.

For the $p=1$ case, we justify this claim in Appendix \ref{App C}, where we derive a rescaled Wigner semicircle law (Proposition (\ref{Semi Circle law})) for the ground state wave function of the matrix model. From this semicircle law, we demonstrate that the central charge $c=\frac{1}{k+1}$ of the current algebra $\widehat{\mathfrak{u}}(1)_{\frac{1}{k+1}}$ is equal to the filling factor of the corresponding quantum Hall fluid. For the $p>1$ case, we provide evidence for this claim in Appendix \ref{App C} (see Proposition \ref{Pro C2}).

\subsection*{Organization of the paper}
\hspace{\parindent}In Section \ref{sec: Chern-Simons Matrix Model}, we describe the action of the Chern-Simons matrix model (\ref{Action}) and the gauge symmetry of the model. We point out that this is a Hamiltonian system with constraints.

In Section \ref{sec: Quantum Operators}, we discuss the quantization of the matrix model and give a precise definition of the algebra of quantum observable of the matrix model, which is denoted by $\mathscr{O}_{N}^{(p)}\left(\epsilon_{1}, \epsilon_{2}\right)$. We present a set of generators of the algebra and state our main results for the basic commutation relations for these generators (Theorem \ref{Main theorem}). Using these commutation relations, we show that all commutation relations can be obtained inductively based on the basic commutation relations. Then we consider the large $N$ limit of the commutation relations of the algebra $\mathscr{O}_{N}^{(p)}\left(\epsilon_{1}, \epsilon_{2}\right)$ (Proposition \ref{prop: filtration}). This can be viewed as a weak version for the actual large $N$ limit, i.e. at the level of algebra.

In Section \ref{sec: Hilbert Space}, we consider the representation of the algebra $\mathscr{O}_{N}^{(p)}\left(\epsilon_{1}, \epsilon_{2}\right)$ on the Hilbert space. Restricting to the case $p=1$, we present our final result of the large $N$ limit of the commutation relations in the 
algebra $\mathscr{O}_{N}^{(p=1)}\left(\epsilon_{1}, \epsilon_{2}\right)$, which is Theorem \ref{thm: large N limit} (a strong version of large $N$ limit). Our idea can be summarized as follows. Using a suitable basis (\ref{Basis non re}) of the Hilbert space of the matrix model, which was first proposed in \cite{Hellerman-Raamsdonk}, we show that the operators in $\mathscr{O}_{N}^{(p=1)}\left(\epsilon_{1}, \epsilon_{2}\right)$ becomes differential operators on a polynomials ring. We explicitly compute the order of all generators in $\mathscr{O}_{N}^{(p=1)}\left(\epsilon_{1}, \epsilon_{2}\right)$ in the large $N$ limit, and show that all correction terms have lower order. In the non-Abelian case $p\geq 2$, we state several conjectures on the large $N$ limit, which are Conjectures (\ref{conjecture 4.1}-\ref{conjecture 4.3}).

In Section \ref{sec: Derivation of the main results}, we present the proofs of all our results that are stated in previous sections.

In Appendix \ref{sec: Diagrammatic notations}, we develop a diagrammatic method for performing computations which plays an important role in deriving all these results, especially for Theorem \ref{thm: large N limit}, and also Sections \ref{subsec: Derivation of the main commutation relations}-\ref{subsec: Proof of the large N limit}.

A byproduct of our diagrammatic method is a new diagrammatic proof of the classical Murnagahan-Nakayama rule: a combinatorial rule for the change of basis between the basis of Schur polynomials and the basis of power sum polynomials. The derivation is given in Section \ref{subsec: Derivation of the main commutation relations} and Appendix \ref{App B}. We expect that these diagrammatic methods will be useful for the study of more general matrix models and quantum systems.

In Appendix \ref{App C}, we derive a rescaled Wigner semicircle law for the ground state wave function of the matrix model (Proposition \ref{Semi Circle law}). This derivation relies on the moment method in random matrix theory and Theorem \ref{thm: large N limit}. Consequently, we justify our claim that the central charge of the large $N$ limit algebra is equal to the filling factor of the corresponding quantum Hall fluid.

\section{Chern-Simons Matrix Model}\label{sec: Chern-Simons Matrix Model}
The action of the Chern-Simons matrix model was given in \cite{Poly}:
\begin{align}\label{S}
S = \int dt \frac{B}{2}\mathrm{Tr}\left \{\sum_{m,n=1}^{2}\epsilon_{mn}(\dot{X}_m+i[A_0,X_m])X_n+2\theta A_0-\omega \sum_{m=1}^{2}X_m^2 \right \}+\sum_{a=1}^{p}{\lambda^{\dagger}}^a(i\dot{\lambda}_a-A_0\lambda_a),
\end{align}
This is a gauged matrix model with a gauge group $U(N)$. Here, $X_m, m=1,2$ are two $N \times N$ Hermitian matrices, $p \in \mathbb{Z}$, $\lambda_a, a=1,\cdots,p$ are $N$-vectors, and $A_{0}$ is a $N \times N$ Hermitian matrix representing the time component of the gauge field. The fields $X_{m}, m=1,2$ transform in the adjoint representation of the gauge group $\mathrm{U}(N)$, and the $p$ vectors $\lambda_{a}$ transform in the fundamental representation of the gauge group:
\begin{equation*}
    X_{m}\longrightarrow UX_mU^{\dagger} ,\quad \lambda_a\longrightarrow U\lambda_a,
\end{equation*}
where $U \in \mathrm{U}(N)$. Physically\cite{Susskind}, the system describes $N$ electrons on the plane subjected to an external constant magnetic field $B$. The matrices $X_{m}, m=1,2$ can be interpreted as non-commutative coordinates of the $N$ electrons. The average electron density is $\rho_{0}=\frac{1}{2\pi \theta}$. The gauge group $\mathrm{U}(N)$ can be understood as an area-preserving diffeomorphism symmetry over the non-commutative plane. The Hamiltonian of the gauged matrix model is given by:
\begin{align}
    H=\omega \sum_{m=1}^{2}X^2_{m},
\end{align}
which represents a harmonic potential confining the electrons near the origin of the plane, where $\omega >0$ is the strength of the potential.
Following \cite{Dorey_Tong_Turner-Matrix}, we work with complex variables:
\begin{align}
   \label{Z} Z:=\sqrt{\frac{B}{2}}(X_1+iX_2)\\
   \label{Z+} Z^{\dagger}:=\sqrt{\frac{B}{2}}(X_1-iX_2)
\end{align}
In terms of $Z, Z^{\dagger}$ and modulo some terms that are integrals of total derivatives, the action becomes:    
\begin{equation}\label{Action}
    S=i\int dt \left [ \mathrm{Tr}\left \{Z^{\dagger}(\dot{Z}+i[A_0,Z])-i B\theta A_0+i \omega (Z^{\dagger}Z)\right \}+\sum_{a=1}^{p}{\lambda^{\dagger}}^{a} \left( \dot{\lambda}_a+iA_0\lambda_a\right) \right ]
\end{equation}
The matrix model can be viewed as a Hamiltonian system with constraints which is given by variation with respect to the auxiliary field $A_0$:
\begin{equation}\label{Classical constraints}
    \frac{\delta S}{\delta A_0}=-[Z,Z^{\dagger}]+\epsilon_2{\bf 1}-\sum_{a=1}^{p}\lambda_a{\lambda^{\dagger}}^{a}=0
\end{equation}
Here $\epsilon_2:=B\theta$. We will consider quantization of the matrix model in the next section.

\section{Quantum Operators}\label{sec: Quantum Operators}
\subsection{Operator Algebra}\label{subsec: Operator Algebra}
We now consider quantization of the Chern-Simons matrix model. The basic fields are $N\times N$ matices $Z=(Z^i_j), Z^{\dagger}=(Z^{\dagger i}_j)$, and $N$-vectors $\lambda=(\lambda^i_a), \lambda^{\dagger}=(\lambda^{\dagger a}_i)$. Upon quantization, they obey the canonical quantization conditions:
\begin{align}\label{CQ1}
    [Z^{i}_{j}, {Z^{\dagger}}^{k}_{l}] &= \epsilon_{1} \delta^{i}_{l} \delta^{k}_{j},\qquad 
    [\lambda_{b}^{j}, {\lambda^{\dagger}}^{a}_{i}] = \epsilon_{1} \delta^{j}_{i} \delta^{a}_{b},\qquad [Z^{i}_{j}\  \text{or}\ {Z^{\dagger}}^{i}_{j} ,\lambda_{a}^{k} \ \text{or}\ {\lambda^{\dagger}}^{a}_{k}]=0.
\end{align}
Here $i,j,k\in\{1,\cdots, N\}, \ a,b\in\{1,\cdots, p\}$. $\epsilon_1$ is the quantization parameter.

\begin{definition}
Let $\cA^{(p)}_N(\epsilon_{1})$ be the associative algebra generated by the symbols $\{Z^i_j, Z^{\dagger i}_j, \lambda^i_a, \lambda^{\dagger a}_i \}$ and a central element $\epsilon_1$, subject to the canonical quantization relations \eqref{CQ1}. $\{Z^{\dagger i}_j, \lambda^{\dagger a}_i\}$'s are called creation operators, and $\{Z^i_j, \lambda^i_a\}$ are called annihilation operators. 

\end{definition}

The algebra $\cA^{(p)}_N(\epsilon_{1})$ carries a natural family of $\mathrm{U}(N)$-action. Such action is generated by the quantum moment map $\mu^i_j$: 
\begin{align}\label{eqn: moment map}
    \mu^{i}_{j} :=  (:[Z, Z^{\dagger}]: + \sum_{a=1}^{p} \lambda_{a} {\lambda^{\dagger}}^{a} - \epsilon_{2} {\bf 1})^{i}_{j}=:[Z, Z^{\dagger}]^i_j: + \sum_{a=1}^{p} \lambda_{a}^i {\lambda^{\dagger}}^{a}_j -\epsilon_2 \delta^i_j.
\end{align}
Here $:[Z, Z^{\dagger}]:$ is the normal ordered commutator where all creation operators are placed to the left of all annihilation operators in the product. Explicitly, 
$$
:[Z, Z^{\dagger}]^i_j:=  \sum_k   ({Z^{\dagger}}^k_j Z^i_k-{Z^{\dagger}}^i_k Z^k_j)= \sum_k  (Z^i_k {Z^{\dagger}}^k_j -{Z^{\dagger}}^i_k Z^k_j)-\epsilon_1 N \delta^i_j
$$
i.e., 
$$
:[Z, Z^{\dagger}]:=[Z, Z^{\dagger}]- \epsilon_1 N {\bf 1}. 
$$

\begin{proposition} \label{prop-action} The following commutation relations hold
\begin{align}\label{mu1}
[\mu^{i}_{j},Z^{k}_{l}]&=\epsilon_{1} \delta^{i}_{l}Z^{k}_{j}-\epsilon_{1}\delta^{k}_{j}Z^{i}_{l},\\
\label{mu2}
[\mu^{i}_{j},{Z^{\dagger}}^{k}_{l}]&=\epsilon_{1} \delta^{i}_{l}{Z^{\dagger}}^{k}_{j}-\epsilon_{1}\delta^{k}_{j}{Z^{\dagger}}^{i}_{l},\\
\label{mu3}
[\mu^{i}_{j},\lambda_{a}^{k}]&=-\epsilon_{1} \delta^{k}_{j} \lambda_{a}^{i},\\
\label{mu4}
[\mu^{i}_{j},{\lambda^{\dagger}}^{a}_{k}]&=\epsilon_{1} \delta^{i}_{k} {\lambda^{\dagger}}^{a}_{j}.
\end{align}
\end{proposition}
\begin{proof} This follows by a direct computation. 
\end{proof}

As a consequence of this proposition, we find 
\begin{align}\label{moment map commutation rel}
[\mu^i_j, \mu^k_l]=\epsilon_1 (\delta^i_l \mu^k_j- \delta^k_j \mu^i_l)
\end{align}
hold for any $\epsilon_2$. Thus $\mu^{i}_{j}$'s generate the Lie algebra ${\mathfrak{g}}=\mathfrak{u}(N)$ \footnote{More precisely, $\mu^i_j$ satisfies commutation relation \eqref{moment map commutation rel} and the hermitian condition $(\mu^i_j)^{\dagger}=\mu^j_i$, so $\sqrt{-1}\mu^i_j$ generates the $\mathfrak{u}(N)$ Lie algebra}.

\begin{definition}
Let $\cA^{(p)}_N(\epsilon_{1},\epsilon_2)$ denote the algebra obtained by adding a central element $\epsilon_2$ to  $\cA^{(p)}_N(\epsilon_{1})$. Let $J_N$ be the \textbf{left} ideal of $\cA^{(p)}_N(\epsilon_{1},\epsilon_{2})$ generated by $\{\mu^{i}_{j}\}$'s. Define ${\cA^{(p)}_N}(\epsilon_{1},\epsilon_2)^{\mathfrak{g}}$ to be the subalgebra of $\cA^{(p)}_N(\epsilon_{1},\epsilon_2)$ generated by $\mathrm{U}(N)$ invariant elements of $\cA^{(p)}_N(\epsilon_{1},\epsilon_2)$, and let $J_N^{\mathfrak{g}}$ be the linear subspace of $J_N$ spanned by $\mathrm{U}(N)$ invariant elements of $J_N$. 

\end{definition}

Explicitly, we have 
$$
\cA^{(p)}_N(\epsilon_{1},\epsilon_2)^{\mathfrak{g}}=\{\alpha\in \cA^{(p)}_N(\epsilon_{1},\epsilon_2)| [\mu^i_j, \alpha]=0, \forall i,j\}.
$$
Since $[\mu^i_j,-]$ generate an $\mathfrak{u}(N)$-action on $\cA^{(p)}_N(\epsilon_{1},\epsilon_2)$, one can verify that $J_N^{\mathfrak{g}}$ is in fact a two sided ideal of $\cA^{(p)}_N(\epsilon_{1},\epsilon_2)^{\mathfrak{g}}$. 

\begin{definition}
We define the operator algebra $\mathscr{O}^{(p)}_N(\epsilon_{1},\epsilon_{2})$ as the quotient:
$$\mathscr{O}^{(p)}_N(\epsilon_{1},\epsilon_{2}):=\cA^{(p)}_N(\epsilon_{1},\epsilon_2)^{\mathfrak{g}}/J^{\mathfrak{g}}_N.$$
\end{definition}

The algebra $\mathscr{O}^{(p)}_N(\epsilon_{1},\epsilon_{2})$ can be viewed as the algebra of quantum observables of the matrix model. 

\begin{remark}
The finite $N$ algebra $\mathscr O_{N}^{(p)}(\epsilon_1,\epsilon_2)$ is known as the quantization of Nakajima quiver variety \cite{gan2006almost,losev2012isomorphisms} for the ADHM quiver (see the quiver diagram below)
\begin{equation}\label{ADHM quiver}
\begin{tikzpicture}[x={(2cm,0cm)}, y={(0cm,2cm)}, baseline=0cm]
  \node[draw,circle,fill=white] (Gauge) at (0,0) {$N$};
  \node[draw,rectangle,fill=white] (Framing) at (1,0) {$p$};
  \node (Z) at (-.5,0) {\scriptsize $Z$};
  \node (Zdag) at (-.73,0.02) {\scriptsize $Z^{\dagger}$};
 \draw[->] (Gauge.15) -- (Framing.155) node[midway,above] {\scriptsize $\lambda^{\dagger}$};
 \draw[<-] (Gauge.345) -- (Framing.205) node[midway,below] {\scriptsize $\lambda$};

  \draw[->,looseness=5] (Gauge.210) to[out=210,in=150] (Gauge.150);
  \draw[<-,looseness=6] (Gauge.240) to[out=210,in=150] (Gauge.120);

\end{tikzpicture}
\end{equation}
In the language of quantization of Nakajima quiver variety, $\epsilon_1$ is the quantization parameter, and $\epsilon_2$ is the deformation parameter (FI parameter). 
\end{remark}

\subsection{Commutation relations}
We consider the following $\mathrm{U}(N)$-invariant operators as elements of $\mathscr{O}^{(p)}_N(\epsilon_{1},\epsilon_{2})$:
\begin{align}\label{generator e}
e^{a}_{b; n, m} = \frac{1}{\epsilon_{1}}{\lambda^{\dagger}}^{a}\mathrm{Sym}(Z^{n}{Z^{\dagger}}^{m})\lambda_{b} \quad \quad n,m\in \mathbb{Z}_{\geq 0}\\
\label{generator t}t_{n,m} = \frac{1}{\epsilon_{1}}\mathrm{Tr}(\mathrm{Sym}(Z^{n}{Z^{\dagger}}^{m}))  \quad \quad n,m\in \mathbb{Z}_{\geq 0}
\end{align}
where $\mathrm{Sym}$ denote the averaged total symmetrization. They can be defined via the generating functions
\begin{align}
\label{3.10}&\operatorname{Tr} \frac{1}{1-(uZ+vZ^{\dagger})}=\sum_{n,m} \frac{(m+n)!}{n!m!}u^{n}v^{m}\mathrm{Tr}(\mathrm{Sym}(Z^{n}{Z^{\dagger}}^{m}))\\
\label{3.11}&{\lambda^{\dagger}}^{a}\frac{1}{1-(uZ+vZ^{\dagger})}\lambda_{b}=\sum_{n,m} \frac{(m+n)!}{n!m!}u^{n}v^{m}{\lambda^{\dagger}}^{a}\mathrm{Sym}(Z^{n}{Z^{\dagger}}^{m})\lambda_{b}
\end{align}

As we will see, these operators will have a consistent large $N$ limit. One of the main objective of this paper is to understand the commutation relations of these operators in $\mathscr{O}^{(p)}_N(\epsilon_{1},\epsilon_{2})$. The results are summarized as follows:
\begin{theorem}\label{Main theorem}
Let $\epsilon_3=\epsilon_2-p\epsilon_1$. Then we have the following relations in the algebra $\mathscr{O}^{(p)}_N(\epsilon_{1},\epsilon_{2})$ \footnote{We thank Davide Gaiotto for commuting the commutation relations \eqref{[eab10,ecdmn]} and \eqref{[t30,eabmn]} in the case $m=0$ to us, which he found by numerical method. The derivation of these commutation relations using Calogero representations is presented in \cite{Gaiotto-Rapcek-Zhou}.}
\begin{equation}\label{trace of e}
    \sum_{a=1}^{p}e^a_{a;n,m}=\epsilon_3 t_{n,m}
\end{equation}
\begin{align}\label{3.23}
    [e^{a}_{b;0,0},e^{c}_{d;n,m}]=\delta^{c}_{b} e^{a}_{d;n,m}-\delta^{a}_{d} e^{c}_{b;n,m}
\end{align}
\begin{align}
\label{3.24}
[t_{2,0},e^{a}_{b;n,m}]&=2me^{a}_{b;n+1,m-1}\\
\label{3.25}
[t_{1,1},e^{a}_{b;n,m}]&=(m-n)e^{a}_{b;n,m}\\
\label{3.26}
[t_{0,2},e^{a}_{b;n,m}]&=-2ne^{a}_{b;n-1,m+1}
\end{align}
 \begin{align}
     \nonumber&[e^a_{b;1,0},e^c_{d;m,n}]=\delta^c_b e^a_{d;m+1,n}-\delta^a_d e^c_{b;m+1,n}+\frac{n\epsilon_2}{2}\{\delta^c_b e^a_{d;m,n-1}+\delta^a_d e^c_{b;m,n-1}\}\\
     \nonumber&-\epsilon_1\frac{1}{(1+m+n)\binom{m+n}{m}}\sum_{m_1+m_2=m}\sum_{n_1+n_2=n-1}(1+m_1+n_1)\binom{m_1+n_1}{m_1}\binom{m_2+n_2}{m_2}\\
     \nonumber&\cross \{(\delta^c_b e^a_{e;m_1,n_1}e^e_{d;m_2,n_2}+\delta^a_d e^c_{e;m_2,n_2}e^e_{b;m_1,n_1})-(e^a_{d;m_1,n_1}e^c_{b;m_2,n_2}+e^a_{d;m_2,n_2}e^c_{b;m_1,n_1})\}\\
     \label{[eab10,ecdmn]}&-n\epsilon_{1}\delta^c_d e^a_{b;m,n-1}
    \end{align}
    \begin{align}
        \nonumber&[t_{3,0},e^a_{b;m,n}]=3ne^a_{b;m+2,n-1}\\
        \nonumber&+\epsilon_1\frac{1}{(m+n+1)\binom{m+n}{m}}\sum_{m_1+m_2=m}\sum_{n_1+n_2=n}(1+m_1+n_1)\binom{m_1+n_1}{m_1}(3m_2+3)\binom{m_2+1+n_2-2}{m_2+1}\\
    \nonumber&\cross \{e^a_{e;m_1,n_1}e^e_{b;m_2+1,n_2-2}-e^a_{e;m_2+1,n_2-2}e^e_{b;m_1,n_1}\}\\
    \nonumber&-\frac{3}{2}\epsilon_1\epsilon_2\frac{1}{\binom{m+n}{m}}\sum_{m_1+m_2=m}\sum_{n_1+n_2=n-3}(1+m_1+n_1)(1+m_2+n_2)\binom{m_1+n_1}{m_1}\binom{m_2+n_2}{m_2}e^a_{e;m_1,n_1}e^e_{b;m_2,n_2}\\
       \label{[t30,eabmn]} &+\epsilon_2^2\frac{n(n-1)(n-2)}{4}e^a_{b;m,n-3}
    \end{align}
 
    \begin{align}
        \nonumber&[t_{2,1},e^a_{b;m,n}]=(2n-m)e^a_{b;m+1,n}\\
        \nonumber&+\epsilon_1\frac{1}{(m+n+1)\binom{m+n}{m}}\sum_{m_1+m_2=m}\sum_{n_1+n_2=n}(1+m_1+n_1)\binom{m_1+n_1}{m_1}(-2m_1+n_2)\binom{m_2+n_2-1}{m_2}\\
    \nonumber&\cross \{e^a_{e;m_1,n_1}e^e_{b;m_2,n_2-1}-e^a_{e;m_2,n_2-1}e^e_{b;m_1,n_1}\}\\
    \nonumber&+\frac{3}{2}\epsilon_1\epsilon_2\frac{1}{\binom{m+n}{m}}\sum_{m_1+m_2=m-1}\sum_{n_1+n_2=n-2}(1+m_1+n_1)(1+m_2+n_2)\binom{m_1+n_1}{m_1}\binom{m_2+n_2}{m_2}e^a_{e;m_1,n_1}e^e_{b;m_2,n_2}\\
       \label{[t21,eabmn]} &-\epsilon_2^2\frac{mn(n-1)}{4}e^a_{b;m-1,n-2}
    \end{align}

    \begin{align}
        \nonumber&[t_{1,2},e^a_{b;m,n}]=(n-2m)e^a_{b;m,n+1}\\
        \nonumber&+\epsilon_1\frac{1}{(m+n+1)\binom{m+n}{m}}\sum_{m_1+m_2=m}\sum_{n_1+n_2=n}(1+m_1+n_1)\binom{m_1+n_1}{m_1}(-2n_2+m_1)\binom{m_2-1+n_2}{m_2-1}\\
    \nonumber&\cross \{e^a_{e;m_1,n_1}e^e_{b;m_2-1,n_2}-e^a_{e;m_2-1,n_2}e^e_{b;m_1,n_1}\}\\
    \nonumber&-\frac{3}{2}\epsilon_1\epsilon_2\frac{1}{\binom{m+n}{m}}\sum_{m_1+m_2=m-2}\sum_{n_1+n_2=n-1}(1+m_1+n_1)(1+m_2+n_2)\binom{m_1+n_1}{m_1}\binom{m_2+n_2}{m_2}e^a_{e;m_1,n_1}e^e_{b;m_2,n_2}\\
     \label{[t12,eabmn]}   &+\epsilon_2^2\frac{m(m-1)n}{4}e^a_{b;m-2,n-1}
    \end{align}

\begin{align}
        \nonumber&[t_{3,0},t_{m,n}]=3nt_{m+2,n-1}\\
        \nonumber&-\frac{3}{2}\epsilon_1\frac{1}{\binom{m+n}{m}}\sum_{m_1+m_2=m}\sum_{n_1+n_2=n-3}(1+m_1+n_1)(1+m_2+n_2)\binom{m_1+n_1}{m_1}\binom{m_2+n_2}{m_2}e^a_{e;m_1,n_1}e^e_{a;m_2,n_2}\\
        \nonumber&-\frac{3}{2}\epsilon_1^2\epsilon_3\frac{1}{\binom{m+n}{m}}\sum_{m_1+m_2=m}\sum_{n_1+n_2=n-3}(1+m_1+n_1)(1+m_2+n_2)\binom{m_1+n_1}{m_1}\binom{m_2+n_2}{m_2}t_{m_1,n_1}t_{m_2,n_2}\\
    \label{[t30,tmn]}&+(\epsilon_1^2+\epsilon_2\epsilon_3)\frac{n(n-1)(n-2)}{4}t_{m,n-3}
    \end{align}
\begin{align}
    \nonumber&[t_{2,1},t_{m,n}]=(2n-m)t_{m+1,n}\\
    \nonumber&+\frac{3}{2}\epsilon_1\frac{1}{\binom{m+n}{m}}\sum_{m_1+m_2=m-1}\sum_{n_1+n_2=n-2}(1+m_1+n_1)(1+m_2+n_2)\binom{m_1+n_1}{m_1}\binom{m_2+n_2}{m_2}e^a_{e;m_1,n_1}e^e_{a;m_2,n_2}\\
    \nonumber&+\frac{3}{2}\epsilon_1^2\epsilon_3\frac{1}{\binom{m+n}{m}}\sum_{m_1+m_2=m-1}\sum_{n_1+n_2=n-2}(1+m_1+n_1)(1+m_2+n_2)\binom{m_1+n_1}{m_1}\binom{m_2+n_2}{m_2}t_{m_1,n_1}t_{m_2,n_2}\\
    \label{[t21,tmn]}&-(\epsilon_1^2+\epsilon_2\epsilon_3)\frac{mn(n-1)}{4}t_{m-1,n-2}
\end{align}
\begin{align}
    \nonumber&[t_{1,2},t_{m,n}]=(n-2m)t_{m,n+1}\\
    \nonumber&-\frac{3}{2}\epsilon_1\frac{1}{\binom{m+n}{m}}\sum_{m_1+m_2=m-2}\sum_{n_1+n_2=n-1}(1+m_1+n_1)(1+m_2+n_2)\binom{m_1+n_1}{m_1}\binom{m_2+n_2}{m_2}e^a_{e;m_1,n_1}e^e_{a;m_2,n_2}\\
    \nonumber&-\frac{3}{2}\epsilon_1^2\epsilon_3\frac{1}{\binom{m+n}{m}}\sum_{m_1+m_2=m-2}\sum_{n_1+n_2=n-1}(1+m_1+n_1)(1+m_2+n_2)\binom{m_1+n_1}{m_1}\binom{m_2+n_2}{m_2}t_{m_1,n_1}t_{m_2,n_2}\\
    \label{[t12,tmn]}&+(\epsilon_1^2+\epsilon_2\epsilon_3)\frac{m(m-1)n}{4}t_{m-2,n-1}
\end{align}
\begin{align}
        \nonumber&[t_{0,3},t_{m,n}]=-3mt_{m-1,n+2}\\
        \nonumber&+\frac{3}{2}\epsilon_1\frac{1}{\binom{m+n}{m}}\sum_{m_1+m_2=m-3}\sum_{n_1+n_2=n}(1+m_1+n_1)(1+m_2+n_2)\binom{m_1+n_1}{m_1}\binom{m_2+n_2}{m_2}e^a_{e;m_1,n_1}e^e_{a;m_2,n_2}\\
        \nonumber&+\frac{3}{2}\epsilon_1^2\epsilon_3\frac{1}{\binom{m+n}{m}}\sum_{m_1+m_2=m-3}\sum_{n_1+n_2=n}(1+m_1+n_1)(1+m_2+n_2)\binom{m_1+n_1}{m_1}\binom{m_2+n_2}{m_2}t_{m_1,n_1}t_{m_2,n_2}\\
        \label{[t03,tmn]}&-(\epsilon_1^2+\epsilon_2\epsilon_3)\frac{m(m-1)(m-2)}{4}t_{m-3,n}
    \end{align}
    
\end{theorem}

The derivation of all the above commutation relations will be given in Section 4. Another way to prove these commutation relations is by Calogero representation based on geometric quantization \cite{Gaiotto-Rapcek-Zhou, Hu-Li-Ye-Zhou}.   A formula for the general commutators 
    $$[e^a_{b;n,m},e^c_{d;r,s}]$$ can be derived inductively from the basic commutation relations (\ref{trace of e})-(\ref{[t03,tmn]}). This will be explained in Corollary \ref{corollary induction} with the help of a filtration structure.


\subsection{Large \texorpdfstring{$N$}{N} limit algebra}\label{subsec: large N limit algebra}

We would like to understand the large $N$ limit of the algebra $\mathscr O_{N}^{(p)}(\epsilon_1,\epsilon_2)$. One might have the first instinct that there might be an inverse limit construction such as $\underset{\substack{\longleftarrow\\ N}}{\lim} \mathscr O_{N}^{(p)}(\epsilon_1,\epsilon_2)$. However, it seems that there is no natural transition map $\mathscr O_{N+1}^{(p)}(\epsilon_1,\epsilon_2)\twoheadrightarrow \mathscr O_{N}^{(p)}(\epsilon_1,\epsilon_2)$ when $\epsilon_2\neq 0$, even at the classical level ($\epsilon_1=0$). The obstruction comes from the moment map equation \eqref{eqn: moment map}, since there is no natural way to embed $N\times N$ matrices into $(N+1)\times (N+1)$ matrices while preserving the $\mathrm{U}(N)$ action, and such that the identity matrix ${\bf 1}_N$ is mapped to identity matrix ${\bf 1}_{N+1}$ \footnote{If $\epsilon_2=0$, then we can simply embed the $N\times N$ matrices into the first $N$ rows and first $N$ columns of $(N+1)\times (N+1)$ matrices, and embed $\mathrm{U}(N)$ into $\mathrm{U}(N+1)$ naturally, this induces a surjective map $\mathscr O_{N+1}^{(p)}(0,0)\twoheadrightarrow \mathscr O_{N}^{(p)}(0,0)$.}.


Nevertheless, we are rescued by the peculiar fact that all the relations \eqref{trace of e}-\eqref{[t03,tmn]} are independent of $N$. This enables us to define an algebra by hand, which is independent of $N$ and captures all the relations \eqref{trace of e}-\eqref{[t03,tmn]}.
\begin{definition}\label{def: large N limit algebra}
We define the large $N$ limit algebra $\mathscr O_{\infty}^{(p)}(\epsilon_1,\epsilon_2)$ to be the $\mathbb C[\epsilon_1,\epsilon_2]$-algebra generated by ${e}^{a}_{b; n, m}$ and $t_{n,m}$ for all $(n,m)\in \mathbb N^2$ and $1\le a,b\le p$, subject to relations \eqref{trace of e}-\eqref{[t03,tmn]}. We also introduce the generators
\begin{align*}
    J^{a}_{b;n,m}:={e}^{a}_{b; n, m} - \frac{\epsilon_{3}}{p} \delta^{a}_{b} t_{n, m}.
\end{align*}
These $J$ operators are related to the $\tilde J$ operators in the introduction by a rescaling, see \eqref{rescaling of J and t} below.
\end{definition}

By definition, the large $N$ limit algebra $\mathscr O_{\infty}^{(p)}(\epsilon_1,\epsilon_2)$ admits a surjective algebra map 
$$\mathscr O_{\infty}^{(p)}(\epsilon_1,\epsilon_2)\twoheadrightarrow \mathscr O_{N}^{(p)}(\epsilon_1,\epsilon_2)
$$ 
for all $N\in \mathbb N_{>0}$.

\begin{remark}
In the work of Costello \cite{Costello}, a uniform in $N$ algebra, denoted by $\mathscr O_{\bullet}^{(p)}(\epsilon_1,\epsilon_2)$, was defined as the $\mathbb C[\epsilon_1,\epsilon_2]$-subalgebra of the product $\prod_{N\in \mathbb N} \mathscr O_{N}^{(p)}(\epsilon_1,\epsilon_2)$ generated by elements
\begin{align*}
    \dot{e}^a_{b;n,m}:=\left(e^a_{b;n,m}\in \mathscr O_{N}^{(p)}(\epsilon_1,\epsilon_2)\right)_{N=1}^{\infty},\quad \dot{t}_{n,m}:=\left(t_{n,m}\in \mathscr O_{N}^{(p)}(\epsilon_1,\epsilon_2)\right)_{N=1}^{\infty}.
\end{align*}
It was conjectured by Costello in \cite{Costello} and proven in \cite{Gaiotto-Rapcek-Zhou} that $\mathscr O_{\bullet}^{(p)}(\epsilon_1,\epsilon_2)$ is isomorphic to the deformed double current algebra studied by Guay \cite{guay2007affine}. It follows from definition that there exists a surjective algebra homomorphism $\mathscr O_{\infty}^{(p)}(\epsilon_1,\epsilon_2)\twoheadrightarrow \mathscr O_{\bullet}^{(p)}(\epsilon_1,\epsilon_2)$, mapping ${e}^a_{b;n,m}$ to $\dot{e}^a_{b;n,m}$ and ${t}_{n,m}$ to $\dot{t}_{n,m}$. The proof of Lemma \ref{lem: PBW basis} below implies that this map is actually injective, thus our large $N$ limit algebra $\mathscr O_{\infty}^{(p)}(\epsilon_1,\epsilon_2)$ is isomorphic to the uniform in $N$ algebra $\mathscr O_{\bullet}^{(p)}(\epsilon_1,\epsilon_2)$. An explicit isomorphism between $\mathscr O_{\infty}^{(p)}(\epsilon_1,\epsilon_2)$ and Guay's deformed double current algebra is presented in \cite{Gaiotto-Rapcek-Zhou}.\\
\end{remark}

\begin{remark}
It will be useful to rewrite \eqref{[eab10,ecdmn]} in terms of $J$ generators introduced above. 
\begin{equation}\label{[Jab10,Jcdmn]}
\begin{split}
&[J^a_{b;1,0},J^c_{d;m,n}]=\delta^c_b J^a_{d;m+1,n}-\delta^a_d J^c_{b;m+1,n}+n\{\delta^a_d\delta^c_b-\frac{1}{p}\delta^c_d\delta^a_b\}\frac{\epsilon_2\epsilon_3}{p}t_{m,n-1}\\
&+\frac{n\epsilon_2}{2}\{\delta^c_b J^a_{d;m,n-1}+\delta^a_d J^c_{b;m,n-1}\}-\frac{n}{p}\{\epsilon_2\delta^c_d J^a_{b;m,n-1}+\epsilon_2\delta^a_bJ^c_{d;m,n-1}\}\\
&-\epsilon_1\frac{1}{(1+m+n)\binom{m+n}{m}}\sum_{m_1+m_2=m}\sum_{n_1+n_2=n-1}(1+m_1+n_1)\binom{m_1+n_1}{m_1}\binom{m_2+n_2}{m_2}\times\\
&\{(\delta^c_b J^a_{e;m_1,n_1}J^e_{d;m_2,n_2}+\delta^a_d J^c_{e;m_2,n_2}J^e_{b;m_1,n_1})-(J^a_{d;m_1,n_1}J^c_{b;m_2,n_2}+J^a_{d;m_2,n_2}J^c_{b;m_1,n_1})\}.
\end{split}
\end{equation}
Equality (\ref{[Jab10,Jcdmn]}) can also be derived using the generating function. This will be shown by (\ref{[Jab10,Jcd]}).
\end{remark}

Motivated by \cite{Dorey_Tong_Turner-Matrix}, we conjecture that the action of the generators $J^{a}_{b; n, m}$ and $t_{m,n}$ (in the algebra $\mathscr{O}^{(p)}_N(\epsilon_{1},\epsilon_{2})$) on the finite $N$ Hilbert space of the matrix model has the asymptotic behavior as $O(\sqrt{N}^{m+n})$ and $O(\sqrt{N}^{m+n+2\delta_{m,n}})$ respectively in the large $N$ limit (Conjecture (\ref{conjecture 4.1})). Further details regarding this point will be discussed in Section \ref{sec 4.2.2}. We will show that this conjecture is true for the Abelian case $p=1$ as a corollary of Theorem (\ref{thm: large N limit}).

Motivated by these rescaling, it turns out that in the algebra $\mathscr O^{(p)}_{\infty}(\epsilon_1,\epsilon_2)$, a certain scaling limit can be defined (Section (\ref{subsec scaling limit})). We will show that a Kac-Moody algebra emerges in the large $N$ limit as a scaling limit of a modified algebra $\widetilde{\mathscr O}_{\infty}^{(p)}(\epsilon_1,\epsilon_2)$ (Theorem (\ref{cor: conj1 implies conj2})).   

It will be useful to define the following degree on the generators: $$\mathrm{deg}(J^{a}_{b; n,m}) = n + m,\quad  \mathrm{deg}(t_{n, m}) = n + m + 2,\quad \deg \epsilon_1=\deg\epsilon_2=0.$$ It induces a grading on the free algebra $$\mathscr O^{(p)}_{\mathrm{free}}:=\mathbb C\langle\epsilon_1,\epsilon_2, J^a_{b;n,m},t_{n,m}:(n,m)\in \mathbb N^2,1\le a,b\le p\rangle$$ by letting the degree for a monomial in $\{J^a_{b,m,n}, t_{m,n}\}$
be the sum of degrees of its components. The grading on $\mathscr O^{(p)}_{\mathrm{free}}$ induces a filtration
$$F_0\mathscr O^{(p)}_{\mathrm{free}}\subset F_1\mathscr O^{(p)}_{\mathrm{free}}\subset F_2\mathscr O^{(p)}_{\mathrm{free}}\subset \cdots\subset \mathscr O^{(p)}_{\mathrm{free}},\quad F_i\mathscr O^{(p)}_{\mathrm{free}}:=\bigoplus_{j=0}^i\mathscr O^{(p)}_{\mathrm{free}}[j],$$
where $\mathscr O^{(p)}_{\mathrm{free}}[j]$ is the homogeneous degree $j$ component of $\mathscr O^{(p)}_{\mathrm{free}}$. Since $\mathscr O_{\infty}^{(p)}(\epsilon_1,\epsilon_2)$ is a quotient of $\mathscr O^{(p)}_{\mathrm{free}}$, it inherits the filtration from $F_{\bullet}\mathscr O^{(p)}_{\mathrm{free}}$ by
\begin{align*}
    F_i\mathscr{O}^{(p)}_{\infty}(\epsilon_{1},\epsilon_{2}):=\text{image of }F_i\mathscr O^{(p)}_{\mathrm{free}}.
\end{align*}
The algebra $\mathscr{O}^{(p)}_{\infty}(\epsilon_{1},\epsilon_{2})$ is not graded with respect to the above degree assignment, but it has the following crucial property.
\begin{proposition}\label{prop: filtration}
Commutators in $\mathscr{O}^{(p)}_{\infty}(\epsilon_{1},\epsilon_{2})$ can be written as
\begin{align}\label{6.1}
  [J^{a}_{b;n,m}, J^{c}_{d;r,s}]\equiv & \delta^{c}_{b}J^{a}_{d;n+r,m+s}-\delta^{a}_{d}J^{c}_{b;n+r,m+s}+ \{\delta^a_d\delta^c_b-\frac{1}{p}\delta^c_d\delta^a_b\} (ns-mr) \frac{\epsilon_{2} \epsilon_{3}}{p} t_{n + r -1, m+ s-1}\\\nonumber
  &\pmod{F_{n+m+r+s-1}\mathscr{O}^{(p)}_{\infty}(\epsilon_{1},\epsilon_{2})} ,\\
  \label{6.2}
  [t_{n, m}, J^{a}_{b;r,s}] \equiv & (ns - mr) J^{a}_{b;n+r-1,m+s-1} \pmod{F_{n+m+r+s-3}\mathscr{O}^{(p)}_{\infty}(\epsilon_{1},\epsilon_{2})} ,\\
  \label{6.3}
  [t_{n, m}, t_{r, s}] \equiv & (ns - mr) t_{n + r -1, m + s -1} \pmod{F_{n+m+r+s-1}\mathscr{O}^{(p)}_{\infty}(\epsilon_{1},\epsilon_{2})} .
\end{align}
\end{proposition}

\begin{proof}

According to relations (\ref{[Jab10,Jcdmn]}), (\ref{[t30,eabmn]})-(\ref{[t12,eabmn]}) and (\ref{[t30,tmn]})-(\ref{[t03,tmn]}), the equation \ref{6.1} is true for $(n,m) = (1,0)$ and all $(r, s)$, hence also true for $(n,m) = (0,1)$ and all $(r, s)$.  \ref{6.2} and \ref{6.3} are true for
$m+n \le 3$ and all $(r,s)$. 

Next we proceed by induction on $n+m$. According to what we have discussed, \ref{6.1} is true for
$n+m \le 1$ and all $(r,s)$. Assume that \ref{6.1} is true for
$n+m \le q$ and all $(r,s)$, then
\begin{align*}
    &[J^a_{b;q+1,0},J^c_{d;r,s}]=-\frac{1}{q}[[t_{2,1},J^a_{b;q,0}],J^c_{d;r,s}]=-\frac{1}{q}([t_{2,1},[J^a_{b;q,0},J^c_{d;r,s}]]-[J^a_{b;q,0},[t_{2,1},J^c_{d;r,s}]])\\
    &=-\frac{1}{q}[t_{2,1},\delta^c_bJ^a_{d;q+r,s}-\delta^a_dJ^c_{b;q+r,s}+\{\delta^a_d\delta^c_b-\frac{1}{p}\delta^c_d\delta^a_b\}sq\frac{\epsilon_2\epsilon_3 }{p}t_{q+r-1,s-1}+\text{terms in }F_{q+r+s-1}\mathscr{O}^{(p)}_{\infty}(\epsilon_{1},\epsilon_{2})]\\
    &~ +\frac{1}{q}[J^a_{b;q,0},(2s-r)J^c_{d;r+1,s}+\text{terms in }F_{r+s}\mathscr{O}^{(p)}_{\infty}(\epsilon_{1},\epsilon_{2})]\\
    &= \delta^c_b J^a_{d;q+1+r,s}-\delta^a_d J^c_{b;q+1+r,s}+\{\delta^a_b\delta^c_b-\frac{1}{p}\delta^c_d\delta^a_b\}(q+1)s\frac{\epsilon_2\epsilon_3 }{p}t_{q+r,s-1}+\text{terms in }F_{q+r+s}\mathscr{O}^{(p)}_{\infty}(\epsilon_{1},\epsilon_{2}),
\end{align*}
so \ref{6.1} is true for $(n,m) = (q+1,0)$ and all $(r,s)$. Using the adjoint action of $t_{0,2}$, we see that \ref{6.1} is true
for $n+m \le q + 1$ and all $(r,s)$. By induction on $s$, we see that \ref{6.1} is true for all $n, m, r, s$.

According to what we have discussed, \ref{6.2} is true for
$n+m \le 3$ and all $(r,s)$. Assume that \ref{6.2} is true for
$n+m \le q$ and all $(r,s)$, then
\begin{align*}
    &[t_{q+1,0},J^a_{b;r,s}]=-\frac{1}{q}[[t_{2,1},t_{q,0}],J^a_{b;r,s}]=-\frac{1}{q}([t_{2,1},[t_{q,0},J^a_{b;r,s}]]-[t_{q,0},[t_{2,1},J^a_{b;r,s}]])\\
    &=-\frac{1}{q}[t_{2,1},qsJ^a_{b;q+r-1,s-1}+\text{terms in }F_{q+r+s-3}\mathscr{O}^{(p)}_{\infty}(\epsilon_{1},\epsilon_{2})]+\frac{1}{q}[t_{q,0},(2s-r)J^a_{b;r+1,s}\\
    &~+\text{terms in }F_{r+s}\mathscr{O}^{(p)}_{\infty}(\epsilon_{1},\epsilon_{2})]\\
    &= (q+1)sJ^a_{b;q+r,s-1}+\text{terms in }F_{q+r+s-2}\mathscr{O}^{(p)}_{\infty}(\epsilon_{1},\epsilon_{2}),
\end{align*}
so \ref{6.2} is true for $(n,m) = (q+1,0)$ and all $(r,s)$. Using the adjoint action of $t_{0,2}$, we see that \ref{6.2} is true
for $n+m \le q + 1$ and all $(r,s)$. By induction on $q$, we see that \ref{6.2} is true for all $n, m, r, s$.

According to what we have discussed, \ref{6.3} is true for
$n+m \le 3$ and all $(r,s)$. Assume that \ref{6.3} is true for
$n+m \le q$ and all $(r,s)$, then
\begin{align*}
    &[t_{q+1,0},t_{r,s}]=-\frac{1}{q}[[t_{2,1},t_{q,0}],t_{r,s}]=-\frac{1}{q}([t_{2,1},[t_{q,0},t_{r,s}]]-[t_{q,0},[t_{2,1},t_{r,s}]])\\
    &=-\frac{1}{q}[t_{2,1},qst_{q+r-1,s-1}+\text{terms in }F_{q+r+s-1}\mathscr{O}^{(p)}_{\infty}(\epsilon_{1},\epsilon_{2})]+\frac{1}{q}[t_{q,0},(2s-r)t_{r+1,s}\\
    &~+\text{terms in }F_{r+s+2}\mathscr{O}^{(p)}_{\infty}(\epsilon_{1},\epsilon_{2})]\\
    &= (q+1)st_{q+r,s-1}+\text{terms in }F_{q+r+s}\mathscr{O}^{(p)}_{\infty}(\epsilon_{1},\epsilon_{2}),
\end{align*}
so \ref{6.3} is true for $(n,m) = (q+1,0)$ and all $(r,s)$. Using the adjoint action of $t_{0,2}$, we see that \ref{6.3} is true
for $n+m \le q + 1$ and all $(r,s)$. By induction on $q$, we see that \ref{6.3} is true for all $n, m, r, s$.

\end{proof}

\begin{corollary} \label{corollary induction}
    A formula for the general commutators 
    $$[e^a_{b;n,m},e^c_{d;r,s}], \quad [t_{n,m},e^c_{d;r,s}], \quad [t_{n,m},t_{r,s}]$$ can be derived inductively from the basic commutation relations (\ref{trace of e})-(\ref{[t03,tmn]}).
\end{corollary}
\begin{proof}
    From the proof of the above proposition. The lower degree terms in (\ref{6.1})-(\ref{6.3}) can be obtained inductively.
\end{proof}

\begin{remark}
    Although, in principle, a formula for the general commutator $[e^a_{b;n,m},e^c_{d;r,s}]$ can be derived inductively from the basic commutation relations (\ref{trace of e})-(\ref{[t03,tmn]}), we do not have a closed formula in general.
\end{remark}

\begin{remark}\label{rmk: PBW generators}
In view of Proposition \ref{prop: filtration}, we see that $\mathscr{O}^{(p)}_{\infty}(\epsilon_{1},\epsilon_{2})$ admits PBW-type generators (as $\mathbb C[\epsilon_1,\epsilon_2]$-module). More precisely, let us define a set
\begin{align}
    \mathfrak G(\mathscr{O}^{(p)}_{\infty}):=\{J^a_{b;n,m},t_{n,m}\:|(n,m)\in \mathbb N^2\}
\end{align}
and fix a total order $\preceq$ on the set $\mathfrak G(\mathscr{O}^{(p)}_{\infty})$. Then define the set of ordered monomials
\begin{align}
    \mathfrak B(\mathscr{O}^{(p)}_{\infty}):=\{1\}\cup \{\mathcal O_1\cdots\mathcal O_n\:|\: n\in \mathbb N_{>0},\mathcal O_1\preceq\cdots \preceq\mathcal O_n\in \mathfrak G(\mathscr{O}^{(p)}_{\infty})\}.
\end{align}
Then we claim that the natural $\mathbb C[\epsilon_1,\epsilon_2]$-module map from the free $\mathbb C[\epsilon_1,\epsilon_2]$-module with basis $\mathfrak B(\mathscr{O}^{(p)}_{\infty})$ to $\mathscr{O}^{(p)}_{\infty}(\epsilon_{1},\epsilon_{2})$ is surjective. This can be shown inductively for each $F_i\mathscr{O}^{(p)}_{\infty}(\epsilon_{1},\epsilon_{2})$. For $F_0\mathscr{O}^{(p)}_{\infty}(\epsilon_{1},\epsilon_{2})$, it is generated as a $\mathbb C[\epsilon_1,\epsilon_2]$-module by monomials in $J^a_{b;0,0}$, which is a quotient of universal enveloping algebra $U(\mathfrak{sl}_p)$, thus $F_0\mathscr{O}^{(p)}_{\infty}(\epsilon_{1},\epsilon_{2})$ is spanned by elements in $\mathfrak B(\mathscr{O}^{(p)}_{\infty})$ by the usual PBW theorem for the universal enveloping algebra of Lie algebra. Assume that $F_i\mathscr{O}^{(p)}_{\infty}(\epsilon_{1},\epsilon_{2})$ is generated by elements in $\mathfrak B(\mathscr{O}^{(p)}_{\infty})$, then every monomial in $F_{i+1}\mathscr{O}^{(p)}_{\infty}(\epsilon_{1},\epsilon_{2})$ can be rearranged into non-decreasing order with respect to the total order $\preceq$, at the cost of adding extra terms in $F_i\mathscr{O}^{(p)}_{\infty}(\epsilon_{1},\epsilon_{2})$ (by Proposition \ref{prop: filtration}). Thus $F_{i+1}\mathscr{O}^{(p)}_{\infty}(\epsilon_{1},\epsilon_{2})$ is also generated by elements in $\mathfrak B(\mathscr{O}^{(p)}_{\infty})$.
\end{remark}

\begin{lemma}\label{lem: PBW basis}
The $\mathbb C[\epsilon_1,\epsilon_2]$-module map $\mathbb C[\epsilon_1,\epsilon_2]\cdot\mathfrak B(\mathscr{O}^{(p)}_{\infty})\to \mathscr{O}^{(p)}_{\infty}(\epsilon_{1},\epsilon_{2})$ is an isomorphism.
\end{lemma}

\begin{proof}
According to Remark \ref{rmk: PBW generators}, it remains to show that $\mathbb C[\epsilon_1,\epsilon_2]\cdot\mathfrak B(\mathscr{O}^{(p)}_{\infty})\to \mathscr{O}^{(p)}_{\infty}(\epsilon_{1},\epsilon_{2})$ is injective. Consider the $\mathbb C[\epsilon_1,\epsilon_2]$-algebra map $\mathscr O_{\infty}^{(p)}(\epsilon_1,\epsilon_2)\to \prod_{N\in\mathbb N_{>0}} \mathscr O_{N}^{(p)}(\epsilon_1,\epsilon_2)$, then it is enough to show that the composition 
\begin{align*}
    \mathrm{pr}:\mathbb C[\epsilon_1,\epsilon_2]\cdot\mathfrak B(\mathscr{O}^{(p)}_{\infty})\to \mathscr O_{\infty}^{(p)}(\epsilon_1,\epsilon_2)\to \prod_{N\in\mathbb N_{>0}} \mathscr O_{N}^{(p)}(\epsilon_1,\epsilon_2)
\end{align*}
is injective. We claim that modulo $\epsilon_1$, $p$ is injective, i.e.
\begin{align*}
    \overline{\mathrm{pr}}:\mathbb C[\epsilon_2]\cdot\mathfrak B(\mathscr{O}^{(p)}_{\infty})\to \mathscr O_{\infty}^{(p)}(0,\epsilon_2)\to \prod_{N\in\mathbb N_{>0}} \mathscr O_{N}^{(p)}(0,\epsilon_2)
\end{align*}
is injective. This claim follows from \cite[Proposition 15.0.2]{Costello}. In fact, let $\overline{\mathrm{pr}}_N$ be the restriction of $\overline{\mathrm{pr}}$ to the $N$-th component in the infinite product, i.e. $\mathscr O_{N}^{(p)}(0,\epsilon_2)$, then $\overline{\mathrm{pr}}_N$ maps our generator $e^a_{b;n,m}$ to $\mathrm{Tr}_{\mathbb C^N}(IE_{ba}JB_1^nB_2^m)$ in the notation of \emph{loc. cit.} \footnote{The notations for the gauge and framing dimensions of \cite{Costello} are related to our notations as follows. Our $p$ is $N$ in the \emph{loc. cit.}, and our $N$ is $K$ in the \emph{loc. cit.}, and we do not have odd dimension in the framing node, this means setting $M=0$ in the \emph{loc. cit.}.}. Now assume that $f\in \mathbb C[\epsilon_2]\cdot\mathfrak B(\mathscr{O}^{(p)}_{\infty})$ such that $\overline{\mathrm{pr}}(f)=0$, schematically written as 
\begin{align*}
    f=\sum_{I}a_I\mathcal O_I,\quad a_I\in \mathbb C[\epsilon_2],\quad \mathcal O_I=\mathcal O_{i_1}\cdots\mathcal O_{i_r},\quad \mathcal O_{i_1}\preceq\cdots\preceq\mathcal O_{i_r},
\end{align*}
where we sum over finitely many index sets $I=(i_1,\cdots,i_r)$. Let $N$ be $\max_{I}|I|$, in other word, $N$ is the degree of $f$ as a polynomial of element in $\mathfrak{G}(\mathscr{O}^{(p)}_{\infty})$. Since $\overline{\mathrm{pr}}_N(f)=0$, Proposition 15.0.2 of \emph{loc. cit.} asserts that all coefficients $a_I$ are zero, thus $f=0$.

Next, we leverage the injectivity result from $\overline{\mathrm{pr}}$ to $\mathrm{pr}$. Suppose that $g\in \mathbb C[\epsilon_1,\epsilon_2]\cdot\mathfrak B(\mathscr{O}^{(p)}_{\infty})$ and such that $\mathrm{pr}(g)=0$, then the injectivity of $\overline{\mathrm{pr}}$ implies that $g$ is divisible by $\epsilon_1$, i.e. $\exists g_1\in \mathbb C[\epsilon_1,\epsilon_2]\cdot\mathfrak B(\mathscr{O}^{(p)}_{\infty})$ such that $g=\epsilon_1g_1$. We claim that $\mathrm{pr}(g_1)$ is also zero. In fact, let $\mathrm{pr}_N$ be the projection of $\mathrm{pr}$ to the $N$-th component $\mathscr O_{N}^{(p)}(\epsilon_1,\epsilon_2)$, since $\epsilon_1\mathrm{pr}_N(g_1)=0$, and $\mathscr O_{N}^{(p)}(\epsilon_1,\epsilon_2)$ is flat over $\mathbb C[\epsilon_1]$ according to \cite[Lemma 15.0.1]{Costello}, we conclude that $\mathrm{pr}_N(g_1)=0$. This holds for arbitrary $N\in \mathbb N_{>0}$, thus $\mathrm{pr}(g_1)=0$. We can keep applying the above argument, and we see that 
\begin{align*}
    \ker(\mathrm{pr})\subset \bigcap_{k}\epsilon_1^k \mathbb C[\epsilon_1,\epsilon_2]\cdot\mathfrak B(\mathscr{O}^{(p)}_{\infty})=0,
\end{align*}
in other words $\mathrm{pr}$ is injective. This concludes the proof.
\end{proof}

\subsection{Scaling limit}\label{subsec scaling limit}
In this subsection, we would like to understand the commutation relations in a certain scaling limit. Namely, later we will conjecture that the rescaled operators \eqref{eqn: rescaled operators} converge in the $N\to \infty$ limit as operators acting on the Hilbert space of the Chern-Simons matrix model (Conjecture \ref{conjecture 4.1}). In this paper, we only prove Conjecture \ref{conjecture 4.1} in the Abelian case $p=1$. Nevertheless, we can still understand the algebra $\mathscr O^{(p)}_{\infty}(\epsilon_1,\epsilon_2)$ in the scaling limit \eqref{eqn: rescaled operators}, which is independent from the Hilbert space of the matrix model.

The first thing we notice is that $t_{0,0}=N$ in the scaling limit \eqref{eqn: rescaled operators}, so we may replace $N$ by $t_{0,0}$. Since $t_{0, 0}$ is central in $\mathscr{O}^{(p)}_{\infty}(\epsilon_{1},\epsilon_{2})$, we may regard $\mathscr{O}^{(p)}_{\infty}(\epsilon_{1},\epsilon_{2})$ as a $\mathbb C[\epsilon_{1}, \epsilon_{2}, t_{0, 0}]$-algebra. We add the inverse square roots $t^{-\frac{1}{2}}_{0,0}$ and $\epsilon_2^{-\frac{1}{2}}$ to the algebra $\mathscr{O}^{(p)}_{\infty}(\epsilon_{1},\epsilon_{2})$, and consider the $\mathbb C[\epsilon_{1}, \epsilon_{2}^{\pm \frac{1}{2}}, t^{- \frac{1}{2}}_{0,0}]$-subalgebra $$\widetilde{\mathscr{O}}^{(p)}_{\infty}(\epsilon_{1},\epsilon_{2}) \subset \mathscr{O}^{(p)}_{\infty}(\epsilon_{1},\epsilon_{2})[t^{-\frac{1}{2}}_{0,0},\epsilon_2^{-\frac{1}{2}}]$$ generated by
\begin{align}\label{rescaling of J and t}
    \tilde{J}^{a}_{b;n, m} = \left(\frac{pt_{0, 0}}{\epsilon_2}\right)^{- \frac{n+m}{2}} J^{a}_{b;n, m},\quad \tilde{t}_{n, m} = \left(\frac{pt_{0, 0}}{\epsilon_2}\right)^{- \frac{n+m}{2}-\delta_{n,m}} t_{n, m}.
\end{align}
Working with the rescaled operators \eqref{rescaling of J and t}, taking the $N\to \infty$ limit is equivalent to modulo $t_{0,0}^{-\frac{1}{2}}$.

\begin{theorem}\label{cor: conj1 implies conj2}
$\widetilde{\mathscr{O}}^{(p)}_{\infty}(\epsilon_{1},\epsilon_{2})$ is a free $\mathbb C[\epsilon_{1},  \epsilon_{2}^{\pm \frac{1}{2}}, t^{- \frac{1}{2}}_{0,0}]$-module, and 
$\widetilde{\mathscr{O}}^{(p)}_{\infty}(\epsilon_{1},\epsilon_{2})/(t^{-\frac{1}{2}}_{0,0})$ is isomorphic to the universal enveloping algebra of the Lie algebra 
$ {\cal O}({\mathbb C}^{2})\otimes \mathfrak{gl}_{p} $ with Lie bracket
\begin{align}\label{Lie bracket}
    [f\otimes A, g\otimes B] =fg\otimes [A, B] +\{f, g\}\otimes  \kappa(A, B)\cdot{\bf{1}}_p
    ,
\end{align}
where $\{-, -\}$ is the bilinear pairing on $\mathcal O(\mathbb C^2)=\mathbb{C}[z,w]$ such that 
$$\{z^nw^m,z^rw^s\}=\delta_{n-m,s-r}(ns-mr)z^{n+r-1}w^{m+s-1}$$
and $\kappa(-,-)$ is the invariant bilinear form on $\mathfrak{gl}_p$ given by
\begin{align}
    \kappa(A,B)=\begin{cases}
        \frac{\epsilon_2\epsilon_3}{p}\mathrm{Tr}(AB), & A\in \mathfrak{sl}_p,\\
        \frac{1}{p}\mathrm{Tr}(AB), & A\in \mathbb{C} {\bf 1}_{p}.
    \end{cases}
\end{align}
The isomorphism is given by ${\tilde{J}}^{a}_{b;n, m}\mapsto z^{n} w^{m}\otimes J^{a}_{b}$ and $\tilde{t}_{n,m}\mapsto z^{n} w^{m}\otimes {\bf 1}_{p} $. Here $J^a_b\in \mathfrak{sl}_p$ is the trace-free part of the elementary matrix $E^a_b$, i.e. $J^a_b=E^a_b-\frac{1}{p}\delta^a_b {\bf 1}_{p}$, and ${\bf 1}_{p}$ is the identity $p\times p$ matrix. 

Moreover there is a surjective Lie algebra map from ${\cal O}({\mathbb C}^{2})\otimes \mathfrak{gl}_{p}$
to the affine Lie algebra $\widehat{\mathfrak{sl}}(p)_{ \epsilon_{3}}\oplus \widehat{\mathfrak{gl}}(1)_{\frac{p}{\epsilon_2}}$,
where the latter has generators $\tilde J^{a}_{b;n}, \alpha_{m}, (m \ne 0)$, and Lie brackets 
\begin{align}
    [\tilde J^{a}_{b;n}, \tilde J^{c}_{d;m}] &= \delta^c_b\tilde J^{a}_{d;n+m}-\delta^a_d\tilde J^{c}_{b;n+m} + \epsilon_{3} n \delta_{n, -m}  \left(\delta^a_d\delta^c_b-\frac{1}{p}\delta^c_d\delta^a_b\right),\\
    [\alpha_{n}, \alpha_{m}] &= \frac{p}{\epsilon_2} n \delta_{n, -m} .
\end{align}
The map is given by
\begin{align}
    \tilde{J}^a_{b;n,m} \mapsto \tilde J^{a}_{b;n-m},\quad {\tilde{t}}_{n,m} \mapsto \alpha_{n-m},\quad \tilde{t}_{n,n} \mapsto  \frac{1}{n+1} \frac{p}{\epsilon_2}.
\end{align}
\end{theorem}

\begin{proof}
Proposition \ref{prop: filtration} implies that the commutators are schematically of the form
\begin{equation}\label{scaled commutators}
\begin{split}
[A_{n,m},B_{p,q}]=&[A,B]_{n+p,m+q}+\delta_{n+p,m+q}\kappa(A,B)(nq-mp)\tilde t_{n+p-1,m+q-1}\\
&+t_{0,0}^{-\frac{1}{2}}\cdot (\text{Polynomial in }t_{0,0}^{-\frac{1}{2}},\tilde J,\tilde t).
\end{split}
\end{equation}
where $A,B$ are $\tilde J$ or $\tilde t$ and we regard $\tilde t$ as the diagonal $\mathfrak{gl}_1$ part of $\mathfrak{gl}_p$. From \eqref{scaled commutators} we conclude that
\begin{enumerate}
    \item $\widetilde{\mathscr{O}}^{(p)}_{\infty}(\epsilon_{1},\epsilon_{2})$ is spanned as an $\mathbb C[\epsilon_{1}, \epsilon_{2}^{\pm \frac{1}{2}}, t^{- \frac{1}{2}}_{0,0}]$-module by monomials in $\tilde J$ and $\tilde t$.
    \item $\widetilde{\mathscr{O}}^{(p)}_{\infty}(\epsilon_{1},\epsilon_{2})/(t_{0,0}^{-\frac{1}{2}})$ is a quotient of the universal enveloping algebra of $\mathcal O(\mathbb C^2)\otimes \mathfrak{gl}_p$ with the Lie brackets \eqref{Lie bracket}.
\end{enumerate}
To show the freeness of $\widetilde{\mathscr{O}}^{(p)}_{\infty}(\epsilon_{1},\epsilon_{2})$ over the base ring $\mathbb C[\epsilon_{1}, \epsilon_{2}^{\pm \frac{1}{2}}, t^{- \frac{1}{2}}_{0,0}]$, let us consider the set of generators
\begin{align}
    \mathfrak G(\widetilde{\mathscr{O}}^{(p)}_{\infty}):=\{\tilde J^a_{b;n,m},\tilde t_{r,s}\:|\: (a,b)\in \{1,\cdots,p\}^2\setminus (p,p),(n,m)\in \mathbb N^2,(r,s)\in\mathbb N^2\setminus (0,0)\},
\end{align}
and fix a total order $\preceq$ on the set $\mathfrak G(\widetilde{\mathscr{O}}^{(p)}_{\infty})$, then define the set of ordered monomials
\begin{align}
    \mathfrak B(\widetilde{\mathscr{O}}^{(p)}_{\infty}):=\{1\}\cup \{\mathcal O_1\cdots\mathcal O_n\:|\: n\in \mathbb N_{>0},\mathcal O_1\preceq\cdots \preceq\mathcal O_n\in \mathfrak G(\widetilde{\mathscr{O}}^{(p)}_{\infty})\}.
\end{align}
The natural $\mathbb C[\epsilon_{1}, \epsilon_{2}^{\pm \frac{1}{2}}, t^{- \frac{1}{2}}_{0,0}]$-module map from the free $\mathbb C[\epsilon_1,\epsilon_2^{\pm \frac{1}{2}},t^{- \frac{1}{2}}_{0,0}]$-module with basis $\mathfrak B(\widetilde{\mathscr{O}}^{(p)}_{\infty})$ to $\widetilde{\mathscr{O}}^{(p)}_{\infty}(\epsilon_{1},\epsilon_{2})$ is surjective, by the same argument as the Remark \ref{rmk: PBW generators}. It remains to show that the natural map $\mathbb C[\epsilon_1,\epsilon_2^{\pm \frac{1}{2}},t^{- \frac{1}{2}}_{0,0}]\cdot\mathfrak B(\widetilde{\mathscr{O}}^{(p)}_{\infty})\to \widetilde{\mathscr{O}}^{(p)}_{\infty}(\epsilon_{1},\epsilon_{2})$ is injective. According to Lemma \ref{lem: PBW basis}, the above map is injective after localization to $\mathbb C[\epsilon_1,\epsilon_2^{\pm \frac{1}{2}},t^{\pm\frac{1}{2}}_{0,0}]$, thus it is injective before localization, because $\mathbb C[\epsilon_1,\epsilon_2^{\pm \frac{1}{2}},t^{- \frac{1}{2}}_{0,0}]\cdot\mathfrak B(\widetilde{\mathscr{O}}^{(p)}_{\infty})$ is a free module over the ring $\mathbb C[t^{-\frac{1}{2}}_{0,0}]$. Therefore $\widetilde{\mathscr{O}}^{(p)}_{\infty}(\epsilon_{1},\epsilon_{2})$ is a free $\mathbb C[\epsilon_1,\epsilon_2^{\pm \frac{1}{2}},t^{- \frac{1}{2}}_{0,0}]$-module with basis $\mathfrak B(\widetilde{\mathscr{O}}^{(p)}_{\infty})$.

Since the image of $\mathfrak B(\widetilde{\mathscr{O}}^{(p)}_{\infty})$ in the universal enveloping algebra $U({\cal O}({\mathbb C}^{2}) \otimes \mathfrak{gl}_{p})$ forms a basis by the usual PBW theorem for the Lie algebra, we see that the surjective map $U({\cal O}({\mathbb C}^{2}) \otimes \mathfrak{gl}_{p})\twoheadrightarrow \widetilde{\mathscr{O}}^{(p)}_{\infty}(\epsilon_{1},\epsilon_{2})/(t_{0,0}^{-\frac{1}{2}})$ is an isomorphism.

Finally, the surjective map to $\widehat{\mathfrak{sl}}(p)_{ \epsilon_{3}}\oplus \widehat{\mathfrak{gl}}(1)_{\frac{p}{\epsilon_2}}$ follows by a direct computation.
\end{proof}

\section{Hilbert Space}\label{sec: Hilbert Space}

\subsection{Fock space and Physical space}

Now we consider the representation of quantum operators of Chern-Simons matrix model with matters. 

Let ${\cal F}$ denote the Fock space generated from the vacuum $\ket{0}$ by applying creation operators $Z^{\dagger}$ and $\lambda^{\dagger}$
to $\ket{0}$. We consider the representation of $\cA_N^{(p)}(\epsilon_1, \epsilon_2)$ at level $k$, which means the central element $\epsilon_2$ takes the value
$$
\epsilon_2= k+p. 
$$
The $\mathfrak{u}(N)$ generators then become 
\begin{align}\label{mu}
    \mu^{i}_{j}=:[Z, Z^{\dagger}]:^{i}_{j} + \sum_{a=1}^{p} \lambda_{a}^{i} {\lambda^{\dagger}}^{a}_{j} - (k +  p) \delta^{i}_{j}.
\end{align}
We define the physical states $\ket{\mathrm{Phys}}$ to be those elements of $\cal F$ subject
to the Gauss law constraints: 
\begin{equation}\label{QMM}
\begin{aligned}
    (:[Z, Z^{\dagger}]: + \sum_{a=1}^{p} \lambda_{a} {\lambda^{\dagger}}^{a} - (k+p) {\bf 1}) \ket{\mathrm{Phys}}=0.
\end{aligned}
\end{equation}
The condition (\ref{QMM}) is the gauge invariance condition of the system, also known as the anomaly free condition, where the gauge group is $\mathrm{U}(N)$.



The $\mathrm{U}(N)$ gauge invariant condition \eqref{QMM} can be separated into $\mathrm{U}(1)$ and $\mathrm{SU}(N)$ parts.
The $\mathrm{U}(1)$ part is by taking the trace of $\mu^{i}_{j}$, which gives
\begin{align*}
\sum_{a=1}^{p}\sum_{i=1}^{N} \lambda^{i}_{a}{\lambda^{\dagger}}^{a}_{i}\ket{\mathrm{Phys}} = (k + p)N\ket{\mathrm{Phys}}.
\end{align*}
This is equivalent to
\begin{align*}
\sum_{a=1}^{p}\sum_{i=1}^{N} {\lambda^{\dagger}}^{a}_{i}\lambda^{i}_{a} \ket{\mathrm{Phys}}= kN\ket{\mathrm{Phys}},
\end{align*}
which says that the number of $\lambda^{\dagger}$'s in a physical state is $kN$.

The second condition is the $\mathrm{SU}(N)$ invariance. In a physical state, we have a certain number of ${Z^{\dagger}}^{i}_{j}$ and ${\lambda^{\dagger}}^{i}$ acting on $\ket{0}$. Typical $\mathrm{SU}(N)$ invariant tensors consist of ${Z^{\dagger}}^{i}_{j}$'s and ${\lambda^{\dagger}}^{i}$'s of the following forms:
\begin{align*}
    C(\{a\},\{n\})=\epsilon^{i_{1}i_{2}\cdots i_{N}}({\lambda^{\dagger}}^{a_{1}}{Z^{\dagger}}^{n_{1}})_{i_{1}}({\lambda^{\dagger}}^{a_{2}}{Z^{\dagger}}^{n_{2}})_{i_{2}}\cdots ({\lambda^{\dagger}}^{a_{N}}{Z^{\dagger}}^{n_{N}})_{i_{N}},
\end{align*}
and
\begin{align*}
    \mathrm{Tr}({Z^{\dagger}}^{n}).
\end{align*}
Here $\epsilon^{i_{1}i_{2}\cdots i_{N}}$ is the Levi-Civita symbol.
The general form of the physical states are:
\begin{align}\label{3.8}
    [\mathrm{Tr}({Z^{\dagger}}^{l_{1}})]^{m_{1}}[\mathrm{Tr}({Z^{\dagger}}^{l_{2}})]^{m_{2}}\cdots [\mathrm{Tr}({Z^{\dagger}}^{l_{s}})]^{m_{s}}C(\{a\},\{n\})C(\{a'\},\{n'\})C(\{a''\},\{n''\})\cdots C(\{a^{k}\},\{n^{k}\})\ket{0}
\end{align}
where the operators $C(\{a\},\{n\})$ appear $k$ times to ensure the number of $\lambda^{\dagger}$'s is $kN$.

The Hamiltonian of the matrix model (\ref{Action}) is 
\begin{equation}\label{Hamiltonian}
    H=\omega \mathrm{Tr}(Z^{\dagger}Z)
\end{equation}
which is proportional to the number operator of $Z^{\dagger}$ i.e. $\mathrm{Tr}(Z^{\dagger}Z)$.  

\begin{remark}
The ground states of the matrix model were described by Dorey, Tong and Turner in \cite{Dorey-Tong-Turner}. We recap their result here. The ground states are those of the form \eqref{3.8} which contain minimal number of $Z^{\dagger}$. With these considerations, one can find:
\begin{enumerate}
\item For $p=1$, there is a unique ground state:
\begin{equation*}
    \ket{\mathrm{ground}}=[\epsilon^{i_{1}i_{2}\cdots i_{N}}({\lambda^{\dagger}}{Z^{\dagger}}^{0})_{i_{1}}({\lambda^{\dagger}}{Z^{\dagger}}^{1})_{i_{2}}\cdots ({\lambda^{\dagger}}{Z^{\dagger}}^{N-1})_{i_{N}}]^{k}\ket{0}.
\end{equation*}
\item For $p>1$, write $N=mp+q$ where $q< p$, and we define operators:
\begin{equation*}
    B(n)_{i_{1}i_{2}\cdots i_{p}}=\epsilon_{a_{1}a_{2}\cdots a_{p}}({\lambda^{\dagger}}^{a_{1}}{Z^{\dagger}}^{n})_{i_{1}}({\lambda^{\dagger}}^{a_{2}}{Z^{\dagger}}^{n})_{i_{2}}\cdots ({\lambda^{\dagger}}^{a_{p}}{Z^{\dagger}}^{n})_{i_{p}}.
\end{equation*}
Then the ground states are:
\begin{equation}\label{3.12}
   \begin{aligned}  \ket{\mathrm{ground};a^{r}_{s}}=\prod_{r=1}^{k}[\epsilon^{i_{1}i_{2}\cdots i_{N}}B(0)_{i_{1}\cdots i_{p}}B(1)_{i_{p+1}\cdots i_{2p}}\cdots B(m-1)_{i_{N-p-q+1}\cdots i_{N-q}}\\
     ({\lambda^{\dagger}}^{a^{r}_{1}}{Z^{\dagger}}^{m})_{i_{N-q+1}}({\lambda^{\dagger}}^{a^{r}_{2}}{Z^{\dagger}}^{m})_{i_{N-q+2}}\cdots ({\lambda^{\dagger}}^{a^{r}_{q}}{Z^{\dagger}}^{m})_{i_{N}}]\ket{0}
\end{aligned}
\end{equation}
\end{enumerate}
In the case $p>1$, \eqref{3.12} can be viewed as a $\mathrm{SU}(p)$ tensor with indices $a^{r}_{s}$ ($r=1,\cdots ,k ; s=1,\cdots ,q$) taking values in $\{1,2,\cdots,p\}$. One can read off the symmetric properties of these indices from \eqref{3.12}. The span of such tensors is an irreducible representation of $\mathrm{SU}(p)$, whose Young diagram is the following

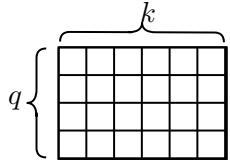
\tikzset{every picture/.style={line width=0.75pt}} 
\begin{center}
\begin{tikzpicture}[x=0.75pt,y=0.75pt,yscale=-0.7,xscale=0.7]

\draw  [draw opacity=0] (250,40) -- (371.2,40) -- (371.2,120.5) -- (250,120.5) -- cycle ; \draw   (250,40) -- (250,120.5)(270,40) -- (270,120.5)(290,40) -- (290,120.5)(310,40) -- (310,120.5)(330,40) -- (330,120.5)(350,40) -- (350,120.5)(370,40) -- (370,120.5) ; \draw   (250,40) -- (371.2,40)(250,60) -- (371.2,60)(250,80) -- (371.2,80)(250,100) -- (371.2,100)(250,120) -- (371.2,120) ; \draw    ;
\draw   (241.1,40.85) .. controls (236.43,40.85) and (234.1,43.18) .. (234.1,47.85) -- (234.1,70.1) .. controls (234.1,76.77) and (231.77,80.1) .. (227.1,80.1) .. controls (231.77,80.1) and (234.1,83.43) .. (234.1,90.1)(234.1,87.1) -- (234.1,112.35) .. controls (234.1,117.02) and (236.43,119.35) .. (241.1,119.35) ;
\draw   (369.1,36.35) .. controls (369.12,31.68) and (366.8,29.34) .. (362.13,29.32) -- (320.13,29.14) .. controls (313.46,29.11) and (310.14,26.77) .. (310.16,22.1) .. controls (310.14,26.77) and (306.8,29.09) .. (300.13,29.06)(303.13,29.07) -- (258.13,28.89) .. controls (253.46,28.87) and (251.12,31.19) .. (251.1,35.86) ;

\draw (211.6,69.4) node [anchor=north west][inner sep=0.75pt]    {$q$};
\draw (306.6,4.9) node [anchor=north west][inner sep=0.75pt]    {$k$};

\end{tikzpicture}
\end{center}
The ground states of the matrix model have interesting connections to fractional quantum Hall wave functions. See the early works of Susskind and Polychronakos \cite{Susskind, Poly}, and see also \cite{Dorey-Tong-Turner}.
\end{remark}

\subsection{Large \texorpdfstring{$N$}{N} Limit}\label{subsec: large N limit representation}
In this section, we state the main result of the large $N$ limit of the representation of $\mathscr{O}^{(p=1)}_N(\epsilon_{1}=1,\epsilon_{2}=k+1)$. In the case $p\geq 2$, we state several conjectures on the large $N$ limit.

\subsubsection{Large \texorpdfstring{$N$}{N} limit for \texorpdfstring{$p=1$}{p=1}}

\begin{proposition}[\cite{Hellerman-Raamsdonk}]
\label{Prop Basis} 

The elements:
\begin{equation}\label{Basis}
\ket{c_{1},c_{2},\cdots,c_{N}}:=\prod_{i=1}^{N}\left(\frac{1}{\sqrt{N}}\mathrm{Tr}({Z^{\dagger}}^{i})\right)^{c_{i}}\ket{\mathrm{ground}}, c_{i}\geq 0
\end{equation}
form a basis for the Hilbert space $\mathcal{H}_{phy}$ spanned by physical states.
\end{proposition}
\begin{remark}
    The basis (\ref{Basis})  is a rescaling of the basis found in \cite{Hellerman-Raamsdonk}. In this paper, we will present an alternative proof of Proposition (\ref{Prop Basis}) based on diagrammatic calculus in Section (\ref{subsec: Proof of the large N limit}) and Appendix (\ref{App B}). 
\end{remark}

For simplicity, we will study the infinite-dimensional representation $\mathcal{H}_N$ linearly spanned by the above elements. $\mathcal{H}_N$ is a dense subspace of $\mathcal{H}_{phy}$. Explicitly, let us introduce variables $p_{i}, 1\leq i \leq N$, and we identify the states $\ket{c_{1},c_{2},\cdots,c_{N}}$ with the monomial $p_{1}^{c_{1}}p_{2}^{c_{2}}\cdots p_{N}^{c_{N}}$. With these identifications, we have an isomorphism of linear spaces 
$$
\mathcal{H}_N\cong \mathbb{C}[p_{1},p_{2},\cdots,p_{N}].$$ 

Hence, $\mathcal{H}_N$ is linearly isomorphic to a polynomial ring in $N$ variables $\{p_{1},p_{2},\cdots,p_{N}\}$. We define the operators $\hat{p_{i}}$ and $\frac{\partial}{\partial p_{i}}$ as operators on the polynomial ring $\mathbb{C}[p_{1},p_{2},\cdots,p_{N}]$, representing multiplication by $p_{i}$ and taking partial derivatives with respect to $p_{i}$, respectively. Equivalently, under the isomorphism $\mathcal{H}_N\cong \mathbb{C}[p_{1},p_{2},\cdots,p_{N}]$, we have the following actions:
\begin{align}
\hat{p_{i}}\ket{c_{1},c_{2},\cdots,c_{N}}:=\ket{\cdots,c_{i}+1,\cdots}\\
\frac{\partial}{\partial p_{i}}\ket{c_{1},c_{2},\cdots,c_{N}}:=c_{i}\ket{\cdots,c_{i}-1,\cdots}
\end{align}
where $\ket{c_{1},c_{2},\cdots,c_{N}}=0$ if $c_{i}<0$ for some $i$. In the following, we will omit the "hat" in $\hat{p_{i}}$ to not distinguish the variables $p_{i}$ and multiplication by $p_{i}$, i.e $p_{i}=\hat{p_{i}}$.\\

The basis $\eqref{Basis}$ is an energy eigenbasis for the Hamiltonian (\ref{Hamiltonian}). For the purpose of this paper, we simply take $\omega=1$. The Hamiltonian (\ref{Hamiltonian}) is simply the number operator for $Z^{\dagger}$. Under the linear isomorphism $\mathcal{H}\cong \mathbb{C}[p_{1},p_{2},\cdots,p_{N}]$, the monomial $p_{i_1}p_{i_2}\cdots p_{i_{r}}$ has energy $\sum_{j=1}^{r}i_{j}$.

To state our theorem on the large $N$ limit of the representation, let us define the energy truncation 
\begin{equation}\label{Hilbert-truncate}
\text{${\mathcal H}^{(\leq E)}_N$= the span of energy eigenstates with energy $\leq E$.}
\end{equation}
Clearly, ${\mathcal H}^{(\leq E)}_N$ is a finite dimensional subspace of $\mathcal H_N$ and they will be naturally isomorphic
\begin{align}\label{sequence H_N}
    {\mathcal H}^{(\leq E)}_N={\mathcal H}^{(\leq E)}_{N+1}={\mathcal H}^{(\leq E)}_{N+2}=\cdots 
\end{align}

for $N$ large enough (in fact, for $N\geq E$). The sequence stabilized as $N \in \infty$. We define the large $N$ limit Hilbert space $\mathcal{H}^{\leq E}$, for each energy cutoff $E$ as the direct limit of the sequence (\ref{sequence H_N}):
\begin{align}\label{direct limit of H_N}
    \mathcal{H}^{\leq E}:=\lim_{\longrightarrow}\mathcal{H}^{\leq E}_{N}
\end{align}
The operator $t_{m,n}$ shifts the energy by $n-m$, hence defines 
$$
t_{m,n}: {\mathcal H}^{(\leq E)}_N\to {\mathcal H}^{(\leq E+n-m)}_N
$$
for every $N$ and $E$. 

\begin{theorem}\label{thm: large N limit}
 Under the isomorphism $\mathcal{H}_N\cong \mathbb{C}[p_{1},p_{2},\cdots,p_{N}]$, the operators $t_{m,n}$ are represented as differential operators. In the $N \rightarrow \infty$ limit, we have the following asymptotic expansions:
\begin{align}
   \label{tmn N}&t_{m,n}=\sqrt{N}^{m+n}(k+1)^{m}p_{n-m}+O(\sqrt{N}^{m+n-2})  \quad \text{if} \quad m < n  \\
   \label{tnn N}&t_{n,n}=\frac{1}{n+1}(k+1)^{n}\sqrt{N}^{2n+2}+
  O(\sqrt{N}^{2n})  \quad \text{if} \quad m=n\\
  \label{tnm N}&t_{n,m}=\sqrt{N}^{n+m}(k+1)^{n-1}(n-m)\frac{\partial}{\partial p_{n-m}}+O(\sqrt{N}^{m+n-2}) \quad \text{if} \quad m < n
\end{align}

\begin{equation}\label{t21tmn N}
    [t_{2,1},t_{m,n}]=(2n-m)t_{m+1,n}+\begin{cases}
    O(\sqrt{N}^{m+n-1}), & \text{if $m+1\neq n$}\\
    O(\sqrt{N}^{m+n+1}), & \text{if $m+1 = n$}
    \end{cases}
\end{equation}
\begin{equation}\label{t12tmn N}
    [t_{1,2},t_{m,n}]=(n-2m)t_{m,n+1}+\begin{cases}
    O(\sqrt{N}^{m+n-1}), & \text{if $m \neq n+1 $}\\
    O(\sqrt{N}^{m+n+1}), & \text{if $m = n+1 $}
    \end{cases}
\end{equation}
Here the asymptotic behavior is understood as for any fixed and sufficiently large energy truncation. To be precise, for any fixed $E>m+n$, the linear operator 
$$
t_{m,n}: {\mathcal H}^{(\leq E)}_N\to {\mathcal H}^{(\leq E+n-m)}_N
$$
will have the asymptotic behavor \eqref{tmn N} when $N$ is sufficiently large (so ${\mathcal H}^{(\leq E)}_N$ will be stable as well). The interpretation for other relations is similar. 

\begin{remark}
    In the following, whenever we talk about large $N$ limit, it is understood in the sense of Theorem (\ref{thm: large N limit}). 
\end{remark}

\end{theorem}

We will prove the above theorem in the next section using an inductive argument. The key to the proof is the explicit representation of $t_{2,1}$ and $t_{1,2}$ in the basis (\ref{Basis}), which will be computed using diagrammatic calculus in the next section.
\\

\leftline{\textbf{Rescaling and emergent current algebra}}
The matrix model describes a quantum hall droplet with radius $R\sim\sqrt{2N\theta}$ \cite{Poly}. Recall $k+p=k+1=\epsilon_2=B\theta $. Following \cite{Dorey_Tong_Turner-Matrix}, we consider the following rescaled operators:
\begin{align}\label{rescaling}
  \tilde{t}_{m,n}:=((k+1)N)^{-\frac{m+n+2\delta_{m,n}}{2}}t_{m,n}
\end{align}

The basis (\ref{Basis}) is also rescaled:
\begin{equation}\label{Basis_rescaled}
\widetilde{{\ket{c_{1},c_{2},\cdots,c_{N}}}}:=\prod_{i=1}^{N}\left(\frac{1}{\sqrt{N(k+1)}}\mathrm{Tr}({Z^{\dagger}}^{i})\right)^{c_{i}}\ket{\mathrm{ground}}, \quad c_{i}\geq 0
\end{equation}
We define the rescaled differential operators:
\begin{align*}
\tilde{p_{i}}\widetilde{\ket{c_{1},c_{2},\cdots,c_{N}}}:=\widetilde{\ket{\cdots,c_{i}+1,\cdots}}\\
\frac{\partial}{\partial \tilde{p_{i}}}\widetilde{\ket{c_{1},c_{2},\cdots,c_{N}}}:=c_{i}\widetilde{\ket{\cdots,c_{i}-1,\cdots}}
\end{align*}

\begin{remark}
    We will justify the choice of this rescaling in Appendix \ref{App C}, where we demonstrate that the radius of the quantum Hall droplet is indeed $R=\sqrt{2N(k+1)}$.
\end{remark}

With the identification $\mathcal{H}\cong \mathbb{C}[\tilde{p_{1}},\tilde{p_{2}},\cdots,\tilde{p_{N}}]$, from (\ref{tmn N})-(\ref{tnm N}), we have the following large $N$ limit (in the sense of Theorem \ref{thm: large N limit}) of the rescaled operators $\tilde{t}_{n,m}$:

\begin{align}
&\tilde{t}_{m,n}=\tilde{p}_{n-m}+O(\sqrt{N}^{-2})  \quad if \quad m < n  \\
   \label{tilde tnn}&\tilde{t}_{n,n}=\frac{1}{k+1}\frac{1}{n+1}+
  O(\sqrt{N}^{-2})\\
  &\tilde{t}_{n,m}=\frac{1}{k+1}(n-m)\frac{\partial}{\partial \tilde{p}_{n-m}}+O(\sqrt{N}^{-2}) \quad if \quad m < n
\end{align}
 In the large $N$ limit, $\tilde{t}_{n,m}$ and $\tilde{t}_{n',m'}$ have the same leading term if $n-m=n'-m'\neq 0$. We define $\tilde{\alpha}_{-n}, n>0$ to be the leading term of $\tilde{t}_{m,m+n}$ and $\tilde{\alpha}_{n}, n>0$ to be the leading term of $\tilde{t}_{m+n,m}$. 
 A $\widehat{\mathfrak{u}}(1)_{\frac{1}{k+1}}$ current algebra $span\{\tilde{\alpha}_{n},c\}$ emerges in the large $N$ limit. The generators $\{\tilde{\alpha}_{n},c\} $ satisfy:
$$[\tilde{\alpha}_{n},\tilde{\alpha}_{m}]=cn\delta_{n,-m}$$
where the central charge is
\begin{align}\label{central}
    c=\lim_{N\rightarrow \infty}(n+1)\tilde{t}_{n,n}=\frac{1}{k+1}
\end{align}

The large $N$ limit representation of the generators are 
\begin{align*}
    &\tilde{\alpha}_{n}=\frac{1}{k+1}n\frac{\partial}{\partial \tilde{p}_{n}}  \quad\quad n>0\\
     &\tilde{\alpha}_{-n}={\tilde{p}}_{n}  \quad\quad\quad\quad\quad n>0
\end{align*}  
They act on the infinite dimensional bosonic Fock space $\mathbb{C}[{\tilde{p}}_{1},{\tilde{p}}_{2},{\tilde{p}}_{3},{\tilde{p}}_{4},\cdots]$.

We will justify the physical meaning of (\ref{central}) in Appendix \ref{App C}. In Appendix \ref{App C}, we will derive a rescaled Wigner semicircle law for the ground state wave function of the matrix model (Proposition \ref{Semi Circle law}). This derivation relies on the moment method in random matrix theory and the second equality of (\ref{central}), namely $\lim_{N\rightarrow \infty}(n+1)\tilde{t}_{n,n}=\frac{1}{k+1}$. Consequently, we justify our claim that the central charge of the large $N$ limit algebra is equal to the filling factor of the corresponding quantum Hall fluid.

\begin{remark}
    The key points of the emergent current algebra is summerized below:
    \begin{itemize}
  \item Let $N$ be a large number and $E$ be an energy level such that $N \gg E$ (More precise meaning of $N \gg E$ will be discussed in certain details in Section \ref{subsec: Proof of the large N limit} and Appendix \ref{App B}). The dimension of the energy eigen-states with energy $E'<E$ is independent of $N$.  
  \item There exist a choice of (properly rescaled) basis $\mathcal{B}$ (\ref{Basis_rescaled}) for the Hilbert space, and a rescaling (the scaling depends on $N$) of the operators $t_{m,n}$  i.e. ${\tilde{t}}_{m,n}$, such that given an arbitrary energy cutoff $E> m+n$, the restriction of ${{\tilde{t}}_{m,n}}$ on ${\mathcal H}^{(\leq E)}_N$ has a well defined limit as $N\rightarrow \infty$, which is independent of $E$, as long as $E>m+n$. The $N\rightarrow \infty$ limit of the operators ${{\tilde{t}}_{m,n}}:{\mathcal H}^{(\leq E)}_N \to {\mathcal H}^{(\leq E+n-m)}_N$ form a certain (cutoff) algebra representing on the cutoff Hilbert space $\mathcal{H}_{\leq E}$ (\ref{direct limit of H_N}).
  \item Since $E$ can be taken arbitrarily large, the cutoff $E$ and $N$ can be both send to $\infty$ in such a way that $E \ll N$. The cutoff algebra representing on a cutoff Hilbert space will tend to an infinite representation of a certain algebra. We emphasize that the leading terms of the operators ${{\tilde{t}}_{m,n}}:{\mathcal H}^{(\leq E)}_N \to {\mathcal H}^{(\leq E+n-m)}_N$ are independent of $E$ and $N$ in the large $N$ limit is a key point here. We simply call these the large $N$ limit algebra and the large $N$ limit representation of the algebra. In the $p=1$ case discussed above the large $N$ limit algebra is a $\mathfrak{u}(1)$ current algebra, and the large $N$ limit representation of the algebra is a highest weight representation of the $\mathfrak{u}(1)$ current algebra.
\end{itemize}
\end{remark}

\subsubsection{Large \texorpdfstring{$N$}{N} limit for \texorpdfstring{$p\geq 2$}{p>=2}}\label{sec 4.2.2}
Let $\epsilon_1=1$. Recall $\epsilon_2=k+p$ and $\epsilon_3=\epsilon_2-p\epsilon_1=k$. We define the rescaled operators
\begin{align}\label{eqn: rescaled operators}
\tilde{J}^a_{b;m,n}:={\left(\frac{k+p}{p}N\right)}^{-\frac{m+n}{2}}J^{a}_{b;m,n} \quad \quad \tilde{t}_{m,n}:={\left(\frac{k+p}{p}N\right)}^{-\frac{m+n+2\delta_{m,n}}{2}}t_{m,n}
\end{align}

Similar to the $p=1$ case, ${\mathcal H}^{(\leq E)}_N$ is a finite dimensional\footnote{Recall that ${\mathcal H}^{(\leq E)}_N$ is generated by $U(N)$-invariant tensors consisting of $Z^{\dagger}$ and $\lambda^{\dagger}$. The number of $Z^{\dagger}$ is bounded above by $E$. Therefore, the dimension of ${\mathcal H}^{(\leq E)}_N$ is finite.} subspace of $\mathcal H_N$ and they will be naturally isomorphic
$$
{\mathcal H}^{(\leq E)}_N={\mathcal H}^{(\leq E)}_{N+p}={\mathcal H}^{(\leq E)}_{N+2p}={\mathcal H}^{(\leq E)}_{N+3p}=\cdots 
$$
for $N$ large enough. Hence, it is possible that for general $p\geq 2$, the large $N$ limit also make sense (in the sense of Theorem \ref{thm: large N limit}). We have the following conjectures.
\begin{conjecture}\label{conjecture 4.1}
    There exists an appropriate energy eigenbasis $\mathcal{B}$ for the Hilbert space.  The operators $\tilde{J}^a_{b;m,n}$ and $\tilde{t}_{m,n}$ converge as operators acting on the large $N$ limit of the Hilbert space. Moreover in the basis $\mathcal{B}$, $\tilde{t}_{n,n}$ has the following asymptotic behaviour in the large $N$ limit
    \begin{align} \label{tnn large N}
    \tilde{t}_{n,n}=\frac{1}{n+1}\frac{p}{k+p}+O({\sqrt{N}}^{-2})
    \end{align}
    and the leading terms of $\{\tilde{J}^a_{m,n}\}$ and $\{\tilde{t}_{m,n}, m\neq n\}$ only depend on the difference $m-n$.
\end{conjecture}

\begin{conjecture}\label{conjecture 4.2}
    We have the following convergence as operators acting on the Hilbert spaces in the large $N$ limit
        \begin{align*}
  [\tilde{J}^{a}_{b;n,m}, \tilde{J}^{c}_{d;r,s}]\rightarrow& \delta^{c}_{b}\tilde{J}^{a}_{d;n+r,m+s}-\delta^{a}_{d}\tilde{J}^{c}_{b;n+r,m+s}+ k\delta_{n+r,m+s}\{\delta^a_d\delta^c_b-\frac{1}{p}\delta^c_d\delta^a_b\} (ns-mr) \frac{k+p}{p} \tilde{t}_{n + r -1, m+ s-1}\\
  [\tilde{t}_{n, m}, \tilde{J}^{a}_{b;r,s}] \rightarrow& 0\\
  [\tilde{t}_{n, m}, \tilde{t}_{r, s}] \rightarrow&\delta_{n+r,m+s} (ns - mr) \tilde{t}_{n + r -1, m + s -1} 
\end{align*}
  
\end{conjecture}


\begin{proposition}
    Conjecture \ref{conjecture 4.1} implies Conjecture \ref{conjecture 4.2}. 
\end{proposition}

\begin{proof}
    This follows from Proposition \ref{prop: filtration}. After rescaling, the lower degree terms in \ref{6.1}-\ref{6.3} converge to zero. 
\end{proof}
Provided that Conjecture \ref{conjecture 4.1} is true, we define 
\begin{align*}
    &\tilde{J}^a_{b;m-n}:=\lim_{N\rightarrow \infty}\tilde{J}^a_{b;m,n}\\
    &\tilde{\alpha}_{m-n}:=\lim_{N\rightarrow \infty}\tilde{t}_{m,n} \quad\text{for}\quad m\neq n
\end{align*}
From Theorem (\ref{cor: conj1 implies conj2}), Conjecture \ref{conjecture 4.1} implies the following 
\begin{conjecture}\label{conjecture 4.3}
   We have the following convergence as operators acting on the Hilbert spaces in the large $N$ limit
    \begin{align*}
    [\tilde{J}^{a}_{b;n}, \tilde{J}^{c}_{d;m}] &\rightarrow \delta^c_b\tilde{J}^{a}_{d;n+m}-\delta^a_d\tilde{J}^{c}_{b;n+m} + k n \delta_{n, -m}  \left(\delta^a_d\delta^c_b-\frac{1}{p}\delta^c_d\delta^a_b\right),\\
    [\tilde{\alpha}_{n}, \tilde{\alpha}_{m}] &\rightarrow \frac{p}{k+p} n \delta_{n, -m}.
\end{align*}
\end{conjecture}

To summarize, Conjecture \ref{conjecture 4.1} implies Conjecture \ref{conjecture 4.2} and Conjecture \ref{conjecture 4.3}. This is essentially Theorem (\ref{cor: conj1 implies conj2}) which says that Kac-Moody algebra is a quotient of the large $N$ limit of a rescaling alegbra.  If Conjecture \ref{conjecture 4.1} is true. We have an emergent $\widehat{\mathfrak{sl}}(p)_{k}\oplus \widehat{\mathfrak{gl}}(1)_{\frac{p}{k+p}}$ Kac-Moody algebra as well as its representation in the large $N$ quantum theory. 
As a corollary of Theorem \ref{thm: large N limit}, Conjecture \ref{conjecture 4.1}-\ref{conjecture 4.3} are true in the Abelian case $p=1$. 


We expect that these are the Fourier modes of boundary excitation of the quantum Hall droplet.

\section{Derivation of the main results}\label{sec: Derivation of the main results}
In this section we will derive the commutation relations (\ref{3.24})-(\ref{3.26}), (\ref{[eab10,ecdmn]}) ,(\ref{[t30,eabmn]})-(\ref{[t12,eabmn]}), and (\ref{[t30,tmn]})-(\ref{[t03,tmn]}). Let us start with some basic relations in the algebra $\mathscr{O}^{(p)}_N(\epsilon_{1},\epsilon_{2})$.
\subsection{Basic relations for the algebra \texorpdfstring{$\mathscr{O}^{(p)}_N(\epsilon_{1},\epsilon_{2})$}{O}}

Given two matrix valued operators $A=\left(A_{j}^{i}\right), B=\left(B_{j}^{i}\right)$, we define their \textit{ordered commutator} by
$$
\widetilde{ [A, B]}_j^ i:=A_{k}^{i} B_{j}^{k}-A_{j}^{k} B_{k}^{i}
$$
For example,
$$
:[Z, Z^{\dagger}]:=-\widetilde{[Z^{\dagger}, Z]}=[Z, Z^{\dagger}]-\epsilon_1 N {\bf 1}
$$
One property of ordered commutator is that
$$
\operatorname{Tr}\widetilde{[A, B]}=0 \text {. }
$$
\begin{lemma}\label{[uZ+vZ,uZ+vZ]=0}
 Matrix elements of $uZ+vZ^{\dagger}$ commute with each other, i.e.
 $$
[(uZ +v Z^{\dagger})^i_j,(u Z +v Z^{\dagger})_{m}^{k}]=0 .
$$
\end{lemma}

\begin{proof}
Direct computation.
\end{proof}

As a consequence, we have the following 
\begin{proposition}\label{Prop-Sym-Relation}
$$\operatorname{Tr}\left(\frac{1}{1-(uZ+vZ^{\dagger})}:\left[Z, Z^{\dagger}\right]:\right)=0.$$
Equivalently in homogeneous components
$$
\mathrm{Tr}(\mathrm{Sym}(Z^{n} {Z^{\dagger}}^m):[Z, Z^{\dagger}]:)=0, \qquad \forall m,n.
$$
\end{proposition}
\begin{proof}
 Using the previous lemma
 $$
 \begin{aligned}
\widetilde{[uZ+vZ^{\dagger},\frac{1}{1-(uZ+vZ^{\dagger})}uZ]}=\frac{1}{1-(uZ+vZ^{\dagger})}\widetilde{[uZ+vZ^{\dagger},uZ]}=-uv\frac{1}{1-(uZ+vZ^{\dagger})}:[Z,Z^{\dagger}]:
\end{aligned}
$$
The proposition follows by taking the trace of both sides.
\end{proof}

\begin{corollary} As elements in $\mathscr{O}_{N}^{(p)}\left(\epsilon_{1}, \epsilon_{2}\right)$, we have

$$
\operatorname{Tr}\left(\mathrm{Sym}\left(Z^{n} {Z^{\dagger}}^m\right) \lambda_{a} {{{\lambda^{\dagger}}}}^a\right)=\epsilon_{2} \operatorname{Tr} \mathrm{Sym}\left(Z^{n} {Z^{\dagger}}^m\right).
$$
\end{corollary}
\begin{proof}
By definition of the algebra $\mathscr{O}_{N}^{(p)}\left(\epsilon_{1}, \epsilon_{2}\right)$, $\mathrm{Tr}\left(\mathrm{Sym}\left(Z^{n} {Z^{\dagger}}^m\right) \mu\right)=0$ in $\mathscr{O}_{N}^{(p)}\left(\epsilon_{1}, \epsilon_{2}\right)$, where $\mu$ is the moment map (Gauss law constraints) in \eqref{mu}. The equality then follows from Proposition \ref{Prop-Sym-Relation}.
\end{proof}

Since $$\operatorname{Tr}\left(\mathrm{Sym}\left(Z^{n} {Z^{\dagger}}^m\right) \lambda_{a} {{{\lambda^{\dagger}}}}^a\right)
={{{\lambda^{\dagger}}}}^a\mathrm{Sym}\left(Z^{n} {Z^{\dagger}}^m\right) \lambda_{a}+p \epsilon_{1} \operatorname{Tr} \mathrm{Sym}\left(Z^{n} {Z^{\dagger}}^m\right),
$$
we derive the following relation from the above corollary.
\begin{proposition}
     \begin{equation}e_{a ; n, m}^{a}=\left(\epsilon_{2}-p \epsilon_{1}\right) t_{n, m}
     \end{equation}
\end{proposition}

We conclude this subsection with some useful identities in $\mathscr{O}_{N}^{(p)}\left(\epsilon_{1}, \epsilon_{2}\right)$:
\begin{proposition}\label{prop 53}
For any two matrix-valued operators (in the adjoint representation of $\mathrm{U}(N)$) $A=\left(A_{j}^{i}\right)$, $B=\left(B_{j}^{i}\right)$ in the algebra $\mathscr{O}_{N}^{(p)}\left(\epsilon_{1}, \epsilon_{2}\right)$, we have
\begin{align}
\label{52}& \operatorname{Tr}\left(A\left[Z, Z^{\dagger}\right] B\right)=\operatorname{Tr}\left(A\left(\epsilon_{2}-\lambda_{a} {{{\lambda^{\dagger}}}}^a\right) B\right)+\epsilon_{1} \operatorname{Tr} A \operatorname{Tr} B \\
& {{{\lambda^{\dagger}}}}^{a} A\left[Z, Z^{\dagger}\right] B \lambda_{b}={{{\lambda^{\dagger}}}}^{a} A\left(\epsilon_{2}-\lambda_{e} {{{\lambda^{\dagger}}}}^{e}\right) B \lambda_{b}
\end{align}
\end{proposition}

\begin{proof}
We prove the first equation, which is equivalent to
$$
\operatorname{Tr} A(\mu+\epsilon_1 N) B=\epsilon_{1} \operatorname{Tr} A \operatorname{Tr} B.
$$
Then
\begin{align*}
\operatorname{Tr} \left(A \mu B\right)=A^{i}_{j} \mu^{j}_{k} B_{k}^{i} =A^{i}_{j}\left[\mu_{k}^{j}, B_{i}^{k}\right]=\epsilon_{1} A_{j}^{i}\left(\delta_{i}^{j} B_{k}^{k}-\delta_{k}^{k} B_{i}^{j}\right)  =\epsilon_{1} \operatorname{Tr} A \operatorname{Tr} B-\epsilon_{1} N \operatorname{Tr}\left( A B\right),
\end{align*}
where in the second equation we have used the moment map equation in $\mathscr{O}_{N}^{(p)}\left(\epsilon_{1}, \epsilon_{2}\right) $. The second equation follows from a similar argument. (A diagrammatic approach to this derivation is provided in Appendix \ref{sec: Diagrammatic notations}.)
\end{proof}

\subsection{Derivation of the main commutation relations}\label{subsec: Derivation of the main commutation relations}
It will be advantageous to compute commutators by using the following generating function
\begin{align}
    &F(u, v):=\frac{1}{1-\left(u Z+v Z^{\dagger}\right)}=\sum_{m, n=0}^{\infty} \frac{(m+n) !}{m ! n !} \mathrm{Sym}\left(Z^{m} {Z^{\dagger}}^n\right) u^{m} v^{n}
\end{align}
$F(u,v)$ is understood to be a formal power series in the variables $u,v$. In the following, all expressions are viewed as formal power series in $u,v$ whose coefficients are operators in $\cA^{(p)}_N(\epsilon_{1})$ or $\mathscr O^{(p)}_N(\epsilon_{1},\epsilon_2)$. We denote the tensor product of $\cA^{(p)}_N(\epsilon_{1})$ and $\mathscr O^{(p)}_N(\epsilon_{1},\epsilon_2)$ with the linear space of formal power series in $u,v$ as $\cA^{(p)}_N(\epsilon_{1})\otimes\mathbb{C}[\![u,v]\!]$ and $\mathscr O^{(p)}_N(\epsilon_{1},\epsilon_2)\otimes\mathbb{C}[\![u,v]\!]$ respectively.
\begin{proposition}
The following identities holds in $\cA^{(p)}_N(\epsilon_{1})\otimes\mathbb{C}[\![u,v]\!]$. In the following expressions, $(Z^i_j)$ and $(F^i_j)$ indicate matrix components of $Z$ and $F(u,v)$. $F^{l}$ indicates the $lth$ power of $F$ with matrix multiplication. Here, $F^{l}$ is a short hand notation for $F(u,v)^{l}$.
    \begin{align}\label{partialu F}
        \partial_{u}^{l} F=l! \underbrace{F Z F Z F Z \cdots ZF}_{l\text{ copies of }Z}   
    \end{align}
    \begin{align}\label{partialv F}
        \partial_{v}^{l} F=l! \underbrace{F Z^{\dagger} F Z^{\dagger} F Z^{\dagger} \cdots Z^{\dagger} F}_{l\text{ copies of }Z^{\dagger}}   
    \end{align}
    \begin{align}\label{[Z,F]}
    [Z,F]=vF[Z,Z^{\dagger}]F
    \end{align}
    \begin{align}\label{[Zdagger,F]}
    [Z^{\dagger},F]=-uF[Z,Z^{\dagger}]F
    \end{align}
    \begin{align}\label{[Z^{i}_{j},F]}
[(Z^i_j),(F^k_l)]=\epsilon_1v(F^k_j)(F^i_l)
    \end{align}
    \begin{align}\label{(1+paritial)F^l}
        (1+\frac{u\partial_u+v\partial_v}{l})F^l=F^{l+1}
    \end{align}
    \begin{align}\label{F^l}
        F^{l}=\left(1+\frac{u \partial_ u+v \partial_v}{l-1}\right)\left(1+\frac{u \partial _u+v \partial_v}{l-2}\right) \cdots\left(1+\frac{u \partial _u+v \partial_v}{1}\right) F
    \end{align}
    
\end{proposition}

\begin{proof}
By taking partial derivative of $u$ on $F$, we have
\begin{align*}
\partial_{u}F(u,v)=&\sum_{n=1}^{\infty}\sum_{n_1+n_2=n-1}(uZ+vZ^{\dagger})^{n_1}Z(uZ+vZ^{\dagger})^{n_2}\\
=&\sum_{n_1=0}^{\infty}\sum_{n_2=0}^{\infty}(uZ+vZ^{\dagger})^{n_1}Z(uZ+vZ^{\dagger})^{n_2}\\
=&FZF
\end{align*}
By taking the repeated partial derivative of $u$ we obtain (\ref{partialu F}). Equality (\ref{partialv F}) can be obtained similarly. Notice that $[Z,-]$, $[Z^{\dagger},-]$, and $[(Z^i_j),-]$ are derivation operators i.e. Leibniz rule is satisfied. Hence (\ref{[Z,F]})-(\ref{[Z^{i}_{j},F]}) follows from similar computations.
 Now we show (\ref{(1+paritial)F^l}). By direct computation we have
\begin{align*}
(u\partial_{u}+v\partial_{v})F=F(uZ+vZ^{\dagger})F=F^2-F
\end{align*}
Here, $F(1-(uZ+vZ^{\dagger}))=1$ is used. Using the above equation, we have
\begin{align*}
    (1+\frac{u\partial_u+v\partial_v}{l})F^l=F^l+\frac{1}{l}\sum_{l_1+l_2=l-1}F^{l_{1}}(F^2-F)F^{l_{2}}=F^{l+1}
\end{align*}
Finally, (\ref{F^l}) follows from (\ref{(1+paritial)F^l}).
\end{proof}

\begin{corollary}
    The coefficient of $u^mv^n$ in $F^l$ is 
    \begin{align*}
        (1+\frac{m+n}{l-1})(1+\frac{m+n}{l-2})\cdots(1+\frac{m+n}{1})\binom{m+n}{m}\mathrm{Sym}(Z^m{Z^{\dagger}}^n)
    \end{align*}
\end{corollary}
We derive the commutation relations \eqref{3.24}-\eqref{[t03,tmn]} using generating functions.
\begin{proposition}
The following equations hold in $\mathscr O^{(p)}_N(\epsilon_{1},\epsilon_2)\otimes\mathbb{C}[\![u,v]\!]$:
\begin{align}
        \label{[t20,Fab]}&[t_{2,0},{\lambda^{\dagger}}^aF\lambda_b]=2v{\lambda^{\dagger}}^aFZF\lambda_b=2v\partial_{u}{\lambda^{\dagger}}^aF\lambda_b\\
        \label{[t02,Fab]}&[t_{0,2},{\lambda^{\dagger}}^aF\lambda_b]=-2u{\lambda^{\dagger}}^aFZ^{\dagger}F\lambda_b=-2u\partial_{v}{\lambda^{\dagger}}^aF\lambda_b\\
        \label{[t11,Fab]}&[t_{1,1},{\lambda^{\dagger}}^aF\lambda_b]=(v\partial_v-u\partial_u){\lambda^{\dagger}}^aF\lambda_b.
\end{align}
\end{proposition}
\begin{proof}
By direct computations.
\end{proof}

\begin{proposition}
The following equation holds in $\mathscr O^{(p)}_N(\epsilon_{1},\epsilon_2)\otimes\mathbb{C}[\![u,v]\!]$:
\begin{equation}\label{CR 1}
\begin{split}
& {\left[{\lambda^{\dagger}}^{a} Z \lambda_{b}, {\lambda^{\dagger}}^c\left(1+u \partial_{u}+v \partial_{v}\right) F \lambda_{d}\right]} \\
& =\epsilon_{1} \delta_{b}^{c} \lambda^{\dagger^{a}} \partial_{u} F \lambda_d-\epsilon_{1} \delta_{d}^{a} {\lambda^{\dagger}}^c \partial_{u} F \lambda_{b} \\
&+\epsilon_{1} \epsilon_{2} v \left\{\delta_{b}^{c} \lambda^{\dagger a}\left(1+\frac{u\partial_{u}+v\partial_{v}}{2}\right)\left(1+\frac{u \partial_{u}+v \partial v}{1}\right) F \lambda_{d}\right. \left.+\delta_{d}^{a} {\lambda^{\dagger}}^c\left(1+\frac{u\partial_{u}+v\partial_{v}}{2}\right)\left(1+\frac{u\partial_{u}+v\partial_{v}}{1}\right) F \lambda_{b}\right\} \\
& -\epsilon_{1} v\left\{\delta_{b}^{c} {\lambda^{\dagger}}^a F \lambda_{e} {\lambda^{\dagger}}^{e}(1+u \partial_u+v \partial_v) F \lambda_{d}+\delta^{a}_{d} {\lambda^{\dagger}}^c(1+u \partial_u+v \partial_v) F \lambda_{e} {\lambda^{\dagger}}^{e} F \lambda_{b}\right\}\\
& +\epsilon_{1} v\left\{{\lambda^{\dagger}}^a(1+u\partial_{u}+v \partial_{v}) F \lambda_{d} {\lambda^{\dagger}}^c F \lambda_{b}+{\lambda^{\dagger}}^a F \lambda_{d} {\lambda^{\dagger}}^c(1+u \partial_{u}+v \partial_{v}) F \lambda_{b}\right\} \\
& -2\epsilon_{1}^{2} v \delta_{d}^{c} {\lambda^{\dagger}}^a\left(1+\frac{u \partial_u+v \partial_{v}}{2}\right)\left(1+\frac{u \partial_{u}+v \partial_{v}}{1}\right) F \lambda_{b}.
\end{split}
\end{equation}
\end{proposition}

\begin{proof}
We have
\begin{align*}
[{\lambda^{\dagger}}^{a} Z \lambda_{b}, {\lambda^{\dagger}}^cF^2 \lambda_{d}]=&\epsilon_1\delta^c_b{\lambda^{\dagger}}^aZF^2\lambda_d-\epsilon_1\delta^a_d{\lambda^{\dagger}}^cF^2Z\lambda_b\\
&+\epsilon_1 v\{{\lambda_{\dagger}}^aF^2\lambda_d{\lambda^{\dagger}}^cF\lambda_b+{\lambda^{\dagger}}^aF\lambda_d{\lambda^{\dagger}}^cF^2\lambda_b\}\\
&-2\epsilon_1^2 v\delta^c_d{\lambda^{\dagger}}^aF^3\lambda_b.
\end{align*}
The first line in RHS of above equation is equal to 
\begin{align*}
\nonumber&\epsilon_1\delta^c_b{\lambda^{\dagger}}^a\{FZF+v F[Z,Z^{\dagger}]F^2\}\lambda_d-\epsilon_1\delta^a_d{\lambda^{\dagger}}^c\{FZF-v F^2[Z,Z^{\dagger}]F\}\lambda_b\\
\nonumber&=\epsilon_1\delta^c_b{\lambda^{\dagger}}^aFZF\lambda_d-\epsilon_1\delta^a_d{\lambda^{\dagger}}^cFZF\lambda_b\\
\nonumber&-\epsilon_1 v \{\delta^c_b{\lambda^{\dagger}}^aF\lambda_e{\lambda^{\dagger}}^eF^2\lambda_d+\delta^a_d{\lambda^{\dagger}}^cF^2\lambda_e{\lambda^{\dagger}}^eF\lambda_b\}\\
&+\epsilon_1\epsilon_2v\{\delta^c_b{\lambda^{\dagger}}^aF^3\lambda_d+\delta^a_d{\lambda^{\dagger}}^cF^3\lambda_b\}
\end{align*}
Altogether, we have
\begin{equation}\label{CR 1.3}
\begin{split}
&[{\lambda^{\dagger}}^{a} Z \lambda_{b}, {\lambda^{\dagger}}^cF^2 \lambda_{d}]=\epsilon_1\delta^c_b{\lambda^{\dagger}}^aFZF\lambda_d-\epsilon_1\delta^a_d{\lambda^{\dagger}}^cFZF\lambda_b\\
&-\epsilon_1 v \{\delta^c_b{\lambda^{\dagger}}^aF\lambda_e{\lambda^{\dagger}}^eF^2\lambda_d+\delta^a_d{\lambda^{\dagger}}^cF^2\lambda_e{\lambda^{\dagger}}^eF\lambda_b\}\\
&+\epsilon_1\epsilon_2v\{\delta^c_b{\lambda^{\dagger}}^aF^3\lambda_d+\delta^a_d{\lambda^{\dagger}}^cF^3\lambda_b\}\\
&+\epsilon_1 v\{{\lambda^{\dagger}}^aF^2\lambda_d{\lambda^{\dagger}}^cF\lambda_b+{\lambda^{\dagger}}^aF\lambda_d{\lambda^{\dagger}}^cF^2\lambda_b\}\\
&-2\epsilon_1^2 v\delta^c_d{\lambda^{\dagger}}^aF^3\lambda_b.
\end{split}
\end{equation}
Applying Proposition (\ref{partialu F}) and (\ref{F^l}) to (\ref{CR 1.3}), we obtain (\ref{CR 1}).
\end{proof}

\begin{proposition}
The following equation holds in $\mathscr O^{(p)}_N(\epsilon_{1},\epsilon_2)\otimes\mathbb{C}[\![u,v]\!]$:
\begin{equation}\label{t30Fab}
\begin{split}
&[\mathrm{Tr}(Z^3),{\lambda^{\dagger}}^a(1+u\partial_u+v\partial_v)F\lambda_b]\\
&=3\epsilon_1 v {\lambda^{\dagger}}^a{\partial_{u}}^2F\lambda_b\\
&+3\epsilon_1 v^2 \{{\lambda^{\dagger}}^a(1+v\partial_v+u\partial_u)F\lambda_e {\lambda^{\dagger}}^e\partial_u F\lambda_b-{\lambda^{\dagger}}^a\partial_uF\lambda_e {\lambda^{\dagger}}^e (1+v\partial_v+u\partial_u)F\lambda_b\}\\
&-3\epsilon_1\epsilon_2 v^3 \left\{
\setlength{\tabcolsep}{50pt}\renewcommand{\arraystretch}{1.8}\begin{array}{l}
{\lambda^{\dagger}}^a(1+\frac{v\partial_v+u\partial_u}{2})(1+v\partial_v+u\partial_u)F\lambda_e {\lambda^{\dagger}}^e(1+v\partial_v+u\partial_u)F\lambda_b\\
+{\lambda^{\dagger}}^a(1+v\partial_v+u\partial_u)F\lambda_e {\lambda^{\dagger}}^e(1+\frac{v\partial_v+u\partial_u}{2})(1+v\partial_v+u\partial_u)F\lambda_b
\end{array}\right\}\\
&+6\epsilon_1\epsilon_2^2 v^3{\lambda^{\dagger}}^a(1+\frac{u\partial_u+v\partial_v}{4})(1+\frac{u\partial_u+v\partial_v}{3})(1+\frac{u\partial_u+v\partial_v}{2})(1+\frac{u\partial_u+v\partial_v}{1})F\lambda_b
\end{split}
\end{equation}
\end{proposition}

\begin{proof}
Let us compute
\begin{align*}
\nonumber& {\left[\operatorname{Tr}\left(Z^{3}\right), {\lambda^{\dagger}}^a F^2 \lambda_{b}\right]} =3 \epsilon_{1} v {\lambda^{\dagger}}^a\left\{F^{2} Z^{2} F+F Z^{2} F^{2}\right\} \lambda_{b} \\
\nonumber& =3 \epsilon_{1} v {\lambda^{\dagger}}^a\left\{\begin{array}{l}
F Z F Z F+F[F, Z] Z F \\
+F Z F Z F+F Z[Z, F] F
\end{array}\right\} \lambda_{b} \\
\nonumber&=6 \epsilon_{1} v \lambda^{\dagger^{a}} FZFZF\lambda _{b}+3 \epsilon_{1} v^{2} {\lambda^{\dagger}}^{a}\left\{-F^{2}\left[Z, Z^{\dagger}\right] F Z F+F Z F\left[Z, Z^{\dagger}\right] F^{2}\right\} \lambda_{b} \\
\nonumber& =6 \epsilon_{1} v \lambda^{\dagger^{a}} F Z F Z F \lambda_{b}+3 \epsilon_{1} v^{2}\left\{{\lambda^{\dagger}}^{a} F^{2} \lambda_{e} {\lambda^{\dagger}}^{e} F Z F \lambda_{b}-{\lambda^{\dagger}}^{a} F Z F {\lambda}_{e} {\lambda^{\dagger}}^{e} F^{2} \lambda_{b}\right\} \\
& \quad+3 \epsilon_{1} \epsilon_{2} v^{2} {\lambda^{\dagger}}^{a}\left\{F Z F^{3}-F^{3} Z F\right\} \lambda_{b}
\end{align*}
The last line in the above equation is proportional to
\begin{equation}\label{t30Fab .2}
\begin{split}
&{\lambda^{\dagger}}^{a}\left\{F Z F^{3}-F^{3} Z F\right\} \lambda_{b}={\lambda^{\dagger}}^{a}\left\{F[Z,F^2]F\right\} \lambda_{b}\\
&~=v{\lambda^{\dagger}}^{a}\left\{F^3[Z,Z^{\dagger}]F^2\right\} \lambda_{b}+v{\lambda^{\dagger}}^{a}\left\{F^2[Z,Z^{\dagger}]F^3\right\} \lambda_{b}\\
&~=-v{\lambda^{\dagger}}^{a}F^3\lambda_e{\lambda^{\dagger}}^eF^2\lambda_{b}-v{\lambda^{\dagger}}^{a}F^2\lambda_e{\lambda^{\dagger}}^eF^3\lambda_{b}+2\epsilon_2 v{\lambda^{\dagger}}^aF^5\lambda_b,
\end{split}
\end{equation}
which implies that
\begin{equation}\label{t30Fab .3}
\begin{split}
&[\operatorname{Tr}\left(Z^{3}\right), {\lambda^{\dagger}}^a F^2 \lambda_{b}]=6 \epsilon_{1} v {\lambda^{\dagger}}^{a} F Z F Z F \lambda_{b}+3 \epsilon_{1} v^{2}\left\{{\lambda^{\dagger}}^{a} F^{2} \lambda_{e} {\lambda^{\dagger}}^{e} F Z F \lambda_{b}-{\lambda^{\dagger}}^{a} F Z F \lambda_{e} {\lambda^{\dagger}}^{e} F^{2} \lambda_{b}\right\} \\
&\quad-3 \epsilon_{1} \epsilon_{2} v^{3}\left\{{\lambda^{\dagger}}^{a} F^{3} \lambda_{e} {\lambda^{\dagger}}^e F^{2} \lambda_{b}+{\lambda^{\dagger}}^a F^{2} \lambda_{e} {\lambda^{\dagger}}^{e} F^{3} \lambda_{b}\right\}+6 \epsilon_{1} \epsilon_{2}^{2} v^{3} {\lambda^{\dagger}}^a F^{5} \lambda_{b}
\end{split}
\end{equation}
Now apply proposition (\ref{partialu F}) and (\ref{F^l}) to (\ref{t30Fab .3}), and we obtain (\ref{t30Fab}).
\end{proof}

Next we take the $\mathfrak{u}(1)$ part of (\ref{t30Fab}), and it leads us to the following
\begin{corollary}
The following equation holds in $\mathscr O^{(p)}_N(\epsilon_{1},\epsilon_2)\otimes\mathbb{C}[\![u,v]\!]$.
\begin{equation}\label{t30TrF}
\begin{split}
&[\mathrm{Tr}(Z^3),\mathrm{Tr}((1+v\partial_v+u\partial_u)F)]=3\epsilon_1 v \mathrm{Tr}({\partial_u}^2F)\\
&-3\epsilon_1v^3\{{\lambda^{\dagger}}^aF^3\lambda_e{\lambda^{\dagger}}^eF^2\lambda_a+{\lambda^{\dagger}}^aF^2\lambda_e{\lambda^{\dagger}}^eF^3\lambda_a\}\\
&-3\epsilon_1^2\epsilon_3 v^3\{\mathrm{Tr}(F^3)\mathrm{Tr}(F^2)+\mathrm{Tr}(F^2)\mathrm{Tr}(F^3)\}\\
&+6\epsilon_1(\epsilon_1^2+\epsilon_2\epsilon_3)v^3\mathrm{Tr}(F^5).
\end{split}
\end{equation}
\end{corollary}

\begin{proof}
Contract the indices $a$ and $b$ in (\ref{t30Fab .3}) and apply the proposition ${\lambda^{\dagger}}^{a}F\lambda_{a}=\epsilon_3 \mathrm{Tr}(F)$. We obtain
\begin{equation}\label{t30TrF .1}
\begin{split}
&[\operatorname{Tr}\left(Z^{3}\right), \epsilon_3 \mathrm{Tr}(F^2) ]=6 \epsilon_{1}\epsilon_{3} v \mathrm{Tr}(F Z F Z F) +3 \epsilon_{1} v^{2}\left\{{\lambda^{\dagger}}^{a} F^{2} \lambda_{e} {\lambda^{\dagger}}^{e} F Z F \lambda_{a}-{\lambda^{\dagger}}^{a} F Z F \lambda_{e} {\lambda^{\dagger}}^{e} F^{2} \lambda_{a}\right\} \\
&\quad-3 \epsilon_{1} \epsilon_{2} v^{3}\left\{{\lambda^{\dagger}}^{a} F^{3} \lambda_{e} {\lambda^{\dagger}}^e F^{2} \lambda_{a}+{\lambda^{\dagger}}^a F^{2} \lambda_{e} {\lambda^{\dagger}}^{e} F^{3} \lambda_{a}\right\}+6 \epsilon_{1} \epsilon_{2}^{2}\epsilon_{3} v^{3} \mathrm{Tr}(F^{5}).
\end{split}
\end{equation}
The second term in right hand side of the above equation can be computed as follows
\begin{equation}\label{t30TrF .2}
\begin{split}
&{\lambda^{\dagger}}^{a} F^{2} \lambda_{e} {\lambda^{\dagger}}^{e} F Z F \lambda_{a}-{\lambda^{\dagger}}^{a} F Z F \lambda_{e} {\lambda^{\dagger}}^{e} F^{2} \lambda_{a}=[{\lambda^{\dagger}}^{a} F^{2} \lambda_{e},{\lambda^{\dagger}}^{a} F Z F \lambda_{e}]\\
&=\epsilon_1 p {\lambda^{\dagger}}^a\{F^3ZF-FZF^3\}\lambda_a-\epsilon_1\epsilon_3^2v\{\mathrm{Tr}(F^3)\mathrm{Tr}(F^2)+\mathrm{Tr}(F^2)\mathrm{Tr}(F^3)\}+2\epsilon_1^2\epsilon_3v\mathrm{Tr}(F^5)\\
&=\epsilon_1 p v\{{\lambda^{\dagger}}^{a}F^3\lambda_e{\lambda^{\dagger}}^eF^2\lambda_{b}+{\lambda^{\dagger}}^{a}F^2\lambda_e{\lambda^{\dagger}}^eF^3\lambda_{b}\}-2\epsilon_1\epsilon_2pv{\lambda^{\dagger}}^aF^5\lambda_a\\
&\quad -\epsilon_1\epsilon_3^2v\{\mathrm{Tr}(F^3)\mathrm{Tr}(F^2)+\mathrm{Tr}(F^2)\mathrm{Tr}(F^3)\}+2\epsilon_1^2\epsilon_3v\mathrm{Tr}(F^5).
\end{split}
\end{equation}
From the second line to the third line in \eqref{t30TrF .2}, we use (\ref{t30Fab .2}). Altogether we get (\ref{t30TrF}).
\end{proof}

We now derive other commutators $[t_{2,1},{\lambda^{\dagger}}^aF\lambda_b]$ ,  $[t_{1,2},{\lambda^{\dagger}}^aF\lambda_b]$ and $[t_{0,3},{\lambda^{\dagger}}^aF\lambda_b]$ from (\ref{t30Fab}). Let us first consider $[t_{2,1},{\lambda^{\dagger}}^aF\lambda_b]$. Notice we have relation $6t_{2,1}=[t_{3,0},t_{0,2}]$.\\

Denote the operator $[t_{0,2},-]$ as $\mathrm{adj}_{t_{0,2}}$.
Now we compute 
\begin{equation*}
\begin{split}
&[t_{2,1},{\lambda^{\dagger}}^aF^2\lambda_b]=-\frac{1}{6}[[t_{0,2},t_{3,0}],{\lambda^{\dagger}}^a F^2\lambda_b]\\
&=-\frac{1}{6}\{\mathrm{adj}_{t_{0,2}}[t_{3,0},{\lambda^{\dagger}}^aF^2\lambda_b]-[t_{3,0},\mathrm{adj}_{t_{0,2}}{\lambda^{\dagger}}^aF^2\lambda_b\}\\
&=-\frac{1}{6}\{\mathrm{adj}_{t_{0,2}}+2u\partial_{v}\}[t_{3,0},{\lambda^{\dagger}}^aF^2\lambda_b].
\end{split}
\end{equation*}
Here, we used the Jacobi identity from the first line to the second line. From the second line to the third line, we use $[t_{0,2},{\lambda^{\dagger}}^aF^2\lambda_b]=-2u\partial_v {\lambda^{\dagger}}^aF^2\lambda_b$, which follows from the next proposition.
\begin{proposition}
For any positive integer $l$, we have
    \begin{align}\label{adj+2upartial v}
        \{\mathrm{adj}_{t_{0,2}}+2u\partial_{v}\}{\lambda^{\dagger}}^aF^l\lambda_b=0
    \end{align}
\end{proposition}
\begin{proof}
(\ref{adj+2upartial v}) is true for $l=1$. Now, it is easy to see that the operator $\mathrm{adj}_{t_{0,2}}+2u\partial_{v}$ commute with $u\partial_u+v\partial_v$. This fact, together with (\ref{F^l}), implies that 
\begin{align*}
        \nonumber&\{\mathrm{adj}_{t_{0,2}}+2u\partial_{v}\}{\lambda^{\dagger}}^aF^l\lambda_b\\
        \nonumber&=\{\mathrm{adj}_{t_{0,2}}+2u\partial_{v}\}\left(1+\frac{u \partial_ u+v \partial_v}{l-1}\right)\left(1+\frac{u \partial _u+v \partial_v}{l-2}\right) \cdots\left(1+\frac{u \partial _u+v \partial_v}{1}\right){\lambda^{\dagger}}^aF\lambda_b\\
        &=\left(1+\frac{u \partial_ u+v \partial_v}{l-1}\right)\left(1+\frac{u \partial _u+v \partial_v}{l-2}\right) \cdots\left(1+\frac{u \partial _u+v \partial_v}{1}\right)\{\mathrm{adj}_{t_{0,2}}+2u\partial_{v}\}{\lambda^{\dagger}}^aF\lambda_b=0
\end{align*}
\end{proof}

Another useful property is that $\mathrm{adj}_{t_{0,2}}+2u\partial_{v}$ is a derivation, i.e. the Leibniz rule is satisfied when it acts on a product of operators. These properties allow us to compute $\{\mathrm{adj}_{t_{0,2}}+2u\partial_{v}\}[t_{3,0},{\lambda^{\dagger}}^aF^2\lambda_b]$ from the expression (\ref{t30Fab}). For example, there is a term $\propto 3v\partial_{u}^2{\lambda^{\dagger}}^aF\lambda_b$ on the RHS of (\ref{t30Fab}). We compute the action of $\mathrm{adj}_{t_{0,2}}+2u\partial_{v}$ on such term
\begin{align*}
\{\mathrm{adj}_{t_{0,2}}+2u\partial_{v}\}3v\partial_{u}^2{\lambda^{\dagger}}^aF\lambda_b&=[\mathrm{adj}_{t_{0,2}}+2u\partial_{v},3v\partial_{u}^2]{\lambda^{\dagger}}^aF\lambda_b=[2u\partial_{v},3v\partial_{u}^2]{\lambda^{\dagger}}^aF\lambda_b\\
&=(6u\partial_{u}^2-12v\partial_u\partial_v){\lambda^{\dagger}}^aF\lambda_b
\end{align*}
Another example is
\begin{align*}
\{\mathrm{adj}_{t_{0,2}}+2u\partial_{v}\}3v^2{\lambda^{\dagger}}^aF^2\lambda_{e}{\lambda^{\dagger}}^e\partial_u F\lambda_b&={\lambda^{\dagger}}^aF^2\lambda_{e}[\mathrm{adj}_{t_{0,2}}+2u\partial_{v},3v^2\partial_u]{\lambda^{\dagger}}^eF\lambda_b\\
&={\lambda^{\dagger}}^aF^2\lambda_{e}(12uv\partial_u-6v^2\partial_v){\lambda^{\dagger}}^eF\lambda_b
\end{align*}
Similarly, we obtain
\begin{equation}\label{t21Fab}
\begin{split}
&[t_{2,1},{\lambda^{\dagger}}^aF^2\lambda_b]=(-u\partial_{u}^2+2v\partial_u\partial_v){\lambda^{\dagger}}^aF\lambda_b\\
&+{\lambda^{\dagger}}^aF^2\lambda_{e}(-2uv\partial_u+v^2\partial_v){\lambda^{\dagger}}^eF\lambda_b-{\lambda^{\dagger}}^a(-2uv\partial_u+v^2\partial_v)F\lambda_{e}{\lambda^{\dagger}}^eF^2\lambda_b\\
&+3\epsilon_2 uv^2\{{\lambda^{\dagger}}^aF^3\lambda_{e}{\lambda^{\dagger}}^eF^2\lambda_b+{\lambda^{\dagger}}^aF^2\lambda_{e}{\lambda^{\dagger}}^eF^3\lambda_b\}
    -6\epsilon_2^2 uv^2{\lambda^{\dagger}}^aF^5\lambda_b
\end{split}
\end{equation}
Now we use similar argument to compute 
\begin{align*}
\nonumber[t_{1,2},{\lambda^{\dagger}}^aF\lambda_b]=\frac{1}{4}[[t_{2,1},t_{0,2}],{\lambda^{\dagger}}^aF\lambda_b]=-\frac{1}{4}\{\mathrm{adj}_{t_{0,2}}+2u\partial_v\}[t_{2,1},{\lambda^{\dagger}}^aF\lambda_b],
\end{align*}
and we find 
\begin{equation}\label{t12Fab}
\begin{split}
&[t_{1,2},{\lambda^{\dagger}}^aF^2\lambda_b]=(v\partial_{v}^2-2u\partial_u\partial_v){\lambda^{\dagger}}^aF\lambda_b\\
&+{\lambda^{\dagger}}^aF^2\lambda_{e}(-2uv\partial_v+u^2\partial_u){\lambda^{\dagger}}^eF\lambda_b-{\lambda^{\dagger}}^a(-2uv\partial_v+u^2\partial_u)F\lambda_{e}{\lambda^{\dagger}}^eF^2\lambda_b\\
&-3\epsilon_2 u^2v\{{\lambda^{\dagger}}^aF^3\lambda_{e}{\lambda^{\dagger}}^eF^2\lambda_b+{\lambda^{\dagger}}^aF^2\lambda_{e}{\lambda^{\dagger}}^eF^3\lambda_b\}
+6\epsilon_2^2 u^2v{\lambda^{\dagger}}^aF^5\lambda_b
\end{split}
\end{equation}
Using $t_{0,3}=\frac{1}{2}[t_{1,2},t_{0,2}]$, and apply similar argument, we obtain
\begin{equation}\label{t03Fab}
\begin{split}
&[t_{0,3},{\lambda^{\dagger}}^aF^2\lambda_b]=-3u\partial_v^2{\lambda^{\dagger}}^aF\lambda_b\\
&+{\lambda^{\dagger}}^aF^2\lambda_{e}(3u^2\partial_v){\lambda^{\dagger}}^eF\lambda_b-{\lambda^{\dagger}}^a(3u^2\partial_v)F\lambda_{e}{\lambda^{\dagger}}^eF^2\lambda_b\\
&+3\epsilon_2 u^3\{{\lambda^{\dagger}}^aF^3\lambda_{e}{\lambda^{\dagger}}^eF^2\lambda_b+{\lambda^{\dagger}}^aF^2\lambda_{e}{\lambda^{\dagger}}^eF^3\lambda_b\}
-6\epsilon_2^2 u^3{\lambda^{\dagger}}^aF^5\lambda_b.
\end{split}
\end{equation}

\begin{remark}
    The algebra has symmetry $Z\leftrightarrow Z^{\dagger}$, $\epsilon_1 \leftrightarrow -\epsilon_1$, $\lambda_e{\lambda^{\dagger}}^e \leftrightarrow -\lambda_e{\lambda^{\dagger}}^e$ , $\epsilon_2 \leftrightarrow -\epsilon_2$ and $u \leftrightarrow v$. As a consistency check, one can see that (\ref{t30Fab}) and (\ref{t21Fab}) become (\ref{t03Fab}) and (\ref{t12Fab}) respectively under this symmetry.
\end{remark}

For the $\mathfrak{u}(1)$ part of $[t_{2,1},{\lambda^{\dagger}}^aF\lambda_b]$ ,  $[t_{1,2},{\lambda^{\dagger}}^aF\lambda_b]$ and $[t_{0,3},{\lambda^{\dagger}}^aF\lambda_b]$, we apply the same argument. From (\ref{t30TrF}), we obtain

\begin{proposition}
\begin{equation}\label{[t30,trF]}
\begin{split}
&[t_{3,0},\mathrm{Tr}(F^2)]=3v\partial_u^2 \mathrm{Tr}(F)-3v^3\{{\lambda^{\dagger}}^aF^3\lambda_e{\lambda^{\dagger}}^eF^2\lambda_a+{\lambda^{\dagger}}^aF^2\lambda_e{\lambda^{\dagger}}^eF^3\lambda_a\}\\
&-3\epsilon_1\epsilon_3 v^3\{\mathrm{Tr}(F^3)\mathrm{Tr}(F^2)+\mathrm{Tr}(F^2)\mathrm{Tr}(F^3)\}+6(\epsilon_1^2+\epsilon_2\epsilon_3)v^3\mathrm{Tr}(F^5)
\end{split}
\end{equation}
\begin{equation}\label{[t21,trF]}
\begin{split}
&[t_{2,1},\mathrm{Tr}(F^2)]=(-u\partial_u^2+2v\partial_u\partial_v) \mathrm{Tr}(F)+3uv^2\{{\lambda^{\dagger}}^aF^3\lambda_e{\lambda^{\dagger}}^eF^2\lambda_a+{\lambda^{\dagger}}^aF^2\lambda_e{\lambda^{\dagger}}^eF^3\lambda_a\}\\
&+3\epsilon_1\epsilon_3 uv^2\{\mathrm{Tr}(F^3)\mathrm{Tr}(F^2)+\mathrm{Tr}(F^2)\mathrm{Tr}(F^3)\}-6(\epsilon_1^2+\epsilon_2\epsilon_3)uv^2\mathrm{Tr}(F^5)
\end{split}
\end{equation}
\begin{equation}\label{[t12,trF]}
\begin{split}
&[t_{1,2},\mathrm{Tr}(F^2)]=(v\partial_v^2-2u\partial_u\partial_v) \mathrm{Tr}(F)-3vu^2\{{\lambda^{\dagger}}^aF^3\lambda_e{\lambda^{\dagger}}^eF^2\lambda_a+{\lambda^{\dagger}}^aF^2\lambda_e{\lambda^{\dagger}}^eF^3\lambda_a\}\\
&-3\epsilon_1\epsilon_3 vu^2\{\mathrm{Tr}(F^3)\mathrm{Tr}(F^2)+\mathrm{Tr}(F^2)\mathrm{Tr}(F^3)\}+6(\epsilon_1^2+\epsilon_2\epsilon_3)vu^2\mathrm{Tr}(F^5)
\end{split}
\end{equation}
\begin{equation}\label{[t03,trF]}
\begin{split}
&[t_{0,3},\mathrm{Tr}(F^2)]=-3u\partial_v^2 \mathrm{Tr}(F)+3u^3\{{\lambda^{\dagger}}^aF^3\lambda_e{\lambda^{\dagger}}^eF^2\lambda_a+{\lambda^{\dagger}}^aF^2\lambda_e{\lambda^{\dagger}}^eF^3\lambda_a\}\\
&+3\epsilon_1\epsilon_3 u^3\{\mathrm{Tr}(F^3)\mathrm{Tr}(F^2)+\mathrm{Tr}(F^2)\mathrm{Tr}(F^3)\}-6(\epsilon_1^2+\epsilon_2\epsilon_3)u^3\mathrm{Tr}(F^5)
\end{split}
\end{equation}
\end{proposition}

Now we obtain all commutation relations (\ref{3.24})-(\ref{3.26}), (\ref{[eab10,ecdmn]}) ,(\ref{[t30,eabmn]})-(\ref{[t12,eabmn]}), and (\ref{[t30,tmn]})-(\ref{[t03,tmn]}) by taking coefficient of $u^mv^n$ in (\ref{[t20,Fab]})-(\ref{[t11,Fab]}), (\ref{CR 1}), (\ref{t30Fab}), (\ref{t21Fab}), (\ref{t12Fab}), and (\ref{[t30,trF]})-(\ref{[t03,trF]}).

Finally, we derive \ref{[Jab10,Jcdmn]}. We decompose ${{\lambda}^{\dagger}}^aF(u,v)\lambda_b$ into $\mathfrak{sl}_{p}$ and $\mathfrak{gl}_{1}$ parts. Introduce
$$J^a_b(u,v):=\frac{1}{\epsilon_1}\left({\lambda}^aF(u,v)\lambda_b-\frac{\epsilon_3}{p}\delta^a_{b}Tr(F(u,v)) \right)\quad, \quad T(u,v):=\frac{1}{\epsilon_1}Tr(F(u,v)).$$
From (\ref{CR 1}), using Lemma \ref{[uZ+vZ,uZ+vZ]=0}, we have
\begin{equation}\label{[Jab10,Jcd]}
    \begin{split}
            &[J^a_{b;1,0},(1+u\partial_u+v\partial_v)J^c_{d}(u,v)]=\\
    &+\delta^c_b\partial_u J^a_d(u,v)-\delta^a_d \partial_u J^c_b(u,v)\\
    &+2v\frac{\epsilon_2\epsilon_3}{p}\{\delta^a_d\delta^c_b-\frac{1}{p}\delta^a_b\delta^c_d\}(1+\frac{u\partial_u+v\partial_v}{2})(1+u\partial_u+v\partial_v)T(u,v)\\
    &+v\epsilon_2\{\delta^c_b(1+\frac{u\partial_u+v\partial_v}{2})(1+u\partial_u+v\partial_v)J^a_d(u,v)+\delta^a_d(1+\frac{u\partial_u+v\partial_v}{2})(1+u\partial_u+v\partial_v)J^c_b(u,v)\}\\
    &-2v\epsilon_2\frac{1}{p}\{\delta^a_b(1+\frac{u\partial_u+v\partial_v}{2})(1+u\partial_u+v\partial_v)J^c_d(u,v)+\delta^c_d(1+\frac{u\partial_u+v\partial_v}{2})(1+u\partial_u+v\partial_v)J^a_b(u,v)\}\\
    & -\epsilon_{1}v\left\{\delta_{b}^{c} J^a_e(u,v)\{(1+u \partial_u+v \partial_v) J^e_d(u,v)\}+ \delta^a_d\{(1+u \partial_u+v \partial_v)J^c_e(u,v)\}J^e_b(u,v)\right\}\\
& +\epsilon_{1}v\left\{\{(1+u\partial_{u}+v \partial_{v})J^a_d(u,v)\}J^c_b(u,v)+J^a_d(u,v) \{(1+u \partial_{u}+v \partial_{v})J^c_b(u,v)\}\right\}. 
    \end{split}
\end{equation}

By taking coefficient of $u^mv^n$ in (\ref{[Jab10,Jcd]}), we obtain (\ref{[Jab10,Jcdmn]}).

\subsection{Proof of the large \texorpdfstring{$N$}{N} limit}\label{subsec: Proof of the large N limit}
\subsubsection*{Representation of the algebra}
In this section, we will apply diagrammatic notations that are introduced in Appendix \ref{sec: Diagrammatic notations}. 
\begin{definition}
We define the totally anti-symmetrization tensor as follows:
\begin{center}
    \tikzset{every picture/.style={line width=0.75pt}} 

\tikzset{every picture/.style={line width=0.75pt}} 



\end{center}
\end{definition}

The tensors $A$ satisfy the following recursive relation:
\begin{equation}\label{Recursive of A}
\hbox{
\tikzset{every picture/.style={line width=0.75pt}} 
%
}
\end{equation}

The dimension of the vector space $V$ is $N$, so $N+1$ is the minimal integer $m$ such that the $m th$ totally anti-symmetrization tensor is zero:

\begin{equation}\label{anti N+1=0}
    \hbox{
    \tikzset{every picture/.style={line width=0.75pt}} 

%

    }
\end{equation}

\begin{lemma}
We have the following relation
\begin{equation}\label{5.41}
    \hbox{
\tikzset{every picture/.style={line width=0.75pt}} 

%
}
\end{equation}
\end{lemma}
\begin{proof}
Let $\left\{e_{i} \mid i=1, \cdots, N\right\}$ be a basis of the vector space $V$. Let $T_{L}$, $T_{R}$ $\in \underbrace{\left(V^{*} \otimes V^{*} \otimes \cdots V^{*}\right)}_{N} \otimes\left(V^{*} \otimes V\right)$ be the tensors of the LHS and RHS of (\ref{5.41}). To show $T_{L}=T_{R}$, it is enough to show that after contracting $\left\{e_{i} \otimes e_{k}^{*} \mid i, k=1,2, \ldots, N\right\}$ with the last 2 components of the tensors, they are equal. We first consider the contraction with $e_{i} \otimes e_{k}^{*}$, for $i \neq k$. In this case, both sides are zero. The left hand side of the diagram is clearly 0 . For the right hand side, using the fact that the tensors $\varepsilon$ and $A$ are both anti-symmetric in their inputs or outputs, the vectors in the basis $\left\{e_{i} \mid i=1, \cdots, N\right\}$ that are propagating in the first $N-1$ vertical lines between $A$ and $\varepsilon$ must be distinct. This forces the propagating vectors in the last input of $\varepsilon$ and the last output of $A$ to be the same. Graphically:

\begin{center}
    \tikzset{every picture/.style={line width=0.75pt}} 

\tikzset{every picture/.style={line width=0.75pt}} 

\tikzset{every picture/.style={line width=0.75pt}} 

\tikzset{every picture/.style={line width=0.75pt}} 

\begin{tikzpicture}[x=0.75pt,y=0.75pt,yscale=-0.75,xscale=0.75]

\draw  [line width=0.75]  (345.54,40.43) -- (485.76,40.43) -- (485.76,61.88) -- (345.54,61.88) -- cycle ;
\draw [line width=0.75]    (354.18,62.86) -- (354.76,131.14) ;
\draw [shift={(354.41,90)}, rotate = 89.52] [color={rgb, 255:red, 0; green, 0; blue, 0 }  ][line width=0.75]    (10.93,-3.29) .. controls (6.95,-1.4) and (3.31,-0.3) .. (0,0) .. controls (3.31,0.3) and (6.95,1.4) .. (10.93,3.29)   ;
\draw [line width=0.75]    (373.7,62.86) -- (374.27,131.14) ;
\draw [shift={(373.93,90)}, rotate = 89.52] [color={rgb, 255:red, 0; green, 0; blue, 0 }  ][line width=0.75]    (10.93,-3.29) .. controls (6.95,-1.4) and (3.31,-0.3) .. (0,0) .. controls (3.31,0.3) and (6.95,1.4) .. (10.93,3.29)   ;
\draw [line width=0.75]    (394.66,62.39) -- (395.23,130.67) ;
\draw [shift={(394.88,89.53)}, rotate = 89.52] [color={rgb, 255:red, 0; green, 0; blue, 0 }  ][line width=0.75]    (10.93,-3.29) .. controls (6.95,-1.4) and (3.31,-0.3) .. (0,0) .. controls (3.31,0.3) and (6.95,1.4) .. (10.93,3.29)   ;
\draw [line width=0.75]    (457.21,62.86) -- (457.79,131.14) ;
\draw [shift={(457.44,90)}, rotate = 89.52] [color={rgb, 255:red, 0; green, 0; blue, 0 }  ][line width=0.75]    (10.93,-3.29) .. controls (6.95,-1.4) and (3.31,-0.3) .. (0,0) .. controls (3.31,0.3) and (6.95,1.4) .. (10.93,3.29)   ;
\draw [line width=0.75]    (476.29,105.21) -- (476.78,131.42) ;
\draw [shift={(476.4,111.32)}, rotate = 88.93] [color={rgb, 255:red, 0; green, 0; blue, 0 }  ][line width=0.75]    (10.93,-3.29) .. controls (6.95,-1.4) and (3.31,-0.3) .. (0,0) .. controls (3.31,0.3) and (6.95,1.4) .. (10.93,3.29)   ;
\draw [line width=0.75]    (476.29,62.59) -- (476.42,99.39) ;
\draw [shift={(476.33,73.99)}, rotate = 89.8] [color={rgb, 255:red, 0; green, 0; blue, 0 }  ][line width=0.75]    (10.93,-3.29) .. controls (6.95,-1.4) and (3.31,-0.3) .. (0,0) .. controls (3.31,0.3) and (6.95,1.4) .. (10.93,3.29)   ;
\draw [line width=0.75]    (476.42,99.39) -- (520.72,61.75) ;
\draw [line width=0.75]    (476.29,105.21) -- (527.41,63.37) ;
\draw    (519.36,61.48) .. controls (518.03,61.95) and (511.29,50.18) .. (511.29,40.43) ;
\draw [shift={(516.44,57.99)}, rotate = 244.81] [color={rgb, 255:red, 0; green, 0; blue, 0 }  ][line width=0.75]    (10.93,-3.29) .. controls (6.95,-1.4) and (3.31,-0.3) .. (0,0) .. controls (3.31,0.3) and (6.95,1.4) .. (10.93,3.29)   ;
\draw    (527.41,63.37) .. controls (537.93,63.84) and (561.6,54.1) .. (564.97,40.43) ;
\draw [shift={(554.59,53.67)}, rotate = 147.86] [color={rgb, 255:red, 0; green, 0; blue, 0 }  ][line width=0.75]    (10.93,-3.29) .. controls (6.95,-1.4) and (3.31,-0.3) .. (0,0) .. controls (3.31,0.3) and (6.95,1.4) .. (10.93,3.29)   ;
\draw   (344.63,131.63) -- (486.85,131.63) -- (486.85,149.8) -- (344.63,149.8) -- cycle ;
\draw [line width=0.75]    (354.54,150.73) -- (355.03,176.94) ;
\draw [shift={(354.65,156.84)}, rotate = 88.93] [color={rgb, 255:red, 0; green, 0; blue, 0 }  ][line width=0.75]    (10.93,-3.29) .. controls (6.95,-1.4) and (3.31,-0.3) .. (0,0) .. controls (3.31,0.3) and (6.95,1.4) .. (10.93,3.29)   ;
\draw [line width=0.75]    (375.04,149.42) -- (375.53,175.63) ;
\draw [shift={(375.16,155.53)}, rotate = 88.93] [color={rgb, 255:red, 0; green, 0; blue, 0 }  ][line width=0.75]    (10.93,-3.29) .. controls (6.95,-1.4) and (3.31,-0.3) .. (0,0) .. controls (3.31,0.3) and (6.95,1.4) .. (10.93,3.29)   ;
\draw [line width=0.75]    (394.45,149.85) -- (394.94,176.07) ;
\draw [shift={(394.56,155.96)}, rotate = 88.93] [color={rgb, 255:red, 0; green, 0; blue, 0 }  ][line width=0.75]    (10.93,-3.29) .. controls (6.95,-1.4) and (3.31,-0.3) .. (0,0) .. controls (3.31,0.3) and (6.95,1.4) .. (10.93,3.29)   ;
\draw [line width=0.75]    (457.98,149.85) -- (458.47,176.07) ;
\draw [shift={(458.09,155.96)}, rotate = 88.93] [color={rgb, 255:red, 0; green, 0; blue, 0 }  ][line width=0.75]    (10.93,-3.29) .. controls (6.95,-1.4) and (3.31,-0.3) .. (0,0) .. controls (3.31,0.3) and (6.95,1.4) .. (10.93,3.29)   ;
\draw [line width=0.75]    (477.43,150.29) -- (477.92,176.51) ;
\draw [shift={(477.55,156.4)}, rotate = 88.93] [color={rgb, 255:red, 0; green, 0; blue, 0 }  ][line width=0.75]    (10.93,-3.29) .. controls (6.95,-1.4) and (3.31,-0.3) .. (0,0) .. controls (3.31,0.3) and (6.95,1.4) .. (10.93,3.29)   ;
\draw   (354.7,183.61) .. controls (354.73,188.28) and (357.07,190.6) .. (361.74,190.58) -- (406.35,190.37) .. controls (413.02,190.34) and (416.36,192.65) .. (416.38,197.32) .. controls (416.36,192.65) and (419.68,190.3) .. (426.35,190.27)(423.35,190.29) -- (470.96,190.06) .. controls (475.63,190.04) and (477.95,187.7) .. (477.92,183.03) ;
\draw   (501.79,30.93) .. controls (501.79,25.94) and (506.03,21.9) .. (511.27,21.9) .. controls (516.51,21.9) and (520.75,25.94) .. (520.75,30.93) .. controls (520.75,35.92) and (516.51,39.96) .. (511.27,39.96) .. controls (506.03,39.96) and (501.79,35.92) .. (501.79,30.93) -- cycle ;
\draw   (555.26,32) .. controls (555.26,27.01) and (559.51,22.97) .. (564.74,22.97) .. controls (569.98,22.97) and (574.22,27.01) .. (574.22,32) .. controls (574.22,36.98) and (569.98,41.02) .. (564.74,41.02) .. controls (559.51,41.02) and (555.26,36.98) .. (555.26,32) -- cycle ;
\draw  [dash pattern={on 4.5pt off 4.5pt}] (344.67,76.71) -- (465.55,76.71) -- (465.55,112.94) -- (344.67,112.94) -- cycle ;
\draw  [dash pattern={on 4.5pt off 4.5pt}]  (343.74,94.49) .. controls (326.17,102.21) and (287.27,100.63) .. (268.84,103.66) ;
\draw [shift={(266.92,104.01)}, rotate = 348.69] [color={rgb, 255:red, 0; green, 0; blue, 0 }  ][line width=0.75]    (10.93,-3.29) .. controls (6.95,-1.4) and (3.31,-0.3) .. (0,0) .. controls (3.31,0.3) and (6.95,1.4) .. (10.93,3.29)   ;

\draw (408.03,43.36) node [anchor=north west][inner sep=0.75pt]  [font=\Large]  {$\epsilon $};
\draw (414.9,90.05) node [anchor=north west][inner sep=0.75pt]  [font=\normalsize]  {$\cdots $};
\draw (408.58,132.52) node [anchor=north west][inner sep=0.75pt]  [font=\normalsize]  {$A$};
\draw (416.05,157.02) node [anchor=north west][inner sep=0.75pt]  [font=\normalsize]  {$\cdots $};
\draw (391.45,198.85) node [anchor=north west][inner sep=0.75pt]  [font=\small]  {$N$ lines};
\draw (554.25,103.02) node [anchor=north west][inner sep=0.75pt]    {$=$};
\draw (597,101.05) node [anchor=north west][inner sep=0.75pt]  [font=\large]  {$\delta _{i,k}$};
\draw (505.61,23.78) node [anchor=north west][inner sep=0.75pt]  [font=\footnotesize]  {$i$};
\draw (558.17,24.36) node [anchor=north west][inner sep=0.75pt]  [font=\footnotesize]  {$k$};
\draw (580,101.96) node [anchor=north west][inner sep=0.75pt]  [font=\large]  {$T$};
\draw (53.5,96.23) node [anchor=north west][inner sep=0.75pt]  [font=\small]  {\begin{tabular}{c} $N-1$ different\\ vectors  propagating \\ in these vertical lines\end{tabular}};


\end{tikzpicture}
\end{center}
So the only nonzero terms are those contraction with $\left\{e_{i} \otimes e_{i}^{*} \mid i=1,2, \ldots, N\right\}$. It remains to show the following equality in the space $V^{*} \otimes V^{*} \otimes \cdots \otimes V^{*}$:
\begin{equation}\label{tensor in proof}
    \hbox{
    
   \tikzset{every picture/.style={line width=0.75pt}} 

\begin{tikzpicture}[x=0.75pt,y=0.75pt,yscale=-0.7,xscale=0.7]

\draw   (186.22,48.58) -- (331.2,48.58) -- (331.2,69.01) -- (186.22,69.01) -- cycle ;
\draw    (195.16,69.53) -- (195.96,178.19) ;
\draw [shift={(195.5,116.86)}, rotate = 89.58] [color={rgb, 255:red, 0; green, 0; blue, 0 }  ][line width=0.75]    (10.93,-3.29) .. controls (6.95,-1.4) and (3.31,-0.3) .. (0,0) .. controls (3.31,0.3) and (6.95,1.4) .. (10.93,3.29)   ;
\draw    (215.34,69.53) -- (215.01,178.19) ;
\draw [shift={(215.19,116.86)}, rotate = 90.17] [color={rgb, 255:red, 0; green, 0; blue, 0 }  ][line width=0.75]    (10.93,-3.29) .. controls (6.95,-1.4) and (3.31,-0.3) .. (0,0) .. controls (3.31,0.3) and (6.95,1.4) .. (10.93,3.29)   ;
\draw    (237,69.08) -- (237.38,178.19) ;
\draw [shift={(237.17,116.64)}, rotate = 89.8] [color={rgb, 255:red, 0; green, 0; blue, 0 }  ][line width=0.75]    (10.93,-3.29) .. controls (6.95,-1.4) and (3.31,-0.3) .. (0,0) .. controls (3.31,0.3) and (6.95,1.4) .. (10.93,3.29)   ;
\draw    (301.68,69.53) -- (302.83,178.19) ;
\draw [shift={(302.18,116.86)}, rotate = 89.4] [color={rgb, 255:red, 0; green, 0; blue, 0 }  ][line width=0.75]    (10.93,-3.29) .. controls (6.95,-1.4) and (3.31,-0.3) .. (0,0) .. controls (3.31,0.3) and (6.95,1.4) .. (10.93,3.29)   ;
\draw    (322.08,69.8) -- (322.71,178.19) ;
\draw [shift={(322.35,117)}, rotate = 89.67] [color={rgb, 255:red, 0; green, 0; blue, 0 }  ][line width=0.75]    (10.93,-3.29) .. controls (6.95,-1.4) and (3.31,-0.3) .. (0,0) .. controls (3.31,0.3) and (6.95,1.4) .. (10.93,3.29)   ;
\draw   (195.28,184.68) .. controls (195.3,189.35) and (197.64,191.67) .. (202.31,191.65) -- (249.02,191.44) .. controls (255.68,191.41) and (259.02,193.73) .. (259.05,198.4) .. controls (259.02,193.73) and (262.34,191.39) .. (269.01,191.36)(266.01,191.37) -- (315.72,191.15) .. controls (320.39,191.13) and (322.71,188.79) .. (322.69,184.12) ;
\draw  [line width=0.75]  (377.54,48.43) -- (517.76,48.43) -- (517.76,69.88) -- (377.54,69.88) -- cycle ;
\draw [line width=0.75]    (386.18,70.86) -- (386.76,139.14) ;
\draw [shift={(386.41,98)}, rotate = 89.52] [color={rgb, 255:red, 0; green, 0; blue, 0 }  ][line width=0.75]    (10.93,-3.29) .. controls (6.95,-1.4) and (3.31,-0.3) .. (0,0) .. controls (3.31,0.3) and (6.95,1.4) .. (10.93,3.29)   ;
\draw [line width=0.75]    (405.7,70.86) -- (406.27,139.14) ;
\draw [shift={(405.93,98)}, rotate = 89.52] [color={rgb, 255:red, 0; green, 0; blue, 0 }  ][line width=0.75]    (10.93,-3.29) .. controls (6.95,-1.4) and (3.31,-0.3) .. (0,0) .. controls (3.31,0.3) and (6.95,1.4) .. (10.93,3.29)   ;
\draw [line width=0.75]    (426.66,70.39) -- (427.23,138.67) ;
\draw [shift={(426.88,97.53)}, rotate = 89.52] [color={rgb, 255:red, 0; green, 0; blue, 0 }  ][line width=0.75]    (10.93,-3.29) .. controls (6.95,-1.4) and (3.31,-0.3) .. (0,0) .. controls (3.31,0.3) and (6.95,1.4) .. (10.93,3.29)   ;
\draw [line width=0.75]    (489.21,70.86) -- (489.79,139.14) ;
\draw [shift={(489.44,98)}, rotate = 89.52] [color={rgb, 255:red, 0; green, 0; blue, 0 }  ][line width=0.75]    (10.93,-3.29) .. controls (6.95,-1.4) and (3.31,-0.3) .. (0,0) .. controls (3.31,0.3) and (6.95,1.4) .. (10.93,3.29)   ;
\draw [line width=0.75]    (508.29,113.21) -- (508.78,139.42) ;
\draw [shift={(508.4,119.32)}, rotate = 88.93] [color={rgb, 255:red, 0; green, 0; blue, 0 }  ][line width=0.75]    (10.93,-3.29) .. controls (6.95,-1.4) and (3.31,-0.3) .. (0,0) .. controls (3.31,0.3) and (6.95,1.4) .. (10.93,3.29)   ;
\draw [line width=0.75]    (508.29,70.59) -- (508.42,107.39) ;
\draw [shift={(508.33,81.99)}, rotate = 89.8] [color={rgb, 255:red, 0; green, 0; blue, 0 }  ][line width=0.75]    (10.93,-3.29) .. controls (6.95,-1.4) and (3.31,-0.3) .. (0,0) .. controls (3.31,0.3) and (6.95,1.4) .. (10.93,3.29)   ;
\draw [line width=0.75]    (508.42,107.39) -- (552.72,69.75) ;
\draw [line width=0.75]    (508.29,113.21) -- (559.41,71.37) ;
\draw    (551.36,69.48) .. controls (550.03,69.95) and (543.29,58.18) .. (543.29,48.43) ;
\draw [shift={(548.44,65.99)}, rotate = 244.81] [color={rgb, 255:red, 0; green, 0; blue, 0 }  ][line width=0.75]    (10.93,-3.29) .. controls (6.95,-1.4) and (3.31,-0.3) .. (0,0) .. controls (3.31,0.3) and (6.95,1.4) .. (10.93,3.29)   ;
\draw    (559.41,71.37) .. controls (569.93,71.84) and (593.6,62.1) .. (596.97,48.43) ;
\draw [shift={(586.59,61.67)}, rotate = 147.86] [color={rgb, 255:red, 0; green, 0; blue, 0 }  ][line width=0.75]    (10.93,-3.29) .. controls (6.95,-1.4) and (3.31,-0.3) .. (0,0) .. controls (3.31,0.3) and (6.95,1.4) .. (10.93,3.29)   ;
\draw   (376.63,139.63) -- (518.85,139.63) -- (518.85,157.8) -- (376.63,157.8) -- cycle ;
\draw [line width=0.75]    (386.54,158.73) -- (386.8,183.2) ;
\draw [shift={(386.59,163.97)}, rotate = 89.39] [color={rgb, 255:red, 0; green, 0; blue, 0 }  ][line width=0.75]    (10.93,-3.29) .. controls (6.95,-1.4) and (3.31,-0.3) .. (0,0) .. controls (3.31,0.3) and (6.95,1.4) .. (10.93,3.29)   ;
\draw [line width=0.75]    (407.04,157.42) -- (407.53,183.63) ;
\draw [shift={(407.16,163.53)}, rotate = 88.93] [color={rgb, 255:red, 0; green, 0; blue, 0 }  ][line width=0.75]    (10.93,-3.29) .. controls (6.95,-1.4) and (3.31,-0.3) .. (0,0) .. controls (3.31,0.3) and (6.95,1.4) .. (10.93,3.29)   ;
\draw [line width=0.75]    (426.45,157.85) -- (426.94,184.07) ;
\draw [shift={(426.56,163.96)}, rotate = 88.93] [color={rgb, 255:red, 0; green, 0; blue, 0 }  ][line width=0.75]    (10.93,-3.29) .. controls (6.95,-1.4) and (3.31,-0.3) .. (0,0) .. controls (3.31,0.3) and (6.95,1.4) .. (10.93,3.29)   ;
\draw [line width=0.75]    (489.98,157.85) -- (490.47,184.07) ;
\draw [shift={(490.09,163.96)}, rotate = 88.93] [color={rgb, 255:red, 0; green, 0; blue, 0 }  ][line width=0.75]    (10.93,-3.29) .. controls (6.95,-1.4) and (3.31,-0.3) .. (0,0) .. controls (3.31,0.3) and (6.95,1.4) .. (10.93,3.29)   ;
\draw [line width=0.75]    (509.43,158.29) -- (509.92,184.51) ;
\draw [shift={(509.55,164.4)}, rotate = 88.93] [color={rgb, 255:red, 0; green, 0; blue, 0 }  ][line width=0.75]    (10.93,-3.29) .. controls (6.95,-1.4) and (3.31,-0.3) .. (0,0) .. controls (3.31,0.3) and (6.95,1.4) .. (10.93,3.29)   ;
\draw   (386.7,191.61) .. controls (386.73,196.28) and (389.07,198.6) .. (393.74,198.58) -- (438.35,198.37) .. controls (445.02,198.34) and (448.36,200.65) .. (448.38,205.32) .. controls (448.36,200.65) and (451.68,198.3) .. (458.35,198.27)(455.35,198.29) -- (502.96,198.06) .. controls (507.63,198.04) and (509.95,195.7) .. (509.92,191.03) ;
\draw   (533.79,38.93) .. controls (533.79,33.94) and (538.03,29.9) .. (543.27,29.9) .. controls (548.51,29.9) and (552.75,33.94) .. (552.75,38.93) .. controls (552.75,43.92) and (548.51,47.96) .. (543.27,47.96) .. controls (538.03,47.96) and (533.79,43.92) .. (533.79,38.93) -- cycle ;
\draw   (587.26,40) .. controls (587.26,35.01) and (591.51,30.97) .. (596.74,30.97) .. controls (601.98,30.97) and (606.22,35.01) .. (606.22,40) .. controls (606.22,44.98) and (601.98,49.02) .. (596.74,49.02) .. controls (591.51,49.02) and (587.26,44.98) .. (587.26,40) -- cycle ;

\draw (252.16,51.76) node [anchor=north west][inner sep=0.75pt]  [font=\Large]  {$\epsilon $};
\draw (255.2,114.86) node [anchor=north west][inner sep=0.75pt]  [font=\normalsize]  {$\cdots $};
\draw (155.61,97.08) node [anchor=north west][inner sep=0.75pt]  [font=\normalsize]  {$\frac{1}{N}$};
\draw (344.48,106.22) node [anchor=north west][inner sep=0.75pt]  [font=\normalsize]  {$=$};
\draw (234.05,198.9) node [anchor=north west][inner sep=0.75pt]  [font=\small]  {$N$ lines};
\draw (440.03,51.76) node [anchor=north west][inner sep=0.75pt]  [font=\Large]  {$\epsilon $};
\draw (446.9,98.05) node [anchor=north west][inner sep=0.75pt]  [font=\normalsize]  {$\cdots $};
\draw (440.58,138.52) node [anchor=north west][inner sep=0.75pt]  [font=\normalsize]  {$A$};
\draw (448.05,165.02) node [anchor=north west][inner sep=0.75pt]  [font=\normalsize]  {$\cdots $};
\draw (423.45,206.85) node [anchor=north west][inner sep=0.75pt]  [font=\small]  {$N$ lines};
\draw (537.61,31.78) node [anchor=north west][inner sep=0.75pt]  [font=\footnotesize]  {$i$};
\draw (592.17,32.36) node [anchor=north west][inner sep=0.75pt]  [font=\footnotesize]  {$i$};

\end{tikzpicture}

 }
\end{equation}
Since both sides are totally anti-symmetric tensors in $\underbrace{\left(V^{*} \otimes V^{*} \otimes \cdots V^{*}\right)}_{N}$, to show this equality (for fixed $i\in\{1,2,\cdots,N\}$), it is enough to show the equality after contracting with $e_{1}\otimes e_{2}\cdots e_{i-1}\otimes  \hat{e_{i}}\otimes e_{i+1} \cdots\otimes e_{N}\otimes e_{i}$. Here, the notation $\hat{e_{i}}$ means $e_{i}$ is missing in the $i$th position. To show this, we substitute the recursive relation (\ref{Recursive of A}) of $A$ into the right-hand side of (\ref{tensor in proof}). Only the first diagram on the right-hand side of (\ref{Recursive of A}) has nonzero contribute when contracting with $e_{1}\otimes e_{2}\cdots e_{i-1}\otimes  \hat{e_{i}}\otimes e_{i+1} \cdots\otimes e_{N}\otimes e_{i}$. Using this, both sides of the equality that we want to show are equal:
\begin{center}
    \tikzset{every picture/.style={line width=0.75pt}} 



 }
\end{equation}
\end{corollary}
\begin{proof}
    Apply the recursive relation (\ref{Recursive of A}) to (\ref{5.41}).
\end{proof}
The Hilbert space $\mathcal{H}$ of the matrix model with level $k$ is generated linearly by the elements which is a product of 
\begin{align}
&t_{n}:=\mathrm{Tr}({Z^{\dagger}}^{n})\\
&C(n_{1},n_{2},\cdots,n_{N}):=\epsilon^{i_{1}i_{2}\cdots i_{N}}(\lambda^{\dagger}{Z^{\dagger}}^{n_{1}})_{i_{1}}(\lambda^{\dagger}{Z^{\dagger}}^{n_{2}})_{i_{2}}\cdots(\lambda^{\dagger}{Z^{\dagger}}^{n_{N}})_{i_{N}}
\end{align}
in which the elements $C(n_{1},n_{2},\cdots,n_{N})$ appear $k$ times.  

These vectors are not linearly independent, and they obey certain algebraic relations. We do not assume any algebraic relations among the components ${\lambda^{\dagger}}^{i}$, ${Z^{\dagger}}^{i}_{j}$. So these algebraic relations are generated only by the relations in the tensor algebra generated by $\epsilon^{i_{1}i_{2}\cdots i_{N}}$ and $\delta^{i}_{j}$. The first kind of relation is generated by (\ref{anti N+1=0}), where we put one $Z^{\dagger}$ on each incoming line in (\ref{anti N+1=0}). By taking trace, we get:
\begin{equation}
    \hbox{

\tikzset{every picture/.style={line width=0.75pt}} 

\begin{tikzpicture}[x=0.75pt,y=0.75pt,yscale=-1,xscale=1]

\draw [line width=0.75]    (251.62,33.87) -- (251.9,52.09) ;
\draw [shift={(251.68,38.38)}, rotate = 89.12] [color={rgb, 255:red, 0; green, 0; blue, 0 }  ][line width=0.75]    (6.56,-1.97) .. controls (4.17,-0.84) and (1.99,-0.18) .. (0,0) .. controls (1.99,0.18) and (4.17,0.84) .. (6.56,1.97)   ;
\draw [line width=0.75]    (263.37,32.8) -- (263.84,52.09) ;
\draw [shift={(263.5,37.85)}, rotate = 88.61] [color={rgb, 255:red, 0; green, 0; blue, 0 }  ][line width=0.75]    (6.56,-1.97) .. controls (4.17,-0.84) and (1.99,-0.18) .. (0,0) .. controls (1.99,0.18) and (4.17,0.84) .. (6.56,1.97)   ;
\draw [line width=0.75]    (276.25,33.07) -- (276.53,52.09) ;
\draw [shift={(276.33,37.98)}, rotate = 89.16] [color={rgb, 255:red, 0; green, 0; blue, 0 }  ][line width=0.75]    (6.56,-1.97) .. controls (4.17,-0.84) and (1.99,-0.18) .. (0,0) .. controls (1.99,0.18) and (4.17,0.84) .. (6.56,1.97)   ;
\draw [line width=0.75]    (313.49,33.87) -- (313.87,52.09) ;
\draw [shift={(313.58,38.38)}, rotate = 88.83] [color={rgb, 255:red, 0; green, 0; blue, 0 }  ][line width=0.75]    (6.56,-1.97) .. controls (4.17,-0.84) and (1.99,-0.18) .. (0,0) .. controls (1.99,0.18) and (4.17,0.84) .. (6.56,1.97)   ;
\draw [line width=0.75]    (324.41,33.87) -- (324.69,52.09) ;
\draw [shift={(324.48,38.38)}, rotate = 89.12] [color={rgb, 255:red, 0; green, 0; blue, 0 }  ][line width=0.75]    (6.56,-1.97) .. controls (4.17,-0.84) and (1.99,-0.18) .. (0,0) .. controls (1.99,0.18) and (4.17,0.84) .. (6.56,1.97)   ;
\draw   (246.71,51.79) -- (332.06,51.79) -- (332.06,64.47) -- (246.71,64.47) -- cycle ;
\draw [line width=0.75]    (252.27,64.68) -- (252.32,86.83) ;
\draw [shift={(252.28,71.16)}, rotate = 89.86] [color={rgb, 255:red, 0; green, 0; blue, 0 }  ][line width=0.75]    (6.56,-1.97) .. controls (4.17,-0.84) and (1.99,-0.18) .. (0,0) .. controls (1.99,0.18) and (4.17,0.84) .. (6.56,1.97)   ;
\draw [line width=0.75]    (264.43,64.93) -- (264.59,84.95) ;
\draw [shift={(264.47,70.34)}, rotate = 89.54] [color={rgb, 255:red, 0; green, 0; blue, 0 }  ][line width=0.75]    (6.56,-1.97) .. controls (4.17,-0.84) and (1.99,-0.18) .. (0,0) .. controls (1.99,0.18) and (4.17,0.84) .. (6.56,1.97)   ;
\draw [line width=0.75]    (276.14,65.21) -- (276.53,84.95) ;
\draw [shift={(276.25,70.49)}, rotate = 88.86] [color={rgb, 255:red, 0; green, 0; blue, 0 }  ][line width=0.75]    (6.56,-1.97) .. controls (4.17,-0.84) and (1.99,-0.18) .. (0,0) .. controls (1.99,0.18) and (4.17,0.84) .. (6.56,1.97)   ;
\draw [line width=0.75]    (314.14,65.5) -- (314.24,84.95) ;
\draw [shift={(314.17,70.63)}, rotate = 89.72] [color={rgb, 255:red, 0; green, 0; blue, 0 }  ][line width=0.75]    (6.56,-1.97) .. controls (4.17,-0.84) and (1.99,-0.18) .. (0,0) .. controls (1.99,0.18) and (4.17,0.84) .. (6.56,1.97)   ;
\draw [line width=0.75]    (325.29,65.5) -- (325.81,84.95) ;
\draw [shift={(325.42,70.63)}, rotate = 88.45] [color={rgb, 255:red, 0; green, 0; blue, 0 }  ][line width=0.75]    (6.56,-1.97) .. controls (4.17,-0.84) and (1.99,-0.18) .. (0,0) .. controls (1.99,0.18) and (4.17,0.84) .. (6.56,1.97)   ;
\draw  [fill={rgb, 255:red, 0; green, 0; blue, 0 }  ,fill opacity=1 ] (249.86,84.41) .. controls (249.86,83.07) and (250.96,81.98) .. (252.32,81.98) .. controls (253.68,81.98) and (254.78,83.07) .. (254.78,84.41) .. controls (254.78,85.74) and (253.68,86.83) .. (252.32,86.83) .. controls (250.96,86.83) and (249.86,85.74) .. (249.86,84.41) -- cycle ;
\draw    (236.22,61.74) .. controls (235.38,89.33) and (247.7,110.5) .. (252.32,86.83) ;
\draw    (236.22,61.74) .. controls (236.5,16.73) and (251.62,22.35) .. (251.62,33.87) ;
\draw    (263.37,32.8) .. controls (263.65,15.92) and (226.42,0.65) .. (227.82,59.32) ;
\draw    (227.82,59.32) .. controls (228.94,110.5) and (261.13,120.68) .. (264.59,86.83) ;
\draw    (276.25,33.07) .. controls (277.09,16.46) and (244.06,-9.26) .. (224.46,28.25) ;
\draw    (224.46,28.25) .. controls (208.78,61.74) and (227.26,119.07) .. (252.74,115.05) ;
\draw    (276.53,84.95) .. controls (276.53,107.01) and (257.5,114.78) .. (252.74,115.05) ;
\draw    (313.49,33.87) .. controls (314.61,15.12) and (331.41,-4.17) .. (342.33,23.96) ;
\draw    (342.33,23.96) .. controls (345.97,41.91) and (347.37,92.01) .. (339.25,109.42) ;
\draw    (339.25,109.42) .. controls (331.69,124.16) and (314.33,104.6) .. (314.24,84.95) ;
\draw    (338.69,58.52) .. controls (339.81,89.06) and (329.45,120.41) .. (325.81,86.83) ;
\draw    (324.41,33.87) .. controls (331.13,-1.49) and (338.97,39.5) .. (338.69,58.52) ;
\draw  [fill={rgb, 255:red, 0; green, 0; blue, 0 }  ,fill opacity=1 ] (262.13,84.95) .. controls (262.13,83.62) and (263.23,82.53) .. (264.59,82.53) .. controls (265.95,82.53) and (267.05,83.62) .. (267.05,84.95) .. controls (267.05,86.29) and (265.95,87.38) .. (264.59,87.38) .. controls (263.23,87.38) and (262.13,86.29) .. (262.13,84.95) -- cycle ;
\draw  [fill={rgb, 255:red, 0; green, 0; blue, 0 }  ,fill opacity=1 ] (274.07,84.95) .. controls (274.07,83.62) and (275.17,82.53) .. (276.53,82.53) .. controls (277.89,82.53) and (279,83.62) .. (279,84.95) .. controls (279,86.29) and (277.89,87.38) .. (276.53,87.38) .. controls (275.17,87.38) and (274.07,86.29) .. (274.07,84.95) -- cycle ;
\draw  [fill={rgb, 255:red, 0; green, 0; blue, 0 }  ,fill opacity=1 ] (311.78,84.95) .. controls (311.78,83.62) and (312.88,82.53) .. (314.24,82.53) .. controls (315.6,82.53) and (316.7,83.62) .. (316.7,84.95) .. controls (316.7,86.29) and (315.6,87.38) .. (314.24,87.38) .. controls (312.88,87.38) and (311.78,86.29) .. (311.78,84.95) -- cycle ;
\draw  [fill={rgb, 255:red, 0; green, 0; blue, 0 }  ,fill opacity=1 ] (323.35,84.95) .. controls (323.35,83.62) and (324.45,82.53) .. (325.81,82.53) .. controls (327.17,82.53) and (328.27,83.62) .. (328.27,84.95) .. controls (328.27,86.29) and (327.17,87.38) .. (325.81,87.38) .. controls (324.45,87.38) and (323.35,86.29) .. (323.35,84.95) -- cycle ;

\draw (282.74,52.20) node [anchor=north west][inner sep=0.75pt]  [font=\small]  {$A$};
\draw (279.59,66.77) node [anchor=north west][inner sep=0.75pt]  [font=\normalsize]  {$\cdots $};
\draw (350.03,48.49) node [anchor=north west][inner sep=0.75pt]    {$=\ 0$};

\end{tikzpicture}

    }
\end{equation}
Here, we use the diagrammatic notation introduced in Appendix (\ref{sec: Diagrammatic notations}). In particular, we denote the generators of the algebra $Z^i_j$, ${Z^{\dagger}}^i_j$, $\lambda_a$, ${\lambda^{\dagger}}^a$ by
\begin{tikzpicture}[decoration={markings, 
    mark= at position 0.30 with {\arrow{stealth}},
    mark= at position 0.88 with {\arrow{stealth}}}
] 
 
\draw [postaction={decorate}] (-0.5,0) -- (0.5,0);\draw[fill=white,line width=1pt](0,0)circle(0.5ex);\end{tikzpicture}, \begin{tikzpicture}[decoration={markings, 
    mark= at position 0.30 with {\arrow{stealth}},
    mark= at position 0.88 with {\arrow{stealth}}}
] 
 
\draw[fill=black,line width=1pt](0,0)circle(0.5ex);\draw [postaction={decorate}] (-0.5,0) -- (0.5,0);\end{tikzpicture},\begin{tikzpicture}[decoration={markings, 
    mark= at position 0.50 with {\arrow{stealth}}}
] \draw [postaction={decorate}] (-0.3,0) -- (0.5,0); \draw (0.8,0) node {$\lambda_a$};\end{tikzpicture}, and \begin{tikzpicture}[decoration={markings, 
    mark= at position 0.50 with {\arrow{stealth}}}
] \draw [postaction={decorate}] (-0.5,0) -- (0.3,0); \draw (-0.8,0) node {${\lambda^{\dagger}}^a$};\end{tikzpicture}. This expression is a sum indexed by all permutations of $N+1$ lines, and each term is a multi-trace term:
\begin{equation}
\propto \mathrm{Tr}({Z^{\dagger}})^{i_{1}}\mathrm{Tr}({Z^{\dagger}}^2)^{i_{2}}\cdots \mathrm{Tr}({Z^{\dagger}}^{N+1})^{i_{N+1}}
\end{equation}
where $(i_{1},i_{2},\cdots,i_{N+1})$ is the cycle type of the corresponding permutation in the antisymmetrization tensor $A$. There are $N!$ single trace terms with the same sign $(-1)^{N}$:
\begin{equation}
    N!(-1)^{N}\mathrm{Tr}({Z^{\dagger}}^{N+1})
\end{equation}
which corresponds to the cycle type $(i_{1}=0,i_{2}=0,\cdots,i_{N+1}=1)$. This shows that $\mathrm{Tr}({Z^{\dagger}}^{N+1})$ can be written as a polynomials in $\mathrm{Tr}({Z^{\dagger}}^{1}),\mathrm{Tr}({Z^{\dagger}}^{2})\cdots \mathrm{Tr}({Z^{\dagger}}^{N})$. By considering a totally antisymmetric tensor with $N+i+1$ inputs and outputs, one can show inductively show that all $\mathrm{Tr}({Z^{\dagger}}^{N+i}),i\geq 1$ can be written as a polynomials in $\mathrm{Tr}({Z^{\dagger}}^{1}),\mathrm{Tr}({Z^{\dagger}}^{2})\cdots \mathrm{Tr}({Z^{\dagger}}^{N})$.

There are also algebraic relations between $\mathrm{Tr}({Z^{\dagger}}^{n})$ and $C(n_{1},n_{2},\cdots,n_{N})$ which is generated by the relation (\ref{bubble}). We illustrate this by the following example: For $N=3$, we have the following relation:
\begin{equation}
C(0,1,2)\mathrm{Tr}({Z^{\dagger}}^{2})=C(2,1,2)+C(0,3,2)+C(0,1,4)=-C(0,2,3)+C(0,1,4)
\end{equation}
This equation comes from the relation (\ref{bubble}). Explicitly, we connect all the lines in the following diagram:

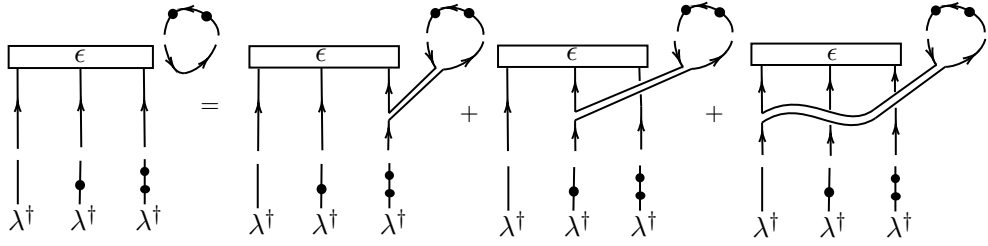
\begin{center}

\tikzset{every picture/.style={line width=0.75pt}} 



\end{center}
In general, we have:
\begin{equation}\label{ti C}
C(n_{1},n_{2},\cdots,n_{N})\mathrm{Tr}({Z^{\dagger}}^{n})=C(n_{1}+n,n_{2},\cdots,n_{N})+C(n_{1},n_{2}+n,\cdots,n_{N})+\cdots+C(n_{1},n_{2},\cdots,n_{N}+n)
\end{equation} 
For the case $0\leq n \leq N$, and $n_{j}=j-1$,  the first $N-n$ terms on the right hand side of (\ref{ti C}) are zero due to the totally anti-symmetric property of $C(n_{1},n_{2},\cdots,n_{N})$. (\ref{ti C}) becomes:
\begin{align}\label{ti C 2}
    C(0,1,\cdots,N-1)\mathrm{Tr}({Z^{\dagger}}^{n})=&C(\cdots,N-n+n,\cdots,N-1)+C(\cdots,N-n,N-(n+1)+n\cdots,N-1) \\ \nonumber
    &+\cdots+C(\cdots,N-n,\cdots,N-1+n)
\end{align}
It turns out that using relations (\ref{ti C}), we can write $C(n_1,n_2,\cdots,n_N)$ into the form:
$$G(\{\mathrm{Tr}({Z^{\dagger}}^n)\})C(0,1,2,\cdots,N-1),$$
where $G(\{\mathrm{Tr}({Z^{\dagger}}^n)\})$ is a polynomial in $\mathrm{Tr}({Z^{\dagger}}^n), n=1,2,\cdots,N$. Moreover, the terms $\mathrm{Tr}({Z^{\dagger}}^n), n=1,2,\cdots,N$ are algebraically independent. We will explain these facts in Appendix \ref{App B}, where we make connection of the terms $C(n_1,n_2,\cdots,n_N)$ and $\mathrm{Tr}({Z^{\dagger}}^n)$ to the Schur polynomials of $N$ variables and to the power sum symmetric polynomials in $N$ variables respectively. In Appendix \ref{App B}, we will show that (\ref{ti C}) is equivalent to the Murnaghan–Nakayama rule, which is a combinatorial rule to write power sum polynomials into Schur polynomials and vice versa. Now we have the following proposition:

\begin{proposition}[Proposition \ref{Prop Basis}]
The elements:
\begin{equation}\label{Basis non re}
\ket{c_{1},c_{2},\cdots,c_{N}}:=\prod_{i=1}^{N}\left(\frac{1}{\sqrt{N}}\mathrm{Tr}({Z^{\dagger}}^{i})\right)^{c_{i}}\ket{\mathrm{ground}}, c_{i}\geq 0
\end{equation}
form a basis for the Hilbert space of physical states. 
\end{proposition}

The basis (\ref{Basis non re}) is a rescaling of the basis found in \cite{Hellerman-Raamsdonk}. They found this basis by making connection to Schur and power sum symmetric polynomials.  Let $\mathcal H_N$ denote the infinite dimensional representation linearly spanned by the above basis.

We introduce variables $p_{i}, 1\leq i \leq N$, and identify the states $\ket{c_{1},c_{2},\cdots,c_{N}}$ with a monomial $p_{1}^{c_{1}}p_{2}^{c_{2}}\cdots p_{N}^{c_{N}}$ in the variables $p_{i}, 1\leq i \leq N$. With these identifications, we have an isomorphism $\mathcal{H}_N\cong \mathbb{C}[p_{1},p_{2},\cdots,p_{N}].$ We define operators: $\hat{p_{i}}$ and $\frac{\partial}{\partial p_{i}}$.
\begin{align*}
\hat{p_{i}}\ket{c_{1},c_{2},\cdots,c_{N}}:=\ket{\cdots,c_{i}+1,\cdots}\\
\frac{\partial}{\partial p_{i}}\ket{c_{1},c_{2},\cdots,c_{N}}:=c_{i}\ket{\cdots,c_{i}-1,\cdots}
\end{align*}
where $\ket{c_{1},c_{2},\cdots,c_{N}}=0$, if $c_{i}<0$ for any $i$. Under the isomorphism $\mathcal{H}_N\cong \mathbb{C}[p_{1},p_{2},\cdots,p_{N}]$,  we define operators: $\hat{p_{i}}$ and $\frac{\partial}{\partial p_{i}}$. The operators $\hat{p_{i}}$ and $\frac{\partial}{\partial p_{i}}$ are just multiplication by $p_{i}$ and partial derivative of $p_{i}$. So we will omit the 'hat' in $\hat{p_{i}}$ and do not distinguish the variables and operators : $p_{i}=\hat{p_{i}}$.\\

In the following, we will consider the actions of the operators $t_{m,n}$ on the Hilbert space using the basis (\ref{Basis non re}) and identify the leading terms in the large $N$ limit. Let us first consider the simple example: $t_{1,0}=\operatorname{tr}(Z)$. We express
\begin{equation}\label{t10}
t_{1,0}\ket{c_{1},c_{2},\cdots,c_{N}}
\end{equation}
graphically, this is:
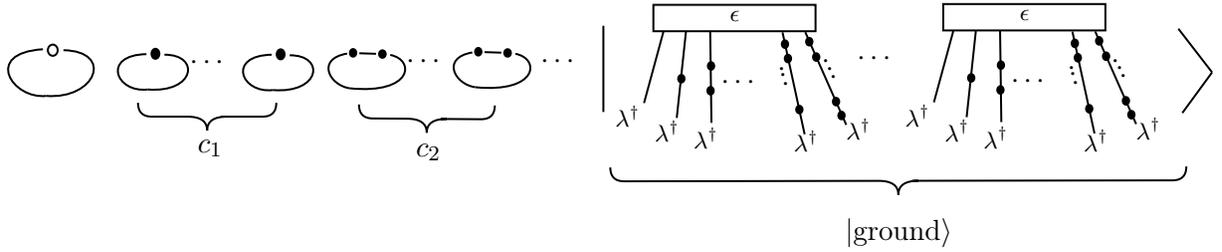
\begin{center}

\tikzset{every picture/.style={line width=0.75pt}} 

\begin{tikzpicture}[x=0.75pt,y=0.75pt,yscale=-1,xscale=1]

\draw   (24.33,44.82) .. controls (24.33,43.26) and (25.44,42) .. (26.81,42) .. controls (28.18,42) and (29.29,43.26) .. (29.29,44.82) .. controls (29.29,46.38) and (28.18,47.65) .. (26.81,47.65) .. controls (25.44,47.65) and (24.33,46.38) .. (24.33,44.82) -- cycle ;
\draw  [fill={rgb, 255:red, 0; green, 0; blue, 0 }  ,fill opacity=1 ] (76.61,46.62) .. controls (76.61,45.23) and (77.49,44.11) .. (78.59,44.11) .. controls (79.68,44.11) and (80.57,45.23) .. (80.57,46.62) .. controls (80.57,48.01) and (79.68,49.14) .. (78.59,49.14) .. controls (77.49,49.14) and (76.61,48.01) .. (76.61,46.62) -- cycle ;
\draw    (32.19,44.82) .. controls (57.84,45.3) and (48.74,69.78) .. (23.5,67.89) ;
\draw    (21.84,44.35) .. controls (-8.78,46.24) and (6.53,69.78) .. (23.5,67.89) ;
\draw    (304.5,32.5) -- (304.5,76) ;
\draw    (83.14,47.1) .. controls (101.29,46.93) and (98.82,66.56) .. (74.71,64.64) ;
\draw    (74.05,46.82) .. controls (51.29,47.21) and (58.51,66.56) .. (74.71,64.64) ;
\draw   (329.72,21.92) -- (411.5,21.92) -- (411.5,35.25) -- (329.72,35.25) -- cycle ;
\draw    (334.76,35.59) -- (325,71.5) ;
\draw    (346.14,35.59) -- (341.5,79) ;
\draw    (358.37,35.3) -- (359,81.5) ;
\draw    (394.85,35.84) -- (406,87.25) ;
\draw    (406.35,35.76) -- (427,82.5) ;
\draw  [fill={rgb, 255:red, 0; green, 0; blue, 0 }  ,fill opacity=1 ] (342.23,59.21) .. controls (342.23,58.15) and (342.94,57.3) .. (343.82,57.3) .. controls (344.7,57.3) and (345.41,58.15) .. (345.41,59.21) .. controls (345.41,60.26) and (344.7,61.12) .. (343.82,61.12) .. controls (342.94,61.12) and (342.23,60.26) .. (342.23,59.21) -- cycle ;
\draw  [fill={rgb, 255:red, 0; green, 0; blue, 0 }  ,fill opacity=1 ] (357.21,52.77) .. controls (357.21,51.71) and (357.93,50.86) .. (358.8,50.86) .. controls (359.68,50.86) and (360.39,51.71) .. (360.39,52.77) .. controls (360.39,53.82) and (359.68,54.68) .. (358.8,54.68) .. controls (357.93,54.68) and (357.21,53.82) .. (357.21,52.77) -- cycle ;
\draw  [fill={rgb, 255:red, 0; green, 0; blue, 0 }  ,fill opacity=1 ] (356.96,65.27) .. controls (356.96,64.21) and (357.68,63.36) .. (358.55,63.36) .. controls (359.43,63.36) and (360.14,64.21) .. (360.14,65.27) .. controls (360.14,66.32) and (359.43,67.18) .. (358.55,67.18) .. controls (357.68,67.18) and (356.96,66.32) .. (356.96,65.27) -- cycle ;
\draw  [fill={rgb, 255:red, 0; green, 0; blue, 0 }  ,fill opacity=1 ] (394.5,41.77) .. controls (394.5,40.71) and (395.21,39.86) .. (396.09,39.86) .. controls (396.97,39.86) and (397.68,40.71) .. (397.68,41.77) .. controls (397.68,42.82) and (396.97,43.68) .. (396.09,43.68) .. controls (395.21,43.68) and (394.5,42.82) .. (394.5,41.77) -- cycle ;
\draw  [fill={rgb, 255:red, 0; green, 0; blue, 0 }  ,fill opacity=1 ] (396,49.02) .. controls (396,47.96) and (396.71,47.11) .. (397.59,47.11) .. controls (398.47,47.11) and (399.18,47.96) .. (399.18,49.02) .. controls (399.18,50.07) and (398.47,50.93) .. (397.59,50.93) .. controls (396.71,50.93) and (396,50.07) .. (396,49.02) -- cycle ;
\draw  [fill={rgb, 255:red, 0; green, 0; blue, 0 }  ,fill opacity=1 ] (401.75,74.77) .. controls (401.75,73.71) and (402.46,72.86) .. (403.34,72.86) .. controls (404.22,72.86) and (404.93,73.71) .. (404.93,74.77) .. controls (404.93,75.82) and (404.22,76.68) .. (403.34,76.68) .. controls (402.46,76.68) and (401.75,75.82) .. (401.75,74.77) -- cycle ;
\draw  [fill={rgb, 255:red, 0; green, 0; blue, 0 }  ,fill opacity=1 ] (406.75,40.52) .. controls (406.75,39.46) and (407.46,38.61) .. (408.34,38.61) .. controls (409.22,38.61) and (409.93,39.46) .. (409.93,40.52) .. controls (409.93,41.57) and (409.22,42.43) .. (408.34,42.43) .. controls (407.46,42.43) and (406.75,41.57) .. (406.75,40.52) -- cycle ;
\draw  [fill={rgb, 255:red, 0; green, 0; blue, 0 }  ,fill opacity=1 ] (410.25,48.27) .. controls (410.25,47.21) and (410.96,46.36) .. (411.84,46.36) .. controls (412.72,46.36) and (413.43,47.21) .. (413.43,48.27) .. controls (413.43,49.32) and (412.72,50.18) .. (411.84,50.18) .. controls (410.96,50.18) and (410.25,49.32) .. (410.25,48.27) -- cycle ;
\draw  [fill={rgb, 255:red, 0; green, 0; blue, 0 }  ,fill opacity=1 ] (420.25,70.77) .. controls (420.25,69.71) and (420.96,68.86) .. (421.84,68.86) .. controls (422.72,68.86) and (423.43,69.71) .. (423.43,70.77) .. controls (423.43,71.82) and (422.72,72.68) .. (421.84,72.68) .. controls (420.96,72.68) and (420.25,71.82) .. (420.25,70.77) -- cycle ;
\draw  [fill={rgb, 255:red, 0; green, 0; blue, 0 }  ,fill opacity=1 ] (423,77.77) .. controls (423,76.71) and (423.71,75.86) .. (424.59,75.86) .. controls (425.47,75.86) and (426.18,76.71) .. (426.18,77.77) .. controls (426.18,78.82) and (425.47,79.68) .. (424.59,79.68) .. controls (423.71,79.68) and (423,78.82) .. (423,77.77) -- cycle ;
\draw    (594.5,34) -- (611.5,56) ;
\draw    (596.5,76.5) -- (611.5,56) ;
\draw   (475.92,21.52) -- (557.7,21.52) -- (557.7,34.85) -- (475.92,34.85) -- cycle ;
\draw    (480.96,35.19) -- (471.2,71.1) ;
\draw    (492.34,35.19) -- (487.7,78.6) ;
\draw    (504.57,34.9) -- (505.2,81.1) ;
\draw    (541.05,35.44) -- (552.2,86.85) ;
\draw    (552.55,35.36) -- (573.2,82.1) ;
\draw  [fill={rgb, 255:red, 0; green, 0; blue, 0 }  ,fill opacity=1 ] (488.43,58.81) .. controls (488.43,57.75) and (489.14,56.9) .. (490.02,56.9) .. controls (490.9,56.9) and (491.61,57.75) .. (491.61,58.81) .. controls (491.61,59.86) and (490.9,60.72) .. (490.02,60.72) .. controls (489.14,60.72) and (488.43,59.86) .. (488.43,58.81) -- cycle ;
\draw  [fill={rgb, 255:red, 0; green, 0; blue, 0 }  ,fill opacity=1 ] (503.41,52.37) .. controls (503.41,51.31) and (504.13,50.46) .. (505,50.46) .. controls (505.88,50.46) and (506.59,51.31) .. (506.59,52.37) .. controls (506.59,53.42) and (505.88,54.28) .. (505,54.28) .. controls (504.13,54.28) and (503.41,53.42) .. (503.41,52.37) -- cycle ;
\draw  [fill={rgb, 255:red, 0; green, 0; blue, 0 }  ,fill opacity=1 ] (503.16,64.87) .. controls (503.16,63.81) and (503.88,62.96) .. (504.75,62.96) .. controls (505.63,62.96) and (506.34,63.81) .. (506.34,64.87) .. controls (506.34,65.92) and (505.63,66.78) .. (504.75,66.78) .. controls (503.88,66.78) and (503.16,65.92) .. (503.16,64.87) -- cycle ;
\draw  [fill={rgb, 255:red, 0; green, 0; blue, 0 }  ,fill opacity=1 ] (540.7,41.37) .. controls (540.7,40.31) and (541.41,39.46) .. (542.29,39.46) .. controls (543.17,39.46) and (543.88,40.31) .. (543.88,41.37) .. controls (543.88,42.42) and (543.17,43.28) .. (542.29,43.28) .. controls (541.41,43.28) and (540.7,42.42) .. (540.7,41.37) -- cycle ;
\draw  [fill={rgb, 255:red, 0; green, 0; blue, 0 }  ,fill opacity=1 ] (542.2,48.62) .. controls (542.2,47.56) and (542.91,46.71) .. (543.79,46.71) .. controls (544.67,46.71) and (545.38,47.56) .. (545.38,48.62) .. controls (545.38,49.67) and (544.67,50.53) .. (543.79,50.53) .. controls (542.91,50.53) and (542.2,49.67) .. (542.2,48.62) -- cycle ;
\draw  [fill={rgb, 255:red, 0; green, 0; blue, 0 }  ,fill opacity=1 ] (547.95,74.37) .. controls (547.95,73.31) and (548.66,72.46) .. (549.54,72.46) .. controls (550.42,72.46) and (551.13,73.31) .. (551.13,74.37) .. controls (551.13,75.42) and (550.42,76.28) .. (549.54,76.28) .. controls (548.66,76.28) and (547.95,75.42) .. (547.95,74.37) -- cycle ;
\draw  [fill={rgb, 255:red, 0; green, 0; blue, 0 }  ,fill opacity=1 ] (552.95,40.12) .. controls (552.95,39.06) and (553.66,38.21) .. (554.54,38.21) .. controls (555.42,38.21) and (556.13,39.06) .. (556.13,40.12) .. controls (556.13,41.17) and (555.42,42.03) .. (554.54,42.03) .. controls (553.66,42.03) and (552.95,41.17) .. (552.95,40.12) -- cycle ;
\draw  [fill={rgb, 255:red, 0; green, 0; blue, 0 }  ,fill opacity=1 ] (556.45,47.87) .. controls (556.45,46.81) and (557.16,45.96) .. (558.04,45.96) .. controls (558.92,45.96) and (559.63,46.81) .. (559.63,47.87) .. controls (559.63,48.92) and (558.92,49.78) .. (558.04,49.78) .. controls (557.16,49.78) and (556.45,48.92) .. (556.45,47.87) -- cycle ;
\draw  [fill={rgb, 255:red, 0; green, 0; blue, 0 }  ,fill opacity=1 ] (566.45,70.37) .. controls (566.45,69.31) and (567.16,68.46) .. (568.04,68.46) .. controls (568.92,68.46) and (569.63,69.31) .. (569.63,70.37) .. controls (569.63,71.42) and (568.92,72.28) .. (568.04,72.28) .. controls (567.16,72.28) and (566.45,71.42) .. (566.45,70.37) -- cycle ;
\draw  [fill={rgb, 255:red, 0; green, 0; blue, 0 }  ,fill opacity=1 ] (569.2,77.37) .. controls (569.2,76.31) and (569.91,75.46) .. (570.79,75.46) .. controls (571.67,75.46) and (572.38,76.31) .. (572.38,77.37) .. controls (572.38,78.42) and (571.67,79.28) .. (570.79,79.28) .. controls (569.91,79.28) and (569.2,78.42) .. (569.2,77.37) -- cycle ;
\draw  [fill={rgb, 255:red, 0; green, 0; blue, 0 }  ,fill opacity=1 ] (176.68,46.16) .. controls (176.68,45.11) and (177.39,44.25) .. (178.27,44.25) .. controls (179.15,44.25) and (179.86,45.11) .. (179.86,46.16) .. controls (179.86,47.22) and (179.15,48.07) .. (178.27,48.07) .. controls (177.39,48.07) and (176.68,47.22) .. (176.68,46.16) -- cycle ;
\draw    (197,47.21) .. controls (209.64,49.57) and (208.58,67.18) .. (176.71,64.07) ;
\draw    (175.26,46.53) .. controls (159.29,49.21) and (166.33,62.68) .. (176.71,64.07) ;
\draw  [fill={rgb, 255:red, 0; green, 0; blue, 0 }  ,fill opacity=1 ] (191.39,46.77) .. controls (191.39,45.73) and (192.14,44.89) .. (193.05,44.89) .. controls (193.97,44.89) and (194.71,45.73) .. (194.71,46.77) .. controls (194.71,47.8) and (193.97,48.64) .. (193.05,48.64) .. controls (192.14,48.64) and (191.39,47.8) .. (191.39,46.77) -- cycle ;
\draw    (181.43,46.36) -- (189.86,46.64) ;
\draw   (69.79,71.36) .. controls (69.79,76.03) and (72.12,78.36) .. (76.79,78.36) -- (94.66,78.36) .. controls (101.33,78.36) and (104.66,80.69) .. (104.66,85.36) .. controls (104.66,80.69) and (107.99,78.36) .. (114.66,78.36)(111.66,78.36) -- (132.54,78.36) .. controls (137.21,78.36) and (139.54,76.03) .. (139.54,71.36) ;
\draw   (181.25,72.21) .. controls (181.22,76.88) and (183.53,79.23) .. (188.2,79.27) -- (205.45,79.39) .. controls (212.12,79.44) and (215.43,81.8) .. (215.4,86.46) .. controls (215.43,81.8) and (218.78,79.49) .. (225.45,79.54)(222.45,79.52) -- (242.7,79.66) .. controls (247.37,79.7) and (249.72,77.39) .. (249.75,72.72) ;
\draw  [fill={rgb, 255:red, 0; green, 0; blue, 0 }  ,fill opacity=1 ] (139.75,46.91) .. controls (139.75,45.52) and (140.64,44.39) .. (141.73,44.39) .. controls (142.83,44.39) and (143.71,45.52) .. (143.71,46.91) .. controls (143.71,48.3) and (142.83,49.42) .. (141.73,49.42) .. controls (140.64,49.42) and (139.75,48.3) .. (139.75,46.91) -- cycle ;
\draw    (146.28,47.39) .. controls (164.43,47.21) and (161.96,66.85) .. (137.86,64.93) ;
\draw    (137.19,47.1) .. controls (114.43,47.5) and (121.66,66.85) .. (137.86,64.93) ;
\draw  [fill={rgb, 255:red, 0; green, 0; blue, 0 }  ,fill opacity=1 ] (239.82,45.59) .. controls (239.82,44.54) and (240.53,43.68) .. (241.41,43.68) .. controls (242.29,43.68) and (243,44.54) .. (243,45.59) .. controls (243,46.64) and (242.29,47.5) .. (241.41,47.5) .. controls (240.53,47.5) and (239.82,46.64) .. (239.82,45.59) -- cycle ;
\draw    (260.14,46.64) .. controls (272.79,49) and (271.73,66.61) .. (239.86,63.5) ;
\draw    (238.41,45.96) .. controls (222.43,48.64) and (229.48,62.11) .. (239.86,63.5) ;
\draw  [fill={rgb, 255:red, 0; green, 0; blue, 0 }  ,fill opacity=1 ] (254.54,46.2) .. controls (254.54,45.16) and (255.28,44.32) .. (256.2,44.32) .. controls (257.11,44.32) and (257.86,45.16) .. (257.86,46.2) .. controls (257.86,47.23) and (257.11,48.07) .. (256.2,48.07) .. controls (255.28,48.07) and (254.54,47.23) .. (254.54,46.2) -- cycle ;
\draw    (244.57,45.79) -- (253,46.07) ;
\draw   (308,104) .. controls (308.01,108.67) and (310.34,111) .. (315.01,110.99) -- (443.26,110.76) .. controls (449.93,110.75) and (453.26,113.08) .. (453.27,117.75) .. controls (453.26,113.08) and (456.59,110.74) .. (463.26,110.73)(460.26,110.73) -- (591.51,110.51) .. controls (596.18,110.5) and (598.51,108.16) .. (598.5,103.49) ;

\draw (366.8,23.4) node [anchor=north west][inner sep=0.75pt]  [font=\footnotesize]  {$\epsilon $};
\draw (362.39,56.17) node [anchor=north west][inner sep=0.75pt]  [font=\normalsize]  {$\cdots $};
\draw (309,71.4) node [anchor=north west][inner sep=0.75pt]  [font=\footnotesize]  {$\lambda ^{\dagger }$};
\draw (329,78.9) node [anchor=north west][inner sep=0.75pt]  [font=\footnotesize]  {$\lambda ^{\dagger }$};
\draw (348.5,82.4) node [anchor=north west][inner sep=0.75pt]  [font=\footnotesize]  {$\lambda ^{\dagger }$};
\draw (399,83.9) node [anchor=north west][inner sep=0.75pt]  [font=\footnotesize]  {$\lambda ^{\dagger }$};
\draw (425.5,77.9) node [anchor=north west][inner sep=0.75pt]  [font=\footnotesize]  {$\lambda ^{\dagger }$};
\draw (390.94,49.54) node [anchor=north west][inner sep=0.75pt]    {$\cdot $};
\draw (391.5,52.65) node [anchor=north west][inner sep=0.75pt]    {$\cdot $};
\draw (392.94,57.29) node [anchor=north west][inner sep=0.75pt]    {$\cdot $};
\draw (413.04,45.92) node [anchor=north west][inner sep=0.75pt]    {$\cdot $};
\draw (415.08,49.87) node [anchor=north west][inner sep=0.75pt]    {$\cdot $};
\draw (416.98,53.67) node [anchor=north west][inner sep=0.75pt]    {$\cdot $};
\draw (431,44.2) node [anchor=north west][inner sep=0.75pt]    {$\cdots $};
\draw (513,23) node [anchor=north west][inner sep=0.75pt]  [font=\footnotesize]  {$\epsilon $};
\draw (508.59,55.77) node [anchor=north west][inner sep=0.75pt]  [font=\normalsize]  {$\cdots $};
\draw (455.2,71) node [anchor=north west][inner sep=0.75pt]  [font=\footnotesize]  {$\lambda ^{\dagger }$};
\draw (475.2,78.5) node [anchor=north west][inner sep=0.75pt]  [font=\footnotesize]  {$\lambda ^{\dagger }$};
\draw (494.7,82) node [anchor=north west][inner sep=0.75pt]  [font=\footnotesize]  {$\lambda ^{\dagger }$};
\draw (545.2,83.5) node [anchor=north west][inner sep=0.75pt]  [font=\footnotesize]  {$\lambda ^{\dagger }$};
\draw (571.7,77.5) node [anchor=north west][inner sep=0.75pt]  [font=\footnotesize]  {$\lambda ^{\dagger }$};
\draw (537.14,49.14) node [anchor=north west][inner sep=0.75pt]    {$\cdot $};
\draw (537.7,52.25) node [anchor=north west][inner sep=0.75pt]    {$\cdot $};
\draw (539.14,56.89) node [anchor=north west][inner sep=0.75pt]    {$\cdot $};
\draw (559.24,45.52) node [anchor=north west][inner sep=0.75pt]    {$\cdot $};
\draw (561.28,49.47) node [anchor=north west][inner sep=0.75pt]    {$\cdot $};
\draw (563.18,53.27) node [anchor=north west][inner sep=0.75pt]    {$\cdot $};
\draw (94.5,46.2) node [anchor=north west][inner sep=0.75pt]    {$\cdots $};
\draw (203.71,45.91) node [anchor=north west][inner sep=0.75pt]    {$\cdots $};
\draw (98.43,88.4) node [anchor=north west][inner sep=0.75pt]    {$c_{1}$};
\draw (208.5,90.4) node [anchor=north west][inner sep=0.75pt]    {$c_{2}$};
\draw (271.43,45.63) node [anchor=north west][inner sep=0.75pt]    {$\cdots $};
\draw (425,127.9) node [anchor=north west][inner sep=0.75pt]    {$\ket{\text{ground}}$};

\end{tikzpicture}

\end{center}
We compute this by Wick theorem: the result is a summation of all Wick contractions, which is given diagrammatically by (\ref{WC}). In a Wick contraction, each white dot has to be contracted with some black dot to the right of it. There are several different types of Wick contractions in (\ref{t10}). The leading term in the limit $N\rightarrow \infty$ is the Wick contraction of $\mathrm{Tr}(Z)$ with $\frac{1}{\sqrt{N}}\mathrm{Tr}({Z^{\dagger}})$:
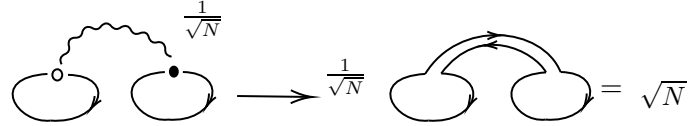
\begin{center}

\tikzset{every picture/.style={line width=0.75pt}} 

\begin{tikzpicture}[x=0.75pt,y=0.75pt,yscale=-1,xscale=1]

\draw   (168.66,36.82) .. controls (168.66,35.26) and (169.77,34) .. (171.14,34) .. controls (172.51,34) and (173.62,35.26) .. (173.62,36.82) .. controls (173.62,38.38) and (172.51,39.65) .. (171.14,39.65) .. controls (169.77,39.65) and (168.66,38.38) .. (168.66,36.82) -- cycle ;
\draw  [fill={rgb, 255:red, 0; green, 0; blue, 0 }  ,fill opacity=1 ] (227.94,35.19) .. controls (227.94,33.81) and (228.83,32.68) .. (229.92,32.68) .. controls (231.02,32.68) and (231.9,33.81) .. (231.9,35.19) .. controls (231.9,36.58) and (231.02,37.71) .. (229.92,37.71) .. controls (228.83,37.71) and (227.94,36.58) .. (227.94,35.19) -- cycle ;
\draw    (176.52,36.82) .. controls (202.18,37.3) and (193.07,61.78) .. (167.83,59.89) ;
\draw [shift={(187.95,54.03)}, rotate = 304.24] [color={rgb, 255:red, 0; green, 0; blue, 0 }  ][line width=0.75]    (6.56,-1.97) .. controls (4.17,-0.84) and (1.99,-0.18) .. (0,0) .. controls (1.99,0.18) and (4.17,0.84) .. (6.56,1.97)   ;
\draw    (166.18,36.35) .. controls (135.56,38.24) and (150.87,61.78) .. (167.83,59.89) ;
\draw    (234.47,35.39) .. controls (258.97,35.87) and (251.07,61.31) .. (226.96,59.39) ;
\draw [shift={(246.24,52.93)}, rotate = 298.77] [color={rgb, 255:red, 0; green, 0; blue, 0 }  ][line width=0.75]    (6.56,-1.97) .. controls (4.17,-0.84) and (1.99,-0.18) .. (0,0) .. controls (1.99,0.18) and (4.17,0.84) .. (6.56,1.97)   ;
\draw    (225.38,35.39) .. controls (196.14,37.31) and (210.76,61.31) .. (226.96,59.39) ;
\draw    (173.33,31.5) .. controls (172.15,29.3) and (172.62,27.62) .. (174.75,26.47) .. controls (176.86,25.72) and (177.59,24.2) .. (176.92,21.93) .. controls (176.55,19.53) and (177.55,18.24) .. (179.92,18.07) .. controls (182.27,18.22) and (183.56,17.22) .. (183.79,15.06) .. controls (184.68,12.75) and (186.25,12.08) .. (188.52,13.04) .. controls (190.49,14.29) and (192.08,14) .. (193.28,12.17) .. controls (195.05,10.42) and (196.7,10.41) .. (198.24,12.13) .. controls (199.81,13.97) and (201.49,14.19) .. (203.26,12.8) .. controls (205.46,11.59) and (207.11,12.01) .. (208.2,14.05) .. controls (209.11,16.11) and (210.68,16.68) .. (212.91,15.75) .. controls (215.22,14.95) and (216.66,15.63) .. (217.25,17.78) .. controls (218,20.11) and (219.45,20.97) .. (221.6,20.36) .. controls (223.93,19.99) and (225.27,21.04) .. (225.6,23.53) .. controls (225.35,25.74) and (226.24,27.06) .. (228.28,27.47) -- (228.33,28) ;
\draw    (261.83,47.5) -- (297.83,47.97) ;
\draw [shift={(299.83,48)}, rotate = 180.75] [color={rgb, 255:red, 0; green, 0; blue, 0 }  ][line width=0.75]    (10.93,-3.29) .. controls (6.95,-1.4) and (3.31,-0.3) .. (0,0) .. controls (3.31,0.3) and (6.95,1.4) .. (10.93,3.29)   ;
\draw    (364.76,36.14) .. controls (390.42,36.61) and (383.57,61.78) .. (358.33,59.89) ;
\draw [shift={(378.08,53.5)}, rotate = 297.05] [color={rgb, 255:red, 0; green, 0; blue, 0 }  ][line width=0.75]    (6.56,-1.97) .. controls (4.17,-0.84) and (1.99,-0.18) .. (0,0) .. controls (1.99,0.18) and (4.17,0.84) .. (6.56,1.97)   ;
\draw    (356.68,36.35) .. controls (326.06,38.24) and (341.37,61.78) .. (358.33,59.89) ;
\draw    (424.97,35.39) .. controls (449.47,35.87) and (441.57,61.31) .. (417.46,59.39) ;
\draw [shift={(436.74,52.93)}, rotate = 298.77] [color={rgb, 255:red, 0; green, 0; blue, 0 }  ][line width=0.75]    (6.56,-1.97) .. controls (4.17,-0.84) and (1.99,-0.18) .. (0,0) .. controls (1.99,0.18) and (4.17,0.84) .. (6.56,1.97)   ;
\draw    (417.62,35.57) .. controls (388.38,37.49) and (401.26,61.31) .. (417.46,59.39) ;
\draw    (356.68,36.35) .. controls (373.9,8.43) and (415.62,12.43) .. (424.97,35.39) ;
\draw [shift={(394.6,16.82)}, rotate = 178.91] [color={rgb, 255:red, 0; green, 0; blue, 0 }  ][line width=0.75]    (6.56,-1.97) .. controls (4.17,-0.84) and (1.99,-0.18) .. (0,0) .. controls (1.99,0.18) and (4.17,0.84) .. (6.56,1.97)   ;
\draw    (364.76,36.14) .. controls (383.9,10.14) and (411.9,25) .. (417.62,35.57) ;
\draw [shift={(386.46,21.97)}, rotate = 357.54] [color={rgb, 255:red, 0; green, 0; blue, 0 }  ][line width=0.75]    (6.56,-1.97) .. controls (4.17,-0.84) and (1.99,-0.18) .. (0,0) .. controls (1.99,0.18) and (4.17,0.84) .. (6.56,1.97)   ;

\draw (232.17,-1.77) node [anchor=north west][inner sep=0.75pt]  [font=\small]  {$\frac{1}{\sqrt{N}}$};
\draw (304.17,27.4) node [anchor=north west][inner sep=0.75pt]  [font=\small]  {$\frac{1}{\sqrt{N}}$};
\draw (443.5,40.73) node [anchor=north west][inner sep=0.75pt]    {$=$};
\draw (463.5,38.07) node [anchor=north west][inner sep=0.75pt]  [font=\small]  {$\sqrt{N}$};

\end{tikzpicture}

\end{center}

The important feature  of such term is that graphically a loop without any black dots is formed. It corresponds to $\mathrm{Tr}(1)=N$ , which dominates in the large $N$ limit. By contrast, the following contraction is a lower order term:

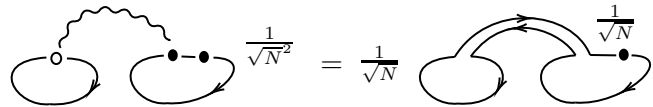
\begin{center}

\tikzset{every picture/.style={line width=0.75pt}} 

\begin{tikzpicture}[x=0.75pt,y=0.75pt,yscale=-1,xscale=1]

\draw   (187.66,51.99) .. controls (187.66,50.43) and (188.77,49.17) .. (190.14,49.17) .. controls (191.51,49.17) and (192.62,50.43) .. (192.62,51.99) .. controls (192.62,53.55) and (191.51,54.82) .. (190.14,54.82) .. controls (188.77,54.82) and (187.66,53.55) .. (187.66,51.99) -- cycle ;
\draw  [fill={rgb, 255:red, 0; green, 0; blue, 0 }  ,fill opacity=1 ] (246.94,50.36) .. controls (246.94,48.97) and (247.83,47.85) .. (248.92,47.85) .. controls (250.02,47.85) and (250.9,48.97) .. (250.9,50.36) .. controls (250.9,51.75) and (250.02,52.88) .. (248.92,52.88) .. controls (247.83,52.88) and (246.94,51.75) .. (246.94,50.36) -- cycle ;
\draw    (195.52,51.99) .. controls (221.18,52.46) and (212.07,76.94) .. (186.83,75.06) ;
\draw [shift={(206.95,69.19)}, rotate = 304.24] [color={rgb, 255:red, 0; green, 0; blue, 0 }  ][line width=0.75]    (6.56,-1.97) .. controls (4.17,-0.84) and (1.99,-0.18) .. (0,0) .. controls (1.99,0.18) and (4.17,0.84) .. (6.56,1.97)   ;
\draw    (185.18,51.52) .. controls (154.56,53.4) and (169.87,76.94) .. (186.83,75.06) ;
\draw    (269.43,51.57) .. controls (293.93,52.05) and (270.07,76.48) .. (245.96,74.56) ;
\draw [shift={(268.3,69.51)}, rotate = 327.01] [color={rgb, 255:red, 0; green, 0; blue, 0 }  ][line width=0.75]    (6.56,-1.97) .. controls (4.17,-0.84) and (1.99,-0.18) .. (0,0) .. controls (1.99,0.18) and (4.17,0.84) .. (6.56,1.97)   ;
\draw    (244.38,50.55) .. controls (215.14,52.47) and (229.76,76.48) .. (245.96,74.56) ;
\draw    (192.33,46.67) .. controls (191.15,44.46) and (191.62,42.79) .. (193.75,41.64) .. controls (195.86,40.89) and (196.59,39.37) .. (195.92,37.1) .. controls (195.55,34.69) and (196.55,33.41) .. (198.92,33.24) .. controls (201.27,33.39) and (202.56,32.38) .. (202.79,30.23) .. controls (203.68,27.92) and (205.25,27.25) .. (207.52,28.21) .. controls (209.49,29.46) and (211.08,29.17) .. (212.28,27.34) .. controls (214.05,25.59) and (215.7,25.58) .. (217.24,27.3) .. controls (218.81,29.13) and (220.49,29.35) .. (222.26,27.96) .. controls (224.46,26.75) and (226.11,27.17) .. (227.2,29.21) .. controls (228.11,31.27) and (229.68,31.84) .. (231.91,30.91) .. controls (234.22,30.11) and (235.66,30.79) .. (236.25,32.94) .. controls (237,35.27) and (238.45,36.13) .. (240.6,35.52) .. controls (242.93,35.15) and (244.27,36.21) .. (244.6,38.69) .. controls (244.35,40.9) and (245.24,42.22) .. (247.28,42.64) -- (247.33,43.17) ;
\draw  [fill={rgb, 255:red, 0; green, 0; blue, 0 }  ,fill opacity=1 ] (262.65,51.22) .. controls (262.65,49.83) and (263.54,48.7) .. (264.64,48.7) .. controls (265.73,48.7) and (266.62,49.83) .. (266.62,51.22) .. controls (266.62,52.61) and (265.73,53.73) .. (264.64,53.73) .. controls (263.54,53.73) and (262.65,52.61) .. (262.65,51.22) -- cycle ;
\draw    (253.29,51) -- (260.29,51.29) ;
\draw    (399.19,51) .. controls (424.85,51.47) and (418,76.63) .. (392.76,74.75) ;
\draw [shift={(412.51,68.35)}, rotate = 297.05] [color={rgb, 255:red, 0; green, 0; blue, 0 }  ][line width=0.75]    (6.56,-1.97) .. controls (4.17,-0.84) and (1.99,-0.18) .. (0,0) .. controls (1.99,0.18) and (4.17,0.84) .. (6.56,1.97)   ;
\draw    (391.1,51.21) .. controls (360.48,53.09) and (375.79,76.63) .. (392.76,74.75) ;
\draw    (480.29,50.43) .. controls (504.78,50.91) and (476,76.17) .. (451.89,74.25) ;
\draw [shift={(475.51,68.98)}, rotate = 329.82] [color={rgb, 255:red, 0; green, 0; blue, 0 }  ][line width=0.75]    (6.56,-1.97) .. controls (4.17,-0.84) and (1.99,-0.18) .. (0,0) .. controls (1.99,0.18) and (4.17,0.84) .. (6.56,1.97)   ;
\draw    (452.05,50.43) .. controls (422.81,52.35) and (435.69,76.17) .. (451.89,74.25) ;
\draw    (391.1,51.21) .. controls (408.33,23.29) and (450.05,27.29) .. (459.4,50.24) ;
\draw [shift={(429.03,31.68)}, rotate = 178.91] [color={rgb, 255:red, 0; green, 0; blue, 0 }  ][line width=0.75]    (6.56,-1.97) .. controls (4.17,-0.84) and (1.99,-0.18) .. (0,0) .. controls (1.99,0.18) and (4.17,0.84) .. (6.56,1.97)   ;
\draw    (399.19,51) .. controls (418.33,25) and (446.33,39.86) .. (452.05,50.43) ;
\draw [shift={(420.89,36.83)}, rotate = 357.54] [color={rgb, 255:red, 0; green, 0; blue, 0 }  ][line width=0.75]    (6.56,-1.97) .. controls (4.17,-0.84) and (1.99,-0.18) .. (0,0) .. controls (1.99,0.18) and (4.17,0.84) .. (6.56,1.97)   ;
\draw    (459.4,50.24) -- (472.57,50.43) ;
\draw  [fill={rgb, 255:red, 0; green, 0; blue, 0 }  ,fill opacity=1 ] (474.37,50.08) .. controls (474.37,48.69) and (475.26,47.56) .. (476.35,47.56) .. controls (477.44,47.56) and (478.33,48.69) .. (478.33,50.08) .. controls (478.33,51.46) and (477.44,52.59) .. (476.35,52.59) .. controls (475.26,52.59) and (474.37,51.46) .. (474.37,50.08) -- cycle ;

\draw (282.83,34.73) node [anchor=north west][inner sep=0.75pt]  [font=\small]  {$\frac{1}{\sqrt{N}^{2}}$};
\draw (321.86,51.97) node [anchor=north west][inner sep=0.75pt]    {$=$};
\draw (340.97,43.93) node [anchor=north west][inner sep=0.75pt]  [font=\small]  {$\frac{1}{\sqrt{N}}$};
\draw (459.84,23.3) node [anchor=north west][inner sep=0.75pt]  [font=\small]  {$\frac{1}{\sqrt{N}}$};

\end{tikzpicture}

\end{center}
The contractions of $t_{1,0}$ with the terms  $C(0,1,2,\cdots,N-1)$ is zero since $0+1+2+\cdots+N-1$ is the lowest energy of the terms $C(n_{1},n_{2},\cdots,n_{N})$, and $t_{1,0}$ lowers the energy by $1$.
Summarising the above discussion, we have:
\begin{equation*}
t_{1,0}\ket{c_{1},c_{2},\cdots,c_{N}}=\sqrt{N}c_{1}\ket{c_{1}-1,c_{2},\cdots,c_{N}}+\frac{1}{\sqrt{N}}\sum_{i= 2}^{N}ic_{i}\ket{\cdots,c_{i-1}+1,c_{i}-1,\cdots}
\end{equation*}
Under the isomorphism $\mathcal{H}_N\cong \mathbb{C}[p_{1},p_{2},\cdots,p_{N}]$, we have:
\begin{equation}\label{t1,0}
t_{1,0}=\sqrt{N}\frac{\partial}{\partial p_{1}}+\frac{1}{\sqrt{N}}\sum_{i=2}^{N}ip_{i-1}\frac{\partial}{\partial p_{i}}
\end{equation}
\subsection{Computation of \texorpdfstring{$t_{2,1}$}{t[2,1]} and \texorpdfstring{$t_{1,2}$}{t[1,2]}}
Fix any energy cutoff $E\geq 3$. Let ${\mathcal H}^{(\leq E)}_N$ denote the truncated Hilbert space \eqref{Hilbert-truncate}, the operator $$
t_{m,n}: {\mathcal H}^{(\leq E)}_N\to {\mathcal H}^{(\leq E+n-m)}_N
$$ is denoted by ${t_{m,n}}_{\leq E}$. Recall \eqref{generator t}: $t_{m,n}=\Tr (Sym(Z^m{Z^{\dagger}}^n))$. We first compute the leading order terms of ${t_{2,1}}_{\leq E}:=t_{2,1}P_{\leq E}$ in the limit $N\rightarrow \infty$ and and subsequently address the lower-order terms. The $t_{2,1}$ contains 2 white dots:
\begin{equation}\label{t21 0}
    \hbox{

\tikzset{every picture/.style={line width=0.75pt}} 

\begin{tikzpicture}[x=0.75pt,y=0.75pt,yscale=-1,xscale=1]

\draw   (269.5,16.83) .. controls (264.83,16.93) and (262.55,19.31) .. (262.65,23.98) -- (262.79,30.73) .. controls (262.93,37.4) and (260.67,40.78) .. (256,40.88) .. controls (260.67,40.78) and (263.07,44.06) .. (263.21,50.73)(263.15,47.73) -- (263.36,57.48) .. controls (263.45,62.15) and (265.83,64.43) .. (270.5,64.33) ;
\draw   (488.5,63.83) .. controls (493.17,63.83) and (495.5,61.5) .. (495.5,56.83) -- (495.5,50.33) .. controls (495.5,43.66) and (497.83,40.33) .. (502.5,40.33) .. controls (497.83,40.33) and (495.5,37) .. (495.5,30.33)(495.5,33.33) -- (495.5,23.83) .. controls (495.5,19.16) and (493.17,16.83) .. (488.5,16.83) ;
\draw   (301.97,30.08) .. controls (301.97,28.52) and (303.08,27.25) .. (304.45,27.25) .. controls (305.83,27.25) and (306.94,28.52) .. (306.94,30.08) .. controls (306.94,31.64) and (305.83,32.9) .. (304.45,32.9) .. controls (303.08,32.9) and (301.97,31.64) .. (301.97,30.08) -- cycle ;
\draw    (325.43,31) .. controls (345.4,32.83) and (327.17,54.48) .. (299.54,54.03) ;
\draw [shift={(321.38,49.46)}, rotate = 330.16] [color={rgb, 255:red, 0; green, 0; blue, 0 }  ][line width=0.75]    (6.56,-1.97) .. controls (4.17,-0.84) and (1.99,-0.18) .. (0,0) .. controls (1.99,0.18) and (4.17,0.84) .. (6.56,1.97)   ;
\draw    (285.07,29.98) .. controls (260.97,30.28) and (278.97,54.88) .. (299.54,54.03) ;
\draw  [fill={rgb, 255:red, 0; green, 0; blue, 0 }  ,fill opacity=1 ] (287.07,29.98) .. controls (287.07,28.59) and (287.96,27.46) .. (289.05,27.46) .. controls (290.15,27.46) and (291.03,28.59) .. (291.03,29.98) .. controls (291.03,31.36) and (290.15,32.49) .. (289.05,32.49) .. controls (287.96,32.49) and (287.07,31.36) .. (287.07,29.98) -- cycle ;
\draw    (292.77,30.08) -- (299.97,30.08) ;
\draw    (308.37,30.28) -- (316.57,30.48) ;
\draw   (364.36,30.55) .. controls (364.36,28.99) and (365.47,27.72) .. (366.84,27.72) .. controls (368.21,27.72) and (369.32,28.99) .. (369.32,30.55) .. controls (369.32,32.11) and (368.21,33.37) .. (366.84,33.37) .. controls (365.47,33.37) and (364.36,32.11) .. (364.36,30.55) -- cycle ;
\draw  [fill={rgb, 255:red, 0; green, 0; blue, 0 }  ,fill opacity=1 ] (379.69,30.85) .. controls (379.69,29.46) and (380.58,28.33) .. (381.68,28.33) .. controls (382.77,28.33) and (383.66,29.46) .. (383.66,30.85) .. controls (383.66,32.24) and (382.77,33.36) .. (381.68,33.36) .. controls (380.58,33.36) and (379.69,32.24) .. (379.69,30.85) -- cycle ;
\draw    (401.29,31.62) .. controls (421.26,33.45) and (405.46,55.05) .. (377.82,54.61) ;
\draw [shift={(399.13,50.04)}, rotate = 328.31] [color={rgb, 255:red, 0; green, 0; blue, 0 }  ][line width=0.75]    (6.56,-1.97) .. controls (4.17,-0.84) and (1.99,-0.18) .. (0,0) .. controls (1.99,0.18) and (4.17,0.84) .. (6.56,1.97)   ;
\draw    (363.36,30.55) .. controls (339.26,30.85) and (357.26,55.45) .. (377.82,54.61) ;
\draw    (371.06,30.65) -- (378.26,30.65) ;
\draw    (385.66,30.85) -- (393.86,31.05) ;
\draw   (442.81,30.08) .. controls (442.81,28.52) and (443.92,27.25) .. (445.29,27.25) .. controls (446.66,27.25) and (447.77,28.52) .. (447.77,30.08) .. controls (447.77,31.64) and (446.66,32.9) .. (445.29,32.9) .. controls (443.92,32.9) and (442.81,31.64) .. (442.81,30.08) -- cycle ;
\draw  [fill={rgb, 255:red, 0; green, 0; blue, 0 }  ,fill opacity=1 ] (474,30.69) .. controls (474,29.31) and (474.88,28.18) .. (475.98,28.18) .. controls (477.07,28.18) and (477.96,29.31) .. (477.96,30.69) .. controls (477.96,32.08) and (477.07,33.21) .. (475.98,33.21) .. controls (474.88,33.21) and (474,32.08) .. (474,30.69) -- cycle ;
\draw    (480.71,31.05) .. controls (500.69,32.88) and (483.17,54.48) .. (455.54,54.03) ;
\draw [shift={(477.22,49.47)}, rotate = 329.75] [color={rgb, 255:red, 0; green, 0; blue, 0 }  ][line width=0.75]    (6.56,-1.97) .. controls (4.17,-0.84) and (1.99,-0.18) .. (0,0) .. controls (1.99,0.18) and (4.17,0.84) .. (6.56,1.97)   ;
\draw    (441.07,29.98) .. controls (416.97,30.28) and (434.97,54.88) .. (455.54,54.03) ;
\draw    (448.77,30.08) -- (455.97,30.08) ;
\draw    (463.37,30.28) -- (471.57,30.48) ;
\draw   (318.83,30.65) .. controls (318.83,29.09) and (319.94,27.82) .. (321.31,27.82) .. controls (322.68,27.82) and (323.79,29.09) .. (323.79,30.65) .. controls (323.79,32.21) and (322.68,33.47) .. (321.31,33.47) .. controls (319.94,33.47) and (318.83,32.21) .. (318.83,30.65) -- cycle ;
\draw   (394.86,31.05) .. controls (394.86,29.49) and (395.97,28.22) .. (397.34,28.22) .. controls (398.71,28.22) and (399.82,29.49) .. (399.82,31.05) .. controls (399.82,32.61) and (398.71,33.87) .. (397.34,33.87) .. controls (395.97,33.87) and (394.86,32.61) .. (394.86,31.05) -- cycle ;
\draw   (456.97,30.08) .. controls (456.97,28.52) and (458.08,27.25) .. (459.45,27.25) .. controls (460.83,27.25) and (461.94,28.52) .. (461.94,30.08) .. controls (461.94,31.64) and (460.83,32.9) .. (459.45,32.9) .. controls (458.08,32.9) and (456.97,31.64) .. (456.97,30.08) -- cycle ;

\draw (189.67,33.4) node [anchor=north west][inner sep=0.75pt]    {$t_{2,1}$};
\draw (217,36.07) node [anchor=north west][inner sep=0.75pt]    {$=$};
\draw (241.17,29.73) node [anchor=north west][inner sep=0.75pt]    {$\frac{1}{3}$};
\draw (336.43,32.35) node [anchor=north west][inner sep=0.75pt]    {$+$};
\draw (413.29,31.78) node [anchor=north west][inner sep=0.75pt]    {$+$};

\end{tikzpicture}

}
\end{equation}
We have the following different ways of Wick contractions:
\begin{enumerate}
    \item One of the white dots is contracted with one of the $\frac{N(N-1)}{2}$ black dots in one of the $k$ $C(0,1,2,\cdots, N-1)$, the other one is contracted with:
    \begin{enumerate}
        \item one $\frac{1}{\sqrt{N}}\mathrm{Tr}(Z^{\dagger})$.
        \item one $\frac{1}{\sqrt{N}^i}\mathrm{Tr}({Z^{\dagger}}^{i})$, $2 \leq i\leq E$. 
    \end{enumerate}
    \item The two white dots are contracted with 
    \begin{enumerate}
        \item two adjacent black dots of one $\frac{1}{\sqrt{N}^2}\mathrm{Tr}({Z^{\dagger}}^{2})$
        \item two adjacent black dots of one $\frac{1}{\sqrt{N}^i}\mathrm{Tr}({{Z^{\dagger}}^i})$, $3 \leq i\leq E$
        \item two non-adjacent black dots of one $\frac{1}{\sqrt{N}^i}\mathrm{Tr}({{Z^{\dagger}}^i})$, $3 \leq i\leq E$
    \end{enumerate}
    \item Two white dots is contracted with two different $\frac{1}{\sqrt{N}^i}\mathrm{Tr}({Z^{\dagger}}^{i})$, $\frac{1}{\sqrt{N}^j}\mathrm{Tr}({Z^{\dagger}}^{j})$.
    \item One of the white dot is contracted with one of the black dot in the expression (\ref{t21 0}) of $t_{2,1}$, after that we get $t_{1,0}$ and we can  apply (\ref{t1,0}).
\end{enumerate}

We will compute all these terms one by one. Graphically, case 1-(a) is :
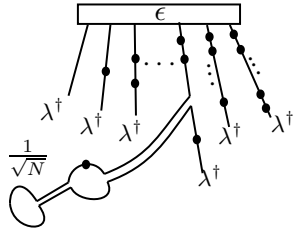
\begin{center}

\tikzset{every picture/.style={line width=0.75pt}} 

\begin{tikzpicture}[x=0.75pt,y=0.75pt,yscale=-1,xscale=1]

\draw   (287.72,9.69) -- (369.5,9.69) -- (369.5,19.51) -- (287.72,19.51) -- cycle ;
\draw    (292.76,19.85) -- (283,55.75) ;
\draw    (304.14,19.85) -- (299.5,63.25) ;
\draw    (316.37,19.55) -- (317,65.75) ;
\draw    (352.85,20.1) -- (364,71.5) ;
\draw    (364.35,20.02) -- (385,66.75) ;
\draw  [fill={rgb, 255:red, 0; green, 0; blue, 0 }  ,fill opacity=1 ] (300.23,43.46) .. controls (300.23,42.41) and (300.94,41.55) .. (301.82,41.55) .. controls (302.7,41.55) and (303.41,42.41) .. (303.41,43.46) .. controls (303.41,44.52) and (302.7,45.37) .. (301.82,45.37) .. controls (300.94,45.37) and (300.23,44.52) .. (300.23,43.46) -- cycle ;
\draw  [fill={rgb, 255:red, 0; green, 0; blue, 0 }  ,fill opacity=1 ] (315.21,37.02) .. controls (315.21,35.97) and (315.93,35.11) .. (316.8,35.11) .. controls (317.68,35.11) and (318.39,35.97) .. (318.39,37.02) .. controls (318.39,38.08) and (317.68,38.93) .. (316.8,38.93) .. controls (315.93,38.93) and (315.21,38.08) .. (315.21,37.02) -- cycle ;
\draw  [fill={rgb, 255:red, 0; green, 0; blue, 0 }  ,fill opacity=1 ] (314.96,49.52) .. controls (314.96,48.47) and (315.68,47.61) .. (316.55,47.61) .. controls (317.43,47.61) and (318.14,48.47) .. (318.14,49.52) .. controls (318.14,50.58) and (317.43,51.43) .. (316.55,51.43) .. controls (315.68,51.43) and (314.96,50.58) .. (314.96,49.52) -- cycle ;
\draw  [fill={rgb, 255:red, 0; green, 0; blue, 0 }  ,fill opacity=1 ] (352.5,26.02) .. controls (352.5,24.97) and (353.21,24.11) .. (354.09,24.11) .. controls (354.97,24.11) and (355.68,24.97) .. (355.68,26.02) .. controls (355.68,27.08) and (354.97,27.93) .. (354.09,27.93) .. controls (353.21,27.93) and (352.5,27.08) .. (352.5,26.02) -- cycle ;
\draw  [fill={rgb, 255:red, 0; green, 0; blue, 0 }  ,fill opacity=1 ] (354,33.27) .. controls (354,32.22) and (354.71,31.36) .. (355.59,31.36) .. controls (356.47,31.36) and (357.18,32.22) .. (357.18,33.27) .. controls (357.18,34.33) and (356.47,35.18) .. (355.59,35.18) .. controls (354.71,35.18) and (354,34.33) .. (354,33.27) -- cycle ;
\draw  [fill={rgb, 255:red, 0; green, 0; blue, 0 }  ,fill opacity=1 ] (359.75,59.02) .. controls (359.75,57.97) and (360.46,57.11) .. (361.34,57.11) .. controls (362.22,57.11) and (362.93,57.97) .. (362.93,59.02) .. controls (362.93,60.08) and (362.22,60.93) .. (361.34,60.93) .. controls (360.46,60.93) and (359.75,60.08) .. (359.75,59.02) -- cycle ;
\draw  [fill={rgb, 255:red, 0; green, 0; blue, 0 }  ,fill opacity=1 ] (364.75,24.77) .. controls (364.75,23.72) and (365.46,22.86) .. (366.34,22.86) .. controls (367.22,22.86) and (367.93,23.72) .. (367.93,24.77) .. controls (367.93,25.83) and (367.22,26.68) .. (366.34,26.68) .. controls (365.46,26.68) and (364.75,25.83) .. (364.75,24.77) -- cycle ;
\draw  [fill={rgb, 255:red, 0; green, 0; blue, 0 }  ,fill opacity=1 ] (368.25,32.52) .. controls (368.25,31.47) and (368.96,30.61) .. (369.84,30.61) .. controls (370.72,30.61) and (371.43,31.47) .. (371.43,32.52) .. controls (371.43,33.58) and (370.72,34.43) .. (369.84,34.43) .. controls (368.96,34.43) and (368.25,33.58) .. (368.25,32.52) -- cycle ;
\draw  [fill={rgb, 255:red, 0; green, 0; blue, 0 }  ,fill opacity=1 ] (378.25,55.02) .. controls (378.25,53.97) and (378.96,53.11) .. (379.84,53.11) .. controls (380.72,53.11) and (381.43,53.97) .. (381.43,55.02) .. controls (381.43,56.08) and (380.72,56.93) .. (379.84,56.93) .. controls (378.96,56.93) and (378.25,56.08) .. (378.25,55.02) -- cycle ;
\draw  [fill={rgb, 255:red, 0; green, 0; blue, 0 }  ,fill opacity=1 ] (381,62.02) .. controls (381,60.97) and (381.71,60.11) .. (382.59,60.11) .. controls (383.47,60.11) and (384.18,60.97) .. (384.18,62.02) .. controls (384.18,63.08) and (383.47,63.93) .. (382.59,63.93) .. controls (381.71,63.93) and (381,63.08) .. (381,62.02) -- cycle ;
\draw    (338.8,18.65) -- (344.2,56.05) ;
\draw    (301,95.25) .. controls (325.4,85.65) and (328.2,74.85) .. (344.2,56.05) ;
\draw    (344.8,60.85) .. controls (331.8,78.85) and (324.6,91.25) .. (302.6,99.25) ;
\draw    (344.8,60.85) -- (350.6,94.85) ;
\draw  [fill={rgb, 255:red, 0; green, 0; blue, 0 }  ,fill opacity=1 ] (338.03,27.86) .. controls (338.03,26.81) and (338.74,25.95) .. (339.62,25.95) .. controls (340.5,25.95) and (341.21,26.81) .. (341.21,27.86) .. controls (341.21,28.92) and (340.5,29.77) .. (339.62,29.77) .. controls (338.74,29.77) and (338.03,28.92) .. (338.03,27.86) -- cycle ;
\draw  [fill={rgb, 255:red, 0; green, 0; blue, 0 }  ,fill opacity=1 ] (339.91,37.35) .. controls (339.91,36.3) and (340.62,35.44) .. (341.5,35.44) .. controls (342.38,35.44) and (343.09,36.3) .. (343.09,37.35) .. controls (343.09,38.4) and (342.38,39.26) .. (341.5,39.26) .. controls (340.62,39.26) and (339.91,38.4) .. (339.91,37.35) -- cycle ;
\draw  [fill={rgb, 255:red, 0; green, 0; blue, 0 }  ,fill opacity=1 ] (346.11,77.85) .. controls (346.11,76.8) and (346.82,75.94) .. (347.7,75.94) .. controls (348.58,75.94) and (349.29,76.8) .. (349.29,77.85) .. controls (349.29,78.9) and (348.58,79.76) .. (347.7,79.76) .. controls (346.82,79.76) and (346.11,78.9) .. (346.11,77.85) -- cycle ;
\draw    (283,99.5) .. controls (283.5,90.25) and (296,86.5) .. (301,95.25) ;
\draw    (284.25,102) .. controls (287.5,110.25) and (308.25,110.5) .. (302.6,99.25) ;
\draw  [fill={rgb, 255:red, 0; green, 0; blue, 0 }  ,fill opacity=1 ] (290.23,90.46) .. controls (290.23,89.41) and (290.94,88.55) .. (291.82,88.55) .. controls (292.7,88.55) and (293.41,89.41) .. (293.41,90.46) .. controls (293.41,91.52) and (292.7,92.37) .. (291.82,92.37) .. controls (290.94,92.37) and (290.23,91.52) .. (290.23,90.46) -- cycle ;
\draw    (283,99.5) -- (267,109) ;
\draw    (268,111.25) -- (284.25,102) ;
\draw    (253.5,110.25) .. controls (253.18,102.43) and (258.75,96) .. (267,109) ;
\draw    (253.5,110.25) .. controls (254.75,125.5) and (278.75,125.5) .. (268,111.25) ;

\draw (324.8,9.9) node [anchor=north west][inner sep=0.75pt]  [font=\normalsize]  {$\epsilon $};
\draw (318.85,35.34) node [anchor=north west][inner sep=0.75pt]  [font=\normalsize]  {$\cdots $};
\draw (267,55.65) node [anchor=north west][inner sep=0.75pt]  [font=\footnotesize]  {$\lambda ^{\dagger }$};
\draw (287,63.15) node [anchor=north west][inner sep=0.75pt]  [font=\footnotesize]  {$\lambda ^{\dagger }$};
\draw (306.5,66.65) node [anchor=north west][inner sep=0.75pt]  [font=\footnotesize]  {$\lambda ^{\dagger }$};
\draw (357,68.15) node [anchor=north west][inner sep=0.75pt]  [font=\footnotesize]  {$\lambda ^{\dagger }$};
\draw (383.5,62.15) node [anchor=north west][inner sep=0.75pt]  [font=\footnotesize]  {$\lambda ^{\dagger }$};
\draw (350.7,40.3) node [anchor=north west][inner sep=0.75pt]    {$\cdot $};
\draw (352.07,45.13) node [anchor=north west][inner sep=0.75pt]    {$\cdot $};
\draw (371.33,32.57) node [anchor=north west][inner sep=0.75pt]    {$\cdot $};
\draw (372.95,35.9) node [anchor=north west][inner sep=0.75pt]    {$\cdot $};
\draw (375.23,39.92) node [anchor=north west][inner sep=0.75pt]    {$\cdot $};
\draw (350.19,36.3) node [anchor=north west][inner sep=0.75pt]    {$\cdot $};
\draw (346.6,89.35) node [anchor=north west][inner sep=0.75pt]  [font=\footnotesize]  {$\lambda ^{\dagger }$};
\draw (250.24,76.71) node [anchor=north west][inner sep=0.75pt]  [font=\small]  {$\frac{1}{\sqrt{N}}$};

\end{tikzpicture}

\end{center}
All three terms in $t_{2,1}$ (\ref{t21 0}) contribute to the same graph in this case, this graph is topologically the same of $C(0,1,2,\cdots,N-1)$. So effectively case 1-(a) contraction kills one of the $\frac{1}{\sqrt{N}}\mathrm{Tr}(Z^{\dagger})$. All $k\frac{N(N-1)}{2}$ ways of contraction gives the same term. Keeping track of a combinatorial factor 2, case 1-(a) gives rise to:
\begin{equation}\label{case 1-a}
    \frac{1}{\sqrt{N}}k\frac{N(N-1)}{2}\cdot 2 \cdot \frac{\partial}{\partial p_{1}}
\end{equation}
For Case 1-(b), only the contraction with the black dots on the $i-1$ rightmost legs of $C(0,1,2,\cdots,N-(i-1),N-(i-2),\cdots,N-1)$ will give non-zero contribution: 
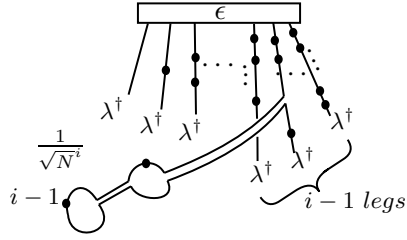
\begin{center}

\tikzset{every picture/.style={line width=0.75pt}} 

\begin{tikzpicture}[x=0.75pt,y=0.75pt,yscale=-1,xscale=1]

\draw   (307.72,43.19) -- (389.5,43.19) -- (389.5,53.01) -- (307.72,53.01) -- cycle ;
\draw    (312.76,53.35) -- (303,89.25) ;
\draw    (324.14,53.35) -- (319.5,96.75) ;
\draw    (336.37,53.05) -- (337,99.25) ;
\draw    (366.28,53.6) -- (367.8,101.4) ;
\draw    (384.07,53.23) -- (404.71,99.97) ;
\draw  [fill={rgb, 255:red, 0; green, 0; blue, 0 }  ,fill opacity=1 ] (320.23,76.96) .. controls (320.23,75.91) and (320.94,75.05) .. (321.82,75.05) .. controls (322.7,75.05) and (323.41,75.91) .. (323.41,76.96) .. controls (323.41,78.02) and (322.7,78.87) .. (321.82,78.87) .. controls (320.94,78.87) and (320.23,78.02) .. (320.23,76.96) -- cycle ;
\draw  [fill={rgb, 255:red, 0; green, 0; blue, 0 }  ,fill opacity=1 ] (335.21,70.52) .. controls (335.21,69.47) and (335.93,68.61) .. (336.8,68.61) .. controls (337.68,68.61) and (338.39,69.47) .. (338.39,70.52) .. controls (338.39,71.58) and (337.68,72.43) .. (336.8,72.43) .. controls (335.93,72.43) and (335.21,71.58) .. (335.21,70.52) -- cycle ;
\draw  [fill={rgb, 255:red, 0; green, 0; blue, 0 }  ,fill opacity=1 ] (334.96,83.02) .. controls (334.96,81.97) and (335.68,81.11) .. (336.55,81.11) .. controls (337.43,81.11) and (338.14,81.97) .. (338.14,83.02) .. controls (338.14,84.08) and (337.43,84.93) .. (336.55,84.93) .. controls (335.68,84.93) and (334.96,84.08) .. (334.96,83.02) -- cycle ;
\draw  [fill={rgb, 255:red, 0; green, 0; blue, 0 }  ,fill opacity=1 ] (364.79,60.95) .. controls (364.79,59.9) and (365.5,59.04) .. (366.38,59.04) .. controls (367.25,59.04) and (367.96,59.9) .. (367.96,60.95) .. controls (367.96,62.01) and (367.25,62.86) .. (366.38,62.86) .. controls (365.5,62.86) and (364.79,62.01) .. (364.79,60.95) -- cycle ;
\draw  [fill={rgb, 255:red, 0; green, 0; blue, 0 }  ,fill opacity=1 ] (365.43,71.92) .. controls (365.43,70.86) and (366.14,70.01) .. (367.02,70.01) .. controls (367.9,70.01) and (368.61,70.86) .. (368.61,71.92) .. controls (368.61,72.97) and (367.9,73.83) .. (367.02,73.83) .. controls (366.14,73.83) and (365.43,72.97) .. (365.43,71.92) -- cycle ;
\draw  [fill={rgb, 255:red, 0; green, 0; blue, 0 }  ,fill opacity=1 ] (365.75,92.52) .. controls (365.75,91.47) and (366.46,90.61) .. (367.34,90.61) .. controls (368.22,90.61) and (368.93,91.47) .. (368.93,92.52) .. controls (368.93,93.58) and (368.22,94.43) .. (367.34,94.43) .. controls (366.46,94.43) and (365.75,93.58) .. (365.75,92.52) -- cycle ;
\draw  [fill={rgb, 255:red, 0; green, 0; blue, 0 }  ,fill opacity=1 ] (384.46,57.99) .. controls (384.46,56.93) and (385.18,56.08) .. (386.05,56.08) .. controls (386.93,56.08) and (387.64,56.93) .. (387.64,57.99) .. controls (387.64,59.04) and (386.93,59.9) .. (386.05,59.9) .. controls (385.18,59.9) and (384.46,59.04) .. (384.46,57.99) -- cycle ;
\draw  [fill={rgb, 255:red, 0; green, 0; blue, 0 }  ,fill opacity=1 ] (387.96,65.74) .. controls (387.96,64.68) and (388.68,63.83) .. (389.55,63.83) .. controls (390.43,63.83) and (391.14,64.68) .. (391.14,65.74) .. controls (391.14,66.79) and (390.43,67.65) .. (389.55,67.65) .. controls (388.68,67.65) and (387.96,66.79) .. (387.96,65.74) -- cycle ;
\draw  [fill={rgb, 255:red, 0; green, 0; blue, 0 }  ,fill opacity=1 ] (397.96,87.38) .. controls (397.96,86.33) and (398.68,85.47) .. (399.55,85.47) .. controls (400.43,85.47) and (401.14,86.33) .. (401.14,87.38) .. controls (401.14,88.44) and (400.43,89.29) .. (399.55,89.29) .. controls (398.68,89.29) and (397.96,88.44) .. (397.96,87.38) -- cycle ;
\draw  [fill={rgb, 255:red, 0; green, 0; blue, 0 }  ,fill opacity=1 ] (400.71,94.38) .. controls (400.71,93.33) and (401.43,92.47) .. (402.3,92.47) .. controls (403.18,92.47) and (403.89,93.33) .. (403.89,94.38) .. controls (403.89,95.44) and (403.18,96.29) .. (402.3,96.29) .. controls (401.43,96.29) and (400.71,95.44) .. (400.71,94.38) -- cycle ;
\draw    (375.94,52.72) -- (381.34,90.12) ;
\draw    (321,128.75) .. controls (345.4,119.15) and (365.34,108.92) .. (381.34,90.12) ;
\draw    (382.14,93.71) .. controls (369.14,111.71) and (344.6,124.75) .. (322.6,132.75) ;
\draw    (382.14,93.71) -- (386.5,118.75) ;
\draw  [fill={rgb, 255:red, 0; green, 0; blue, 0 }  ,fill opacity=1 ] (375.18,61.93) .. controls (375.18,60.88) and (375.89,60.02) .. (376.77,60.02) .. controls (377.64,60.02) and (378.35,60.88) .. (378.35,61.93) .. controls (378.35,62.99) and (377.64,63.84) .. (376.77,63.84) .. controls (375.89,63.84) and (375.18,62.99) .. (375.18,61.93) -- cycle ;
\draw  [fill={rgb, 255:red, 0; green, 0; blue, 0 }  ,fill opacity=1 ] (377.05,71.42) .. controls (377.05,70.37) and (377.77,69.51) .. (378.64,69.51) .. controls (379.52,69.51) and (380.23,70.37) .. (380.23,71.42) .. controls (380.23,72.48) and (379.52,73.33) .. (378.64,73.33) .. controls (377.77,73.33) and (377.05,72.48) .. (377.05,71.42) -- cycle ;
\draw  [fill={rgb, 255:red, 0; green, 0; blue, 0 }  ,fill opacity=1 ] (383.45,108.8) .. controls (383.45,107.75) and (384.17,106.89) .. (385.04,106.89) .. controls (385.92,106.89) and (386.63,107.75) .. (386.63,108.8) .. controls (386.63,109.86) and (385.92,110.71) .. (385.04,110.71) .. controls (384.17,110.71) and (383.45,109.86) .. (383.45,108.8) -- cycle ;
\draw    (303,133) .. controls (303.5,123.75) and (316,120) .. (321,128.75) ;
\draw    (304.25,135.5) .. controls (307.5,143.75) and (328.25,144) .. (322.6,132.75) ;
\draw  [fill={rgb, 255:red, 0; green, 0; blue, 0 }  ,fill opacity=1 ] (310.23,123.96) .. controls (310.23,122.91) and (310.94,122.05) .. (311.82,122.05) .. controls (312.7,122.05) and (313.41,122.91) .. (313.41,123.96) .. controls (313.41,125.02) and (312.7,125.87) .. (311.82,125.87) .. controls (310.94,125.87) and (310.23,125.02) .. (310.23,123.96) -- cycle ;
\draw    (303,133) -- (287,142.5) ;
\draw    (288,144.75) -- (304.25,135.5) ;
\draw    (272.5,150) .. controls (269.5,139.75) and (278.75,129.5) .. (287,142.5) ;
\draw    (272.5,150) .. controls (273.75,165.25) and (298.75,159) .. (288,144.75) ;
\draw  [fill={rgb, 255:red, 0; green, 0; blue, 0 }  ,fill opacity=1 ] (269.95,144.21) .. controls (269.95,143.16) and (270.66,142.3) .. (271.54,142.3) .. controls (272.41,142.3) and (273.13,143.16) .. (273.13,144.21) .. controls (273.13,145.27) and (272.41,146.12) .. (271.54,146.12) .. controls (270.66,146.12) and (269.95,145.27) .. (269.95,144.21) -- cycle ;
\draw   (369.5,141) .. controls (372.07,144.89) and (375.3,145.56) .. (379.2,142.99) -- (387,137.85) .. controls (392.57,134.18) and (396.63,134.29) .. (399.2,138.19) .. controls (396.63,134.29) and (398.13,130.51) .. (403.7,126.84)(401.2,128.49) -- (411.51,121.7) .. controls (415.4,119.13) and (416.07,115.89) .. (413.5,112) ;
\draw    (368,110.53) -- (368.06,122.76) ;

\draw (344.8,43.4) node [anchor=north west][inner sep=0.75pt]  [font=\normalsize]  {$\epsilon $};
\draw (374.85,74.55) node [anchor=north west][inner sep=0.75pt]  [font=\normalsize]  {$\cdots $};
\draw (287,89.15) node [anchor=north west][inner sep=0.75pt]  [font=\footnotesize]  {$\lambda ^{\dagger }$};
\draw (307,96.65) node [anchor=north west][inner sep=0.75pt]  [font=\footnotesize]  {$\lambda ^{\dagger }$};
\draw (326.5,100.15) node [anchor=north west][inner sep=0.75pt]  [font=\footnotesize]  {$\lambda ^{\dagger }$};
\draw (362.43,121.23) node [anchor=north west][inner sep=0.75pt]  [font=\footnotesize]  {$\lambda ^{\dagger }$};
\draw (403.21,94.51) node [anchor=north west][inner sep=0.75pt]  [font=\footnotesize]  {$\lambda ^{\dagger }$};
\draw (358.56,74.66) node [anchor=north west][inner sep=0.75pt]    {$\cdot $};
\draw (358.5,79.49) node [anchor=north west][inner sep=0.75pt]    {$\cdot $};
\draw (390.76,62.93) node [anchor=north west][inner sep=0.75pt]    {$\cdot $};
\draw (392.95,67.69) node [anchor=north west][inner sep=0.75pt]    {$\cdot $};
\draw (394.98,71.9) node [anchor=north west][inner sep=0.75pt]    {$\cdot $};
\draw (358.19,70.8) node [anchor=north west][inner sep=0.75pt]    {$\cdot $};
\draw (383.24,114.14) node [anchor=north west][inner sep=0.75pt]  [font=\footnotesize]  {$\lambda ^{\dagger }$};
\draw (254.34,105.89) node [anchor=north west][inner sep=0.75pt]  [font=\small]  {$\frac{1}{\sqrt{N}^{i}}$};
\draw (241.57,135.06) node [anchor=north west][inner sep=0.75pt]  [font=\footnotesize]  {$i-1$};
\draw (338.14,69.83) node [anchor=north west][inner sep=0.75pt]    {$\cdots $};
\draw (390.25,136.52) node [anchor=north west][inner sep=0.75pt]  [font=\footnotesize]  {$i-1\ legs$};

\end{tikzpicture}

\end{center}
In the above graph,  we put a "$i-1$" label to indicate there are $i-1$ black dots. Up to same factor, the above contractions give rise to a term: 
\begin{align*}
    &(N-(i-1))C(0,1,2,\cdots,N-(i-1)+i-1,N-(i-2),\cdots,N-1)\\ \nonumber
    &+(N-(i-2))C(0,1,2,\cdots,N-(i-1),N-(i-2)+i-1,\cdots,N-1)\\ \nonumber
    &+\cdots \\ \nonumber
    &+(N-1)C(0,1,2,\cdots,N-(i-1),N-(i-2),\cdots,N-1+i-1)
\end{align*}
We reorganize this as the summation of two terms:
\begin{align}
    \nonumber&NC(0,1,2,\cdots,N-(i-1)+i-1,N-(i-2),\cdots,N-1)\\ 
    \nonumber&+NC(0,1,2,\cdots,N-(i-1),N-(i-2)+i-1,\cdots,N-1)\\ 
    \nonumber&+\cdots \\ 
     \label{1-b eq 1}&+NC(0,1,2,\cdots,N-(i-1),N-(i-2),\cdots,N-1+i-1)
\end{align}
and 
\begin{align}
    \nonumber&-(i-1)C(0,1,2,\cdots,N-(i-1)+i-1,N-(i-2),\cdots,N-1)\\ 
    \nonumber&-(i-2)C(0,1,2,\cdots,N-(i-1),N-(i-2)+i-1,\cdots,N-1)\\ 
    \nonumber&-\cdots \\ 
    \label{1-b eq 2}&-C(0,1,2,\cdots,N-(i-1),N-(i-2),\cdots,N-1+i-1)
\end{align}
From (\ref{ti C 2}), the term (\ref{1-b eq 1}) is equal to:
\begin{equation*}
    N \mathrm{Tr}({Z^{\dagger}}^{i-1})C(0,1,2,\cdots,N-1)
\end{equation*}
The point is that for a fixed $E$, in the limit $N\rightarrow \infty$, (\ref{1-b eq 2}) has order $O(1)=O(N^{0})$. We remark on this point at the end of Appendix \ref{App B}. Keeping track of all combinatorial factors, case 1-(b) gives the following term:

\begin{equation}\label{case 1b}
    \sqrt{N}\cdot k \cdot 2 \cdot \sum_{i=2}^{E} ip_{i-1}\frac{\partial}{\partial p_{i}}+O(\frac{1}{\sqrt{N}})
\end{equation}
\\
For case 2-(a), we have: 
\begin{center}

\tikzset{every picture/.style={line width=0.75pt}} 

\begin{tikzpicture}[x=0.75pt,y=0.75pt,yscale=-1,xscale=1]

\draw    (142.43,72.33) .. controls (162.4,74.16) and (144.17,95.81) .. (116.54,95.37) ;
\draw [shift={(138.38,90.8)}, rotate = 330.16] [color={rgb, 255:red, 0; green, 0; blue, 0 }  ][line width=0.75]    (6.56,-1.97) .. controls (4.17,-0.84) and (1.99,-0.18) .. (0,0) .. controls (1.99,0.18) and (4.17,0.84) .. (6.56,1.97)   ;
\draw    (102.07,71.31) .. controls (77.97,71.61) and (95.97,96.21) .. (116.54,95.37) ;
\draw  [fill={rgb, 255:red, 0; green, 0; blue, 0 }  ,fill opacity=1 ] (103.65,71.46) .. controls (103.65,70.07) and (104.54,68.94) .. (105.64,68.94) .. controls (106.73,68.94) and (107.62,70.07) .. (107.62,71.46) .. controls (107.62,72.85) and (106.73,73.97) .. (105.64,73.97) .. controls (104.54,73.97) and (103.65,72.85) .. (103.65,71.46) -- cycle ;
\draw    (109.77,71.41) -- (116.97,71.41) ;
\draw    (121,71.45) -- (137.91,71.82) ;
\draw  [fill={rgb, 255:red, 0; green, 0; blue, 0 }  ,fill opacity=1 ] (299.44,74.25) .. controls (299.44,72.86) and (300.33,71.74) .. (301.42,71.74) .. controls (302.52,71.74) and (303.41,72.86) .. (303.41,74.25) .. controls (303.41,75.64) and (302.52,76.76) .. (301.42,76.76) .. controls (300.33,76.76) and (299.44,75.64) .. (299.44,74.25) -- cycle ;
\draw    (306.16,74.6) .. controls (326.13,76.43) and (308.62,98.03) .. (280.98,97.59) ;
\draw [shift={(302.67,93.03)}, rotate = 329.75] [color={rgb, 255:red, 0; green, 0; blue, 0 }  ][line width=0.75]    (6.56,-1.97) .. controls (4.17,-0.84) and (1.99,-0.18) .. (0,0) .. controls (1.99,0.18) and (4.17,0.84) .. (6.56,1.97)   ;
\draw    (266.52,73.53) .. controls (242.42,73.83) and (260.42,98.43) .. (280.98,97.59) ;
\draw    (270.89,73.56) -- (284,73.56) ;
\draw    (288.82,73.83) -- (297.02,74.03) ;
\draw    (124,34.67) .. controls (129.67,14.67) and (178.67,30.33) .. (155,47) ;
\draw [shift={(152.54,27.88)}, rotate = 198.11] [color={rgb, 255:red, 0; green, 0; blue, 0 }  ][line width=0.75]    (6.56,-1.97) .. controls (4.17,-0.84) and (1.99,-0.18) .. (0,0) .. controls (1.99,0.18) and (4.17,0.84) .. (6.56,1.97)   ;
\draw    (122,45.33) .. controls (120.33,41.33) and (123,36.67) .. (124,34.67) ;
\draw    (125.18,48.18) .. controls (126.7,55.7) and (142.09,53.09) .. (150.27,49.82) ;
\draw    (116.97,71.41) .. controls (122.27,63.09) and (119,58.73) .. (122,45.33) ;
\draw    (121,71.45) .. controls (126.3,63.14) and (122.18,61.58) .. (125.18,48.18) ;
\draw    (137.91,71.82) .. controls (148.82,62.73) and (144.27,58.55) .. (150.27,49.82) ;
\draw    (142.43,72.33) .. controls (150.97,63.61) and (147.91,60) .. (155,47) ;
\draw  [fill={rgb, 255:red, 0; green, 0; blue, 0 }  ,fill opacity=1 ] (204.52,72.74) .. controls (204.52,71.35) and (205.4,70.22) .. (206.5,70.22) .. controls (207.59,70.22) and (208.48,71.35) .. (208.48,72.74) .. controls (208.48,74.13) and (207.59,75.25) .. (206.5,75.25) .. controls (205.4,75.25) and (204.52,74.13) .. (204.52,72.74) -- cycle ;
\draw    (224.09,73.36) .. controls (244.06,75.19) and (230.72,97.14) .. (203.08,96.7) ;
\draw [shift={(223.84,92.06)}, rotate = 325.73] [color={rgb, 255:red, 0; green, 0; blue, 0 }  ][line width=0.75]    (6.56,-1.97) .. controls (4.17,-0.84) and (1.99,-0.18) .. (0,0) .. controls (1.99,0.18) and (4.17,0.84) .. (6.56,1.97)   ;
\draw    (188.62,72.64) .. controls (164.52,72.94) and (182.52,97.54) .. (203.08,96.7) ;
\draw    (193.36,72.64) -- (203.52,72.74) ;
\draw    (210.92,72.94) -- (219.12,73.14) ;
\draw    (203.09,38.03) .. controls (208.76,18.03) and (256.3,35.06) .. (232.64,51.73) ;
\draw [shift={(231.37,32.27)}, rotate = 201.44] [color={rgb, 255:red, 0; green, 0; blue, 0 }  ][line width=0.75]    (6.56,-1.97) .. controls (4.17,-0.84) and (1.99,-0.18) .. (0,0) .. controls (1.99,0.18) and (4.17,0.84) .. (6.56,1.97)   ;
\draw    (201.55,51.18) .. controls (199.88,47.18) and (202.09,40.03) .. (203.09,38.03) ;
\draw    (204.27,51.55) .. controls (205.79,59.06) and (221.18,56.45) .. (229.36,53.18) ;
\draw    (188.62,72.64) .. controls (193.92,64.32) and (195.55,64.09) .. (201.55,51.18) ;
\draw    (193.36,72.64) .. controls (198.66,64.32) and (198.45,66.09) .. (204.27,52.27) ;
\draw    (219.12,73.14) .. controls (226.45,67) and (223.55,61.18) .. (229.55,52.45) ;
\draw    (224.09,73.36) .. controls (228.27,67.55) and (227,64.45) .. (232.64,51.73) ;
\draw    (277.98,41.59) .. controls (283.65,21.59) and (331.19,38.62) .. (307.53,55.28) ;
\draw [shift={(306.25,35.83)}, rotate = 201.44] [color={rgb, 255:red, 0; green, 0; blue, 0 }  ][line width=0.75]    (6.56,-1.97) .. controls (4.17,-0.84) and (1.99,-0.18) .. (0,0) .. controls (1.99,0.18) and (4.17,0.84) .. (6.56,1.97)   ;
\draw    (276.43,54.74) .. controls (274.77,50.74) and (276.98,43.59) .. (277.98,41.59) ;
\draw    (279.16,55.1) .. controls (280.68,62.62) and (296.07,60.01) .. (304.25,56.74) ;
\draw    (266.52,73.53) .. controls (271.82,65.21) and (270.43,67.65) .. (276.43,54.74) ;
\draw    (270.89,73.56) .. controls (276.19,65.24) and (273.34,70.09) .. (279.16,56.27) ;
\draw    (284,73.56) .. controls (291.34,67.42) and (298.25,65.46) .. (304.25,56.74) ;
\draw    (288.82,73.83) .. controls (293,68.01) and (301.89,68.01) .. (307.53,55.28) ;
\draw    (444.63,76.94) .. controls (464.6,78.77) and (446.37,100.42) .. (418.73,99.98) ;
\draw [shift={(440.58,95.41)}, rotate = 330.16] [color={rgb, 255:red, 0; green, 0; blue, 0 }  ][line width=0.75]    (6.56,-1.97) .. controls (4.17,-0.84) and (1.99,-0.18) .. (0,0) .. controls (1.99,0.18) and (4.17,0.84) .. (6.56,1.97)   ;
\draw    (404.27,75.92) .. controls (380.17,76.22) and (398.17,100.82) .. (418.73,99.98) ;
\draw  [fill={rgb, 255:red, 0; green, 0; blue, 0 }  ,fill opacity=1 ] (406.27,75.92) .. controls (406.27,74.53) and (407.16,73.41) .. (408.25,73.41) .. controls (409.34,73.41) and (410.23,74.53) .. (410.23,75.92) .. controls (410.23,77.31) and (409.34,78.44) .. (408.25,78.44) .. controls (407.16,78.44) and (406.27,77.31) .. (406.27,75.92) -- cycle ;
\draw    (411.97,76.02) -- (419.17,76.02) ;
\draw    (423.2,76.07) -- (440.11,76.43) ;
\draw  [fill={rgb, 255:red, 0; green, 0; blue, 0 }  ,fill opacity=1 ] (601.64,75.86) .. controls (601.64,74.47) and (602.53,73.35) .. (603.62,73.35) .. controls (604.71,73.35) and (605.6,74.47) .. (605.6,75.86) .. controls (605.6,77.25) and (604.71,78.38) .. (603.62,78.38) .. controls (602.53,78.38) and (601.64,77.25) .. (601.64,75.86) -- cycle ;
\draw    (608.36,76.21) .. controls (628.33,78.04) and (610.81,99.64) .. (583.18,99.2) ;
\draw [shift={(604.86,94.64)}, rotate = 329.75] [color={rgb, 255:red, 0; green, 0; blue, 0 }  ][line width=0.75]    (6.56,-1.97) .. controls (4.17,-0.84) and (1.99,-0.18) .. (0,0) .. controls (1.99,0.18) and (4.17,0.84) .. (6.56,1.97)   ;
\draw    (568.71,75.14) .. controls (544.61,75.44) and (562.61,100.04) .. (583.18,99.2) ;
\draw    (573.09,75.17) -- (586.2,75.17) ;
\draw    (591.01,75.44) -- (599.21,75.64) ;
\draw    (426.2,39.28) .. controls (431.86,19.28) and (480.86,34.94) .. (457.2,51.61) ;
\draw [shift={(454.74,32.5)}, rotate = 198.11] [color={rgb, 255:red, 0; green, 0; blue, 0 }  ][line width=0.75]    (6.56,-1.97) .. controls (4.17,-0.84) and (1.99,-0.18) .. (0,0) .. controls (1.99,0.18) and (4.17,0.84) .. (6.56,1.97)   ;
\draw    (424.2,49.94) .. controls (422.53,45.94) and (425.2,41.28) .. (426.2,39.28) ;
\draw    (427.38,52.79) .. controls (428.89,60.31) and (444.29,57.7) .. (452.47,54.43) ;
\draw    (419.17,76.02) .. controls (424.47,67.7) and (449.47,67.82) .. (452.47,54.43) ;
\draw    (423.2,76.07) .. controls (429.71,71.43) and (454.43,69.29) .. (457.2,51.61) ;
\draw    (434.86,66.43) .. controls (433.14,64) and (423.57,65) .. (424.2,49.94) ;
\draw  [fill={rgb, 255:red, 0; green, 0; blue, 0 }  ,fill opacity=1 ] (503.15,75.55) .. controls (503.15,74.16) and (504.04,73.04) .. (505.13,73.04) .. controls (506.23,73.04) and (507.11,74.16) .. (507.11,75.55) .. controls (507.11,76.94) and (506.23,78.06) .. (505.13,78.06) .. controls (504.04,78.06) and (503.15,76.94) .. (503.15,75.55) -- cycle ;
\draw    (522.29,75.97) .. controls (542.26,77.8) and (528.91,99.75) .. (501.28,99.31) ;
\draw [shift={(522.04,94.67)}, rotate = 325.73] [color={rgb, 255:red, 0; green, 0; blue, 0 }  ][line width=0.75]    (6.56,-1.97) .. controls (4.17,-0.84) and (1.99,-0.18) .. (0,0) .. controls (1.99,0.18) and (4.17,0.84) .. (6.56,1.97)   ;
\draw    (486.81,75.25) .. controls (462.71,75.55) and (480.71,100.15) .. (501.28,99.31) ;
\draw    (491.56,75.25) -- (501.71,75.35) ;
\draw    (509.11,75.55) -- (517.31,75.75) ;
\draw    (501.29,40.64) .. controls (506.95,20.64) and (554.5,37.67) .. (530.83,54.34) ;
\draw [shift={(529.56,34.88)}, rotate = 201.44] [color={rgb, 255:red, 0; green, 0; blue, 0 }  ][line width=0.75]    (6.56,-1.97) .. controls (4.17,-0.84) and (1.99,-0.18) .. (0,0) .. controls (1.99,0.18) and (4.17,0.84) .. (6.56,1.97)   ;
\draw    (502.14,55.43) .. controls (499.29,50.14) and (500,45.57) .. (501.29,40.64) ;
\draw    (506,56.43) .. controls (513.14,61) and (516.1,59.84) .. (524.29,56.57) ;
\draw    (486.81,75.25) .. controls (498.14,62.14) and (515.29,67.57) .. (524.29,56.57) ;
\draw    (491.56,75.25) .. controls (496.86,66.93) and (525.02,68.16) .. (530.83,54.34) ;
\draw    (517.31,75.75) .. controls (517,71.14) and (512.71,73.29) .. (511,68.43) ;
\draw    (522.29,75.97) .. controls (520.43,71.57) and (517.43,72.43) .. (513.43,67.29) ;
\draw    (580.18,43.2) .. controls (585.84,23.2) and (633.39,40.23) .. (609.72,56.89) ;
\draw [shift={(608.45,37.44)}, rotate = 201.44] [color={rgb, 255:red, 0; green, 0; blue, 0 }  ][line width=0.75]    (6.56,-1.97) .. controls (4.17,-0.84) and (1.99,-0.18) .. (0,0) .. controls (1.99,0.18) and (4.17,0.84) .. (6.56,1.97)   ;
\draw    (578.63,56.35) .. controls (576.96,52.35) and (579.18,45.2) .. (580.18,43.2) ;
\draw    (581.36,56.71) .. controls (589.14,62.29) and (598.27,61.62) .. (606.45,58.35) ;
\draw    (568.71,75.14) .. controls (573.57,64.14) and (608.14,66.86) .. (606.45,58.35) ;
\draw    (573.09,75.17) .. controls (578.39,66.85) and (616.29,70.57) .. (609.72,56.89) ;
\draw    (582.71,66) .. controls (581.57,63) and (577,65.86) .. (578.63,56.35) ;
\draw    (585.71,64.86) .. controls (584.57,60.86) and (580.43,64) .. (581.36,56.71) ;
\draw    (438.14,72.07) .. controls (438.57,74) and (439.43,74.29) .. (440.11,76.43) ;
\draw    (427.38,52.79) .. controls (426.43,60) and (435.86,63.43) .. (437.29,65.29) ;
\draw    (441.14,70.43) .. controls (441.86,73.43) and (443.57,71.57) .. (444.63,76.94) ;
\draw    (502.14,55.43) .. controls (503.29,58) and (507,62) .. (507.29,63.29) ;
\draw    (506,56.43) .. controls (506.71,56.71) and (509.14,62.14) .. (510.29,62.71) ;
\draw    (585.43,71.43) .. controls (586.43,73.29) and (586.14,72.43) .. (586.2,75.17) ;
\draw    (589,70.71) .. controls (590,72.57) and (591.14,72.43) .. (591.2,75.17) ;
\draw   (80,48.5) .. controls (75.33,48.5) and (73,50.83) .. (73,55.5) -- (73,60.5) .. controls (73,67.17) and (70.67,70.5) .. (66,70.5) .. controls (70.67,70.5) and (73,73.83) .. (73,80.5)(73,77.5) -- (73,85.5) .. controls (73,90.17) and (75.33,92.5) .. (80,92.5) ;
\draw   (390,52) .. controls (385.33,52) and (383,54.33) .. (383,59) -- (383,64) .. controls (383,70.67) and (380.67,74) .. (376,74) .. controls (380.67,74) and (383,77.33) .. (383,84)(383,81) -- (383,89) .. controls (383,93.67) and (385.33,96) .. (390,96) ;
\draw   (323,91) .. controls (327.67,91.11) and (330.05,88.83) .. (330.16,84.16) -- (330.27,78.91) .. controls (330.42,72.24) and (332.83,68.96) .. (337.5,69.07) .. controls (332.83,68.96) and (330.57,65.58) .. (330.72,58.91)(330.66,61.91) -- (330.84,53.66) .. controls (330.95,48.99) and (328.67,46.61) .. (324,46.5) ;
\draw   (628.5,97.5) .. controls (633.17,97.45) and (635.47,95.09) .. (635.42,90.42) -- (635.36,85.67) .. controls (635.29,79) and (637.58,75.64) .. (642.25,75.59) .. controls (637.58,75.64) and (635.21,72.34) .. (635.13,65.67)(635.17,68.67) -- (635.08,60.92) .. controls (635.03,56.25) and (632.67,53.95) .. (628,54) ;

\draw (157.82,68.95) node [anchor=north west][inner sep=0.75pt]    {$+$};
\draw (236.48,68.95) node [anchor=north west][inner sep=0.75pt]    {$+$};
\draw (458.02,73.56) node [anchor=north west][inner sep=0.75pt]    {$+$};
\draw (538.68,73.56) node [anchor=north west][inner sep=0.75pt]    {$+$};
\draw (344.5,69.4) node [anchor=north west][inner sep=0.75pt]    {$+$};
\draw (51.5,62.9) node [anchor=north west][inner sep=0.75pt]    {$\frac{1}{3}$};
\draw (362.5,64.9) node [anchor=north west][inner sep=0.75pt]    {$\frac{1}{3}$};
\draw (17.67,62.9) node [anchor=north west][inner sep=0.75pt]    {$\frac{1}{\sqrt{N}^{2}}$};

\end{tikzpicture}

\end{center}
This is:
\begin{equation}\label{2 b eq 1}
    \sqrt{N}\cdot 2 \cdot p_{1}\frac{\partial}{\partial p_{2}}
\end{equation}
For case 2-(b), we have:
\begin{center}

\tikzset{every picture/.style={line width=0.75pt}} 

\begin{tikzpicture}[x=0.75pt,y=0.75pt,yscale=-1,xscale=1]

\draw    (157.43,92.33) .. controls (177.4,94.16) and (159.17,115.81) .. (131.54,115.37) ;
\draw [shift={(153.38,110.8)}, rotate = 330.16] [color={rgb, 255:red, 0; green, 0; blue, 0 }  ][line width=0.75]    (6.56,-1.97) .. controls (4.17,-0.84) and (1.99,-0.18) .. (0,0) .. controls (1.99,0.18) and (4.17,0.84) .. (6.56,1.97)   ;
\draw    (117.07,91.31) .. controls (92.97,91.61) and (110.97,116.21) .. (131.54,115.37) ;
\draw  [fill={rgb, 255:red, 0; green, 0; blue, 0 }  ,fill opacity=1 ] (119.07,91.31) .. controls (119.07,89.92) and (119.96,88.8) .. (121.05,88.8) .. controls (122.15,88.8) and (123.03,89.92) .. (123.03,91.31) .. controls (123.03,92.7) and (122.15,93.82) .. (121.05,93.82) .. controls (119.96,93.82) and (119.07,92.7) .. (119.07,91.31) -- cycle ;
\draw    (124.77,91.41) -- (131.97,91.41) ;
\draw    (136,91.45) -- (152.91,91.82) ;
\draw  [fill={rgb, 255:red, 0; green, 0; blue, 0 }  ,fill opacity=1 ] (314.44,94.25) .. controls (314.44,92.86) and (315.33,91.74) .. (316.42,91.74) .. controls (317.52,91.74) and (318.41,92.86) .. (318.41,94.25) .. controls (318.41,95.64) and (317.52,96.76) .. (316.42,96.76) .. controls (315.33,96.76) and (314.44,95.64) .. (314.44,94.25) -- cycle ;
\draw    (321.16,94.6) .. controls (341.13,96.43) and (323.62,118.03) .. (295.98,117.59) ;
\draw [shift={(317.67,113.03)}, rotate = 329.75] [color={rgb, 255:red, 0; green, 0; blue, 0 }  ][line width=0.75]    (6.56,-1.97) .. controls (4.17,-0.84) and (1.99,-0.18) .. (0,0) .. controls (1.99,0.18) and (4.17,0.84) .. (6.56,1.97)   ;
\draw    (281.52,93.53) .. controls (257.42,93.83) and (275.42,118.43) .. (295.98,117.59) ;
\draw    (285.89,93.56) -- (299,93.56) ;
\draw    (303.82,93.83) -- (312.02,94.03) ;
\draw    (139,54.67) .. controls (144.67,34.67) and (193.67,50.33) .. (170,67) ;
\draw [shift={(167.54,47.88)}, rotate = 198.11] [color={rgb, 255:red, 0; green, 0; blue, 0 }  ][line width=0.75]    (6.56,-1.97) .. controls (4.17,-0.84) and (1.99,-0.18) .. (0,0) .. controls (1.99,0.18) and (4.17,0.84) .. (6.56,1.97)   ;
\draw    (137,65.33) .. controls (135.33,61.33) and (138,56.67) .. (139,54.67) ;
\draw    (140.18,68.18) .. controls (141.7,75.7) and (157.09,73.09) .. (165.27,69.82) ;
\draw    (131.97,91.41) .. controls (137.27,83.09) and (134,78.73) .. (137,65.33) ;
\draw    (136,91.45) .. controls (141.3,83.14) and (137.18,81.58) .. (140.18,68.18) ;
\draw    (152.91,91.82) .. controls (163.82,82.73) and (159.27,78.55) .. (165.27,69.82) ;
\draw    (157.43,92.33) .. controls (165.97,83.61) and (162.91,80) .. (170,67) ;
\draw  [fill={rgb, 255:red, 0; green, 0; blue, 0 }  ,fill opacity=1 ] (219.95,92.94) .. controls (219.95,91.55) and (220.84,90.42) .. (221.94,90.42) .. controls (223.03,90.42) and (223.92,91.55) .. (223.92,92.94) .. controls (223.92,94.33) and (223.03,95.45) .. (221.94,95.45) .. controls (220.84,95.45) and (219.95,94.33) .. (219.95,92.94) -- cycle ;
\draw    (239.09,93.36) .. controls (259.06,95.19) and (245.72,117.14) .. (218.08,116.7) ;
\draw [shift={(238.84,112.06)}, rotate = 325.73] [color={rgb, 255:red, 0; green, 0; blue, 0 }  ][line width=0.75]    (6.56,-1.97) .. controls (4.17,-0.84) and (1.99,-0.18) .. (0,0) .. controls (1.99,0.18) and (4.17,0.84) .. (6.56,1.97)   ;
\draw    (203.62,92.64) .. controls (179.52,92.94) and (197.52,117.54) .. (218.08,116.7) ;
\draw    (208.36,92.64) -- (218.52,92.74) ;
\draw    (225.92,92.94) -- (234.12,93.14) ;
\draw    (218.09,58.03) .. controls (223.76,38.03) and (271.3,55.06) .. (247.64,71.73) ;
\draw [shift={(246.37,52.27)}, rotate = 201.44] [color={rgb, 255:red, 0; green, 0; blue, 0 }  ][line width=0.75]    (6.56,-1.97) .. controls (4.17,-0.84) and (1.99,-0.18) .. (0,0) .. controls (1.99,0.18) and (4.17,0.84) .. (6.56,1.97)   ;
\draw    (216.55,71.18) .. controls (214.88,67.18) and (217.09,60.03) .. (218.09,58.03) ;
\draw    (219.27,71.55) .. controls (220.79,79.06) and (236.18,76.45) .. (244.36,73.18) ;
\draw    (203.62,92.64) .. controls (208.92,84.32) and (210.55,84.09) .. (216.55,71.18) ;
\draw    (208.36,92.64) .. controls (213.66,84.32) and (213.45,86.09) .. (219.27,72.27) ;
\draw    (234.12,93.14) .. controls (241.45,87) and (238.55,81.18) .. (244.55,72.45) ;
\draw    (239.09,93.36) .. controls (243.27,87.55) and (242,84.45) .. (247.64,71.73) ;
\draw    (292.98,61.59) .. controls (298.65,41.59) and (346.19,58.62) .. (322.53,75.28) ;
\draw [shift={(321.25,55.83)}, rotate = 201.44] [color={rgb, 255:red, 0; green, 0; blue, 0 }  ][line width=0.75]    (6.56,-1.97) .. controls (4.17,-0.84) and (1.99,-0.18) .. (0,0) .. controls (1.99,0.18) and (4.17,0.84) .. (6.56,1.97)   ;
\draw    (291.43,74.74) .. controls (289.77,70.74) and (291.98,63.59) .. (292.98,61.59) ;
\draw    (294.16,75.1) .. controls (295.68,82.62) and (311.07,80.01) .. (319.25,76.74) ;
\draw    (281.52,93.53) .. controls (286.82,85.21) and (285.43,87.65) .. (291.43,74.74) ;
\draw    (285.89,93.56) .. controls (291.19,85.24) and (288.34,90.09) .. (294.16,76.27) ;
\draw    (299,93.56) .. controls (306.34,87.42) and (313.25,85.46) .. (319.25,76.74) ;
\draw    (303.82,93.83) .. controls (308,88.01) and (316.89,88.01) .. (322.53,75.28) ;
\draw    (459.63,96.94) .. controls (479.6,98.77) and (461.37,120.42) .. (433.73,119.98) ;
\draw [shift={(455.58,115.41)}, rotate = 330.16] [color={rgb, 255:red, 0; green, 0; blue, 0 }  ][line width=0.75]    (6.56,-1.97) .. controls (4.17,-0.84) and (1.99,-0.18) .. (0,0) .. controls (1.99,0.18) and (4.17,0.84) .. (6.56,1.97)   ;
\draw    (419.27,95.92) .. controls (395.17,96.22) and (413.17,120.82) .. (433.73,119.98) ;
\draw  [fill={rgb, 255:red, 0; green, 0; blue, 0 }  ,fill opacity=1 ] (420.85,96.07) .. controls (420.85,94.68) and (421.74,93.55) .. (422.83,93.55) .. controls (423.93,93.55) and (424.81,94.68) .. (424.81,96.07) .. controls (424.81,97.46) and (423.93,98.58) .. (422.83,98.58) .. controls (421.74,98.58) and (420.85,97.46) .. (420.85,96.07) -- cycle ;
\draw    (426.97,96.02) -- (434.17,96.02) ;
\draw    (438.2,96.07) -- (455.11,96.43) ;
\draw  [fill={rgb, 255:red, 0; green, 0; blue, 0 }  ,fill opacity=1 ] (616.64,95.86) .. controls (616.64,94.47) and (617.53,93.35) .. (618.62,93.35) .. controls (619.71,93.35) and (620.6,94.47) .. (620.6,95.86) .. controls (620.6,97.25) and (619.71,98.38) .. (618.62,98.38) .. controls (617.53,98.38) and (616.64,97.25) .. (616.64,95.86) -- cycle ;
\draw    (623.36,96.21) .. controls (643.33,98.04) and (625.81,119.64) .. (598.18,119.2) ;
\draw [shift={(619.86,114.64)}, rotate = 329.75] [color={rgb, 255:red, 0; green, 0; blue, 0 }  ][line width=0.75]    (6.56,-1.97) .. controls (4.17,-0.84) and (1.99,-0.18) .. (0,0) .. controls (1.99,0.18) and (4.17,0.84) .. (6.56,1.97)   ;
\draw    (583.71,95.14) .. controls (559.61,95.44) and (577.61,120.04) .. (598.18,119.2) ;
\draw    (588.09,95.17) -- (601.2,95.17) ;
\draw    (606.01,95.44) -- (614.21,95.64) ;
\draw    (441.2,59.28) .. controls (446.86,39.28) and (495.86,54.94) .. (472.2,71.61) ;
\draw [shift={(469.74,52.5)}, rotate = 198.11] [color={rgb, 255:red, 0; green, 0; blue, 0 }  ][line width=0.75]    (6.56,-1.97) .. controls (4.17,-0.84) and (1.99,-0.18) .. (0,0) .. controls (1.99,0.18) and (4.17,0.84) .. (6.56,1.97)   ;
\draw    (439.2,69.94) .. controls (437.53,65.94) and (440.2,61.28) .. (441.2,59.28) ;
\draw    (442.38,72.79) .. controls (443.89,80.31) and (459.29,77.7) .. (467.47,74.43) ;
\draw    (434.17,96.02) .. controls (439.47,87.7) and (464.47,87.82) .. (467.47,74.43) ;
\draw    (438.2,96.07) .. controls (444.71,91.43) and (469.43,89.29) .. (472.2,71.61) ;
\draw    (449.86,86.43) .. controls (448.14,84) and (438.57,85) .. (439.2,69.94) ;
\draw  [fill={rgb, 255:red, 0; green, 0; blue, 0 }  ,fill opacity=1 ] (518.45,95.76) .. controls (518.45,94.37) and (519.34,93.25) .. (520.44,93.25) .. controls (521.53,93.25) and (522.42,94.37) .. (522.42,95.76) .. controls (522.42,97.15) and (521.53,98.28) .. (520.44,98.28) .. controls (519.34,98.28) and (518.45,97.15) .. (518.45,95.76) -- cycle ;
\draw    (537.29,95.97) .. controls (557.26,97.8) and (543.91,119.75) .. (516.28,119.31) ;
\draw [shift={(537.04,114.67)}, rotate = 325.73] [color={rgb, 255:red, 0; green, 0; blue, 0 }  ][line width=0.75]    (6.56,-1.97) .. controls (4.17,-0.84) and (1.99,-0.18) .. (0,0) .. controls (1.99,0.18) and (4.17,0.84) .. (6.56,1.97)   ;
\draw    (501.81,95.25) .. controls (477.71,95.55) and (495.71,120.15) .. (516.28,119.31) ;
\draw    (506.56,95.25) -- (516.71,95.35) ;
\draw    (524.11,95.55) -- (532.31,95.75) ;
\draw    (516.29,60.64) .. controls (521.95,40.64) and (569.5,57.67) .. (545.83,74.34) ;
\draw [shift={(544.56,54.88)}, rotate = 201.44] [color={rgb, 255:red, 0; green, 0; blue, 0 }  ][line width=0.75]    (6.56,-1.97) .. controls (4.17,-0.84) and (1.99,-0.18) .. (0,0) .. controls (1.99,0.18) and (4.17,0.84) .. (6.56,1.97)   ;
\draw    (517.14,75.43) .. controls (514.29,70.14) and (515,65.57) .. (516.29,60.64) ;
\draw    (521,76.43) .. controls (528.14,81) and (531.1,79.84) .. (539.29,76.57) ;
\draw    (501.81,95.25) .. controls (513.14,82.14) and (530.29,87.57) .. (539.29,76.57) ;
\draw    (506.56,95.25) .. controls (511.86,86.93) and (540.02,88.16) .. (545.83,74.34) ;
\draw    (532.31,95.75) .. controls (532,91.14) and (527.71,93.29) .. (526,88.43) ;
\draw    (537.29,95.97) .. controls (535.43,91.57) and (532.43,92.43) .. (528.43,87.29) ;
\draw    (595.18,63.2) .. controls (600.84,43.2) and (648.39,60.23) .. (624.72,76.89) ;
\draw [shift={(623.45,57.44)}, rotate = 201.44] [color={rgb, 255:red, 0; green, 0; blue, 0 }  ][line width=0.75]    (6.56,-1.97) .. controls (4.17,-0.84) and (1.99,-0.18) .. (0,0) .. controls (1.99,0.18) and (4.17,0.84) .. (6.56,1.97)   ;
\draw    (593.63,76.35) .. controls (591.96,72.35) and (594.18,65.2) .. (595.18,63.2) ;
\draw    (596.36,76.71) .. controls (604.14,82.29) and (613.27,81.62) .. (621.45,78.35) ;
\draw    (583.71,95.14) .. controls (588.57,84.14) and (623.14,86.86) .. (621.45,78.35) ;
\draw    (588.09,95.17) .. controls (593.39,86.85) and (631.29,90.57) .. (624.72,76.89) ;
\draw    (597.71,86) .. controls (596.57,83) and (592,85.86) .. (593.63,76.35) ;
\draw    (600.71,84.86) .. controls (599.57,80.86) and (595.43,84) .. (596.36,76.71) ;
\draw    (453.14,92.07) .. controls (453.57,94) and (454.43,94.29) .. (455.11,96.43) ;
\draw    (442.38,72.79) .. controls (441.43,80) and (450.86,83.43) .. (452.29,85.29) ;
\draw    (456.14,90.43) .. controls (456.86,93.43) and (458.57,91.57) .. (459.63,96.94) ;
\draw    (517.14,75.43) .. controls (518.29,78) and (522,82) .. (522.29,83.29) ;
\draw    (521,76.43) .. controls (521.71,76.71) and (524.14,82.14) .. (525.29,82.71) ;
\draw    (600.43,91.43) .. controls (601.43,93.29) and (601.14,92.43) .. (601.2,95.17) ;
\draw    (604,90.71) .. controls (605,92.57) and (606.14,92.43) .. (606.2,95.17) ;
\draw   (95,68.5) .. controls (90.33,68.5) and (88,70.83) .. (88,75.5) -- (88,80.5) .. controls (88,87.17) and (85.67,90.5) .. (81,90.5) .. controls (85.67,90.5) and (88,93.83) .. (88,100.5)(88,97.5) -- (88,105.5) .. controls (88,110.17) and (90.33,112.5) .. (95,112.5) ;
\draw   (405,72) .. controls (400.33,72) and (398,74.33) .. (398,79) -- (398,84) .. controls (398,90.67) and (395.67,94) .. (391,94) .. controls (395.67,94) and (398,97.33) .. (398,104)(398,101) -- (398,109) .. controls (398,113.67) and (400.33,116) .. (405,116) ;
\draw   (338,111) .. controls (342.67,111.11) and (345.05,108.83) .. (345.16,104.16) -- (345.27,98.91) .. controls (345.42,92.24) and (347.83,88.96) .. (352.5,89.07) .. controls (347.83,88.96) and (345.57,85.58) .. (345.72,78.91)(345.66,81.91) -- (345.84,73.66) .. controls (345.95,68.99) and (343.67,66.61) .. (339,66.5) ;
\draw   (643.5,117.5) .. controls (648.17,117.45) and (650.47,115.09) .. (650.42,110.42) -- (650.36,105.67) .. controls (650.29,99) and (652.58,95.64) .. (657.25,95.59) .. controls (652.58,95.64) and (650.21,92.34) .. (650.13,85.67)(650.17,88.67) -- (650.08,80.92) .. controls (650.03,76.25) and (647.67,73.95) .. (643,74) ;
\draw  [fill={rgb, 255:red, 0; green, 0; blue, 0 }  ,fill opacity=1 ] (149.85,45.46) .. controls (149.85,44.07) and (150.74,42.94) .. (151.84,42.94) .. controls (152.93,42.94) and (153.82,44.07) .. (153.82,45.46) .. controls (153.82,46.85) and (152.93,47.97) .. (151.84,47.97) .. controls (150.74,47.97) and (149.85,46.85) .. (149.85,45.46) -- cycle ;
\draw  [fill={rgb, 255:red, 0; green, 0; blue, 0 }  ,fill opacity=1 ] (229.45,49.46) .. controls (229.45,48.07) and (230.34,46.94) .. (231.44,46.94) .. controls (232.53,46.94) and (233.42,48.07) .. (233.42,49.46) .. controls (233.42,50.85) and (232.53,51.97) .. (231.44,51.97) .. controls (230.34,51.97) and (229.45,50.85) .. (229.45,49.46) -- cycle ;
\draw  [fill={rgb, 255:red, 0; green, 0; blue, 0 }  ,fill opacity=1 ] (303.45,53.08) .. controls (303.45,51.69) and (304.34,50.56) .. (305.44,50.56) .. controls (306.53,50.56) and (307.42,51.69) .. (307.42,53.08) .. controls (307.42,54.47) and (306.53,55.59) .. (305.44,55.59) .. controls (304.34,55.59) and (303.45,54.47) .. (303.45,53.08) -- cycle ;
\draw  [fill={rgb, 255:red, 0; green, 0; blue, 0 }  ,fill opacity=1 ] (452.8,49.46) .. controls (452.8,48.07) and (453.69,46.94) .. (454.78,46.94) .. controls (455.88,46.94) and (456.76,48.07) .. (456.76,49.46) .. controls (456.76,50.85) and (455.88,51.97) .. (454.78,51.97) .. controls (453.69,51.97) and (452.8,50.85) .. (452.8,49.46) -- cycle ;
\draw  [fill={rgb, 255:red, 0; green, 0; blue, 0 }  ,fill opacity=1 ] (525.2,52.03) .. controls (525.2,50.65) and (526.09,49.52) .. (527.18,49.52) .. controls (528.28,49.52) and (529.16,50.65) .. (529.16,52.03) .. controls (529.16,53.42) and (528.28,54.55) .. (527.18,54.55) .. controls (526.09,54.55) and (525.2,53.42) .. (525.2,52.03) -- cycle ;
\draw  [fill={rgb, 255:red, 0; green, 0; blue, 0 }  ,fill opacity=1 ] (603.2,55.06) .. controls (603.2,53.67) and (604.09,52.54) .. (605.18,52.54) .. controls (606.28,52.54) and (607.16,53.67) .. (607.16,55.06) .. controls (607.16,56.45) and (606.28,57.57) .. (605.18,57.57) .. controls (604.09,57.57) and (603.2,56.45) .. (603.2,55.06) -- cycle ;

\draw (172.82,88.95) node [anchor=north west][inner sep=0.75pt]    {$+$};
\draw (251.48,88.95) node [anchor=north west][inner sep=0.75pt]    {$+$};
\draw (473.02,93.56) node [anchor=north west][inner sep=0.75pt]    {$+$};
\draw (553.68,93.56) node [anchor=north west][inner sep=0.75pt]    {$+$};
\draw (359.5,89.4) node [anchor=north west][inner sep=0.75pt]    {$+$};
\draw (66.5,82.9) node [anchor=north west][inner sep=0.75pt]    {$\frac{1}{3}$};
\draw (377.5,84.9) node [anchor=north west][inner sep=0.75pt]    {$\frac{1}{3}$};
\draw (32.67,82.9) node [anchor=north west][inner sep=0.75pt]    {$\frac{1}{\sqrt{N}^{i}}$};
\draw (139.6,29.6) node [anchor=north west][inner sep=0.75pt]  [font=\small]  {$i-2$};
\draw (217.6,33.2) node [anchor=north west][inner sep=0.75pt]  [font=\small]  {$i-2$};
\draw (294.8,36) node [anchor=north west][inner sep=0.75pt]  [font=\small]  {$i-2$};
\draw (442,33.2) node [anchor=north west][inner sep=0.75pt]  [font=\small]  {$i-2$};
\draw (514.8,36) node [anchor=north west][inner sep=0.75pt]  [font=\small]  {$i-2$};
\draw (593.2,36.8) node [anchor=north west][inner sep=0.75pt]  [font=\small]  {$i-2$};

\end{tikzpicture}

\end{center}
The leading order terms are the first, third, fifth terms in the above graph. Putting some combinatorial factors, this contributes to:
\begin{equation}\label{2 b eq 2}
    \sqrt{N}\cdot i \cdot p_{i-1}\frac{\partial}{\partial p_{i}}+O(\frac{1}{\sqrt{N}})
\end{equation}
Notice the leading terms in (\ref{2 b eq 1}), (\ref{2 b eq 2}) has the same form, so case 2-(a) and 2-(b) together gives:
\begin{equation}\label{2 a+b eq}
    \sqrt{N}\sum_{i=2}^{E}i\cdot p_{i-1}\frac{\partial}{\partial p_{i}}+O(\frac{1}{\sqrt{N}})
\end{equation}
Now it's clear that for the case 2-(c), as well as case 3, we do not form loops without any black dots. Hence all these terms have order $O(\frac{1}{\sqrt{N}})$.

Finally, let's consider case 4. After one of the white dot is contracted with one of the black dot in the expression (\ref{t21 0}) of $t_{2,1}$, we get:
\begin{center}

\tikzset{every picture/.style={line width=0.75pt}} 

\begin{tikzpicture}[x=0.75pt,y=0.75pt,yscale=-1,xscale=1]

\draw   (289.5,31.83) .. controls (284.83,31.93) and (282.55,34.31) .. (282.65,38.98) -- (282.79,45.73) .. controls (282.93,52.4) and (280.67,55.78) .. (276,55.88) .. controls (280.67,55.78) and (283.07,59.06) .. (283.21,65.73)(283.15,62.73) -- (283.36,72.48) .. controls (283.45,77.15) and (285.83,79.43) .. (290.5,79.33) ;
\draw   (527.9,80.43) .. controls (532.57,80.43) and (534.9,78.1) .. (534.9,73.43) -- (534.9,66.93) .. controls (534.9,60.26) and (537.23,56.93) .. (541.9,56.93) .. controls (537.23,56.93) and (534.9,53.6) .. (534.9,46.93)(534.9,49.93) -- (534.9,40.43) .. controls (534.9,35.76) and (532.57,33.43) .. (527.9,33.43) ;
\draw   (321.97,45.08) .. controls (321.97,43.52) and (323.08,42.25) .. (324.45,42.25) .. controls (325.83,42.25) and (326.94,43.52) .. (326.94,45.08) .. controls (326.94,46.64) and (325.83,47.9) .. (324.45,47.9) .. controls (323.08,47.9) and (321.97,46.64) .. (321.97,45.08) -- cycle ;
\draw    (345.43,46) .. controls (365.4,47.83) and (347.17,69.48) .. (319.54,69.03) ;
\draw [shift={(341.38,64.46)}, rotate = 330.16] [color={rgb, 255:red, 0; green, 0; blue, 0 }  ][line width=0.75]    (6.56,-1.97) .. controls (4.17,-0.84) and (1.99,-0.18) .. (0,0) .. controls (1.99,0.18) and (4.17,0.84) .. (6.56,1.97)   ;
\draw    (305.07,44.98) .. controls (280.97,45.28) and (298.97,69.88) .. (319.54,69.03) ;
\draw  [fill={rgb, 255:red, 0; green, 0; blue, 0 }  ,fill opacity=1 ] (306.65,45.14) .. controls (306.65,43.74) and (307.66,42.61) .. (308.91,42.61) .. controls (310.16,42.61) and (311.17,43.74) .. (311.17,45.14) .. controls (311.17,46.53) and (310.16,47.67) .. (308.91,47.67) .. controls (307.66,47.67) and (306.65,46.53) .. (306.65,45.14) -- cycle ;
\draw    (312.77,45.08) -- (319.97,45.08) ;
\draw    (328.37,45.28) -- (336.57,45.48) ;
\draw    (421.29,46.62) .. controls (441.26,48.45) and (425.46,70.05) .. (397.82,69.61) ;
\draw [shift={(419.13,65.04)}, rotate = 328.31] [color={rgb, 255:red, 0; green, 0; blue, 0 }  ][line width=0.75]    (6.56,-1.97) .. controls (4.17,-0.84) and (1.99,-0.18) .. (0,0) .. controls (1.99,0.18) and (4.17,0.84) .. (6.56,1.97)   ;
\draw    (383.36,45.55) .. controls (359.26,45.85) and (377.26,70.45) .. (397.82,69.61) ;
\draw    (387.8,45.7) -- (401.4,45.7) ;
\draw    (405.66,45.85) -- (413.86,46.05) ;
\draw   (468.41,27.08) .. controls (468.41,25.52) and (469.52,24.25) .. (470.89,24.25) .. controls (472.26,24.25) and (473.37,25.52) .. (473.37,27.08) .. controls (473.37,28.64) and (472.26,29.9) .. (470.89,29.9) .. controls (469.52,29.9) and (468.41,28.64) .. (468.41,27.08) -- cycle ;
\draw    (506.31,28.05) .. controls (526.29,29.88) and (508.77,51.48) .. (481.14,51.03) ;
\draw [shift={(502.82,46.47)}, rotate = 329.75] [color={rgb, 255:red, 0; green, 0; blue, 0 }  ][line width=0.75]    (6.56,-1.97) .. controls (4.17,-0.84) and (1.99,-0.18) .. (0,0) .. controls (1.99,0.18) and (4.17,0.84) .. (6.56,1.97)   ;
\draw    (466.67,26.98) .. controls (442.57,27.28) and (460.57,51.88) .. (481.14,51.03) ;
\draw    (474.37,27.08) -- (481.57,27.08) ;
\draw    (487.2,27.3) -- (500.8,27.3) ;
\draw   (338.83,45.65) .. controls (338.83,44.09) and (339.94,42.82) .. (341.31,42.82) .. controls (342.68,42.82) and (343.79,44.09) .. (343.79,45.65) .. controls (343.79,47.21) and (342.68,48.47) .. (341.31,48.47) .. controls (339.94,48.47) and (338.83,47.21) .. (338.83,45.65) -- cycle ;
\draw   (414.86,46.05) .. controls (414.86,44.49) and (415.97,43.22) .. (417.34,43.22) .. controls (418.71,43.22) and (419.82,44.49) .. (419.82,46.05) .. controls (419.82,47.61) and (418.71,48.87) .. (417.34,48.87) .. controls (415.97,48.87) and (414.86,47.61) .. (414.86,46.05) -- cycle ;
\draw    (383.37,45.28) .. controls (392.2,23.5) and (403.8,34.5) .. (405.66,45.85) ;
\draw    (387.8,45.7) .. controls (393.4,30.1) and (401.2,39.5) .. (401.4,45.7) ;
\draw    (482.37,27.28) .. controls (491.2,5.5) and (504.46,16.7) .. (506.31,28.05) ;
\draw    (487.2,27.3) .. controls (487.42,26.69) and (487.64,26.11) .. (487.87,25.57) .. controls (493.4,12.44) and (500.61,21.34) .. (500.8,27.3) ;
\draw   (483.61,70.88) .. controls (483.61,69.32) and (484.72,68.05) .. (486.09,68.05) .. controls (487.46,68.05) and (488.57,69.32) .. (488.57,70.88) .. controls (488.57,72.44) and (487.46,73.7) .. (486.09,73.7) .. controls (484.72,73.7) and (483.61,72.44) .. (483.61,70.88) -- cycle ;
\draw    (507.11,71.85) .. controls (527.09,73.68) and (509.57,95.28) .. (481.94,94.83) ;
\draw [shift={(503.62,90.27)}, rotate = 329.75] [color={rgb, 255:red, 0; green, 0; blue, 0 }  ][line width=0.75]    (6.56,-1.97) .. controls (4.17,-0.84) and (1.99,-0.18) .. (0,0) .. controls (1.99,0.18) and (4.17,0.84) .. (6.56,1.97)   ;
\draw    (467.47,70.78) .. controls (443.37,71.08) and (461.37,95.68) .. (481.94,94.83) ;
\draw    (472.4,70.7) -- (482.37,70.88) ;
\draw    (489.6,70.9) -- (501.6,71.1) ;
\draw    (467.47,70.78) .. controls (476.3,49) and (505.4,57.1) .. (507.11,71.85) ;
\draw    (472.4,70.7) .. controls (484.6,50.1) and (502.8,67.7) .. (501.6,71.1) ;

\draw (261.17,44.73) node [anchor=north west][inner sep=0.75pt]    {$\frac{1}{3}$};
\draw (356.43,47.35) node [anchor=north west][inner sep=0.75pt]    {$+$};
\draw (433.29,46.78) node [anchor=north west][inner sep=0.75pt]    {$+$};

\end{tikzpicture}

\end{center}
For the first term in the above expression, the two white dots have to be contracted with black dots in $\ket{c_{1},c_{2},\cdots,c_{N}}$, which belongs to case 1-3. The remaining terms add up to $N\mathrm{Tr}(Z)$ which we have computed: (\ref{t1,0}). Hence case 4 gives rise to:
\begin{equation}\label{case 4}
{\sqrt{N}}^{3}\frac{\partial}{\partial p_{1}}+\sqrt{N}\sum_{i= 2}^{E}ip_{i-1}\frac{\partial}{\partial p_{i}}
\end{equation}

Summarizing all these, from (\ref{case 1-a}),(\ref{case 1b}),(\ref{2 a+b eq}), in the limit $N \rightarrow \infty$, we have:
\begin{align}
     t_{2,1}=&\frac{1}{\sqrt{N}}k\frac{N(N-1)}{2}\cdot 2 \cdot \frac{\partial}{\partial p_{1}}+ \sqrt{N}\cdot k \cdot 2 \cdot \sum_{i= 2}^{E} ip_{i-1}\frac{\partial}{\partial p_{i}}+\sqrt{N}\sum_{i=2}^{E}i\cdot p_{i-1}\frac{\partial}{\partial p_{i}}\\ \nonumber
    &+{\sqrt{N}}^{3}\frac{\partial}{\partial p_{1}}+\sqrt{N}\sum_{i= 2}^{E}i\cdot p_{i-1}\frac{\partial}{\partial p_{i}}+O(\frac{1}{\sqrt{N}})
\end{align}
Equivalently 
\begin{equation}\label{t2,1 eq1}
    t_{2,1}=((k+1){\sqrt{N}}^{3}-k\sqrt{N})\frac{\partial}{\partial p_{1}}+2(k+1)\sqrt{N}\sum_{i=2}^{E}i\cdot p_{i-1}\frac{\partial}{\partial p_{i}}+O(\frac{1}{\sqrt{N}})
\end{equation}
Apply (\ref{t1,0}), we have:
\begin{equation}\label{t21}
     {t_{2,1}}_{\leq E}=(N(k+1)-k){t_{1,0}}_{\leq E}+(k+1)\sqrt{N}\sum_{i= 2}^{E}i\cdot p_{i-1}\frac{\partial}{\partial p_{i}}+O(\frac{1}{\sqrt{N}})
\end{equation}

${t_{1,2}}_{\leq E}$ can be computed in a similar way. The result is:
\begin{align}\label{t12 eq1}
    {t_{1,2}}_{\leq E}=(k+1)\sqrt{N}^{3}p_{1}-k\sqrt{N}p_{1}+\sqrt{N}\sum_{i= 1}^{E}ip_{i+1}\frac{\partial}{\partial p_{i}}
\end{align}
Substitute $t_{0,1}=\sqrt{N}p_{1}$ into (\ref{t12 eq1}), we get another expression for ${t_{1,2}}_{\leq E}$:
\begin{equation}\label{t12 eq2}
    {t_{1,2}}_{\leq E}=(k+1)\sqrt{N}^{2}{t_{0,1}}_{\leq E}-k{t_{0,1}}_{\leq E}+\sqrt{N}\sum_{i= 1}^{E}ip_{i+1}\frac{\partial}{\partial p_{i}}
\end{equation}

\subsection{Large \texorpdfstring{$N$}{N} limit of \texorpdfstring{$t_{m,n}$}{t[m,n]}}
We first compute the leading terms of $t_{n,0}, n\geq 1$ in the energy cutoff large $N$ limit. We have the following result
\begin{proposition}
    In $N\rightarrow \infty$ limit, under the isomorphism $\mathcal{H}_N\cong \mathbb{C}[p_{1},p_{2},\cdots,p_{N}]$, we have
    \begin{equation}\label{tn0}
        t_{n,0}=(k+1)^{n-1}{\sqrt{N}}^{n}n\cdot\frac{\partial}{\partial p_{n}}+O({\sqrt{N}}^{n-2})
    \end{equation}
\end{proposition}
\begin{proof}
    We show this by induction over $n$. For fix energy cutoff $E>1$, (\ref{t1,0}) implies
    \begin{align*}
    {t_{1,0}}_{\leq E}=\sqrt{N}\frac{\partial}{\partial p_{1}}+\frac{1}{\sqrt{N}}\sum_{i= 2}^{E}ip_{i-1}\frac{\partial}{\partial p_{i}}=\sqrt{N}\frac{\partial}{\partial p_{1}}+O(\frac{1}{\sqrt{N}})
    \end{align*}
    Hence, for $n=1$ the statement is true.
    Assume for $n$, (\ref{tn0}) is true. We show this is also true for $n+1$. Now we fix any arbitrary energy cutoff $E\geq n+1$. The commutation relation
    \begin{equation*}
        t_{n+1,0}=\frac{1}{n}[t_{n,0},t_{2,1}]
    \end{equation*}
    implies
    \begin{align*}
        {t_{n+1,0}}_{\leq E}=\frac{1}{n}[{t_{n,0}}_{\leq E},{t_{2,1}}_{\leq E}]
    \end{align*}
    By inductive assumption (\ref{tn0}) for $n$, we have:
    \begin{align*}
        {t_{n,0}}_{\leq E}=(k+1)^{n-1}{\sqrt{N}}^{n}n\cdot\frac{\partial}{\partial p_{n}}+O({\sqrt{N}}^{n-2})
    \end{align*}
    Using (\ref{t21}), we compute:
    \begin{align*}
        {t_{n+1,0}}_{\leq E}=&\frac{1}{n}[{t_{n,0}}_{\leq E},(N(k+1)-k){t_{1,0}}_{\leq E}+(k+1)\sqrt{N}\sum_{i= 2}^{E}i\cdot p_{i-1}\frac{\partial}{\partial p_{i}}+O(\frac{1}{\sqrt{N}})]\\
        =&\frac{1}{n}[{t_{n,0}}_{\leq E},(k+1)\sqrt{N}\sum_{i= 2}^{E}i\cdot p_{i-1}\frac{\partial}{\partial p_{i}}+O(\frac{1}{\sqrt{N}})]\\
        =&\frac{1}{n}[(k+1)^{n-1}{\sqrt{N}}^{n}n\cdot\frac{\partial}{\partial p_{n}}+O({\sqrt{N}}^{n-2}),(k+1)\sqrt{N}\sum_{i= 2}^{E}i\cdot p_{i-1}\frac{\partial}{\partial p_{i}}+O(\frac{1}{\sqrt{N}})]\\
        =&(k+1)^{n}{\sqrt{N}}^{n+1}(n+1)\frac{\partial}{\partial p_{n+1}}+O({\sqrt{N}}^{n-1})
    \end{align*}
\end{proof}

From (\ref{t2,1 eq1}), and $t_{0,1}=\sqrt{N}p_{1}$ we also have:
\begin{equation}\label{t11}
    t_{1,1}=\frac{1}{2}[t_{2,1},t_{0,1}]=\frac{1}{2}(k+1)\sqrt{N}^{4}+O(\sqrt{N}^{2})
\end{equation}
in the large $N$ limit.

\leftline{\textbf{Main result}}
Fix $m,n$, and fix any energy cutoff $E>m+n$. In the limit $N \gg E$, $N \rightarrow \infty$, under the isomorphism $\mathcal{H}_N\cong \mathbb{C}[p_{1},p_{2},\cdots,p_{N}]$, we have:
\begin{align}
   \label{tmn} &{t_{m,n}}_{\leq E}=\sqrt{N}^{m+n}(k+1)^{m}p_{n-m}+O(\sqrt{N}^{m+n-2})  \quad if \quad m < n  \\
   \label{tnn} &{t_{n,n}}_{\leq E}=\frac{1}{n+1}(k+1)^{n}\sqrt{N}^{2n+2}+
  O(\sqrt{N}^{2n})   \quad if \quad m=n\\
  \label{tnm}  &{t_{n,m}}_{\leq E}=\sqrt{N}^{n+m}(k+1)^{n-1}(n-m)\frac{\partial}{\partial p_{n-m}}+O(\sqrt{N}^{m+n-2}) \quad if \quad m < n
\end{align}

\begin{equation}\label{t21tmn}
    [{t_{2,1}}_{\leq E},{t_{m,n}}_{\leq E}]=(2n-m){t_{m+1,n}}_{\leq E}+\begin{cases}
    O(\sqrt{N}^{m+n-1}), & \text{if $m+1\neq n$}\\
    O(\sqrt{N}^{m+n+1}), & \text{if $m+1 = n$}
    \end{cases}
\end{equation}
\begin{equation}\label{t12tmn}
    [{t_{1,2}}_{\leq E},{t_{m,n}}_{\leq E}]=(n-2m){t_{m,n+1}}_{\leq E}+\begin{cases}
    O(\sqrt{N}^{m+n-1}), & \text{if $m \neq n+1 $}\\
    O(\sqrt{N}^{m+n+1}), & \text{if $m = n+1 $}
    \end{cases}
\end{equation}

\begin{proof}
     We prove (\ref{tmn})-(\ref{tnm}), (\ref{t21tmn}) and (\ref{t12tmn}) by induction over $s=m+n$. We have already shown the statement is true for $s=0,1,2$, see (\ref{t2,1 eq1})(\ref{t12 eq1})(\ref{tn0}) and (\ref{t11}). Assume now for $m+n\leq s$, (\ref{tmn})-(\ref{tnm}), (\ref{t21tmn}) and (\ref{t12tmn}) are true, we'll show all these are true for $m+n=s+1$.\\ 
     We first show (\ref{tmn})-(\ref{tnm}) are true for $m+n=s+1$. Consider ${t_{m,n}}_{\leq E}$ for $m<n , m+n=s+1$, from the induction assumption, (\ref{t21tmn}) is true for $m-1+n=s$, we have:
     \begin{equation}
         {t_{m,n}}_{\leq E}=\frac{1}{2n-m+1}[{t_{2,1}}_{\leq E},{t_{m-1,n}}_{\leq E}]+O(\sqrt{N}^{m+n-2})
     \end{equation}
    The right hand side of the above equation can be computed using the induction assumption (\ref{tmn}) for ${t_{m-1,n}}_{\leq E}$:
    \begin{equation}
        {t_{m-1,n}}_{\leq E}=\sqrt{N}^{m+n-1}(k+1)^{m-1}p_{n-m+1}+O(\sqrt{N}^{m+n-3})
    \end{equation}
    and also the expression (\ref{t21}) for ${t_{2,1}}_{\leq E}$, we have:
    \begin{align} \label{tmn in the proof}
        &{t_{m,n}}_{\leq E}=\frac{1}{2n-m+1}\{\sqrt{N}^{2}(k+1)[{t_{1,0}}_{\leq E},{t_{m-1,n}}_{\leq E}]+(k+1)^{m}\sqrt{N}^{m+n}[\sum_{i=2}^{E} ip_{i-1}\frac{\partial}{\partial p_{i}},p_{n-m+1}]\}+O(\sqrt{N}^{m+n-2})\\ \nonumber
      &=\frac{1}{2n-m+1}\{\sqrt{N}^{2}(k+1)n{t_{m-1,n-1}}_{\leq E}+(k+1)^{m}\sqrt{N}^{m+n}(n-m+1)p_{n-m})\}+O(\sqrt{N}^{m+n-2})
    \end{align}
    From the induction assumption, $${t_{m-1,n-1}}_{\leq E}=\sqrt{N}^{m+n-2}(k+1)^{m-1}p_{n-m}+O(\sqrt{N}^{m+n-4})$$. Substitute this into the right hand side of (\ref{tmn in the proof}), we get
\begin{align}
    {t_{m,n}}_{\leq E}&=\frac{1}{2n-m+1}\{(n+n-m+1)\sqrt{N}^{m+n}(k+1)^{m}p_{n-m}\}+O(\sqrt{N}^{m+n-2})\\ \nonumber
    &=\sqrt{N}^{m+n}(k+1)^{m}p_{n-m}+O(\sqrt{N}^{m+n-2})
\end{align}
    hence (\ref{tmn}) is true for $m+n=s+1$. 
Next, in the case $s+1$ is even, we compute ${t_{n,n}}_{\leq E}$ for $n=\frac{s+1}{2}$ similarly by using the expression (\ref{t21}) for ${t_{2,1}}_{\leq E}$, and the induction assumption (\ref{tmn}) for ${t_{n-1,n}}_{\leq E}$:
\begin{align}
     {t_{n,n}}_{\leq E}&=\frac{1}{n+1}[{t_{2,1}}_{\leq E},{t_{n-1,n}}_{\leq E}]+O(\sqrt{N}^{2n})\\ \nonumber
     &=\frac{1}{n+1}\sqrt{N}^{2}(k+1)[{t_{1,0}}_{\leq E},\sqrt{N}^{2n-1}(k+1)^{n-1}p_{1}]+O(\sqrt{N}^{2n})\\ \nonumber
     &=\frac{1}{n+1}\sqrt{N}^{2n+2}(k+1)^{n}+O(\sqrt{N}^{2n})
\end{align}
Hence (\ref{tnn}) is true for $2n=s+1$. Then we compute ${t_{n,m}}_{\leq E}$ for the case $m < n$. For this, we use similar argument. By the induction assumption (\ref{tmn}), and the expression (\ref{t12 eq2}), we have:
\begin{align}
    {t_{n,m}}_{\leq E}&=\frac{1}{2n-m+1}[{t_{n,m-1}}_{\leq E},{t_{1,2}}_{\leq E}]+O(\sqrt{N}^{n+m-2})\\ \nonumber
    &=\frac{1}{2n-m+1}\{(k+1)\sqrt{N}^{2}[{t_{n,m-1}}_{\leq E},{t_{0,1}}_{\leq E}]+\sqrt{N}^{n+m}(k+1)^{n-1}(n-m+1)(n-m)\frac{\partial}{\partial p_{n-m}}\}\\
    \nonumber&+O(\sqrt{N}^{n+m-2})\\ \nonumber
    &=\frac{1}{2n-m+1}{\sqrt{N}^{n+m}(k+1)^{n-1}(n+n-m+1)(n-m)\frac{\partial}{\partial p_{n-m}}}+O(\sqrt{N}^{n+m-2})\\ \nonumber
    &=\sqrt{N}^{n+m}(k+1)^{n-1}(n-m)\frac{\partial}{\partial p_{n-m}}+O(\sqrt{N}^{n+m-2})
\end{align}
Hence, (\ref{tnm}) is also true for $n+m=s+1$.\\
Finally, (\ref{t21tmn}) and (\ref{t12tmn}) are true for $m+n=s+1$. This follows from the induction assumption and commutation relations (\ref{[t21,tmn]}) and (\ref{[t12,tmn]}).
\end{proof}

\section{Summary and outlook}\label{sec: Summary and outlook}
In this paper, we investigate the algebra of observables for the Chern-Simons matrix model. A set of basic commutation relations between generators $e^a_{b;m,n}$ and $t_{m,n}$ are found in (\ref{trace of e})-(\ref{[t03,tmn]}), which leads to our definition of large $N$ limit algebra $\mathscr{O}_{\infty}^{(p)}$ (Definition \ref{def: large N limit algebra}). In $\mathscr{O}_{\infty}^{(p)}$, the commutation relations between generators $e^a_{b;m,n}$ and $t_{m,n}$ can be found inductively as described in the proof of Proposition \ref{prop: filtration}. Although we do not provide explicit formula for all commutation relations except for the basic ones, we expect that such explicit formula would be useful for further exploration of the matrix model. We anticipate that the generating functions (\ref{3.10}) and (\ref{3.11}) are applicable in deriving such a formula.

We also study the action of $\mathscr{O}_{N}^{(p)}$ on the Hilbert space $\mathcal H_N$. We construct the large $N$ limit Hilbert space in the $p=1$ case and show that $\mathscr{O}_{\infty}^{(1)}$ acts on it, and then derive the leading terms of a set of generators $\{t_{m,n}\}_{m,n \geq 0}$ (Theorem \ref{thm: large N limit}). Finally, we demonstrate in Appendix \ref{App C} that a normalized semicircle law is satisfied in the ground state wave function. The physical implication of this is that the central charge of the large $N$ limit algebra is equal to the filling factor of the corresponding quantum Hall droplet, which is $\frac{1}{1+k}$. To extend these results to the case $p>2$, we believe that it is crucial to perform computations using an appropriate basis, which we leave for future research.

The edge excitations of the Chern-Simons matrix model have been investigated in \cite{Poly, Rodriguez_2009, Frenkel:2021yql} using semiclassical approach. In contrast, this paper presents a full quantum approach to analyzing the excitation modes of the matrix model in the large $N$ limit. However, it remains unclear how to establish a connection between the excitation modes of the matrix model in the large $N$ limit and the edge excitation modes from semiclassical methods. Presumably, one could address this issue by studying the random matrix ensembles defined in \cite{Karabali-Sakita}. (Further discussions on this topic are given in Appendix \ref{App C}.) 
Furthermore, it is interesting to understand the bulk-boundary correspondence within the framework of the Chern-Simons matrix model. We anticipate that the large $N$ limit of the Chern-Simons matrix model have deep connections with ordinary Chern-Simons theory and conformal field theory \cite{Witten}.


We have also developed another approach to the quantization of the Chern-Simons matrix model based on geometric quantization. Using geometric methods we can also obtain the Large $N$ limit of the algebra of observables as well as its representation. We will present the method elsewhere \cite{Hu-Li-Ye-Zhou}.
         
The method developed in this paper should be in principal applicable to the supergroup generalization of the Chern-Simons matrix model \cite{okazaki2018matrix}. The obvious generalization would be studying the algebra of observables $\mathscr O^{(p|q)}_N(\epsilon_1,\epsilon_2)$ in a $U(N)$ Chern-Simons matrix model with $U(p|q)$ internal symmetry. We expect that there exists a large $N$ limit algebra $\mathscr O^{(p|q)}_{\infty}(\epsilon_1,\epsilon_2)$ defined similarly to the $q=0$ case considered in this paper, and that it should be isomorphic to Costello's uniform in $N$ algebra $\mathscr O^{(p|q)}_{\bullet}(\epsilon_1,\epsilon_2)$ defined in \cite[Section 13]{Costello}. We also expect that $\mathscr O^{(p|q)}_{\infty}(\epsilon_1,\epsilon_2)$ is isomorphic to the deformed double current algebra of type $\mathfrak{gl}(p|q)$. It is conjectured in \cite{okazaki2018matrix} that $\mathscr O^{(p|q)}_N(\epsilon_1,\epsilon_2)$ in the large $N$ limit admits a scaling limit which truncates to the affine Kac-Moody algebra of $\mathfrak u(p|q)$, and the latter acts on the large $N$ Hilbert space of the matrix model. It will be interesting to prove their conjecture using the technique developed in this paper. A more challenging task is to study the $U(N|M)$ Chern-Simons matrix model with $U(p|q)$ internal symmetry, where both $N$ and $M$ can be large.

It will also be interesting to generalize our method to matrix models in higher dimensions, for example 4+1d non-commutative Chern-Simons theory arising from describing 4d quantum Hall effect \cite{barns2018adhm}, and matrix models from string theory and M theory, in particular the BFSS model \cite{banks1999m} and the IKKT model \cite{ishibashi1997large}. It may shed new light on the non-perturbative construction of string/M theory.

\section*{Acknowledgements}
We would like to thank Andrey Losev and Shing-Tung Yau for discussions and encouragements concerning this work. Special thank goes to D. Gaiotto for communicating his conjecture about the commutation relations \eqref{[eab10,ecdmn]} and \eqref{[t30,eabmn]} in the case $m=0$ to us. Part of the work were done at the Yau Mathematical Science Center at Tsinghua University and we are indebted to its support. S. Hu and D. Ye is partially supported by the Wu Wen-Tsun Key Lab of Mathematics of Chinese Academy of Sciences at USTC. S. Li is supported by the National Key R and D Program of China (No. 2020YFA0713000). Y. Zhou is supported by Kavli IPMU through World Premier International Research Center Initiative (WPI), MEXT, Japan. 
 
\appendix
\section{Diagrammatic notations}\label{sec: Diagrammatic notations}
We introduce some diagrammatic notations used in the computations of commutation relations of the algebra $\mathscr{O}_{N}^{(p)}\left(\epsilon_{1}, \epsilon_{2}\right)$. Let $V$ be the $N$-dimensional vector space, which is the fundamental representation space of $\mathrm{U}(N)$. We use the following diagrams to represent a tensor $T^{i_{1}i_{2}\cdots i_{m}}_{j_{1}j_{2}\cdots j_{n}}$ in 
${V^{*}}^{\otimes m}\otimes V^{\otimes n}$:
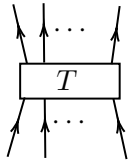
\begin{center}
\tikzset{every picture/.style={line width=0.75pt}} 
\begin{tikzpicture}[x=0.75pt,y=0.75pt,yscale=-1,xscale=1]

\draw   (307.85,56.19) -- (358.04,56.19) -- (358.04,73.81) -- (307.85,73.81) -- cycle ;
\draw    (301.5,103.16) -- (310.15,73.22) ;
\draw [shift={(306.83,84.73)}, rotate = 106.12] [color={rgb, 255:red, 0; green, 0; blue, 0 }  ][line width=0.75]    (6.56,-1.97) .. controls (4.17,-0.84) and (1.99,-0.18) .. (0,0) .. controls (1.99,0.18) and (4.17,0.84) .. (6.56,1.97)   ;
\draw    (320.25,73.81) -- (320.54,103.75) ;
\draw [shift={(320.35,84.18)}, rotate = 89.45] [color={rgb, 255:red, 0; green, 0; blue, 0 }  ][line width=0.75]    (6.56,-1.97) .. controls (4.17,-0.84) and (1.99,-0.18) .. (0,0) .. controls (1.99,0.18) and (4.17,0.84) .. (6.56,1.97)   ;
\draw    (354.87,73.81) -- (361.5,104.33) ;
\draw [shift={(357.21,84.57)}, rotate = 77.74] [color={rgb, 255:red, 0; green, 0; blue, 0 }  ][line width=0.75]    (6.56,-1.97) .. controls (4.17,-0.84) and (1.99,-0.18) .. (0,0) .. controls (1.99,0.18) and (4.17,0.84) .. (6.56,1.97)   ;
\draw    (354,56.46) -- (358.04,27.11) ;
\draw [shift={(356.51,38.22)}, rotate = 97.83] [color={rgb, 255:red, 0; green, 0; blue, 0 }  ][line width=0.75]    (6.56,-1.97) .. controls (4.17,-0.84) and (1.99,-0.18) .. (0,0) .. controls (1.99,0.18) and (4.17,0.84) .. (6.56,1.97)   ;
\draw    (319.67,25.67) -- (319.96,55.61) ;
\draw [shift={(319.77,36.04)}, rotate = 89.45] [color={rgb, 255:red, 0; green, 0; blue, 0 }  ][line width=0.75]    (6.56,-1.97) .. controls (4.17,-0.84) and (1.99,-0.18) .. (0,0) .. controls (1.99,0.18) and (4.17,0.84) .. (6.56,1.97)   ;
\draw    (304.38,26.25) -- (311.31,56.19) ;
\draw [shift={(306.81,36.74)}, rotate = 76.98] [color={rgb, 255:red, 0; green, 0; blue, 0 }  ][line width=0.75]    (6.56,-1.97) .. controls (4.17,-0.84) and (1.99,-0.18) .. (0,0) .. controls (1.99,0.18) and (4.17,0.84) .. (6.56,1.97)   ;

\draw (324.7,57.75) node [anchor=north west][inner sep=0.75pt]    {$T$};
\draw (322.4,81.24) node [anchor=north west][inner sep=0.75pt]    {$\cdots $};
\draw (322.98,34.86) node [anchor=north west][inner sep=0.75pt]    {$\cdots $};

\end{tikzpicture}

\end{center}
In this diagram, we have $m$ incoming lines and $n$ outgoing lines. Contracting upper and lower indices are represented by connecting incoming and outgoing lines. We denote the generators of the algebra as $Z^i_j$, ${Z^{\dagger}}^i_j$, $\lambda_a$, and ${\lambda^{\dagger}}^a$ by the following diagrams:
\begin{equation*}
\begin{tikzpicture}[decoration={markings, 
    mark= at position 0.30 with {\arrow{stealth}},
    mark= at position 0.88 with {\arrow{stealth}}}
] 
\draw [postaction={decorate}] (-0.5,0) -- (0.5,0);
\draw[fill=white,line width=1pt](0,0)circle(0.5ex);
\end{tikzpicture}, \quad
\begin{tikzpicture}[decoration={markings, 
    mark= at position 0.30 with {\arrow{stealth}},
    mark= at position 0.88 with {\arrow{stealth}}}
] 
\draw[fill=black,line width=1pt](0,0)circle(0.5ex);
\draw [postaction={decorate}] (-0.5,0) -- (0.5,0);
\end{tikzpicture}, \quad
\begin{tikzpicture}[decoration={markings, 
    mark= at position 0.50 with {\arrow{stealth}}}
] 
\draw [postaction={decorate}] (-0.3,0) -- (0.5,0);
\draw (0.8,0) node {$\lambda_a$};
\end{tikzpicture}, \quad \text{and} \quad
\begin{tikzpicture}[decoration={markings, 
    mark= at position 0.50 with {\arrow{stealth}}}
] 
\draw [postaction={decorate}] (-0.5,0) -- (0.3,0);
\draw (-0.8,0) node {${\lambda^{\dagger}}^a$};
\end{tikzpicture}.
\end{equation*}
All elements appearing in this paper can be expressed using diagrams, which consist of usual tensors (that commute with everything) decorated with white and black points on the oriented lines, along with \begin{tikzpicture}[decoration={markings, 
    mark= at position 0.50 with {\arrow{stealth}}}
] \draw [postaction={decorate}] (-0.3,0) -- (0.5,0); \draw (0.8,0) node {$\lambda_a$};\end{tikzpicture}, and \begin{tikzpicture}[decoration={markings, 
    mark= at position 0.50 with {\arrow{stealth}}}
] \draw [postaction={decorate}] (-0.5,0) -- (0.3,0); \draw (-0.8,0) node {${\lambda^{\dagger}}^a$};\end{tikzpicture}
at the endpoints of the oriented lines. 

We envision every element in the diagram that consists of $Z^i_j$, ${Z^{\dagger}}^i_j$, $\lambda_a$, ${\lambda^{\dagger}}^a$ as having a vertical dashed line (serve as auxiliary lines) passing through it. Whenever we continuously move an element $A$ around and cross such a vertical line going through another element $B$, there may be some quantum correction term proportional to $\epsilon_{1}$. This phenomenon arises from the non-commutativity of the matrix components (\ref{CQ1}). Diagrammatically, this phenomenon is illustrated in the following diagram, where we use a wavy line to denote quantum contractions:
\begin{equation}
\hbox{   
\tikzset{every picture/.style={line width=0.75pt}} 


}
\end{equation}     

    These quantum corrections are generated by the canonical quantization conditions (\ref{CQ1}). Diagrammatically:
\begin{equation}\label{WC}
  \hbox{

\tikzset{every picture/.style={line width=0.75pt}} 

%

  }  
\end{equation}

Let us illustrate how to apply the diagrammatic notation for performing calculations on the quantum algebra of the matrix model by considering two examples.

\begin{example}
    Let us demonstrate that $\mu^i_j$ generates the $\mathfrak{u}(N)$ action, as stated in Proposition (\ref{prop-action}), using our diagrammatic notation. The operator $\mu^i_j$ is depicted diagrammatically as:
    \begin{equation}
        \hbox{

\tikzset{every picture/.style={line width=0.75pt}} 

%

        }
    \end{equation}

Next, we show (\ref{mu1}): $[\mu^{i}_{j},Z^{k}_{l}]=\epsilon_{1} \delta^{i}_{l}Z^{k}_{j}-\epsilon_{1}\delta^{k}_{j}Z^{i}_{l}$ and (\ref{mu3}) $[\mu^{i}_{j},\lambda_{a}^{k}]=-\epsilon_{1} \delta^{k}_{j} \lambda_{a}^{i}$ diagrammatically:

\begin{equation}
    \hbox{

\tikzset{every picture/.style={line width=0.75pt}} 

%
}
\end{equation}

\end{example}

\begin{example}
    Another example is to show (\ref{52}) in Proposition (\ref{prop 53}): $$\operatorname{Tr}\left(A\left[Z, Z^{\dagger}\right] B\right)=\operatorname{Tr}\left(A\left(\epsilon_{2}-\lambda_{a} {{{\lambda^{\dagger}}}}^a\right) B\right)+\epsilon_{1} \operatorname{Tr} A \operatorname{Tr} B$$ diagrammatically. To show this, we consider the following identity:
\begin{equation}
\hbox{

\tikzset{every picture/.style={line width=0.75pt}} 

%

}
\end{equation}
Apply the following commutation relation ($\mu^i_j$ generates $\mathfrak{u}(N)$ action):
\begin{equation}
\hbox{

\tikzset{every picture/.style={line width=0.75pt}} 

%

}
\end{equation}
On the other hand:
\begin{equation}\label{4.31}
\hbox{

\tikzset{every picture/.style={line width=0.75pt}} 

%

\end{example}

\section{Connection to Murnaghan–Nakayama rule}\label{App B}
In this appendix, we establish connections between the terms:
\begin{align}
&C(n_{1},n_{2},\cdots,n_{N}):=\epsilon^{i_{1}i_{2}\cdots i_{N}}(\lambda^{\dagger}{Z^{\dagger}}^{n_{1}})_{i_{1}}(\lambda^{\dagger}{Z^{\dagger}}^{n_{2}})_{i_{2}}\cdots(\lambda^{\dagger}{Z^{\dagger}}^{n_{N}})_{i_{N}}\\
&t_{0,n}:=\mathrm{Tr}({Z^{\dagger}}^{n})
\end{align}
to the Schur polynomials and power sum polynomials in $N$ variables. We demonstrate that equation (\ref{ti C}) is equivalent to the Murnaghan–Nakayama rule, which is a combinatorial rule governing the transformation between the bases of Schur polynomials and power sum polynomials within the ring of symmetric polynomials. Since equation (\ref{ti C}) is derived using the diagrammatic method in this paper, this can be viewed as a new diagrammatic proof of the classical Murnaghan–Nakayama rule. For further details on the Murnaghan–Nakayama rule, we refer to \cite{stanley2023enumerative} \cite{mendes2019combinatorics}.

Following \cite{Hellerman-Raamsdonk}, one can view $C(n_1,n_2,\cdots,n_N)$ and $\mathrm{Tr}({Z^{\dagger}}^n)$ as polynomials in the variables $\{{Z^{\dagger}}^i_j, 1\leq i,j \leq N;{\lambda^{\dagger}}_i, 1 \leq i \leq N\}$. The connections between $C(n_1,n_2,\cdots,n_N)$ and $\mathrm{Tr}({Z^{\dagger}}^n)$ to the Shur polynomials and power sum polynomials are established by evaluating them on diagonal matrices $Z^{\dagger}=diag(z_1,z_2,\cdots,z_N)$.  Consequently, $C(n_1,n_2,\cdots,n_N)$ and $\mathrm{Tr}({Z^{\dagger}}^n)$ becomes

\begin{align}\label{evl}
    \det[z^{n_{i}}_{j}]_{1\leq i,j\leq N}\prod_{i=1}^{N}{\lambda^{\dagger}}_{i} \quad \text{and}  \quad
\sum_{i=1}^{N}z_{i}^{n} 
\end{align}
We assume $0 \leq n_1 < n_2 < n_3 < \cdots <n_N$. Let $\mu_{i}:=n_{N-i+1}-(N-i)$, then 
 $\mu_1\geq \mu_2\geq \mu_3\geq \cdots \mu_N\geq 0$. Hence (\ref{evl}) becomes (up to an overall sign)
    
\begin{align}\label{evl2}
    \det[z^{\mu_{i}+N-i}_{j}]_{1\leq i,j\leq N}\prod_{i=1}^{N}{\lambda^{\dagger}}_{i} \quad \text{and}  \quad
P_{n}=\sum_{i=1}^{N}z_{i}^{n} 
\end{align}    
where $P_{n}$ is the power sum polynomial.
    The classical definition of Schur polynomials states:
    $$s_{\mu}(z_1,z_2,\cdots,z_N)=\frac{\det[z^{\mu_{i}+N-i}_{j}]_{1\leq i,j\leq N}}{\det[z^{N-i}_{j}]_{1\leq i,j\leq N}}$$
Now, let us evaluate equation (\ref{ti C}) at $Z^{\dagger}=diag(z_1,z_2,\cdots,z_N)$, and divide both sides by $C(0,1,2,\cdots,N-1)$. This yields an equation of the following form:
\begin{equation}\label{ps}
    P_{n}(z_1,z_2\cdots,z_N)s_{\mu}(z_1,z_2,\cdots,z_N)=\sum_{\nu \in \Lambda}\pm s_{\nu}(z_1,z_2\cdots,z_N)
\end{equation}

Here, $\Lambda$ represents a set of specific integer partitions. Let us identify the integer partitions $\nu \in \Lambda$ that appear in the summation on the right-hand side of (\ref{ps}), along with the sign of each term. We illustrate this in terms of Young diagrams. We will consider two types of Young diagrams associated with  $C(n_1,n_2,\cdots,n_N)$, where the condition $0 \leq n_1 < n_2 < n_3 < \cdots <n_N$ do not need to be satisfied. The first kind of diagrams are the non-standard one: we simply place $n_{N-i+1}$ boxes on the $ith$ row. We allow the number of boxes to be exchanged between two different rows, and after exchanging, we assign a negative sign. This reflects the totally antisymmetric property of $C(n_1,n_2,\cdots,n_N)$. 

For example for $N=5$, and $C(1,3,8,5,7)$, we have the following (non-standard) Young diagram:
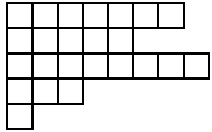
\begin{center}
\ytableausetup{smalltableaux}
\ydiagram{7,5,8,3,1}
\end{center}

Let us illustrate how to obtain the second type of Young diagrams associated with $C(n_1,n_2,\cdots,n_N)$. These are the standard diagrams, each carrying a sign. We begin with the first type of Young diagram (typically non-standard) associated with $C(n_1,n_2,\cdots,n_N)$, which carries a sign. We rearrange the rows such that the number of boxes does not increase, and then we mark the first $N-i$ boxes in each 
$i$th row with a cross. Each row exchange results in a minus sign. After removing all the crosses and shifting all the boxes to the leftmost position, we obtain the second tpye of Young diagram (a standard one) corresponding to a partition $\nu$ and carrying a sign. By the correspondence described previously, $C(n_1,n_2,\cdots,n_N)$ corresponds to the Schur polynomial $\pm s_{\nu}$. 

An example will best illustrate this. Consider the above example $C(1,3,8,5,7)$: 
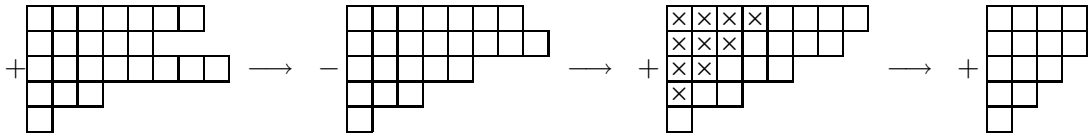
\begin{center}
\ytableausetup{centertableaux}
\[
+\ydiagram{7,5,8,3,1}
~\longrightarrow~
-\ydiagram{7,8,5,3,1}
~\longrightarrow~ +
\begin{ytableau}
\cross & \cross & \cross & \cross &  &  &  &   \\
\cross & \cross & \cross &  &  &  & \\
\cross & \cross &  &  &\\
\cross &  &\\
\\
\end{ytableau}
~\longrightarrow~
+\ydiagram{4,4,3,2,1}
\]

\end{center}

In the following, let us identify the integer partitions corresponding to the terms on the right-hand side of equation (\ref{ti C}) in terms of (standard) Young diagrams. It's better to proceed by considering an example. Let us take $n=5$ and $C(n_1=1,n_2=3,n_3=5,n_4=7,n_5=8)$ in (\ref{ti C}). There are 5 terms on the right hand side of equation (\ref{ti C}). The first term is $C(1+5,3,5,7,8)$, which corresponds to the following Young diagram of the second type:
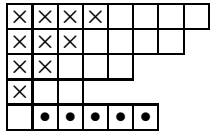
\begin{center}
    \ytableausetup{centertableaux}
\begin{ytableau}
\cross & \cross & \cross & \cross &  &  &  &   \\
\cross & \cross & \cross &  &  &  & \\
\cross & \cross &  &  &\\
\cross &  &\\
& \bullet& \bullet& \bullet& \bullet& \bullet\\
\end{ytableau}
\end{center}
In the above diagram, we decorate the newly added boxes with black dots. Now, we permute the number of boxes in each row until we arrive at a diagram with a non-increasing number of boxes in each row from top to bottom. This permutation is achieved by first exchanging the number of boxes in the 4th and the 5th rows, and then exchanging the 3rd and the 4th rows, and so on. These exchanges can be visualized by moving the boxes containing black dots:
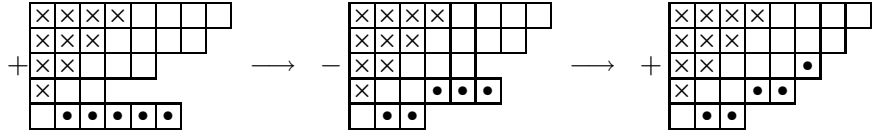
\begin{center}
    \ytableausetup{centertableaux}
\[+ \begin{ytableau}
\cross & \cross & \cross & \cross &  &  &  &   \\
\cross & \cross & \cross &  &  &  & \\
\cross & \cross &  &  &\\
\cross &  &\\
& \bullet& \bullet& \bullet& \bullet& \bullet\\
\end{ytableau}
~\longrightarrow~
- \begin{ytableau}
\cross & \cross & \cross & \cross &  &  &  &   \\
\cross & \cross & \cross &  &  &  & \\
\cross & \cross &  &  &\\
\cross &  & & \bullet& \bullet& \bullet\\
& \bullet& \bullet\\
\end{ytableau}
~\longrightarrow~
+ \begin{ytableau}
\cross & \cross & \cross & \cross &  &  &  &   \\
\cross & \cross & \cross &  &  &  & \\
\cross & \cross &  &  & & \bullet\\
\cross &  & & \bullet& \bullet\\
& \bullet& \bullet\\
\end{ytableau}
\]
\end{center}
Finally, by removing all the crosses and shifting all the boxes to the leftmost position, we obtain the standard diagram corresponding to the partition labeling the Schur polynomial corresponding to this pariticular term $C(1+5,3,8,5,7)$ in $\Tr ({Z^{\dagger}}^5)C(1,3,5,7,8)$:    
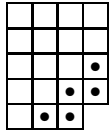
\begin{center}
\ydiagram[*(white) \bullet]
{4+0,4+0,3+1,2+2,1+2}
*[*(white)]{4,4,4,4,3}
\end{center}
The sign of this term is $+$.

By examining several additional examples, it becomes evident that the newly introduced boxes (depicted as boxes containing black dots) typically take on the following shape:

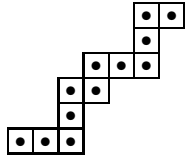
\begin{center}
\ytableaushort{\none\none\none\none\none\bullet\bullet,\none\none\none\none\none\bullet,
\none\none\none\bullet\bullet\bullet,\none\none\bullet\bullet,\none\none\bullet,\bullet\bullet\bullet}
\end{center}

This shape is commonly referred to as a border-strip diagram. More precisely:

\begin{definition}
Let $\nu$ be a standard Young diagram and $\mu$ another standard Young diagram contained within $\nu$. We denote $\nu \backslash \mu$ as the set of boxes in the Young diagram of $\nu$ that do not appear in $\mu$. We say that $\nu \backslash \mu$ forms a skew shape diagram. A skew shape diagram is considered a border-strip diagram if it is connected and does not contain a $2 \times 2$ arrangement of boxes. The height of a border-strip diagram $\nu \backslash \mu$ is defined as $ht(\nu \backslash \mu) = \text{number of rows of } \nu \backslash \mu - 1$. The size of $\nu \backslash \mu$ is defined as the number of boxes it contains.
\end{definition}

The preceding demonstration by example indicates that the partitions $\nu \in \Lambda$ appearing on the right-hand side of (\ref{ps}) correspond to Young diagrams $\nu$ such that $\nu \backslash \mu$ forms a border-strip diagram. The sign associated with each term $s_{\nu}(z_1,z_2,\cdots,z_N)$ on the right-hand side of (\ref{ps}) is $(-1)^{ht(\nu \backslash \mu)}$, where $ht(\nu \backslash \mu)$ represents the height of the border-strip diagram $\nu \backslash \mu$ as defined above. In summary, we state the following proposition:

\begin{proposition}
The following identity holds:
\begin{equation}\label{MN rule}
     P_{n}(z_1,z_2,\cdots,z_N)s_{\mu}(z_1,z_2,\cdots,z_N) = \sum_{\nu \in \Lambda} (-1)^{ht(\nu \backslash \mu)} s_{\nu}(z_1,z_2,\cdots,z_N)
\end{equation}
where $\Lambda=\{\nu \supseteq \mu |\quad \nu \backslash \mu \text{ is a border-strip diagram with } n \text{ boxes}\}$.
\end{proposition}

This represents one version of the Murnaghan–Nakayama rule.

Here's a refined version:

Let $\nu \vdash n$ and $\mu$ be a composition of $n$, i.e., an ordered tuple $\mu=(\mu_1, \mu_2, \ldots, \mu_k)$ of positive integers such that $\sum_{i=1}^{k} \mu_i = n$. We define a border-strip diagram $T$ of shape $\nu$ and content $\mu$ as a filling of the cells of the Young diagram of $\nu$ with border-strips of sizes $\mu_1,\cdots,\mu_{l}$ labeled with $1,\cdots,l$, such that the removal of the last $i$ border-strips results in a smaller (standard) Young diagram for all $i$. The sign of $T$ is determined by the product of the signs of these border-strips. The set $BST(\mu,\nu)$ is defined as the collection of border-strip diagrams of shape $\nu$ and content $\mu$.

Utilizing the rule (\ref{MN rule}), one can expand the power sum symmetric function $P_{\mu}:=P_{\mu_1}P_{\mu_2}\cdots P_{\mu_l}$ inductively in terms of the base of Schur symmetric functions $\{s_{\mu}:\mu\vdash n\}$. Consequently, we obtain the following combinatorial rule for the change of bases:
\begin{align}
    P_{\mu} = \sum_{\nu \vdash n} \left\{\sum_{T\in BST(\mu,\nu)}(-1)^{ht(T)}\right\}s_{\nu}
\end{align}
Through the Frobenius characteristic map, the ring of symmetric functions is related to the ring of characters of symmetric groups. This yields:
\begin{align}
\nonumber P_{\mu} = \sum_{\nu \vdash n}{\chi}^{\nu}_{\mu}s_{\nu} \quad \text{and} \quad s_{\nu} = \sum_{\mu \vdash n}{\chi}^{\nu}_{\mu}\frac{P_{\mu}}{z_{\mu}}
\end{align}
Here, $\chi^{\nu}_{\mu}$ denotes the value of the character of the irreducible representation of the symmetric group $S_{n}$ labeled by the partition $\mu \vdash n$ on the conjugacy class $C_{\mu}$ of cycle type $\mu$, and $z_{\mu} = \frac{n!}{|C_{\mu}|}$.
 Thus, the Murnaghan-Nakayama rule provides a combinatorial method for computing $\chi^{\nu}_{\mu}$.

\begin{remark}\label{Remark O(1)}
Using these correspondences, the coefficients of (\ref{1-b eq 2}) in terms of the basis (\ref{Basis non re}) are approximately equal to $i$ times the number of ways of decomposing a Young diagram of $i$ boxes into border-strips. For a fixed energy cutoff $E$ and $i \leq E$, these coefficients are finite fixed numbers independent of $N$. Consequently, (\ref{1-b eq 2}) has order $O(1)$ in the limit $N \gg E$, $N\rightarrow \infty$.
\end{remark}

\section{Renormalized semicircle distribution and filling factor}\label{App C}

In this appendix, we provide a physical interpretation of our assertion that the central charge of the large $N$ limit algebra $\widehat{\mathfrak{u}(1)}_{\frac{1}{k+1}}$ equals the filling factor of the corresponding quantum Hall fluid.

To establish connections between the matrix model and the quantum Hall effect, Karabali and Sakita introduced the notion of wave functions in the Chern-Simons matrix model\cite{Karabali-Sakita}. Specifically, they introduced states $\ket{X,\lambda}$, parametrized by $N \cross N$ Hermitian matrices $X$ and an $N$-vector $\lambda$. These states are defined by the following properties:
\begin{align}
\label{X eigen}&\hat{X}\ket{X,\lambda}=X\ket{X,\lambda},\\
\label{lambda eigen}&\hat{\lambda}_i\ket{X,\lambda}=\lambda_{i}\ket{X,\lambda},\\
\label{complete}& \int [dX] d\lambda d\lambda^{\dagger} \ket{X,\lambda}\bra{X,\lambda}=\mathrm{Id}.
\end{align}
\textbf{Notation}:
In this section, we use  $\hat{Z},\hat{Z}^{\dagger},\hat{X},\hat{Y}, \hat{\lambda}, \hat{\lambda}^{\dagger}$ to denote operators in the quantum matrix model. $X$ is a classical $N \cross N$ matrix, $\lambda$ is a classical $N$-vector. The operators $\hat{Z},\hat{Z}^{\dagger},\hat{X},\hat{Y}$ are related by
\begin{align}
    &\hat{Z}=\sqrt{\frac{B}{2}}(\hat{X}+i\hat{Y})\\
    &\hat{Z}^{\dagger}=\sqrt{\frac{B}{2}}(\hat{X}-i\hat{Y})
\end{align}
Condition \eqref{X eigen} states that the states $\{\ket{X,\lambda}\}$ can be viewed as a "coordinate" basis for the canonical variables $\hat{X}$ and $\hat{Y}$. Conditions \eqref{lambda eigen} and \eqref{complete} state that $\ket{X,\lambda}$ are coherent states for the Heisenberg algebra generated by $\hat{\lambda}^{i}, \hat{\lambda}^{\dagger}_{i}$.

In the following, we examine the ground state wave function of the matrix model. The ground state wave function is expressed as follows\cite{Karabali-Sakita}:
\begin{align}
    \Psi_{gr}(X,\Bar{\lambda})=\bra{X,\lambda}\ket{ground}=(\sqrt{2B})^{lN(N-1)/2}[\epsilon^{i_{1}i_{2}\cdots i_{N}}\Bar{\lambda}_{i_{1}}(\Bar{\lambda}X)_{i_{2}}(\Bar{\lambda}X^2)_{i_{3}}\cdots (\Bar{\lambda}X^{N-1})_{i_{N}}]^{k}e^{-\frac{B}{2}\mathrm{Tr}(X^2)}
\end{align}

Using the completeness condition (\ref{complete}), we express the norm square of the normalized ground state as:
\begin{align}\label{gensemble}
    1=\widetilde{\bra{ground}}\widetilde{\ket{ground}}=\int [dX]d\lambda d\lambda^{\dagger}\widetilde{\bra{ground}}\ket{X,\lambda}\bra{X,\lambda}\widetilde{\ket{ground}}
\end{align}
where $[dX]=b_{N}\prod_{i=1}^{N}dX_{ii}\prod_{i<j}d\: \mathrm{Re}(X_{ij})\prod_{i<j}d \: \mathrm{Im}(X_{ij})$. $b_{N}$ is a normalization constant. By integrating out  $\lambda, \lambda^{\dagger}$ in \eqref{gensemble}, we obtain a probability measure $Q_{g}(X)[dX]$ in the space of $N\cross N$ Hermitian matrices:
\begin{align}\label{def of QdX}
    \int [dX]d\lambda d\lambda^{\dagger}\widetilde{\bra{ground}}\ket{X,\lambda}\bra{X,\lambda}\widetilde{\ket{ground}}=\int [dX] Q_{g}(X)
\end{align}
We then apply the standard moment method in random matrix theory\cite{anderson2010introduction, mingo2017free} along with the large $N$ limit of $t_{n,n}=\mathrm{Tr}(\mathrm{Sym}(Z^n{Z^{\dagger}}^n))$ in our main theorem \ref{thm: large N limit}:
\begin{align}
    t_{n,n}=\frac{1}{n+1}(k+1)^{n}\sqrt{N}^{2n+2}+
  O(\sqrt{N}^{2n})
\end{align}
to derive a rescaled semicircle law for the random matrix $X$ in the probability measure $Q_{g}(X)[dX]$.

We introduce a rescaled variable:
\begin{align}
    \tilde{X}=\sqrt{\frac{2B}{(k+1)N}}X.
\end{align}
Let $\tilde{x}_{1}\leq \tilde{x}_{2} \leq \cdots \leq \tilde{x}_{N}$ be the eigenvalues of the rescaled Hermitian matrix $\tilde{X}$. We define the spectral density of the eigenvalues $\{\tilde{x}_{i}, 1\leq i \leq N\}$ of $\tilde{X}$ on a real line $\mathbb{R}$:
$$\mu_{N}(\tilde{x})=\frac{1}{N}\sum_{i=1}^{N}\delta(\tilde{x}-\tilde{x}_{i}).$$
The resolvent $\frac{1}{N}\mathrm{tr}(\frac{1}{w-\tilde{X}})$ of the random matrix $\tilde{X}$ is equal to the Stieltjes transformation $S_{\mu_{N}}(w)$ of the density $\mu_{N}$:
    $$S_{\mu_{N}}(w)=\int_{\mathbb{R}}\frac{1}{w-\tilde{x}}d\mu_{N}(\tilde{x})=\frac{1}{N}\sum_{i=1}^{N}\frac{1}{w-\tilde{x}_{i}}=\frac{1}{N}\Tr(\frac{1}{w-\tilde{X}}).$$
We then compute the large $N$ limit of the expectation value of the resolvent with respect to the probability measure $Q_{g}(X)[dX]$ defined earlier:
\begin{align*}
    &E_{g}(\frac{1}{N}\Tr (\frac{1}{w-\tilde{X}}))=\int [dX] \frac{1}{N}\Tr (\frac{1}{w-\sqrt{\frac{2B}{(k+1)N}}X})Q_{g}(X)\\
    &=\int [dX]d\lambda d\lambda^{\dagger}\widetilde{\bra{ground}}\frac{1}{N}\Tr (\frac{1}{w-\sqrt{\frac{2B}{(k+1)N}}\hat{X}})\ket{X,\lambda}\bra{X,\lambda}\widetilde{\ket{ground}}\\
    &=\widetilde{\bra{ground}}\frac{1}{N}\Tr (\frac{1}{w-\sqrt{\frac{2B}{(k+1)N}}\hat{X}})\widetilde{\ket{ground}}.
\end{align*}
Here, we denote expectation values with respect to the probability measure $Q(X)_{g}[dX]$ as $E_{g}(\cdots)$.
Recalling $\hat{X}=\frac{1}{\sqrt{2B}}(\hat{Z}+\hat{Z^{\dagger})}$ and $F(u,v)=\frac{1}{1-(u\hat{Z}+v\hat{Z^{\dagger}})}=\sum_{m,n\geq 0}\binom{m+n}{m}t_{m,n}u^mv^n$, we apply the result $t_{n,n}=\frac{1}{n+1}(k+1)^{n}\sqrt{N}^{2n+2}+
  O(\sqrt{N}^{2n})$. This yields:
\begin{align*}
    &\widetilde{\bra{ground}}\frac{1}{N}\Tr\left(\frac{1}{w-\sqrt{\frac{2B}{(k+1)N}}\hat{X}}\right)\widetilde{\ket{ground}}=\frac{1}{Nw}\widetilde{\bra{ground}}F\left(\frac{1}{w\sqrt{(k+1)N}},\frac{1}{w\sqrt{(k+1)N}}\right)\widetilde{\ket{ground}}\\
    &=\frac{1}{Nw} \sum_{n=0}^{\infty}\binom{2n}{n}\widetilde{\bra{ground}}t_{n,n}\widetilde{\ket{ground}}(w\sqrt{(k+1)N})^{-2n}= \sum_{n=0}^{\infty}C_{n}w^{-2n-1}+O(N^{-1}).
\end{align*}
Here, $\{C_{n}=\frac{1}{n+1}\binom{2n}{n}, n=0,1,\cdots\}$ are the Catalan numbers. From these calculations, we obtain the large $N$ limit of the expectation value of $S_{\mu_{N}}(w)$:
\begin{align*}
    \lim_{N \to \infty}E_{g}(S_{\mu_{N}}(w))=\sum_{n=0}^{\infty}C_{n}w^{-2n-1}=\frac{w-\sqrt{w^{2}-4}}{2}  \quad w \in \mathbb{C}^{+}=\{z\in \mathbb{C}| \mathrm{Im}(z)>0\}.
\end{align*}
Here, $w$ is constrained to the upper half of the complex plane. The large $N$ limit of the expectation value of the density $\mu_{N}(\tilde{x}), \tilde{x}\in \mathbb{R}$ is obtained by the inverse Stieltjes transformation:
\begin{align*}
    \lim_{N \to \infty} E_{g}(\mu_{N}(\tilde{x}))=-\lim_{\epsilon \to 0^{+}}\mathrm{Im}\{E_{g}(S_{\mu_{N}}(\tilde{x}+i\epsilon))\}=\frac{1}{2\pi}\sqrt{4-\tilde{x}^2} \mathbb{1}_{[-2,2]}(\tilde{x}).
\end{align*}
Here, $\mathbb{1}_{[-2,2]}(\tilde{x})$ is the characteristic function on the interval $[-2,2]$. This result is the renowned Wigner semicircle law. Converting back to the original matrix $X=\sqrt{\frac{(k+1)N}{2B}}\tilde{X}$, we obtain the following renormalized Wigner semicircle law:
\begin{proposition}
    Let $Q_{g}(X)[dX]$ be the probability measure on the space of $N \cross N$ Hermitian matrices as defined by (\ref{gensemble}). Let ${x_{i}, 1\leq i \leq N}$ be the set of eigenvalues of $X$. The density of eigenvalues is given by $\rho_{N}(x)=\frac{1}{N}\sum_{i=1}^{N}\delta(x-x_{i})$. Then, in the large $N$ limit, the leading term of the expectation value of $\rho_{N}(x)$ satisfies a renormalized (rescaled) semicircle distribution:
    \begin{align}\label{Semi Circle law}
        E_{g}(\rho_{N}(x))=\frac{2}{\pi R^2}\sqrt{R^2-x^2}\mathbb{1}_{[-R,R]}(x)+\text{lower order terms}   \quad \text{as} \quad N \to \infty
    \end{align}
    where the radius is given by $R=\frac{2(k+1)N}{B}$.
\end{proposition}

\leftline{\textbf{Physical interpretation: $\frac{1}{k+1}$ is the filling factor of the quantum hall fluid}}

In the matrix model, the matrices $\hat{X}$ and $\hat{Y}$ can be interpreted as non-commutative coordinates of electrons moving on a plane subjected to an external constant magnetic field $B$ \cite{Susskind}\cite{Poly}. The harmonic potential $Tr(\hat{X}^2+\hat{Y}^2)$ confines the electrons within a disc centered around the origin of the plane. By diagonalizing the matrix $X=UxU^{-1}$, where $U \in U(N)$ and $x=diag(x_1,x_{2},\cdots,x_{N})$, the wave functions of the matrix model can be seen as functions of the variables $U$, $x=diag(x_1,x_{2},\cdots,x_{N})$, and $\Bar{\lambda}$:

\begin{align}
\Psi(x,U,\lambda)=\bra{X=UxU^{-1},\lambda}\ket{phys}
\end{align}

Integrating out the extra variables $U$ and $\Bar{\lambda}$ yields a resulting theory that describes a many-particle system, known as the Calogero many-particle system\cite{Karabali_2002}. The eigenvalues $(x_{1},x_{2},\cdots,x_{N})$ are considered as the positions of the Calogero particles. In the context of the quantum Hall effect, the $y$ coordinates of the particles correspond to the momenta of the eigenvalues $(x_{1},x_{2},\cdots,x_{N})$. This implies that the $y$ coordinates are spread out across the plane. The emergence of the semi-circle law distribution of the ground state wave function (\ref{Semi Circle law}) can be understood as arising from a circle law distribution from the perspective of the quantum Hall fluid. The following picture\footnote{This picture is inspired by a picture in \cite{Frenkel:2021yql}.} illustrates our physical interpretation:
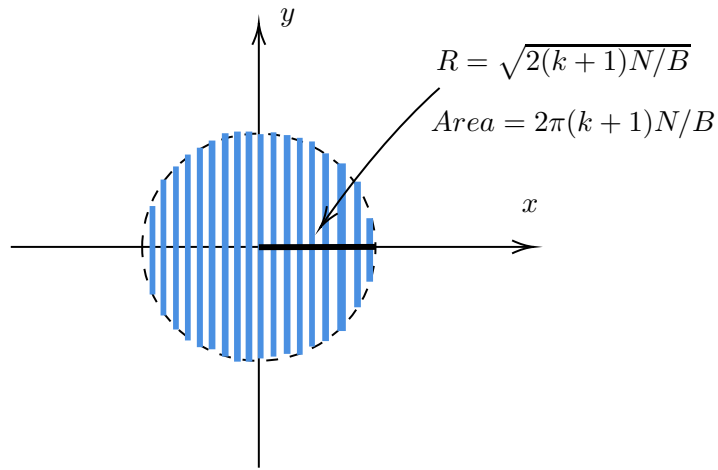
\begin{equation*}
    \hbox{

\tikzset{every picture/.style={line width=0.75pt}} 

\begin{tikzpicture}[x=0.75pt,y=0.75pt,yscale=-1,xscale=1]

\draw  [dash pattern={on 4.5pt off 4.5pt}] (109.26,131.22) .. controls (109.26,99.56) and (135.6,73.9) .. (168.1,73.9) .. controls (200.6,73.9) and (226.94,99.56) .. (226.94,131.22) .. controls (226.94,162.89) and (200.6,188.55) .. (168.1,188.55) .. controls (135.6,188.55) and (109.26,162.89) .. (109.26,131.22) -- cycle ;
\draw    (43,131.07) -- (304.77,131.07) ;
\draw [shift={(306.77,131.07)}, rotate = 180] [color={rgb, 255:red, 0; green, 0; blue, 0 }  ][line width=0.75]    (10.93,-3.29) .. controls (6.95,-1.4) and (3.31,-0.3) .. (0,0) .. controls (3.31,0.3) and (6.95,1.4) .. (10.93,3.29)   ;
\draw    (168.1,242.5) -- (168.1,20.1) ;
\draw [shift={(168.1,18.1)}, rotate = 90] [color={rgb, 255:red, 0; green, 0; blue, 0 }  ][line width=0.75]    (10.93,-3.29) .. controls (6.95,-1.4) and (3.31,-0.3) .. (0,0) .. controls (3.31,0.3) and (6.95,1.4) .. (10.93,3.29)   ;
\draw  [color={rgb, 255:red, 74; green, 144; blue, 226 }  ,draw opacity=1 ][fill={rgb, 255:red, 74; green, 144; blue, 226 }  ,fill opacity=1 ] (156.11,73.36) -- (158.48,73.36) -- (158.48,188.48) -- (156.11,188.48) -- cycle ;
\draw  [color={rgb, 255:red, 74; green, 144; blue, 226 }  ,draw opacity=1 ][fill={rgb, 255:red, 74; green, 144; blue, 226 }  ,fill opacity=1 ] (115.41,110.94) -- (113.52,110.94) -- (113.52,154.59) -- (115.41,154.59) -- cycle ;
\draw  [color={rgb, 255:red, 74; green, 144; blue, 226 }  ,draw opacity=1 ][fill={rgb, 255:red, 74; green, 144; blue, 226 }  ,fill opacity=1 ] (120.93,97.57) -- (119.04,97.57) -- (119.04,165.19) -- (120.93,165.19) -- cycle ;
\draw  [color={rgb, 255:red, 74; green, 144; blue, 226 }  ,draw opacity=1 ][fill={rgb, 255:red, 74; green, 144; blue, 226 }  ,fill opacity=1 ] (125.35,90.96) -- (127.24,90.96) -- (127.24,172.26) -- (125.35,172.26) -- cycle ;
\draw  [color={rgb, 255:red, 74; green, 144; blue, 226 }  ,draw opacity=1 ][fill={rgb, 255:red, 74; green, 144; blue, 226 }  ,fill opacity=1 ] (131.34,84.96) -- (133.24,84.96) -- (133.24,177.79) -- (131.34,177.79) -- cycle ;
\draw  [color={rgb, 255:red, 74; green, 144; blue, 226 }  ,draw opacity=1 ][fill={rgb, 255:red, 74; green, 144; blue, 226 }  ,fill opacity=1 ] (143.33,77.89) -- (145.7,77.89) -- (145.7,182.25) -- (143.33,182.25) -- cycle ;
\draw  [color={rgb, 255:red, 74; green, 144; blue, 226 }  ,draw opacity=1 ][fill={rgb, 255:red, 74; green, 144; blue, 226 }  ,fill opacity=1 ] (200.6,84.35) -- (202.96,84.35) -- (202.96,178.1) -- (200.6,178.1) -- cycle ;
\draw  [color={rgb, 255:red, 74; green, 144; blue, 226 }  ,draw opacity=1 ][fill={rgb, 255:red, 74; green, 144; blue, 226 }  ,fill opacity=1 ] (208.17,89.42) -- (211.48,89.42) -- (211.48,173.03) -- (208.17,173.03) -- cycle ;
\draw  [color={rgb, 255:red, 74; green, 144; blue, 226 }  ,draw opacity=1 ][fill={rgb, 255:red, 74; green, 144; blue, 226 }  ,fill opacity=1 ] (219.06,98.64) -- (216.69,98.64) -- (216.69,161.04) -- (219.06,161.04) -- cycle ;
\draw  [color={rgb, 255:red, 74; green, 144; blue, 226 }  ,draw opacity=1 ][fill={rgb, 255:red, 74; green, 144; blue, 226 }  ,fill opacity=1 ] (222.84,117.08) -- (225.21,117.08) -- (225.21,148.13) -- (222.84,148.13) -- cycle ;
\draw  [color={rgb, 255:red, 74; green, 144; blue, 226 }  ,draw opacity=1 ][fill={rgb, 255:red, 74; green, 144; blue, 226 }  ,fill opacity=1 ] (168.1,74.66) -- (169.99,74.66) -- (169.99,186.86) -- (168.1,186.86) -- cycle ;
\draw  [color={rgb, 255:red, 74; green, 144; blue, 226 }  ,draw opacity=1 ][fill={rgb, 255:red, 74; green, 144; blue, 226 }  ,fill opacity=1 ] (149.96,74.2) -- (152.33,74.2) -- (152.33,185.94) -- (149.96,185.94) -- cycle ;
\draw  [color={rgb, 255:red, 74; green, 144; blue, 226 }  ,draw opacity=1 ][fill={rgb, 255:red, 74; green, 144; blue, 226 }  ,fill opacity=1 ] (174.57,73.74) -- (176.46,73.74) -- (176.46,185.94) -- (174.57,185.94) -- cycle ;
\draw  [color={rgb, 255:red, 74; green, 144; blue, 226 }  ,draw opacity=1 ][fill={rgb, 255:red, 74; green, 144; blue, 226 }  ,fill opacity=1 ] (181.19,75.13) -- (183.56,75.13) -- (183.56,184.56) -- (181.19,184.56) -- cycle ;
\draw  [color={rgb, 255:red, 74; green, 144; blue, 226 }  ,draw opacity=1 ][fill={rgb, 255:red, 74; green, 144; blue, 226 }  ,fill opacity=1 ] (187.82,76.51) -- (189.71,76.51) -- (189.71,185.02) -- (187.82,185.02) -- cycle ;
\draw  [color={rgb, 255:red, 74; green, 144; blue, 226 }  ,draw opacity=1 ][fill={rgb, 255:red, 74; green, 144; blue, 226 }  ,fill opacity=1 ] (193.97,78.35) -- (195.87,78.35) -- (195.87,180.87) -- (193.97,180.87) -- cycle ;
\draw    (259.28,50.84) .. controls (237.95,69.82) and (216.16,96.33) .. (200,119.95) ;
\draw [shift={(199.02,121.39)}, rotate = 304.13] [color={rgb, 255:red, 0; green, 0; blue, 0 }  ][line width=0.75]    (10.93,-3.29) .. controls (6.95,-1.4) and (3.31,-0.3) .. (0,0) .. controls (3.31,0.3) and (6.95,1.4) .. (10.93,3.29)   ;
\draw [line width=2.25]    (168.1,131.22) -- (226.94,131.07) ;
\draw  [color={rgb, 255:red, 74; green, 144; blue, 226 }  ,draw opacity=1 ][fill={rgb, 255:red, 74; green, 144; blue, 226 }  ,fill opacity=1 ] (137.34,81.5) -- (139.39,81.5) -- (139.39,181.18) -- (137.34,181.18) -- cycle ;
\draw  [color={rgb, 255:red, 74; green, 144; blue, 226 }  ,draw opacity=1 ][fill={rgb, 255:red, 74; green, 144; blue, 226 }  ,fill opacity=1 ] (161.95,73.36) -- (164.31,73.36) -- (164.31,188.48) -- (161.95,188.48) -- cycle ;

\draw (299.22,105.68) node [anchor=north west][inner sep=0.75pt]    {$x$};
\draw (177.43,9.16) node [anchor=north west][inner sep=0.75pt]    {$y$};
\draw (255.99,26.91) node [anchor=north west][inner sep=0.75pt]    {$R=\sqrt{2( k+1) N/B}$};
\draw (253.36,60.14) node [anchor=north west][inner sep=0.75pt]    {$Area=2\pi ( k+1) N/B\ $};

\end{tikzpicture}

    }
\end{equation*}

Let us compute the filling factor. From the rescaled semicircle law (\ref{Semi Circle law}) distribution, the radius of the circle is $R=\sqrt{2(k+1)N/B}$. The area of the circle is $Area=2\pi (k+1)N/B$. The number of magnetic flux that passes through the droplet is
\begin{align}
    N_{B}=Area \cross \frac{B}{2\pi}=(k+1)N.
\end{align}
Here, the charge of elecrons $e$ and the planck constant $\hbar$ are set to one. From these we obtain the filling factor as the ratio of the number of electrons and the number of magnetic flux that passes through the droplet:
\begin{align}
    \nu=\frac{N}{N_{B}}=\frac{1}{k+1}.
\end{align}
Let us calculate the filling factor. Using the rescaled semicircle law (C.11), we find that the radius of the circle is \( R = \sqrt{\frac{2(k + 1)N}{B}} \). The area of the circle is \( \text{Area} = \frac{2\pi(k + 1)N}{B} \). The number of magnetic flux quanta passing through the droplet is given by:
\begin{align*}
    N_{B}=Area \cross \frac{B}{2\pi}=(k+1)N.
\end{align*}
Here, we have set the charge of electrons $e$ and the Planck constant $\hbar$ to one. With this information, we can compute the filling factor as the ratio of the number of electrons to the number of magnetic flux quanta passing through the droplet:
\begin{align}
     \nu = \frac{N}{N_{B}} = \frac{1}{k + 1}.
\end{align}
Recall in in Section \ref{subsec: large N limit representation}, we show that the central charge of the large $N$ limit algebra $\widehat{\mathfrak{u}(1)}_{\frac{1}{k+1}}$ is expressed as:
\begin{align*}
    c=\lim_{N\rightarrow \infty}(n+1)\tilde{t}_{n,n}=\frac{1}{k+1}.
\end{align*}
Here, $\tilde{t}_{n,n}=\sqrt{(k+1)N}^{-(2n+2)} t_{n,n} $.
Additionally, recall that our derivation of the rescaled semicircle law (\ref{Semi Circle law}) relies on the result $\lim_{N\rightarrow \infty}(n+1)\tilde{t}_{n,n}=\frac{1}{k+1}$ . Therefore, we justify our claim that the central charge of the large $N$ limit algebra is equal to the filling factor of the corresponding quantum Hall droplet.

This result is consistent with previous works on the matrix model concerning the quantum Hall effect \cite{Susskind}\cite{Poly}\cite{Karabali_2002}, providing validation for the scaling factor chosen in this paper (See (\ref{rescaling})(\ref{Basis_rescaled})). Additionally, we note the work \cite{Frenkel:2021yql}, where the authors also derived a semicircle law using saddle point approximation, while our derivation of this semicircle law relies on the moment method and the large $N$ limit of the operators $t_{n,n}, n\geq0$. The radius of the semicircle law (\ref{Semi Circle law}) agrees with the one in \cite{Frenkel:2021yql}.\footnote{The parameter $k$ in this paper corresponds to $k-1$ in \cite{Frenkel:2021yql}.}.\\ 

\leftline{\textbf{The case $p>1$: $\frac{p}{k+p}$ is the filling factor of the quantum Hall fluid }}
In the case $p>1$, we consider a special situation: $N$ is divided by $p$. Let $N=pM$. In this situation  (see \ref{3.12}), the ground state is unique (i.e., there is no degeneracy):
\begin{align}  \ket{\mathrm{ground};N=pM}=\prod_{r=1}^{k}[\epsilon^{i_{1}i_{2}\cdots i_{N}}B(0)_{i_{1}\cdots i_{p}}B(1)_{i_{p+1}\cdots i_{2p}}\cdots B(M-1)_{i_{N-p+1}\cdots i_{N}}]^{k}\ket{0}
\end{align}
where $$B(n)_{i_{1}i_{2}\cdots i_{p}}=\sum_{1\leq a_{1},a_{2},\cdots a_{p} \leq p} \epsilon_{a_{1}a_{2}\cdots a_{p}}({\lambda^{\dagger}}^{a_{1}}{Z^{\dagger}}^n)_{i_{1}} ({\lambda^{\dagger}}^{a_{2}}{Z^{\dagger}}^n)_{i_{2}} \cdots ({\lambda^{\dagger}}^{a_{p}}{Z^{\dagger}}^n)_{i_{p}}$$

Let $\widetilde{\ket{\mathrm{ground};N=pM}}$ be a normalized ground state:
$$\widetilde{\bra{\mathrm{ground};N=pM}}\widetilde{\ket{\mathrm{ground};N=pM}}=1$$

Before we show the main proposition \ref{Pro C2}, we need a lemma:
\begin{lemma}\label{lemma C}
    \begin{align}\label{t12 ground}
t_{1,2}\ket{\mathrm{ground};N=pM}=(k(M-1)+N)\Tr (Z^{\dagger})\ket{\mathrm{ground};N=pM}.
    \end{align}
\end{lemma}

\begin{proof}
    Recall $t_{1,2}=\frac{1}{3}(\Tr (Z^{\dagger}ZZ^{\dagger})+\Tr (Z^{\dagger}Z^{\dagger}Z)+\Tr (ZZ^{\dagger}Z^{\dagger}))$. When computing $t_{1,2}\ket{\mathrm{ground};N=pM}$, the element $Z$ in $t_{1,2}$ will contract with:
    \begin{enumerate}
        \item $Z^{\dagger}$'s in $t_{1,2}$. These contractions yield the term $N\Tr (Z^{\dagger})\ket{\mathrm{ground};N=pM}$ on the right-hand side of \eqref{t12 ground}.
        \item $Z^{\dagger}$'s in $\prod_{r=1}^{k}[\epsilon^{i_{1}i_{2}\cdots i_{N}}B(0)_{i_{1}\cdots i_{p}}B(1)_{i_{p+1}\cdots i_{2p}}\cdots B(M-1)_{i_{N-p+1}\cdots i_{N}}]^{k}$. Using the elementary fact that \begin{align*}
    C(\{a\},\{n\})=\epsilon^{i_{1}i_{2}\cdots i_{N}}({\lambda^{\dagger}}^{a_{1}}{Z^{\dagger}}^{n_{1}})_{i_{1}}({\lambda^{\dagger}}^{a_{2}}{Z^{\dagger}}^{n_{2}})_{i_{2}}\cdots ({\lambda^{\dagger}}^{a_{N}}{Z^{\dagger}}^{n_{N}})_{i_{N}}=0
\end{align*}
if there exist $1<l_{1}<l_{2}<\cdots<l_{p+1}<N$ such that $n_{l_{1}}=n_{l_{2}}=\cdots=n_{l_{p+1}}$. It follows that the Wick contractions of $Z$ in $t_{1,2}$ with $Z^{\dagger}$'s in $B(0)_{i_{1}\cdots i_{p}}B(1)_{i_{p+1}\cdots i_{2p}}\cdots B(M-2)_{i_{N-2p+1}\cdots i_{N-p}}$ will contribute zero. The Wick contractions of $Z$ with $Z^{\dagger}$'s in $B(M-1)_{i_{N-p+1}\cdots i_{N}}$ give:
\begin{align*}
   & k(M-1)\{\epsilon^{i_{1}i_{2}\cdots i_{N}}B(0)_{i_{1}\cdots i_{p}}B(1)_{i_{p+1}\cdots i_{2p}}\cdots B(M-2)_{i_{N-2p+1}\cdots i_{N-p}}B(M)_{i_{N-p+1}\cdots i_{N}}\cdot \\
   &\prod_{r=1}^{k-1}[\epsilon^{i_{1}i_{2}\cdots i_{N}}B(0)_{i_{1}\cdots i_{p}}B(1)_{i_{p+1}\cdots i_{2p}}\cdots B(M-1)_{i_{N-p+1}\cdots i_{N}}]^{k-1}\}\ket{0}\\
   &=k(M-1)\Tr (Z^{\dagger})\prod_{r=1}^{k}[\epsilon^{i_{1}i_{2}\cdots i_{N}}B(0)_{i_{1}\cdots i_{p}}B(1)_{i_{p+1}\cdots i_{2p}}\cdots B(M-1)_{i_{N-p+1}\cdots i_{N}}]^{k}\ket{0}.
\end{align*}
Here, we have used:
\begin{align*}
    &\epsilon^{i_{1}i_{2}\cdots i_{N}}B(0)_{i_{1}\cdots i_{p}}B(1)_{i_{p+1}\cdots i_{2p}}\cdots B(M-2)_{i_{N-2p+1}\cdots i_{N-p}}B(M)_{i_{N-p+1}\cdots i_{N}}\\
    &=\Tr (Z^{\dagger})\epsilon^{i_{1}i_{2}\cdots i_{N}}B(0)_{i_{1}\cdots i_{p}}B(1)_{i_{p+1}\cdots i_{2p}}\cdots B(M-1)_{i_{N-p+1}\cdots i_{N}},
\end{align*}
which follows from a $p>1$ generalization of relation \eqref{ti C} (as a corollary of \eqref{bubble}).
\end{enumerate}

\end{proof}

The following proposition serves as evidence supporting Conjecture \ref{conjecture 4.1}.
\begin{proposition}\label{Pro C2}
    Assume in the limit: $N=pM$, $M \to \infty$, the operators $\{t_{n,n}, n=0,1,2\cdots \}$ have leading order $\sqrt{N}^{n+m+2\delta_{n,m}}$, i.e
    $$t_{n,n} \sim O(\sqrt{N}^{n+m+2\delta_{n,m}}).$$

Then we have: 
\begin{align}\label{A t21tnn+1}
\widetilde{\bra{\mathrm{ground};N=pM}}t_{n,n}\widetilde{\ket{\mathrm{ground};N=pM}}=\frac{1}{n+1}\left(\frac{k+p}{p}\right)^{n}N^{n+1}+O(N^{n})  ,\quad N \to \infty.
\end{align}
\end{proposition}
\begin{proof}
Assume $$t_{n,n} \sim O(\sqrt{N}^{n+m+2\delta_{n,m}}),$$ then Proposition \ref{prop: filtration} implies that we have the following large $N$ commutation relations
   $$[t_{2,1},t_{n,n+1}]=(n+2)t_{n+1,n+1}+O(N^{n}).$$
Let us denote $$d(n):=\widetilde{\bra{\mathrm{ground};N=pM}}t_{n,n}\widetilde{\ket{\mathrm{ground};N=pM}}.$$
The ground state expectation value of \ref{A t21tnn+1} is  
\begin{align}\label{ground t21tnn+1}
    \widetilde{\bra{\mathrm{ground};N=pM}}[t_{2,1},t_{n,n+1}]\widetilde{\ket{\mathrm{ground};N=pM}}=(n+2)d(n+1)+O(N^{n}).
\end{align}
It follows from Lemma \ref{lemma C} that
\begin{align}\label{t12g}
t_{1,2}\widetilde{\ket{\mathrm{ground};N=pM}}=(k(M-1)+N)\Tr (Z^{\dagger})\widetilde{\ket{\mathrm{ground};N=pM}}.
\end{align}
Taking the adjoint of the above and we get
\begin{align}\label{ground t21}
\widetilde{\bra{\mathrm{ground};N=pM}}t_{2,1}=(k(M-1)+N) \widetilde{\bra{\mathrm{ground};N=pM}}\Tr (Z).
\end{align}
Using \eqref{ground t21} and $t_{2,1}\widetilde{\ket{\mathrm{ground};N=pM}}=0$, \eqref{ground t21tnn+1} implies that
\begin{align}
    (k(M-1)+N)\widetilde{\bra{\mathrm{ground};N=pM}}t_{1,0}t_{n,n+1}\widetilde{\ket{\mathrm{ground};N=pM}}=(n+1)d(n+1)+O(N^{n}).
\end{align}
The right-hand-side of the above equation equals to
\begin{align*}
    &(k(M-1)+N)\widetilde{\bra{\mathrm{ground};N=pM}}t_{1,0}t_{n,n+1}\widetilde{\ket{\mathrm{ground};N=pM}}\\
    &=(k(M-1)+N)\widetilde{\bra{\mathrm{ground};N=pM}}[t_{1,0},t_{n,n+1}]\widetilde{\ket{\mathrm{ground};N=pM}}\\
    &=(k(M-1)+N)(n+1)\widetilde{\bra{\mathrm{ground};N=pM}}t_{n,n}\widetilde{\ket{\mathrm{ground};N=pM}}\\
    &=(k(M-1)+N)(n+1)d(n).
\end{align*}
Altogether, we have 
\begin{align}
    (k(M-1)+N)(n+1)d(n)=(n+2)d(n+1)+O(N^{n}).
\end{align}
Recall $M=\frac{N}{p}$. We obtain a recursion relation 
\begin{align}
    \frac{k+p}{p}\frac{(n+1)d(n)}{N^{n+1}}=\frac{(n+2)d(n+1)}{N^{n+2}}+O(\frac{1}{N}).
\end{align}
Together with the initial condition $d(0)=N$, we have
\begin{align}
d(n)=\widetilde{\bra{\mathrm{ground};N=pM}}t_{n,n}\widetilde{\ket{\mathrm{ground};N=pM}}=\frac{1}{n+1}\left(\frac{k+p}{p}\right)^{n}N^{n+1}+O(N^{n}).
\end{align}
\end{proof}
Applying exactly the same arguments as in the case $p=1$, we see that the radius of the quantum Hall droplet is
$$R=\sqrt{2\frac{k+p}{p}\frac{N}{B}},$$ hence the filling factor is
$$\nu=\frac{p}{k+p}.$$
On the other hand, the same factor is the $\mathfrak{u}(1)$ central charge of the large $N$ current algebra:
$$c=\lim_{N \to \infty}(n+1)\sqrt{\frac{k+p}{p}N}^{-(2n+2)} t_{n,n}=\frac{p}{k+p}.$$




\bibliographystyle{unsrt}
\bibliography{NCCS-Bib}

\begin{thebibliography}{10}

\bibitem{Susskind}
Leonard Susskind.
\newblock {The quantum Hall fluid and non-commutative Chern-Simons theory}.
\newblock {\em arXiv preprint hep-th/0101029}, 2001.

\bibitem{Tong_2015}
David Tong and Carl Turner.
\newblock {Quantum Hall effect in supersymmetric Chern-Simons theories}.
\newblock {\em Physical Review B}, 92(23), 2015.

\bibitem{Poly}
Alexios~P Polychronakos.
\newblock {Quantum Hall states as matrix Chern-Simons theory}.
\newblock {\em Journal of High Energy Physics}, 2001(04):011, 2001.

\bibitem{Hellerman-Raamsdonk}
Simeon Hellerman and Mark Van~Raamsdonk.
\newblock {Quantum Hall physics = noncommutative field theory}.
\newblock {\em Journal of High Energy Physics}, 2001(10):039, 2001.

\bibitem{Karabali-Sakita}
Dimitra Karabali and Bunji Sakita.
\newblock {Chern-Simons matrix model: coherent states and relation to Laughlin
  wave functions}.
\newblock {\em Physical Review B}, 64(24):245316, 2001.

\bibitem{Dorey-Tong-Turner}
Nick Dorey, David Tong, and Carl Turner.
\newblock {Matrix model for non-Abelian quantum Hall states}.
\newblock {\em Physical Review B}, 94(8):085114, 2016.

\bibitem{Polychronakos_2006}
Alexios~P Polychronakos.
\newblock {The physics and mathematics of Calogero particles}.
\newblock {\em Journal of Physics A: Mathematical and General},
  39(41):12793--12845, sep 2006.

\bibitem{Karabali_2002}
Dimitra Karabali and B.~Sakita.
\newblock {Orthogonal basis for the energy eigenfunctions of the Chern-Simons
  matrix model}.
\newblock {\em Physical Review B}, 65(7), jan 2002.

\bibitem{Dorey_Tong_Turner-Matrix}
Nick Dorey, David Tong, and Carl Turner.
\newblock {A matrix model for WZW}.
\newblock {\em Journal of High Energy Physics}, 2016(8):1--31, 2016.

\bibitem{gan2006almost}
Wee~Liang Gan and Victor Ginzburg.
\newblock {Almost-commuting variety, D-modules, and Cherednik algebras}.
\newblock {\em International Mathematics Research Papers}, 2006:26439, 2006.

\bibitem{losev2012isomorphisms}
Ivan Losev.
\newblock Isomorphisms of quantizations via quantization of resolutions.
\newblock {\em Advances in Mathematics}, 231(3-4):1216--1270, 2012.

\bibitem{ginzburg2009lectures}
Victor Ginzburg.
\newblock {Lectures on Nakajima's quiver varieties}.
\newblock {\em arXiv:0905.0686}, 2009.

\bibitem{Gaiotto-Rapcek-Zhou}
Davide Gaiotto, Miroslav Rapčák, and Yehao Zhou.
\newblock {Deformed Double Current Algebras, Matrix Extended $\mathcal
  W_{\infty}$ Algebras, Coproducts, and Intertwiners from the M2-M5
  Intersection}.
\newblock {\em arXiv:2309.16929}, 2023.

\bibitem{guay2007affine}
Nicolas Guay.
\newblock {Affine Yangians and deformed double current algebras in type A}.
\newblock {\em Advances in Mathematics}, 211(2):436--484, 2007.

\bibitem{Costello}
Kevin Costello.
\newblock {Holography and Koszul duality: the example of the $ M2 $ brane}.
\newblock {\em arXiv:1705.02500}, 2017.

\bibitem{gaiotto2019aspects}
Davide Gaiotto and Jihwan Oh.
\newblock {Aspects of $\Omega$-deformed M-theory}.
\newblock {\em arXiv preprint:1907.06495}, 2019.

\bibitem{Hu-Li-Ye-Zhou}
Sen Hu, Si~Li, Dongheng Ye, and Yehao Zhou.
\newblock {\em To appear}, 2024.

\bibitem{Rodriguez_2009}
Ivan~D Rodriguez.
\newblock {Edge excitations of the Chern-Simons matrix theory for the {FQHE}}.
\newblock {\em Journal of High Energy Physics}, 2009(07):100--100, 2009.

\bibitem{Frenkel:2021yql}
Alexander Frenkel and Sean~A. Hartnoll.
\newblock {Entanglement in the Quantum Hall Matrix Model}.
\newblock {\em JHEP}, 05:130, 2022.

\bibitem{Witten}
Edward Witten.
\newblock {Quantum field theory and the Jones polynomials}.
\newblock {\em Comm. Math. Phys.}, 121(3):351--399, 1989.

\bibitem{okazaki2018matrix}
Tadashi Okazaki and Douglas~J Smith.
\newblock {Matrix supergroup Chern-Simons models for vortex-antivortex
  systems}.
\newblock {\em Journal of High Energy Physics}, 2018(2):1--53, 2018.

\bibitem{barns2018adhm}
Alec Barns-Graham, Nick Dorey, Nakarin Lohitsiri, David Tong, and Carl Turner.
\newblock {ADHM and the 4d quantum Hall effect}.
\newblock {\em Journal of High Energy Physics}, 2018(4):1--43, 2018.

\bibitem{banks1999m}
Tom Banks, Willy Fischler, Steven~H Shenker, and Leonard Susskind.
\newblock {M theory as a matrix model: A conjecture}.
\newblock In {\em The World in Eleven Dimensions}, pages 435--451. CRC Press,
  1999.

\bibitem{ishibashi1997large}
Nobuyuki Ishibashi, Hikaru Kawai, Yoshihisa Kitazawa, and Asato Tsuchiya.
\newblock {A large-N reduced model as superstring}.
\newblock {\em Nuclear Physics B}, 498(1-2):467--491, 1997.

\bibitem{stanley2023enumerative}
Richard Stanley.
\newblock {\em Enumerative Combinatorics: Volume 2}.
\newblock CAMBRIDGE University Press, 2023.

\bibitem{mendes2019combinatorics}
Anthony Mendes.
\newblock The combinatorics of rim hook tableaux.
\newblock {\em Australas. J Comb.}, 73:132--148, 2019.

\bibitem{anderson2010introduction}
Greg~W Anderson, Alice Guionnet, and Ofer Zeitouni.
\newblock {\em An introduction to random matrices}.
\newblock Number 118. Cambridge university press, 2010.

\bibitem{mingo2017free}
James~A Mingo and Roland Speicher.
\newblock {\em Free probability and random matrices}, volume~35.
\newblock Springer, 2017.

\end{thebibliography}

\end{document}